\documentclass[10pt,letterpaper,oneside,british]{amsart}

\usepackage[T1]{fontenc}
\usepackage[latin9]{inputenc}
\usepackage{color}
\definecolor{note_fontcolor}{rgb}{0.800781, 0.800781, 0.800781}
\usepackage{babel}
\usepackage{textcomp}
\usepackage{mathrsfs}
\usepackage{mathtools}
\usepackage{enumitem}
\usepackage{amstext}
\usepackage{amsthm}
\usepackage{amssymb}
\usepackage{stmaryrd}
\usepackage{graphicx}
\usepackage{esint}
\usepackage[unicode=true,pdfusetitle,
 bookmarks=true,bookmarksnumbered=false,bookmarksopen=false,
 breaklinks=true,pdfborder={0 0 0},pdfborderstyle={},backref=page,colorlinks=true]
 {hyperref}

\makeatletter

\pdfpageheight\paperheight
\pdfpagewidth\paperwidth

\newcommand{\noun}[1]{\textsc{#1}}
\providecommand{\tabularnewline}{\\}
\newenvironment{lyxgreyedout}
  {\textcolor{note_fontcolor}\bgroup\ignorespaces}
  {\ignorespacesafterend\egroup}

\numberwithin{equation}{section}
\numberwithin{figure}{section}
\theoremstyle{plain}
\newtheorem{thm}{\protect\theoremname}
\theoremstyle{definition}
\newtheorem{example}[thm]{\protect\examplename}
\theoremstyle{definition}
\newtheorem{defn}[thm]{\protect\definitionname}
\theoremstyle{plain}
\newtheorem{lem}[thm]{\protect\lemmaname}
\theoremstyle{plain}
\newtheorem{prop}[thm]{\protect\propositionname}
\theoremstyle{remark}
\newtheorem{rem}[thm]{\protect\remarkname}
\theoremstyle{plain}
\newtheorem{cor}[thm]{\protect\corollaryname}
\theoremstyle{plain}
\newtheorem{conjecture}[thm]{\protect\conjecturename}

\DeclareMathOperator{\var}{Var}
\DeclareMathOperator{\sgn}{sgn}

\makeatother

\providecommand{\conjecturename}{Conjecture}
\providecommand{\corollaryname}{Corollary}
\providecommand{\definitionname}{Definition}
\providecommand{\examplename}{Example}
\providecommand{\lemmaname}{Lemma}
\providecommand{\propositionname}{Proposition}
\providecommand{\remarkname}{Remark}
\providecommand{\theoremname}{Theorem}

\begin{document}
\title{Diophantine Approximation of Anergodic Birkhoff Sums over Rotations}
\author{Paul Verschueren\\
(paul@verschueren.org.uk)}
\thanks{The author wishes to thank Professor Sebastian van Strien of Imperial
College, without whose encouragement and support this paper would
never have seen publication, and also the College itself for providing
access to the literature.\\
The author also wishes to thank Ben Mestel of the Open University
for introducing him to the subject, for his infectious enthusiasm,
and for many wonderful discussions. \\
This work was partly supported by ERC AdG RGDD No 339523}
\email{paul@verschueren.org.uk}
\date{Started: 05/2020}
\date{This version: 03/2023}
\begin{abstract}

We study Birkhoff sums over rotations (series of the form $\sum_{r=1}^{N}\phi(r\alpha)$),
in which the summed function $\phi$ may be unbounded at the origin.
Estimates of these sums have been of significant interest and application
in pure mathematics since the late 1890s, but in recent years they
have also appeared in numerous areas of applied mathematics, and have
enjoyed significant renewed interest. Functions which have been intensively
studied include the reciprocals of number theoretical functions such
as $\phi(x)=1/\{x\},1/\{\{x\}\},1/\left\Vert x\right\Vert $  and
trigonometric functions such as $\phi(x)=\cot\pi x$ or $\left|\csc\pi x\right|$.
Classically the Birkhoff sum of each function has been studied in
relative isolation using function specific tools, and the results
have frequently been restricted to Bachmann-Landau estimates. We
introduce here a more general unified theory which is applicable to
all of the above functions. The theory uses only ``elementary''
tools (no tools of complex analysis), is capable of giving ``effective''
results (explicit bounds), and generally matches or improves on previously
available results. 

\end{abstract}

\maketitle
\tableofcontents{}

\section{Introduction}

We study estimates for series of the form $\sum_{r=1}^{N}\phi(r\alpha)$
where $\phi$ is a real function of period 1, and $\alpha$ is an
irrational real number. Such series can be classified as Birkhoff
sums over rotations of the circle (with initial condition $0$). When
$\phi$ is Lebesgue integrable, a first order estimate is easily available
via the powerful theorems of Ergodic Theory. In particular Birkhoff's
Ergodic Theorem gives us $\frac{1}{N}\sum_{r=1}^{N}\phi(r\alpha)\rightarrow\int\phi$.
Series which cannot be estimated in this way we call anergodic Birkhoff
sums. In this paper we study the case where $\phi$ is not Lebesgue
integrable due to being unbounded. 

Such series have been of great historical interest in pure mathematics
(eg Diophantine Approximation, Discrepancy Theory, q-series) but there
has been a recent resurgence of interest as new applications have
emerged in a number of areas including KAM theory, Quantum Field Theory,
Quantum Chaos, Quantum Computing, and String Theory.  

These series seem most naturally situated at the intersection of the
disciplines of Diophantine Approximation and Dynamical Systems, although
some sophisticated techniques of Complex Analysis were also famously
deployed by Hardy and Littlewood in their studies. We will briefly
position our objects of study within the two aforementioned disciplines. 

One notable difference between the two disciplines is that studies
in Diophantine Approximation have almost always focused on the sums
$\sum_{r=1}^{N}\phi(r\alpha)$ as the natural objects of study, where
$\phi$ is unbounded at the origin. In Dynamical Systems the natural
objects of study are the more general sums $\sum_{r=1}^{N}\phi(x_{0}+r\alpha)$
where $x_{0}$ is an 'initial condition'. This broadening of perspective
proves important. 

Whilst the general approach of this paper is designed to be applicable
to quite general sums, space will restrict us to developing detailed
estimates only for a single class of functions $\phi$ and for the
homogeneous $(x_{0}=0)$ case. This case does however cover the results
of a surprisingly high proportion of previous studies. We hope to
address inhomogeneous sums and other classes of functions in later
papers. 

\subsection{The context of Diophantine Approximation}

Particular examples of series $\sum_{r=1}^{N}\phi(r\alpha)$ have
long been studied within the discipline of Diophantine Approximation,
particularly exploiting the theory of Continued Fractions. These studies
are typically challenging, and individual papers have focused on developing
results for a particular given function $\phi$. Perhaps as a result,
techniques have often been dependent upon identities tied to the particular
function being studied. The major contribution of this paper is an
approach which separates the study of the underlying dynamics (an
irrational rotation) from the characteristics of the function $\phi$.
We develop new results about the dynamics, and these are then immediately
available for use with any chosen function $\phi$. 

The simplest such series is undoubtedly the ``sum of remainders'',
namely the series $\sum_{r=1}^{N}\{r\alpha\}$ (ie $\phi(x)=\{x\}$)
in which $\{x\}$ is the fractional part of $x$ (or remainder modulo
1). This case is ergodic and a simple application of ergodic theory
immediately gives $\frac{1}{N}\sum_{r=1}^{N}\{r\alpha\}\rightarrow\frac{1}{2}$.
However the rate of convergence is of great interest, and was closely
studied by many illustrious mathematicians in the early 20th century,
including Lerch\cite{Lerch1904}, Sierpinsky\cite{Sierpinski1909,Sierpinski1910},
Ostrowski\cite{ostrowski1922bemerkungen}, Hardy \& Littlewood\cite{HardyLattice1924,HardyLattice1922},
Behnke\cite{behnke1924theorie}, Hecke\cite{hecke1922analytische}
who all contributed insights and techniques. Even after such intense
study, Lang still saw value in developing a new approach as late
as the 1960's (see \cite{Lang1995introduction}), and there have been
more papers on the topic in more recent years.

Of all the tools and techniques introduced in these papers, we will
make the greatest use of Ostrowski's representation of integers developed
in \cite{ostrowski1922bemerkungen} above. The way in which Ostrowski
used this technique within his paper is not easily generalisable,
and the technique seems to have been forgotten for many years (in
fact it was independently rediscovered in 1952 by Lekkerkerker \cite{Lekkerkerker1952}
and in 1972 by Zeckendorf\cite{Zeckendorf1972}, but then only for
the case of Fibonacci numbers). However in recent years it has been
recognised as a powerful technique quite independent of its original
``sum of remainders'' context. 

Although the approach in our paper is applicable to the ``sum of
remainders'' problem, that is not our focus. Rather we are concerned
with series in which $\phi$ is not integrable due to unbounded singularities.
In such cases the Ergodic Theorems do not apply, and and other techniques
are necessary to estimate the sums. A remarkable range of techniques
have been applied. Hardy \& Littlewood studied $\phi(x)=\csc\pi x$.
using a number of advanced techniques including double zeta functions,
functional equations, and Cesaro means. Lang \cite{Lang1995introduction}studied
$\phi(x)=1/\{x\}$ and $\left|\csc\pi x\right|$ using an elegant
and elementary recursive technique he developed for the purpose. Sudler\cite{Sudler1964}
and Wright\cite{Wright1964} studied $\phi(x)=\log|\sin\pi x|$ giving
perhaps the first explicit bounds for such a sum. This sum occurs
in a remarkable range of application areas and guises (see the Bibliographies
of \cite{Lubinsky1998,VerschuerenSudler} for a list of 30 papers
on the topic prior to 2016). An open question of Erdos-Szekeres\cite{Erdos1959}
from 1959 was settled positively in 1998 for almost all rotation numbers
$\alpha$ by Lubinsky\cite{Lubinsky1998} who felt certain it would
prove to apply to all $\alpha$. It was finally settled negatively
in 2016 by Verschueren\cite{VerschuerenThesis}, but using techniques
rather different from those in this paper. Recently (2020) Beresnevich,
Haynes and Velani\cite{VelaniRecipFractional2020} studied $\phi(x)=1/\left\Vert x\right\Vert $
and explored connections between this particular Birkhoff sum and
recent advances in the long standing Littlewood conjecture, using
both elementary techniques and Minkowski's Theorem. Sinai \& Ulcigrai\cite{Sinai2009}
(2009) studied the case $\phi(x)=\cot\pi x$ arising out of a problem
in Quantum Computing, and used the ``Cut and Stack'' technique developed
within the Dynamical Systems discipline.

In this paper we  develop instead a single unified approach with
which to tackle these problems. We also show how we can apply it with
relatively little effort to a particular class ($\sum_{r=1}^{N}\phi(r\alpha)$
where $\phi$ is has a single unbounded point at the origin), where
it gives results equivalent to, or improving upon, those currently
existing in the literature.

\subsection{The context of Dynamical Systems}

In the Dynamical Systems context, our series $\sum_{r=1}^{N}\phi(r\alpha)$
is an example of an additive co-cycle (or equivalently a skew product)
and more specifically, a Birkhoff sum. The base phase space is the
circle, and the space morphism (or evolution function) is an irrational
rotation. Birkhoff sums occur widely in the modelling of physical
systems, representing sums of an observable along an orbit of the
phase space. Although we are restricting here our study of Birkhoff
sums to one of the simplest of Dynamical Systems (the irrational rotation),
we recall that many physical systems can be studied via, or even reduced
to, rotations of the circle. 

Key tools of study in this area are normally Ergodic theorems (eg
Birkhoff, von Neumann, Oseledets) which establish that averages of
the Birkhoff sums converge (ie the sequence\footnote{or the multiplicative analogue $\left(\prod_{r=1}^{N}\phi(r\alpha)\right)^{1/N}$ }
$\frac{1}{N}\sum_{r=1}^{N}\phi(r\alpha)$ has a limit) almost everywhere
under suitable constraints (primarily that $\phi$ be $L^{1}-$integrable).
We will call the Birkhoff sum \noun{anergodic} when ergodic theory
cannot be applied. 

Such cases already occur in Quantum Field Theory and String Theory
(where tools of renormalisation theory replace those of ergodic theory)
and where the techniques of this paper (seen as a form of renormalisation)
may provide the beginnings of a further alternative approach. Sinai
\& Ulcigrai report the occurrence of the case $\phi(x)=\cot\pi x$
in a problem of quantum computing. They developed their own approach
to the problem, but the approach developed in this paper is also applicable
to the problem and provides an alternative approach (see the application
in \ref{subsec:The-double-exponential}). 

More recently, the advent of ``Big Data'' has refocused attention
on Koopman operators and their adjoints, Perron-Frobenius operators,
as these seem better suited to data-driven (rather than equation-driven)
analysis, and high dimensional systems. The original theory was developed
in the context of ergodic theory, and is limited to Banach spaces
of observables (usually Hilbert spaces). Birkhoff sums are an important
type of Koopman operator, and our study of anergodic sums extends
the study of Koopman operators into non-Banach spaces of observables
with singularities. It seems entirely possible that Big Data problems
will at some point need such extensions. We therefore expend some
effort in situating our theory within the context of operator theory
(see particularly Section \ref{sec:Dualities-on-Operators}).

\subsection{Overview of main results}
\begin{enumerate}
\item We develop general upper and lower bounds for the Birkhoff sum $S_{N}\phi=\sum_{r=1}^{N}\phi(r\alpha)$
where $\phi$ is a monotonic period-1 function which may be unbounded
at the origin. These bounds are independent of any other characteristics
of $\phi$, and hence widely applicable. 
\item These bounds are easily extended to the related Birkhoff sums $S_{N}\theta$
in the cases $\theta(x)=\phi(1-x)$, $\theta(x)=\phi(x)+\phi(1-x)$,
and $\theta(x)=\phi(x)-\phi(1-x)$. 
\item We compute specific bounds for the particular family $\phi(x)=x^{-\beta}$
for $\beta\ge1$. (The techniques can also be used for $\beta<1$
but this case is $L^{1}$-integrable and is already covered by existing
methods.) We show that that in most cases this leads relatively quickly
to results which match or improve upon results in the literature. 
\item The underlying theory includes a new analysis of the distribution
of the sequence of fractional parts $\{r\alpha\}$ across $[0,1)$
using an extension of Ostrowski's representation of integers which
seems of interest in its own right.
\item 
\end{enumerate}

\subsection{Overview of this paper}

The aim of this paper is to develop a general theory of Birkhoff sums
of the form $\sum_{1}^{N}\phi(r\alpha)$ where $\alpha$ is a real
irrational and $\phi$ is a period 1 real function which may be unbounded
at the origin. As we shall use a number of Dynamical Systems concepts
in the paper, we note that the Birkhoff sum is equivalent to an additive
co-cycle or skew-product $\sum_{1}^{N}\phi(R_{\alpha}^{r}0)$ where
$R_{\alpha}$ is an irrational rotation of the circle, $0$ is the
initial condition, and $\phi$ is a observable on the circle which
may be unbounded at the initial condition.

\textbf{In section 2} we introduce a collection of concepts and results
which we use in the rest of the paper. The paper is slightly unusual
in that we do not require any advanced results from any particular
mathematical discipline, but we will however use basic tools and results
from a range of disciplines. This gives us two problems. The first
is that finding a level of introduction which will suit all readers
is an impossible task, and we apologise to every reader in advance
for both insulting her intelligence in some sections, and for assuming
too much for him in others. The second is that we have encountered
numerous notation collisions between different disciplines. Rather
than embark on the Sisyphian task of unifying mathematical notation,
we have reused the concept of namespaces from Computer Science, albeit
with an informal implementation which we feel is more suited to mathematical
writing.

\textbf{In section 3 }we begin developing a theory of hom-set magmas
which provides the algebraic context for our later sections. This
gives us a set of dualities which save us much duplication of effort
later, and also provides a structure within which to make sense of
later sections, particularly in Section \ref{sec:HomBirkhoffSums}.

\textbf{In section 4} we introduce and develop the idea of separation
of concerns, a concept borrowed from the world of software engineering.
The sum $\sum_{1}^{n}\phi(r\alpha)$ mixes values of the observable
$\phi$ with the dynamics of an irrational rotation. Previous approaches
to these sums have developed identities concerning the observable
$\phi$ and developed specific techniques for summing these particular
identities. In software engineering terms the concerns of summation
and dynamics are tightly bound. The result is that each observable
with distinct characteristics tends to result in a different approach.
In the approach of this paper, we will consciously aim to separate
the two concerns. This allows us to study the dynamics once only,
independent of any particular observable. The summation process then
becomes a separate concern, and its study is simplified by being unbound
from the dynamics. 

The way we achieve this is to use an approach from Category Theory/Operator
Theory to recast the normal Birkhoff sum $\sum_{1}^{n}\phi(r\alpha)$
into the sum $\phi^{*}\sum_{1}^{*n}(R_{\alpha}^{r}0)$ where $\phi^{*},\sum^{*}$
are now operators on the space of orbits under the irrational rotation
$R_{\alpha}$. This achieves the desired separation albeit at the
cost of a more abstract approach. We are now left with studying the
distribution of the orbit $\sum_{1}^{*n}(R_{\alpha}^{r}0)$ (independently
of $\phi$), and then studying the problem of applying operators $\phi^{*}$
to such sequences. Like with many things in basic Category Theory,
the approach becomes almost trivial once the concepts are understood.
The value of Category Theory lies in encouraging thinking which leads
to a fruitful recasting.

\textbf{In section 5} we study our first concern which is the distribution
of the orbit $\sum_{1}^{n}(R_{\alpha}^{r}0)$. The equi-distribution
of the points of such an orbit regarded as a \emph{set}, is well understood
and much studied. However our concern is rather to study the distribution
as a \emph{sequence}, ie it is to develop an estimate of the location
of the point $r\alpha$ for each $r$. It turns out we can do this
remarkably well using the number theoretic tool of Ostrowski representation,
but lifting this tool to the level of an operator acting on ``quasiperiod''
orbit segments - namely the segments of an orbit which lie between
closest returns. The error terms in these positional estimates are
well controlled and effectively obey a renormalisation law.

\textbf{In section 6} we develop some basic theory on unbounded functions
and their spaces, as this does not appear to be a well studied area. 

\textbf{In section 7} we move to our second concern of applying the
derived operator $\phi^{*}$ to the orbit, and this is now largely
the problem of applying $\phi^{*}$ to the quasiperiod segments identified
in Section 5. This results in upper and lower bounds for the Birkhoff
sum $\sum_{1}^{n}\phi(r\alpha)$ which of course involve $\phi$,
but which are perfectly general, ie they are independent of the particular
characteristics of the specific observable $\phi$. The approach also
provides a way of identifying important equivalence classes of observables
which share the same higher order growth estimates. This completes
our development of the general theory.

\textbf{In section 8} we make an application of the general theory
to the particular family of functions $\phi(x)=x^{-\beta}$ for real
$\beta\ge1$. The general theory can of course be applied to the family
$\beta<1$, but as $x^{-\beta}$ is $L^{1}-$integrable on the circle,
these Birkhoff sums can be estimated via existing methods and we shall
ignore them here. For $\beta\ge1$ the fact that we now have a general
theory enables us to develop estimates using a variety of estimating
techniques developed independently in the literature. The theory enables
us to develop the estimates relatively rapidly compared with earlier
papers, perhaps with greater clarity. In addition, as they are developed
within a unifying framework, the the relative merits of the different
techniques can be assessed, something which was not easy to do previously. 

\textbf{In section 9}, we compare the estimates from the general theory
to the results available in the literature. With a few special case
exceptions, we show that the results from the general theory are equivalent
to, or improve upon those in the literature. (It is possible that
the exceptions could be improved on with more careful application
of the theory, as we have made no attempt to achieve best possible
results at this stage).

\section{Preliminaries}

This paper requires little in the way of pre-requisites from any
one mathematical discipline, but we do draw some basic ideas and results
from a range of disciplines. Some of these disciplines have several
notations in common use and in addition there are some notation collisions
between disciplines. In this section, we standardise our notation,
summarise the necessary disciplinary background, and develop from
these some simple theory and results for use in the following sections. 

\subsection{Basic Terminology \& Notation}

We will adopt the following notation for logical operations: $\&$
(AND), $|$ (OR), $!$ (NOT), $\coloneqq$ (Definition), $\left\llbracket .\right\rrbracket $
(Iverson bracket - see \ref{def:The-Iverson-bracket}).

We will use symbols from the $\mathbb{B}$lackboard font for our
main spaces of interest: we assign $\mathbb{Z},\mathbb{\mathbb{\mathbb{R}}},\mathbb{T}$
their usual meanings as the ring of integers, the topological field
of real numbers with Lebesgue measure, and the circle (as a Lie Group/real
manifold - see \ref{subsec:Circle} below). We assign $\mathbb{N}$
the ``semi-standard'' meaning as an additive group (and so $0\in\mathbb{N}$
and we define $\mathbb{N}^{+}\coloneqq\mathbb{N}\backslash\{0\}$).
Finally we define two real semi-open intervals $\mathbb{I}\coloneqq[0,1)$
and $\mathbb{I}^{\slash}\coloneqq\mathbb{I}-\frac{1}{2}=[-\frac{1}{2},\frac{1}{2})$
which together form the atlas (the covering of coordinate spaces)
for our circle manifold.  

We define the set $\{x_{r}\}_{r=1}^{N}\coloneqq\{x_{1},x_{2},\ldots,x_{N}\}$
with the analogous notation $(x_{r})_{r=1}^{N}$ for sequences.

\subsubsection{Contextual Annotation}

Motivation
In borrowing from a number of disciplines in this paper, we quickly
hit a problem with homonyms - symbols which have different meanings
in different disciplines. The problem is one of disambiguation - how
do we ensure that the reader understands which meaning is signified
by a homonym at any particular occurrence.

In mathematics we have historically aimed to solve this problem by
eliminating it, stressing the importance of unambiguous notation.
In theory we can simply define sufficient new symbols. However we
found two obstacles to this approach in this paper. First, we are
using very basic concepts with very familiar notation, and this means
that introducing new notation may be unhelpful. For example the notation
$\{x\}$ signifies different meanings in set theory and number theory,
but it seems undesirable to replace it with new notation in either
discipline. Second, although the supply of new symbols is theoretically
limitless, modern mathematics is in reality constrained by the supply
of defined Unicode symbols. The latter is very finite and can be exhausted
surprisingly quickly.

Both of these problems were encountered in the process of writing
down the theory presented in this paper. This leads us to considering
ambiguous notations, and in practice this proves to be much less of
a problem than it sounds. 

Firstly, the meaning of any particular homonym is often easily derived
from the immediate context, and then the danger of ambiguity is theoretical
rather than real. In fact we frequently recognise this in mathematics
by saying we are ``abusing notation''. The phrase suggests that
we know we are doing something wrong or lazy, but perhaps feel we
can ``get away with it''. However we now argue that, provided that
the context does remove ambiguity, we are not ``abusing'' so much
as ``reusing'' notation, and this is an effective and laudable approach
to managing a finite resource. As an example, consider the trivial
function $f:\mathbb{N}\rightarrow\mathbb{N}$ defined by $n\mapsto n+1$.
We now extend this to a function $f:\mathbb{C}\rightarrow\mathbb{C}$
defined by $z\mapsto z+1$. We refer to this as ``abuse of notation''
since we are using the symbol $f$ to signify two different functions,
but it would serve no good cause to use different symbol for the two
functions, and indeed doing so would lose sight of the deep connection
between the functions. 

We will therefore feel free to subject our notation to abusive practices,
subject to an additional authorial responsibility to ensure that ambiguity
is resolved by the context. The reader must judge how well we succeed.

Secondly, the normal way of resolving ambiguity in formulae or equations
is to add explanatory detail to the surrounding text. However this
can become clumsy, and the relevant text can seem somewhat dissociated
from the symbols themselves. We borrow from Computer Science the idea
of ``inlining'' context by the direct annotation of ambiguous tokens
(such as variable names) with context (``namespace'') tokens which
determine the correct disambiguation. In Computer Science the goal
is to provide a formal disambiguation syntax suitable for processing
by machine, and existing implementations we have considered are probably
too heavy handed to suit mathematical writing. We have instead adopted
a lighter and more mathematical notational approach with the more
modest goal of providing sufficient disambiguation to suit an intelligent
mathematical reader, rather than an efficient algorithm. Like any
new notation there is a learning hurdle, but we feel the benefits
for this paper make it worthwhile. A quick-start summary follows here,
but see Appendix \ref{sec:Context-Annotation} for a more formal explanation.

\subsubsection{Contextual Annotation Quickstart}

Embracing the positive reuse of notation as described above, enables
us to write $X\coloneqq(X,+)$ in order to define the group $X$ with
underlying set $X$ and operator $'+'$. We call $'X'$ here a homonym
because we have given two different mathematical objects (the set
and the group) the same name $X$. Depending upon the document we
are writing, we might judge the intended meaning to be clear from
surrounding text. If there is need for disambiguation, we can make
the meaning of each occurrence of the homonym $X$ explicit by annotating
each symbol with a suitable contextual label. For example we could
write $X_{Grp}\coloneqq(X_{Set},+)$. In a more formal setting we
might decide to be even more explicit with something like $X_{\epsilon Grp}\coloneqq(X_{::Set},+_{AssocOp})$.
Whether we needed to define the context labels $Grp,Set,AssocOp$
would again depend on authorial judgment.  In these examples, the
label represents a category or class of mathematical objects but it
is useful not to restrict ourselves to this. Also given subscripted
$X_{i}$ we could write $(X_{i})_{Grp}$. Use $:$ or $\epsilon$
only when appropriate.

\begin{example}
Classically in mathematics, the type of a variable is given in accompanying
text, visually separated from the reasoning which follows. The annotation
approach complements the classical approach and the two can be usefully
intermixed. In some situations annotation may be distracting, cluttering
an equation or formula, and in this case the classical approach will
be best. In other situations it may enhance the flow of information
by concisely inlining all necessary context. For example we might
write $z_{\mathbb{C}}=x_{\mathbb{R}}+iy_{\mathbb{R}}$, and this may
be either highly useful or highly distracting depending upon the circumstances.
\end{example}

\subsubsection{The abuttal operator (Annotation conventions)}

Rather than writing $Y^{X}$ we will use $XY$ as a concise notation
for the morphisms from the object $X$ to the object $Y$. In addition
to providing more natural directionality, this notation is particularly
convenient for use as an annotation. For example, writing $\phi_{XY}(a)=b$
tell us that $\phi$ is a morphism from $X$ to $Y$, and hence we
can also deduce that $a\in X,b\in Y$: it would be redundant to write
$\phi_{XY}(a_{X})=b_{Y}$, although it is permitted. Note that writing
$\phi_{XY}$ means the same as $\phi\in XY$ and $\phi:X\rightarrow Y$,
and sometimes it also makes sense to say simply ``$\phi$ is $XY$''.
If $X=Y$ it often makes sense to abbreviate $\phi_{XX}$ to simply
$\phi_{X}$ but then it needs to be clear from other context that
$\phi$ is a function on $X$, and not an element of $X$. This is
an example of why we regard $X$ here a contextual label rather than
the category or class $X$.

We also borrow from mathematical logic the convention of using abuttal
to represent a suitable distinguished operator in a particular context.
For example in the context of an algebraic ring we retain the usual
convention $xy\coloneqq x\times y$, ie the abuttal operator represents
the multiplication operator. However in the context of the function
space $XY$, we define abuttal as evaluation, ie $\phi_{XY}x\coloneqq\phi(x)$.
Given also $\theta_{YZ}$ we will define abuttal to be functional
composition, ie $(\theta\phi)_{XZ}x\coloneqq\theta_{YZ}(\phi_{XY}x)$.

\subsection{Logic and Abstract Algebra }

Functions of Propositions
\begin{defn}
\label{def:The-Iverson-bracket}Let $\mathscr{P}$ be a Boolean algebra
of logical propositions, $\mathscr{B}$ the Boolean integers $\{0,1\}$,
and $\mathscr{B}^{\prime}$ the alternate Boolean integers $\{-1,1\}$.
The Iverson bracket $\left\llbracket .\right\rrbracket _{\mathscr{PB}}$
is a function which maps logical propositions to $\{0,1\}$ as follows:
for any proposition $P$ we define $\left\llbracket P\right\rrbracket :=1$
if $P$ is true, and $0$ if $P$ is false. The alternate Iverson
bracket $\left\llbracket .\right\rrbracket _{\mathscr{PB}^{\prime}}^{\prime}$
is analogous: we define $\left\llbracket P\right\rrbracket :=1$ if
$P$ is true, and $-1$ if $P$ is false.
\end{defn}

Note the useful identities $\left\llbracket !P\right\rrbracket =1-\left\llbracket P\right\rrbracket $,
and that $\left\llbracket P\right\rrbracket ^{\prime}=(-1)^{\left\llbracket !P\right\rrbracket }=\left\llbracket P\right\rrbracket -\left\llbracket !P\right\rrbracket $.
These functions allows us to define discontinuous functions very efficiently:
for example $f_{\mathbb{R}\mathbb{R}}(x)\coloneqq\left\llbracket x>0\right\rrbracket x$
gives us the Heaviside function (the real function which takes the
value $x$ for $x>0$, and $0$ otherwise. 

We will often need to distinguish odd and even cases, and it useful
to define a set of constants for this purpose.
\begin{defn}
For $r\in\mathbb{Z}$ we define $E_{r}\coloneqq\left\llbracket r\,\textrm{is\,even}\right\rrbracket ,O_{r}\coloneqq\left\llbracket r\,\textrm{is\,odd}\right\rrbracket $
\end{defn}

Note the useful identities $E_{r}+O_{r}=1$ (so $x_{r}=E_{r}x_{r}+O_{r}x_{r}$),
and $E_{r}-O_{r}=(-1)^{r}$.

Abstract Algebra
\label{par:AlgBasics}Basic terminology
We recall some basic algebraic concepts and terminology, we fix precisely
how we will use them here, and we establish notation.
\begin{description}
\item [{Set}] We will not be concerned with foundational issues here, so
we content ourselves with a naive approach and in particular we allow
the definition of sets by use of the naive notation $X\coloneqq\{x:Px\}$
where $P$ is a logical proposition.
\item [{Function}] A function $\phi_{XY}$ here always means a single valued
relation on the set $X\times Y$. The annotation $XY$ represents
the set of morphisms from $X$ to $Y$, which in the case of sets
is precisely the set of all possible functions from $X$ to $Y$.
The latter is often denoted $Y^{X}$, but we prefer the more natural
direction implied by $XY$. When it seems more appropriate, we will
also use the traditional notation $\phi:X\rightarrow Y$.\\
A function is defined explicitly using the notation $\phi_{XY}\coloneqq x_{X}\mapsto y_{Y}$,
which also allows the definition of anonymous functions, eg $x_{X}\mapsto y_{Y}$.
\\
We will also use placeholder notation to allow for function notation
involving decorations (eg $\bar{.}$) or braces (eg $\{.\}$ ). Here
the character $'.'$ is the placeholder. \\
We define the $n_{\mathbb{N}}$-ary extension of $\phi_{XY}$ to be
$\phi_{X^{n}Y^{n}}:(x_{i})_{i=1}^{n}\mapsto(\phi x_{i})_{i=1}^{n}$
and the skew extension $\phi_{X^{n}Y^{n}}^{Op}:(x_{i})\mapsto(\phi x_{i})^{Op}$
where $\left(.^{Op}\right)_{X^{n}X^{n}}$ is the function which reverses
sequences in $X^{n}$. (This is of course a specialisation of the
concept of the $Op$ functor from Category Theory but we wish to avoid
deploying the full generality here. We retain the ``skew'' terminology
as it is more familiar in the contexts we are using). \\
If $X=Y$ then $\phi$ is a set endomorphism, and any $x_{X}$ satisfying
$\phi x=x$ is called a fixed point of $\phi$.
\item [{Algebraic~Operation}] A function $\circ_{X^{n}X}$ is called an
$n_{\mathbb{N}}$-ary algebraic operation on a set $X$, abbreviated
to simply an operation on $X$ when $n$ is unimportant. \\
A unary operation is also a set endomorphism, which will often be
denoted by the decoration of a symbol, eg $\overline{.}$ denotes
the function $x\mapsto\overline{x}$. \\
Binary operations have the formal notation $\circ(x_{1},x_{2})$ and
this is occasionally helpful, but we shall of course normally use
the more common notation $x_{1}\circ x_{2}$. Often it is useful to
use abuttal (eg $ab$) to indicate the presence of a well known binary
operation - when needed we will use $''$ to indicate the abuttal
operator, eg $''\coloneqq(a,b)\mapsto a\times b$ means $ab$ is to
be interpreted as $a\times b$. \\
$n$-ary operations on sequences $(x_{i})_{X^{n}}$ are often constructed
by extending unary or binary operations, eg by $\overline{(x_{i})}\coloneqq(\overline{x_{i}})$
or $\text{\ensuremath{\circ(x_{i})\coloneqq}\ensuremath{\circ(x_{1},\circ(x_{2},\ldots\circ(x_{n-1},x_{n})\ldots))}}$
and we will take these these extensions as given unless otherwise
specified. If we also omit specification of $n$ then we intend that
the extension may be made to any $n\ge1$ (unary) or $n\ge2$ (binary
case). Examples: the unary operation $-:x\mapsto-x$ gives $-(x_{1},\ldots x_{n})=(-x_{1},\ldots,-x_{n})$,
a commutative binary operation $+$ gives $,+(x_{1},\ldots,x_{n})=\sum_{1}^{n}x_{i}$
whereas the non-commutative binary \noun{skew} operation $\circ^{Op}:(x_{1},x_{2})\mapsto x_{2}\circ x_{1}$
gives $\circ^{Op}(x_{1},\ldots,x_{n})=\circ(x_{n},\ldots,x_{1})$.
(Borrowed from Universal Algebra).
\item [{Algebraic~Structures}] An algebraic structure $X_{Struc}$ is
a an ordered pair of sets $(X,\{\circ_{\alpha}\}_{\alpha\in A})$
where $X_{Set}$ is called the underlying set, and $\{\circ_{\alpha}\}$is
a set of operations on $X$ indexed by a set $A$. We will also call
$X$ a $\left|A\right|-Struc$. In the case of a singleton set $A$
(ie $X$ is a $1-Struc$), we simplify the notation to $(X,\circ)$
rather than $(X,\{\circ\})$. For example, a bare set has $A=\emptyset$
(meaning there is no additional structure), a bare monoid or group
has $A$ a singleton, a bare ring or field has $|A|=2$. For $|A|\ge2$
there will usually be relations defining interaction between operators.
\item [{Generated~Structures}] Let $(U,\circ)$ be a structure with $\circ$
a binary operation. If $U$ has a unit it is unital. Note that if
$U$ is not unital we can always extend it with a unit, and there
can be only one unit for $\circ$. Let $X\subseteq U$ and $X$ is
closed under $\circ$, $X$ is a substructure of $U$. The smallest
substructure containing $X$ (which could be $U$ or $X$ or something
in between) is the closure of $X$, designated $<X,\circ>$ (and also
called the structure generated by $X$). \\
If $\circ$ is not associative (following Bourbaki \#\#paper) the
structure is a magma and its elements are represented by binary trees
with leaves in $X$, eg $w_{1}\circ w_{2}$ can be regarded as a
tree rooted at $\circ$ with $w_{1},w_{2}$ as left and right sub-trees
(which may be simply leaves). When $\circ$ is associative, the need
for brackets disappears, the structure becomes a semi-group, and elements
such as $u_{1}\circ u_{2}\circ u_{3}$ can be represented as strings
$u_{1}u_{2}u_{3}$ or sequences $(u_{1},u_{2},u_{3})$. If one of
these structures lacks a unit, we can always extend it with a unique
unit, but this is more important in the case of a semi-group - in
this case the structure becomes a monoid, and the unit can be represented
as an empty string or sequence, typically designated $\epsilon$.
\\
The structure is called free if there are no relations between elements,
ie if $w_{1}=w_{2}$ means that any representations of $w_{1},w_{2}$
(trees or strings) are identical. Note that even if $(U,\circ)$ is
not free, we can always introduce a second operator $\circ_{2}$,
forget $\circ_{1}$, and consider $U_{Set}$ as a set of generators
for a free magma, semi-group or monoid $(U^{*},\circ_{2})$. In this
case elements of $U_{Set}$ are called the base elements. \\
We shall be particularly interested in using free monoids to represent
orbits in a Dynamical System, and in magmas of hom-sets to develop
our theory of induced morphisms (see \ref{sec:Dualities-on-Operators}).\\
We can regard a free magma as a naturally graded structure (ie there
is a function $U^{*}\rightarrow\mathbb{N}_{0}$) in two ways. The
first, introduced by Bourbaki is effectively the width of a term binary
tree, measured as the width of the base of the tree, ie the number
of leaf elements in the tree. This grading also extends to associative
free magmas where it reduces to being simply the length of a string).
We introduce a second grading for non-associative magmas in \ref{sec:Dualities-on-Operators},
which we will call the height of the tree, which is more useful for
our purposes. We define base elements to be trees of height $0$,
and inductively a tree of height $n$ is one which has one or both
sub-trees of height $n-1$. Note this means $a\circ b$ has a height
of $1$, $(a\circ b)\circ(c\circ d)$ has a height of $2$, whereas
$(a\circ(b\circ(c\circ d)))$ has a height of 3. 
\item [{Homomorphisms/Morphisms}] A structure homomorphism $\phi_{XY}$
between two algebraic structures $X_{Sruc}=(X_{Set},\{\circ_{\alpha}\}),Y_{Struc}=(Y_{Set},\{\circ_{\beta}\})$
is a function $\phi_{X_{Set}Y_{Set}}$ which preserves (is compatible
with/distributes over) the algebraic structure. This means that for
each $n$-ary ($n>0$) algebraic operation $\circ_{\alpha}$ on $X$
there is a $n$-ary $\circ_{\beta}$ on $Y$ such that $\phi\left(\circ_{\alpha}(x_{i})_{n-Seq}\right)=\circ_{\beta}\left(\phi x_{i}\right)_{n-Seq}$
for all sequences $(x_{i})_{X^{n}}$. It is useful to define for
each $n$ the Cartesian product function $\phi_{X^{n}Y^{n}}=\prod_{i=1}^{n}\phi_{XY}\coloneqq(x_{i})_{X^{n}}\mapsto(\phi_{XY}x_{i})_{Y^{n}}$,
which gives us $\phi_{XY}\circ\circ_{\alpha}=\circ_{\beta}\circ\phi_{X^{n}Y^{n}}$.
We can also define the involution $Op_{X^{n}}\coloneqq(x_{i})_{i=1}^{n}\mapsto(x_{i})_{i=n}^{1}$,
which pulls up to an involution on $X^{n}Y^{n}$ defined by $Op(\phi_{X^{n}Y^{n}})=\phi_{X^{n}Y^{n}}^{Op}\coloneqq\phi_{X^{n}Y^{n}}\circ Op_{X^{n}}$.
We call $\phi^{Op}$ an anti-structure homomorphism as it reverses
sequences. For example, given groups $(X,\circ_{X}),(Y,\circ_{Y})$,
then $\phi_{XY}$ is an anti-group homomorphism iff $\phi(x_{1}\circ_{X}x_{2})=(\phi x_{2})\circ_{Y}(\phi x_{1})$.
\\
Note the standard additional terminology for specialised homomorphisms:
if the function $\phi_{X_{Set}Y_{Set}}$ is a bijection then $\phi$
is an isomorphism, and if $X=Y$ then $\phi_{XX}$ is an endomorphism,
and if both conditions hold then $\phi_{XX}$ is an automorphism.
If $\phi_{XX}x=x$, $x$ is a fixed point of the endomorphism $\phi$.
\item [{hom-sets}] Note that $XY$ represents the hom-set (the set of homomorphisms
or morphisms) from $X$ to $Y$. and that the endomorphism hom-set
$XX$ of endomorphisms on $X$ is a subset of the hom-set of unary
operators on $X$. We therefore need to be careful $XY$ can often
be equipped with algebraic operations defined by the operations of
the underlying algebraic structures $X$ and $Y$. We call these new
operations \noun{derived/induced} and $XY$ is then a (second order)
algebraic structure in its own right. We will make significant use
of this type of construction - see section (\ref{sec:Dualities-on-Operators})
below.\\
The abuttal map $":=(X,Y)\mapsto XY$ is itself a binary operation
on sets, and so given a set of sets $X=\{X_{\alpha}\}$ we will write
$X^{*}=<X,''>$ for the set of sets generated from $X$ by $''$.
Note that since $''$ is not associative, $(X^{*},'')$ is a magma
rather than a semi-group (its elements are represented by binary trees
rather than strings), and we will call it the magma of hom-sets. Also
the elements of $X^{*}$ contain typically morphisms between sets
of morphisms rather than morphisms between morphisms, so that these
are distinct from the ``higher order'' morphisms encountered in
higher order category theory.\\
second cant use of eg $Y$ group gives $XY$ group, endomorphisms
of group form group (or comp of monids XYZ). 
\item [{}]~
\end{description}

\begin{description}
\item [{Involution}] An involution $\left(\overline{.}\right)_{XX}$ on
an algebraic structure $X=(X_{Set},\{\circ\}_{\alpha})$ is a unary
operation of order 2 (ie for every $x_{X}$ we have $\overline{\overline{x}}=x$),
and which is also a skew-automorphism which is a skew-anti-commutes
with each $\circ_{\alpha}$, ie $\overline{x\circ_{\alpha}y}=\overline{y}\circ_{\alpha}\overline{x}$.
The second constraint is automatically satisfied if $\{\circ\}_{\alpha}$
is empty. Given another set $Y$, an involution on $X$ induces an
involution on $(XY)_{Set}$ by $\overline{\phi_{XY}}x=\phi\overline{x}$.
However this is involution does not generally commute with operators
on $XY$ such as composition or addition, so it is not generally an
involution on $XY$ as an algebraic structure. 
\end{description}

We call $\overline{x}$ the involute of $x$, and self-involute when
$x=\overline{x}$. The set of self-involute points is the set of fixed
points of $\overline{.}$ denoted $FP(X,\bar{.}$) or just $FP(X)$
when $\bar{.}$ is clear. A subset $W\subseteq X$ is fixed point
free under $\bar{.}$ if it is disjoint from $FP(X,\bar{.})$. Note
that this means $W,\bar{W}$ are mutually disjoint, and that $W\ne X$
unless $X=\emptyset$. A partition of $X$ under $\bar{.}$ is a triple
of mutually disjoint subsets $(W,\overline{W},FP(X))$ which cover
$X$.
\begin{lem}
Let $X,Y$ be algebraic structures with involutions $\bar{.}_{X},\bar{.}_{Y}$,
and let $W\subset X$ be fixed point free under $\bar{.}_{X}$. Suppose
further that $A,B$ are real-valued functions on $X\times Y$ with
the property that $A(x,y)>B(x,y)$ and $A(\overline{x},y)<B(\overline{x},y)$
for any $y$ and any $x\in W$. Then $\overline{A}(x,y)<\overline{B}(x,y)$
and $\overline{A}(\overline{x},y)>\overline{B}(\overline{x},y)$ for
any $y$ and any $x\in W$.
\end{lem}

\begin{proof}
$\overline{A}(x,y)=A(\overline{x},\overline{y})<B(\overline{x},\overline{y})=\overline{B}(x,y)$
and the second result follows similarly. 
\end{proof}
Involutions

\begin{defn}
An involution $\overline{.}$ on a set $X$ is a self-inverse bijection,
ie $\overline{\overline{x}}=x$. An involution on an algebraic structure
$(X,\circ)$ is an involution on $X$ which is also an anti-isomorphism,
ie $\overline{x\circ y}=\overline{y}\circ\overline{x}$. If $\circ$
is commutative, $\overline{.}$ becomes an automorphism. 
\end{defn}

If $(X,\circ)$ is a group, we call $\overline{x}=x^{-1}$ the inverse
involution, noting $\overline{xy}=\overline{y}\,\overline{x}$. When
$X=\mathbb{R}/\mathbb{Z}$ this gives $\overline{x}=-x$, $\overline{x+y}=(-y)+(-x)$,
and we will use this extensively.  

On the free monoid $(X^{*},+),$ reversal is an involution where $\overline{x_{1}x_{2}\ldots x_{N}}=x_{N}x_{N-1}\ldots x_{1}$
and $\overline{w_{1}+w_{2}}=\overline{w_{2}}+\overline{w_{1}}$. Note
that reversal restricted to $X$ is the identity.

We will use some useful abstractions inspired by Category Theory in
order to simplify and structure our theory. We start with concrete
categories, namely categories whose objects are mathematical spaces
each with an underlying set. By abuse of notation we will denote the
underlying space of $X_{Obj}$ by $X_{Set}$. Also in a concrete category,
the morphism set $XY$ of morphisms from $X_{Obj}$ to $Y_{Obj}$
is a set of functions $\{\phi_{X_{Set}Y_{Set}}\}$ between the underlying
sets. Of course $XY$ will normally be a strict subset of the full
set of functions from $X_{Set}$ to $Y_{Set}$ (usually denoted $Y^{X}$),
and will be chosen to preserve in some way the additional structure
of the spaces $X_{Obj},Y_{Obj}$. 

Given a concrete category $C$ we construct the second order category
$C^{+}$ whose objects are $Obj(C)\bigcup\{XY:X,Y\in Obj(C)\}$, ie
we extend $C$ by making the morphism sets $XY$ of $C$ into objects
of $C^{+}$. This allows us to construct morphisms between morphism
sets of $C$, and between morphism sets of $C$ and objects of $C$.

We use this for separating $S_{N},\phi$ and light use of duality
to reduce the number of cases we need to consider.

\label{par:Induced-Morphisms}Induced/Constructed/Dependent Morphisms
General: $A_{(XY)(VW)}:\phi_{XY}\mapsto\phi_{VW}$ Need to check whether
$\phi_{VW}$ is a morphism.

Sequence extension $\prod_{(XY)(X^{n}Y^{n})}^{n}:\phi_{XY}\mapsto\phi_{X^{n}Y^{n}}$
where $\left(\prod^{n}\phi\right)(x_{i})\coloneqq(\phi x_{i})_{Y^{n}}$
and hence to Kleene Star $\phi_{X^{*}Y^{*}}$ which is $\phi_{X^{n}Y^{n}}$
on seq of length $n\ge0$. But note image $\prod_{(XY)(X^{n}Y^{n})}^{n}XY$
is usually a strict subset of $X^{n}Y^{n}$ consisting of the diagonal
entries: general term of $X^{n}Y^{n}$ is $(\phi_{j})_{X^{n}Y^{n}}:(x_{j})_{X^{n}}\mapsto(\phi_{j}x_{j})_{Y^{n}}$
- diagonal element of $X^{n}Y^{n}$, $X^{n}$ has $\phi_{j}=\phi$,
$x_{j}=x$ respectively.

$(XY)^{n}=\left\{ (\phi_{1},\ldots,\phi_{n})\right\} $ is ambiguous
over domain - is it $XY^{n}$ or $X^{n}Y^{n}$? - don't use - use
either of the latter which are well defined. But then $\phi_{XY^{n}}\in(XY)_{Set}^{n}$-
but application of $(XY)^{n}$ as a map is undefined, better to say
$XY^{n}$ is naturally isomorphic to $X^{n}Y^{n}\mid_{Diag(X^{n})}$

Also $(\phi_{i})_{XY^{n}}:x\mapsto(\phi_{i}x)_{Y^{n}}$ - restriction
of domain of $X^{n}Y^{n}$ from $X^{n}$ to diagonal elements of $X^{n}$,
$Diag(X^{n})$

Pullback/Koopman Operator: $C_{\psi}:\phi_{XY}\mapsto\left(\phi\psi_{WX}\right)_{WY}$
$C_{\psi}$ is linear in $\phi$ and has signature $(XY)(WY)$, $C$
is $\left(WX\right)\left((XY)\left((XY)(WY)\right)\right)$ function
to operator (meta operator, 2nd order operator)

Given $(Y,\circ_{Y^{n}Y}),$ $\circ_{(XY^{n})(XY)}:(\phi_{i})_{XY^{n}}\mapsto\left(\circ_{Y^{n}Y}\cdot(\phi_{i})_{XY^{n}}\right)_{XY}$
eg (n=2) $(\phi_{1}\circ_{XY}\phi_{2})x\coloneqq\left(\phi_{1}x\right)\circ_{Y}(\phi_{2}x)$
(pushforward of $\circ$ ) Equivalent to $f(g)=f\circ g$ not $g(f)=f\circ g$
- simple composition for 1-functions, extended to sequences. What
about $\circ_{Y^{n}Z}$ ? $(Y,\circ_{Y^{n}Z}),\circ_{(XY^{n})(XZ)}:(\phi_{i})_{XY^{n}}\mapsto\left(\circ_{Y^{n}Z}\cdot(\phi_{i})_{XY^{n}}\right)_{XZ}$
- this normal comp with $Y$ replaced by $Y^{n}.$ Generalisation
of distribution when $X=Y=Z$. 

PushForward/Koopman Adjoint: $C_{\phi}:\psi_{WX}\mapsto\left(\phi_{XY}\psi_{WX}\right)_{WY}$
is $(WX)(WY)$ instead of $(XY)(WY)$. $C$ is linear in $\phi$,
but $C_{\phi}$ is not linear in $\psi$. So the codomain $X$ of
$\phi$ is pushed forward to $Y$, rather than the domain of $X$
being pulled back to $W$. In the above we have $(W,X,Y)$ replaced
by $(X,Y^{n},Y)$.

Let $X,Y$ be two objects of some concrete category $C$, and suppose
$Y$ is equipped with a binary operation $+_{Y}$ (not necessarily
commutative). Then $+_{Y}$ induces a dependent binary operation $+_{XY}$
on the set of functions $Y^{X}$ as follows: $(\phi_{XY}\,+_{XY}\,\theta_{XY})x\coloneqq\phi x\,+_{Y}\,\theta x$.
Of course the subset $XY$ of $Y^{X}$ may not be closed under $+_{XY}$.
If it is closed, then we say the space $XY\coloneqq(XY,+_{XY})$ is
a dependent space of $Y$. 

Now given an endomorphism $T_{XX}$ and a morphism $\phi_{XY}$, then
the composition $\phi_{XY}T_{XX}$ is itself in $XY$. In other words
$T$ acts on $XY$ by $\phi\mapsto\phi T$. When $XY$ is a dependent
space on $Y$, the action is also linear in the sense that $(\phi_{1}+\phi_{2})(Tx)=\phi_{1}(Tx)+\phi_{2}(Tx)=(\phi_{1}T)(x)+(\phi_{2}T)(x)=(\phi_{1}T+\phi_{2}T)(x)$. 

We now wish to be able to reason further about the map $\phi\mapsto\phi T$
which is an endomorphism of $XY$ and consequently does not exist
within $C$. Informally we will just add it in, but we will also describe
how to this formally. We move to the higher order category $C^{+}$
where we can define morphisms between sets of morphisms - in particular
a morphism between $XY$ and $AB$ will lie in the morphism set $(XY)(AB)$.
In particular we can now define the dependent morphism $\mathbb{T}_{(XY)(XY)}(\phi_{XY})\coloneqq\phi T_{XX}$.
Now by the result above, $+_{XY}$ in $XY$ induces a dependent binary
operation $+_{(XY)(XY)}$ in $XY^{(XY)}$ defined by $(T+U)\phi\coloneqq T\phi+U\phi$.
Again $(XY)(XY)$ is a dependent space of $XY$ if closed under $+$. 

The endomorphisms $T_{XX}$ of $X$ form a monoid under composition,
and so also $(XY)(XY)$ and also a semi-group if closed under the
operator $+_{XY}$ pulled up from $+_{Y}$. Also $T^{n}\phi=\phi T^{n}$
for $n\ge0$. 

$X^{*}=(X^{*},+)$ is the free linear monoid on $X$, and we extend
$T$ to $X^{*}$ by $T(\sum_{r=1}^{N}x_{r})\coloneqq\sum_{r=1}^{N}Tx_{r}$
so the induced $T$ is $X^{*}X^{*}$ and linear. Now the orbit segment
of $x_{0}$ with length $N$ in $(X,T)$ is $(Tx_{0},T^{2}x_{0},\ldots,T^{N}x_{0})$
and now regard this as the word $\sum_{r=1}^{N}\left(T^{r}\left(x_{0}\right)_{X^{*}}\right)$
in $X^{*}$ where $\left(x_{0}\right)_{X^{*}}$ is the single letter
word in $X^{*}$ corresponding to $(x_{0})_{X}$. Using the linearity
of $T^{r}$ we can rewrite this as $\left(\sum_{r=1}^{N}T^{r}\right)x_{0}$
where the summation is the induced operator $+_{(X^{*}X^{*})}$. We
now define the linear operator $S_{N}=\sum_{r=1}^{N}T^{r}$ and we
can write the orbit segment as $S_{N}x_{0}$. Finally we similarly
extend $\phi_{XY}$ to $X^{*}$by $\phi(\sum_{r=1}^{N}x_{r})=\sum_{r=1}^{N}\phi x_{r}$
where now the second summation uses $+_{Y}$. Hence we can use the
rewriting rules of (\ref{sec:Separation-of-Concerns}) $\phi(S_{N}x_{0})=\sum_{r=1}^{N}\phi(T^{r}x_{0})=\sum_{r=1}^{N}(T^{r}\phi)x_{0}=\left(\sum_{r=1}^{N}T^{r}\right)\phi x_{0}=\left(S_{N}\phi\right)x_{0}$
where the second homonym $S_{N}$ is now $(X^{*}Y)Y$ and $S_{N}\phi$
is $X^{*}Y$ (ie an observable on $X^{*}$) and its restriction to
$X$ is now the Birkhoff sum operator 

Free Monoids
\begin{defn}
Given a set $X$, the free monoid $(X^{*},+)$ on $X$ is defined
as the set of finite words (or strings) $x_{1}x_{2}\ldots x_{N}$
with each letter $x_{i}\in X$, including the empty word $0$, and
with the operation $+$ being the concatenation of words (which of
course is NOT commutative). We write $\sum_{r=1}^{N}w_{r}$ to represent
$w_{1}+w_{2}\ldots+w_{N}$ where $w_{r}$ is a word of $X^{*}$. We
write the reverse sum $w_{N}+w_{N-1}\ldots+w_{1}$ as $\sum_{r=1}^{N}w_{N+1-r}$.
\end{defn}

\label{par:Duality}Duality
A duality between two spaces allows us to obtain dual results automatically
in one space from results in the other. We will exploit duality in
a number of places to reduce the number of cases we need to consider.
However the dualities are quite simple, and the theory development
required to make them formal and explicit does not repay the effort.
Hence we will simply identify the dualities informally and use them
to structure the cases to be considered. 

We will exploit 2 types of duality: domain dualities which arise through
the existence of useful involutions on underlying sets and groups,
and a parity duality (a duality between odd and even results) which
arises in connection with the study of Continued Fractions.

\subsection{\label{subsec:Dynamical-Systems}Dynamical Systems  }

Although we noted above that Birkhoff sums were studied pre-Birkhoff,
they are normally today associated with Ergodic Theory and Dynamical
Systems, and we will introduce appropriate terminology and notation
here.

A Dynamical System $(X,T_{X})$ is simply a space $X$ (a set $X$
equipped with some suitable algebraic and analytical structure) with
an endomorphism $T:X\rightarrow X$. Given $x_{0}\in X$ (called an
initial condition), the point $x_{N}=T^{N}x_{0}$ is the $N-$th iterate
of $x_{0}$, and the sequence $(x_{1},x_{2},\ldots)$ is the (forward)
orbit of $x_{0}$ (NB it is very convenient NOT to regard $x_{0}$
itself as the first point of the orbit). For $n,m\ge0$, the subsequence
of the orbit $(x_{r})_{r=m+1}^{m+n}$ is called a segment of length
$n$, and an initial segment if $m=0$.

We will focus in this paper on the Dynamical system $(\mathbb{T},R_{\alpha})$
where $\mathbb{T}=[0,1)$ is the circle (a compact 1 dimensional real
manifold) and $R_{\alpha}$ is a rotation through $\alpha_{\mathbb{R}}$
revolutions. Note that $R_{\alpha}=R_{\{\alpha\}}$ 

When $X$ has a distance function $d$ we say $x_{N}$ is a \noun{point
of closest return} to $x_{0}$ if $x_{N}=T^{N}x_{0}$ satisfies $d(x_{0},x_{N})<d(x_{0},x_{r})$
for $r<N$, ie $x_{N}$ is the first point of minimum distance from
$x_{0}$ in the initial orbit segment of length $N$. We call $N$
a \noun{quasiperiod} with quasiperiod error $d(x_{0},x_{N})$. Note
that $1$ is trivially a quasiperiod.

Note also that if $x_{N}=x_{0}$ then: the orbit is periodic, and
$N$ is a multiple of the period $p=\min\{r>0:R^{r}x=x\}$, and the
period is itself a quasiperiod with quasiperiod error $0$. 

We now introduce the concept of an observable of the Dynamical System
as a function $\phi_{XY}$ with values in a semi-group $(Y,\circ_{Y})$.
By \ref{par:Induced-Morphisms} this induces a natural binary function
$\circ_{XY}$ on $XY$. In general $\circ_{Y},\circ_{XY}$ need not
be commutative (as in the theory of co-cycles. where $(Y,\circ)$
is a space of matrices with matrix multiplication). In this paper
we restrict ourselves to the simple case $(Y,\circ)=(\mathbb{\mathbb{R}},+)$,
and so in the sequel an observable on $X$ will be a real function
on $X$, and the set of observables $X\mathbb{R}$ is a commutative
group, namely $(X\mathbb{R},+_{X\mathbb{R}})$. 
\begin{defn}[Birkhoff Sum]
The $N$th Birkhoff sum over the Dynamical System $(X,T)$ of the
observable $\phi_{X\mathbb{R}}$ with initial condition $x$ is $S_{N}^{T}(\phi,x)\coloneqq\sum_{r=1}^{N}\phi(T^{r}x)$.
\end{defn}

We shall find it convenient to define a number of homonymous forms
derived from $S_{N}$. We allow the omission of $T$ or $\phi$ when
these are understood, obtaining the forms $S_{N}(\phi,x)\coloneqq S_{N}^{T}(\phi,x)$,
$S_{N}x\coloneqq S_{N}(\phi,x)$. When $T$ is $R_{\alpha}$, namely
a rotation of the circle through $\alpha_{\mathbb{R}}$ revolutions,
we will also write $S_{N}^{\alpha}\coloneqq S_{N}^{R_{\alpha}}$.
Finally we define the \noun{homogeneous} form $S_{N}\phi\coloneqq S_{N}(\phi,0)$,

Separation of concerns
Let $(X,T)$ be a Dynamical system with a Value system $(Y,+)$ and
Observables system $(XY,+)$.

In (\ref{par:Induced-Morphisms}) we noted that given $T:X\rightarrow X$
and $\text{\ensuremath{\phi}}\in(XY,+)$, that $T$ has a homomorphic
action on $T:XY\rightarrow XY$ defined by 
\[
(T\phi)x=(\phi\circ T)x=\phi(Tx)
\]
For the purposes of this paper, the right hand side is more tractable
since it means we can study a given Dynamical System $(X,T)$ independent
of $\phi$. However our goal is to study Birkhoff sums which have
the form $(S_{N}\phi)x_{0}$. Noting that this has the same form as
the left hand side of (\ref{par:Induced-Morphisms}), it would be
nice to rewrite $(S_{N}\phi)x_{0}$ as $\phi(S_{N}x_{0})$. Unfortunately
$S_{N}x_{0}$ is undefined, but we can easily remedy this by extending
our study from $X$ to $X^{*}$, the free monoid on $X$ (see \ref{par:Induced-Morphisms}),
which provides a richer context in which to work.

Note that we can write any word $w=x_{1}\ldots x_{N}\in X^{*}$ as
a concatenation $\sum_{r=1}^{N}x_{i}$. We can also identify the initial
orbit segment $(T^{r}x_{0})_{r=1}^{N}$ of $(X,T)$ with the concatenation
$\sum_{r=1}^{N}T^{r}x_{0}$.

We now extend the function $\phi:X\rightarrow Y$ to a monoid homomorphism
$\phi:X^{*}\rightarrow Y$ by $\phi(\sum_{r=1}^{N}x_{r})=\sum_{r=1}^{N}\phi x_{r}$.
Similarly we extend each function $T^{r}:X\rightarrow X$ to a monoid
homomorphism $T^{r}:X^{*}\rightarrow X^{*}$ by $T^{r}(\sum_{r=1}^{N}x_{r})=\sum_{r=1}^{N}T^{r}x_{r}$,
noting by (\ref{par:Induced-Morphisms}) that $+$ on $X^{*}$ induces
$+$ on $X^{*}X^{*}$ so that $M=\left((X^{*}X^{*}),+\right)$ is
a monoid of monoid endomorphisms. In particular we can now write $S_{N}=\sum_{r=1}^{N}T^{r}\in M$
with $S_{N}w=\sum_{r=1}^{N}T^{r}w$ for each word $w\in X^{*}$. We
can now use (\ref{par:Induced-Morphisms}) with $U:X^{*}\rightarrow X^{*}\in M$
and $\phi\in(X^{*}Y,+)=M_{2}$ that $U$ has a homomorphic action
$U:M_{2}\rightarrow M_{2}$ by $(U\phi)w=\phi(Uw)$. In particular
taking $U=S_{N},w=x_{0}$ gives us (as desired) $(S_{N}\phi)x_{0}=\phi(S_{N}x_{0})$,
and indeed much more.

Noting that we can now identify $S_{N}x_{0}$ with the initial orbit
segment $(T^{r}x_{0})_{r=1}^{N}$ in the Dynamical System $(X,T)$,
we use the identity above to separate the study of Birkhoff sums $S_{N}\phi x_{0}$
into the study of initial orbit segments, and subsequently the action
of $\phi$ on these segments. We do this in Sections (\ref{sec:Distribution},\ref{sec:Analysis})
respectively.

\subsection{\label{subsec:Circle}Geometry of the circle}

As usual we will adopt the topological quotient group $\mathbb{T}=\mathbb{\mathbb{\mathbb{R}}}/\mathbb{Z}$
as our model of the circle, together with its quotient topology and
(pushed forward from $\mathbb{I}$) Lebesgue measure. Recall that
the elements of $\mathbb{R}/\mathbb{Z}$ are the cosets $x_{\mathbb{\mathbb{R}}}+\mathbb{Z}$.
We write $x_{\mathbb{T}}\coloneqq x_{\mathbb{\mathbb{R}}}+\mathbb{Z}$
and designate the quotient homomorphism $\pi_{\mathbb{\mathbb{\mathbb{R}}}\mathbb{T}}$
so that $\pi(x_{\mathbb{\mathbb{\mathbb{R}}}})=x_{\mathbb{\mathbb{R}}}+\mathbb{Z}=x_{\mathbb{T}}$.
Given $x,x_{0}\in\mathbb{\mathbb{\mathbb{R}}}$, we also regard $\mathbb{T}$
as a compact manifold with the atlas of charts\footnote{Strictly a chart is a be a homeomorphism between open sets, but we
can easily extend these charts to the full manifold. extension to
the full manifold/interval we must first define $\chi_{x_{0}}:\mathbb{T}\backslash\{x_{0}+\mathbb{Z}\}_{Set}\rightarrow\mathbb{I}\backslash\{x_{0}\}_{Set}$
and then extend it to include $x_{0}+\mathbb{Z}\mapsto\{x_{0}\}.$Similarly
for $\chi_{x_{0}}^{\slash}$} $\chi_{x_{0}}^{\mathbb{T}\mathbb{I}}:x_{\mathbb{T}}=x+\mathbb{Z}\mapsto\{x-x_{0}\}$.
Note that $\chi_{x_{0}}^{\mathbb{T}\mathbb{I}}\chi^{\mathbb{\mathbb{\mathbb{R}}}\mathbb{T}}=\chi_{x_{0}}^{\mathbb{\mathbb{\mathbb{R}}}\mathbb{I}}:x\mapsto\{x-x_{0}\}$.
We also have a similar set of charts $\chi_{x_{0}}^{\mathbb{T}\mathbb{I}^{\slash}}:x_{\mathbb{T}}\mapsto\{\{x-x_{0}\}\}$
with $\chi_{x_{0}}^{\mathbb{T}\mathbb{I}^{\slash}}\chi^{\mathbb{\mathbb{\mathbb{R}}}\mathbb{T}}=\chi_{x_{0}}^{\mathbb{\mathbb{\mathbb{R}}}\mathbb{I}^{\slash}}:x\mapsto\{x-x_{0}\}$.
When $x_{0}=0$ we will drop the subscript from each of the maps $\chi$
and regard these as canonical maps. In particular we take $\chi^{\mathbb{T}\mathbb{I}},\chi^{\mathbb{T}\mathbb{I}^{\slash}}$
as a canonical atlas, and call $x_{\mathbb{I}}=\chi^{\mathbb{T}\mathbb{I}}(x_{\mathbb{T}}),x_{\mathbb{I}^{\slash}}=\chi^{\mathbb{T}\mathbb{I}^{\slash}}(x_{\mathbb{T}})$
the natural coordinate, and signed coordinate of $x_{\mathbb{T}}$
respectively.

Note the transition map or change of coordinates $\chi_{x_{1}}^{\mathbb{T}\mathbb{I}}\left(\chi_{x_{0}}^{\mathbb{T}\mathbb{I}^{\slash}}\right)^{-1}x=\{x-(x_{1}-x_{0})\}$,
and so for the canonical coordinates the morphism $\chi_{x_{1}}^{\mathbb{T}\mathbb{I}}\left(\chi_{x_{0}}^{\mathbb{T}\mathbb{I}^{\slash}}\right)^{-1}$
lies in the hom-set $\mathbb{I}^{\slash}\mathbb{I}$ and its inverse
lies in $\mathbb{I}\mathbb{I}^{\slash}$.

The natural order on $\mathbb{\mathbb{\mathbb{R}}}$and $\mathbb{I}$
is ambiguous on $\mathbb{T}$. Whilst it is normal to define circle
order as a ternary relation (without reference to coordinates), it
will suit us better to use the order induced from the natural coordinate,
so that we define $x_{\mathbb{T}}\le y_{\mathbb{T}}\,\Leftrightarrow\,0\le x_{\mathbb{I}}\le y_{\mathbb{I}}<1$.
We will then say that a sequence $\left(x_{i}\right)_{i=1}^{n}$ for
$n\ge3$ is in cyclic order if there is a rotation $R_{\alpha}$ under
which the full sequence is placed in natural order, ie $\{x_{i}+\alpha\}_{\mathbb{I}}\le\{x_{i+1}+\alpha\}_{\mathbb{I}}$
for $i=1..n-1$. We say $\left(x_{i}\right)_{i=1}^{n}$ for $n\ge3$
is in anti-cyclic order if the reverse sequence $\left(x_{i}\right)_{i=n}^{1}$
is in cyclic order. A sequence which is in cyclic or anti-cyclic order
is simply ordered. A sequence of length $0\le n\le2$ is always both
in cyclic and anti-cyclic order. 

The ambiguity of order means that a pair of distinct circle points
$a_{\mathbb{T}}<b_{\mathbb{T}}$ identifies 2 circle intervals defined
by the direction of travel from $a$ to $b$. We will find it useful
to have 2 ways of resolving this ambiguity. First we use the coordinate
order: the cyclic interval $\mathbb{I}$ contains interior points
$x$ such that $(a,x,b)$ are in cyclic order; and the anti-cyclic
interval contains $x$ such that $(a,x,b)$ are in anti-cyclic order.
We call these directed intervals, and if we do not specify a direction,
we will take cyclic by default. Secondly we define undirected intervals:
the minor interval is the one of smaller measure, the major interval
is the one of larger measure. If the measure of both is precisely
$1/2$, we take the minor interval to be the cyclic directed interval,
although we will not make use of such intervals in this paper. Note
that for directed intervals $(a,b)\ne(b,a)$ whereas for undirected
intervals (of measure not $\frac{1}{2}$) we have $(a,b)=(b,a)$ for
both minor and major intervals.

In addition to open and closed intervals, we are equally concerned
with semi-open intervals. We will use the usual notation of open $()$
and closed $[]$ brackets to indicate whether endpoints are included
or not. We will also use $|$ to avoid specifying which is the case,
so that $|a,b|$ represents a particular but unspecified one of the
4 possible interval configurations of endpoints.

We will also use $1$ as a synonym for $0$ where this aids readability,
so $[\frac{3}{4},1]$ represents the subset $[\frac{3}{4},1)\cup\{0\}$in
$\mathbb{I}$, and the cyclic directed interval $[\frac{3}{4},0]$
in $\mathbb{T}$.

Any circle interval $J$ is also a topological subspace of $\mathbb{T}$
equipped with the subspace topology. One subtlety to note is that
a subspace may have boundary points in the parent topology, but these
are not boundary points in the subspace topology: for example if $J=[a,b)$
then the interval $[a,c_{J})$ is \emph{open} (not \emph{semi-open})
in the subspace topology of $J$. With this in mind, it becomes simple
to extend the well known characterisation about open sets of $\mathbb{R}$
to arbitrary real or circle intervals. We include a proof for completeness:
\begin{prop}
\label{prop: Countable intervals}In the subspace topology of an interval
$J$ of either $\mathbb{R}$ or $\mathbb{T}$, an open set is an at
most countable union of disjoint open intervals 
\end{prop}

\begin{proof}
Let $X$ be an open set in the topology of $J$. For each $x_{X}$
let $I_{x}$ be the maximal interval in $X$ containing $x$. If $I_{x},I_{y}$
are not disjoint, their union is an interval, so their maximality
requires $I_{x}=I_{x}\cup I_{y}=I_{y}$. Hence $X$ is a union of
disjoint intervals. However $X$ is also open and so contains no boundary
points. Hence each $I_{x}$ is an open interval. Finally each interval
with interior contains a rational, so any family of disjoint intervals
with interior is at most countable. 
\end{proof}
A partition of the circle of length $n\ge1$ is a an ordered sequence
$\left(x_{i}\right)_{i=1}^{n}$. A partition of a circle interval
$J$ is a partition of the circle, all of whose points lie in $J$.
(Note that this is a slight generalisation of the usual definition:
we do not require $J$ to be closed, nor do we require the endpoints
of $J$ to be included in the partition).

We introduce the natural involution on $\mathbb{T}$ define by $\overline{x_{\mathbb{T}}}=-x_{\mathbb{T}}$
which induces coordinate involutions $\overline{\{\chi x_{\mathbb{T}}\}}=\{\chi\overline{x_{\mathbb{T}}}\}=\{-x_{I}\}=\{1-x\}$
and $\overline{\{\{\chi^{\slash}x_{\mathbb{T}}\}\}}=\{\{\chi^{\slash}\overline{x_{\mathbb{T}}}\}\}=\{\{-x_{\mathbb{I}^{\slash}}\}\}=-\{\{x_{\mathbb{I}^{\slash}}\}\}$.
Note that for $x_{\mathbb{I}}\ne0$ we have $\{1-x\}=1-x$, whereas
$\{1-0\}=0$. The duality which results from application of this involution
we call \noun{Circle Duality}.

\label{Rotations}We now introduce the real translation $R_{\alpha}^{\mathbb{\mathbb{\mathbb{RR}}}}:x\mapsto x+\alpha$.
This induces the translations (which we will here call rotations)
$R_{\alpha_{\mathbb{T}}}^{\mathbb{TT}}:x_{\mathbb{T}}\mapsto x_{\mathbb{T}}+\alpha_{\mathbb{T}}$
and $R_{\alpha_{\mathbb{I}}}^{\mathbb{I}\mathbb{I}}:x_{\mathbb{I}}\mapsto\{x_{\mathbb{I}}+\alpha_{\mathbb{I}}\},R_{\alpha_{\mathbb{I}^{\slash}}}^{\mathbb{I}^{\slash}}:x_{\mathbb{I}^{\slash}}\mapsto\{\{x_{\mathbb{I}}+\alpha_{\mathbb{I}}\}\}$.
Since each of these has identical structure, we will abuse notation
and regard all of them as $R_{\alpha}$. 

Note that $\overline{R_{\alpha}}=R_{\overline{\alpha}}=R_{1-\alpha}$
(and $R_{\alpha}=R_{\{\alpha\}}=R_{\alpha+n_{\mathbb{\mathbb{Z}}}}$)

Note that in these systems we have $R_{\alpha}^{N}x=\{x+N\alpha\}$
and $S_{N}(\phi,x,R_{\alpha})=\sum_{r=1}^{N}\phi(\{x+r\alpha\})$.
Note we can always lift $\phi:\mathbb{T}\rightarrow\mathbb{\mathbb{\mathbb{\mathbb{R}}}}$
to $\phi^{\uparrow}:\mathbb{\mathbb{\mathbb{\mathbb{\mathbb{R}}}\rightarrow}\mathbb{\mathbb{\mathbb{R}}}}$
by defining $\phi^{\uparrow}(x)=\phi(\{x\})$, and by abuse of notation
we will simply write $\phi$ for $\phi^{\uparrow}$ so for example
we can simply write $S_{N}(\phi,x,R_{\alpha})=\sum_{r=1}^{N}\phi(x+r\alpha)$.
We call $S_{N}(\phi,x,R_{\alpha})$ a homogeneous sum for $\{x\}=0$,
and an inhomogeneous sum for $\{x\}\ne0$.

\label{DSDuality}Duality results:. $S_{N}^{\overline{T}}\overline{\phi}\overline{x}=S_{N}^{T}\phi x$
$S_{N}(-\phi)x=-S_{N}\phi x$ $S_{N}^{\overline{T}}(-\overline{\phi})\overline{x}=-S_{N}^{T}\phi x$

A partition of the circle of length $q\ge1$ is a sequence of $q$
circle points with coordinates satisfying $0\le x_{1}<x_{2}\ldots<x_{q}<1$.
Note that any sequence of distinct points can be reordered to give
a partition. The intervals of the partition are the images under $\chi_{\mathbb{T}\mathbb{I}}^{-1}$
of the $q-1$ intervals $(x_{i},x_{i+1})$ for $r<q$, and the interval
$(x_{q},1)\bigcup(0,x_{1})$.

The partition is regular if $x_{r+1}-x_{r}=\frac{1}{q}$ for $1\le r<q$,
and also (if $q>1)$ $\{x_{1}-x_{q}\}=\frac{1}{q}$. We will only
deal with regular partitions, so we will drop the word regular. Note
that a partition of length 1 is always regular.

\subsection{Number Theory}

Algebraic Context
Given some set $A\subseteq\mathbb{\mathbb{\mathbb{R}}}$ we define
$x_{\mathbb{R}}A=\{xa:a\in A\}$. If $A$ is an additive group then
so is $xA$ and there is a quotient group $\mathbb{R}/xA$ which is
also an additive group with elements $y_{(\mathbb{R}/xA)}=y+xA\coloneqq\{y+xa:a\in A\}$.
Further $xA$ is also closed under multiplication iff $z\in\mathbb{Z}$.
When $z\in\mathbb{Z}$, $z\mathbb{Z}$ is a commutative ring (without
unit for $|z|>1$), and the quotient $\mathbb{Z}/z\mathbb{Z}$ is
a commutative ring with unit (a field if $z$ is a prime, but with
zero divisors if $z$ is composite).

In this paper we will work primarily in the projective closure of
the positive reals, ie the interval $\mathbb{P}=[0,\infty]$ equipped
with the involution $R:x\mapsto\frac{1}{x}$, where we define $\infty\coloneqq R0$
(and hence $R\infty=0$, and $R[0,1]=[1,\infty]$). We also define
$\infty$ to be an integer, together with the usual operator extensions
$\infty+x_{\mathbb{P}}\coloneqq\infty,\,\infty.x_{\mathbb{P}}\coloneqq\infty,\,x_{\mathbb{P}}\le\infty$.
\label{subsec:NumFuncs}Number Functions
The most important group of functions for our purposes are the remainder
(residue) and floor functions. Recall $\mathbb{I}\coloneqq[0,1)$
and for $x_{\mathbb{R}}\ne0$ $x\mathbb{I}=[0,x)$.

For each $z_{\mathbb{R}}\ne0$ we define the $z-$\noun{remainder}
function $\{.\}_{z}:\mathbb{R}\rightarrow\left|z\right|\mathbb{I}$
by $\{x\}_{z}\coloneqq\min\{x-k_{\mathbb{Z}}z:kz\le x\}$, ie $\{x\}_{z}$
is the positive remainder after integer division by $z$. Note as
$z\rightarrow0$, $\{x\}_{z}\rightarrow0$ and so we extend to $\{x\}_{0}=0$.
Note $\{x+nz\}_{z}=\{x\}_{z}$ so there is a natural bijection $x+z\mathbb{Z}\longleftrightarrow\{x\}_{z}$,
ie $\mathbb{\mathbb{\mathbb{R}}}/z\mathbb{Z}\leftrightarrow|z|\mathbb{I}$.
This means we can give $|z|\mathbb{I}$ group structure via $x+_{z\mathbb{I}}\,y\coloneqq\{x+y\}_{z}$.
Similarly $\mathbb{Z}/q\mathbb{Z}\longleftrightarrow|q|\mathbb{I}\bigcap\mathbb{Z}$
and we can add distributive multiplication via $r\cdot_{|q|\mathbb{I}}s=\{rs\}_{q}$.

The \noun{signed remainder} is $\{\{x\}\}_{z}=\{x\}_{z}-\left\llbracket \{x\}_{z}\ge\frac{\left|z\right|}{2}\right\rrbracket \left|z\right|\in\left|z\right|\mathbb{I}^{\slash}$.
The remainder and signed remainder are inverse bijections $|z|\mathbb{I}^{\slash}\longleftrightarrow|z|\mathbb{I}$.
Together they provide an involution on $|z|\left(\mathbb{I}\bigcup\mathbb{I}^{\slash}\right)$
which fixes $|z|[0,\frac{1}{2})$ and transposes $[-\frac{1}{2},0)$
with $[\frac{1}{2},1)$. 

We also define the \noun{floor} function $\left\lfloor x\right\rfloor _{z}\coloneqq x-\{x\}_{z}$
and the \noun{smallest distance} from $x$ to $z\mathbb{Z}$ given
by $\left\Vert x\right\Vert _{z}=\left|\{\{x\}\}_{z}\right|\in\frac{|z|}{2}\mathbb{I}$.
Note that the latter is a metric on $\mathbb{I}$ though not a norm
(it fails the requirement of scalar multiplication). It is useful
to extend these functions by defining $\{\infty\}=0$ from which we
also get $\{\{\infty\}\}=0,\left\lfloor \infty\right\rfloor =\infty,\left\Vert \infty\right\Vert =0$.

When $z=1$ we will omit it from the notation. The functions $\{x\},\{\{x\}\}$
are then also called \noun{fractional parts} of $x$. 

We can use $\{x\}_{z}$ to derive other common functions:

The \noun{highest common factor} is the binary operation $(x_{\mathbb{R}},y_{\mathbb{R}})\coloneqq\max\{z_{\mathbb{R}}:\{x\}_{z}=0\&\{y\}_{z}=0\}$.
Note $(x,x)=(x,0)=\left|x\right|$.

The reflexive \& transitive relation '\noun{divides}' is given by
$z|x\coloneqq\;\{x\}_{z}=0$ (so we allow $0|x$), and the equivalence
relation '\noun{congruent}' by $x\equiv y\bmod z\;\coloneqq\;\{x\}_{z}=\{y\}_{z}$.
Now $\{x\}_{z}=\{y\}_{z}\Longleftrightarrow x+z\mathbb{Z}=y+z\mathbb{Z}$
so that there is a natural bijection between $[0,z)$ and $\mathbb{\mathbb{\mathbb{R}}}/z\mathbb{Z}$
giving $[0,z)$ a group structure by $x_{[0,z)}\leftrightarrow[x]$,
and a ring structure on $\{r\}_{r=0}^{q-1}$ induced from $\mathbb{Z}/q\mathbb{Z}$.

The signum function $\sgn$ maps $\alpha_{\mathbb{R}}$ to its sign:$\sgn(x_{\mathbb{R}})=\left\llbracket x>0\right\rrbracket -\left\llbracket x<0\right\rrbracket $.
(Note $\sgn(0)=0$).
Residues
When $z,x$ are natural numbers we call $\{x\}_{z}$ a residue of
$z$, noting that there are just $z$ possible residues, namely the
set $\{i_{\mathbb{N}}:0\le i<z\}$. We will also tend to use We have
basic results, if $(p_{\mathbb{N}},q_{\mathbb{N}})=1$ then $\left\{ \,\{r_{\mathbb{N}}p\}_{q}\,\right\} _{r=1}^{q-1}=\{r\}_{r=1}^{q-1}$,
ie the postfix operator $\times p$ induces a permutation on non-zero
residues of $q$ - in particular for $q>1$ there is a unique $1\le r_{p}\le q-1$
with $\{r_{p}p\}_{q}=1$ which we designate $\{p\}_{q}^{-1}$.

$\{.\}_{q}$ induces $+_{q\mathbb{I}}$ by: $\{x\}_{q}+_{q\mathbb{I}}\{y\}_{q}=\{x+y\}_{q}$
- need to show indt of choice of $x,y$.

$\{x+y\}_{q}=\left\{ \{x\}_{q}+\{y\}_{q}\right\} _{q}=\{x\}_{q}+\{y\}_{q}-\left\llbracket \{x\}_{q}+\{y\}_{q}\ge|z|\right\rrbracket |z|$.
$\{rs\}_{q}=\left\{ \{r\}_{q}\{s\}_{q}\right\} _{q}$

Diophantine Approximation
Given $a_{\mathbb{R}}$, the theory of Continued Fractions provides
a sequence $\left(\frac{p_{r}}{q_{r}}\right)_{r=0}^{\infty}$ (\noun{convergents})
of exponentially improving rational approximations of $a$. These
have long provided a primary tool in our area of study, and there
are many good introductions, eg Hardy \& Wright\cite{HardyWright1975BookNumTheory}.
We will give a very brief summary from the slightly novel viewpoint
of our needs in this paper.

Recall that we can split any $x_{\mathbb{R}}$ into its integer and
fractional parts, ie $x=\left\lfloor x\right\rfloor +\left\{ x\right\} $.
In the sequel it will be useful to adopt the convention that symbols
with $.^{\slash}$ indicate real numbers, and symbols without indicate
integer numbers. In addition we will regard $\infty$ as a positive
integer reciprocal with $0$, so that $\frac{1}{0}\coloneqq\infty=\left\lfloor \infty\right\rfloor +\frac{1}{\infty}$.
An integer, real therefore means an element of $\widehat{\mathbb{Z}}=\mathbb{Z}\bigcup\{\infty\}$,$\widehat{\mathbb{R}}=\mathbb{R}\bigcup\{\infty\}$
respectively.

We start with a modernisation of the method of \noun{anthyphairesis}
(alternating subtraction) developed by the early Greeks, probably
by the School of Pythagoras, and later a central component of what
we now call Euclid's algorithm. 
\begin{defn}
The \noun{anthyphairetic relation} is the recurrence relation $a_{n}^{\slash}=a_{n}+\frac{1}{a_{n+1}^{\slash}}$
where $a_{n}=\left\lfloor a_{n}^{\slash}\right\rfloor $. Given an
initial real number $a^{\slash}$ and setting $a_{0}^{\slash}=a^{\slash}$
the relation defines (for $n\ge0$) 2 infinite sequences $(a_{n}^{\slash}),(a_{n})$
of real numbers and their integer parts respectively. \footnote{Technical note: Classically, expositions distinguish between finite
sequences (for $a_{0}^{\slash}$ rational) and infinite sequences
(for $a_{0}^{\slash}$ irrational). For our purposes it is simpler
to regard all sequences as infinite, extending finite sequences with
the integer $\infty$. In practice this makes little difference to
the discussion.} For historical reasons, we call $(a_{n}^{\slash})$ the sequence
of \noun{complete quotients} of $a^{\slash}$, and $(a_{n})$ the
sequence of \noun{partial quotients}. The integer sequence $(a_{n})$
is also called the \noun{anthyphairesis} of the real number $a^{\slash}.$ 

Note for $n\ge0$ that $\frac{1}{a_{n+1}^{\slash}}$ is the fractional
part of $a_{n}^{\slash}$, and so $1<a_{n+1}^{\slash}\le\infty$ which
also means $a_{n+1}\ge1$. Also by definition $a_{n}\le a_{n}^{\slash}<a_{n}+1$
unless $a_{n}^{\slash}=a_{n}=\infty$. 
\end{defn}

In modern terms, the relation defines the anthyphairetic map $\mathscr{A}:a^{\slash}\mapsto(a_{n})$
of real numbers to sequences of integers. It is easy to see that this
map is injective, and so it provides an invertible encoding of reals
by integers. The map is not surjective: if $a_{n}=\infty$ then also
$a_{n+1}=\infty$ and $a_{n-1}\ne1$ (with the single exception of
$A(1)=(1,\infty,\infty,\ldots)$). The image of $\mathscr{A}$ is
therefore set of anthyphaireses, a subset of all integer sequences.

The natural question arises as to whether we can construct the inverse
map $A^{-1}$ from the set of anthyphaireses to $[1,\infty]$. Note
first that using the anthyphairetic relation recursively gives us
$a_{0}^{\slash}=a_{0}+\frac{1}{a_{1}^{\slash}}=a_{0}+\frac{1}{a_{1}+\frac{1}{a_{2}^{\slash}}}=...$.
These expressions are called \noun{Continued Fractions}, although
a better modern term might be recursive fractions, namely fractions
$\frac{a}{b}$ in which $a,b$ may themselves be recursive fractions.
Let us write the $n-$th Continued Fraction as $CF_{n}^{\slash}(a_{0},\ldots,a_{n-1},a_{n}^{\slash})$,
noting that if we simplify the continued fraction back to a normal
fraction we will end up with $a_{0}^{\slash}=CF_{n}^{\slash}(a_{0},\ldots,a_{n}^{\slash})=\frac{P_{n}(a_{0},\ldots,a_{n}^{\slash})}{Q_{n}(a_{0},\ldots,a_{n}^{\slash})}$,
a rational polynomial in $a_{0},\ldots,a_{n}^{\slash}$. If $a_{0}^{\slash}$
is irrational, so is $a_{n}^{\slash}$ and we have simply shifted
the problem of constructing the inverse from $a_{0}^{\slash}$ to
$a_{n}^{\slash}.$\footnote{There is an important exception: if the anthyphairesis is eventually
periodic, Euler/Lagrange showed that $a_{0}^{\slash}$ is a root of
a quadratic equation which can be recovered from the sequence.}). 

However if $a_{n}^{\slash}$ is itself an integer, $a_{0}^{\slash}$
is rational and easy to calculate, and so we the effects of replacing
$a_{n}^{\slash}$ with its integer part $a_{n}$. This gives us a
derived sequence of Continued Fractions $CF_{n}=CF_{n}^{\slash}(a_{0}...a_{n})$,
each of which now simplifies to some normal fraction $\frac{p_{n}}{q_{n}}$.
For example $CF_{0}=\frac{p_{0}}{q_{0}}=a_{0}$, $CF_{1}=\frac{p_{1}}{q_{1}}=a_{0}+\frac{1}{a_{1}}=\frac{a_{0}a_{1}+1}{a_{1}}$,
$CF_{2}=\frac{p_{2}}{q_{2}}=a_{0}+\frac{1}{a_{1}+\frac{1}{a_{2}}}=\frac{a_{0}a_{1}a_{2}+a_{0}+a_{2}}{a_{1}a_{2}+1}$
etc. 

The remarkable results of Continued Fraction theory are now that 
\begin{equation}
a_{0}^{\slash}-\frac{p_{n}}{q_{n}}=\frac{(-1)^{n}}{q_{n}\left(a_{n+1}^{\slash}q_{n}+q_{n-1}\right)}\label{eq:CFApproximation}
\end{equation}
This means that $\left(\frac{p_{n}}{q_{n}}\right)$ is a sequence
whose limit is $a_{0}^{\slash}$, ie $A^{-1}(a_{r})=\lim_{n\rightarrow\infty}\frac{p_{n}}{q_{n}}$,
and we call $\frac{p_{n}}{q_{n}}$ a \noun{convergent} of $a_{0}^{\slash}$.
Further, these are the best possible rational approximations given
the size of denominator, meaning that $\left|a_{0}^{\slash}-\frac{p}{q}\right|<\left|a_{0}^{\slash}-\frac{p_{n}}{q_{n}}\right|$
requires $q>q_{n}$. Indeed each approximation is actually stronger:
the (stricter) inequality $\left|qa_{0}^{\slash}-p\right|<\left|q_{n}a_{0}^{\slash}-p_{n}\right|$
in fact requires $q\ge q_{n+1}$.

We will need two other basic results of Continued Fraction theory.
We assume each $\frac{p_{n}}{q_{n}}$ is in lowest terms. Then first
there is an invariant of the sequence given by the identity:
\begin{equation}
p_{n+1}q_{n}-p_{n}q_{n+1}=(-1)^{n}\label{eq:pnqn}
\end{equation}
Second, $(p_{n}),(q_{n})$ both satisfy the \noun{continuant} recurrence
relation $r_{n}=a_{n}r_{n-1}+r_{n-2},$which means:
\begin{align*}
p_{n} & =a_{n}p_{n-1}+p_{n-2}\\
q_{n} & =a_{n}q_{n-1}+q_{n-2}
\end{align*}
Note that from $\frac{p_{0}}{q_{0}}=a_{0}$ we get $p_{0}=a_{0},q_{0}=1$
and from $\frac{p_{1}}{q_{1}}=a_{0}+\frac{1}{a_{1}}=\frac{a_{0}a_{1}+1}{a_{1}}$
we get $p_{1}=a_{0}a_{1}+1$,$q_{1}=a_{1}$. Noting the recurrence
relation can also be written in descending form as $r_{n-2}=r_{n}-a_{n}r_{n-1}$,
we can then deduce the simpler initial conditions: 
\begin{align*}
p_{-1} & =1,p_{-2}=0\\
q_{-1} & =0,q_{-2}=1
\end{align*}

Continued Fraction theory began from the Greeks' algebraic interest
in the relations between natural numbers, rationals and reals, and
this algebraic focus continued into exploring their connection with
the solution of Diophantine equations (ie solutions in integers or
rationals). Our interest by contrast is in using the results of the
theory to study Dynamical Systems consisting of rotations of the circle.
This changes our view of the first class citizens of the theory. So
whereas the first class citizens in classical theory are rational
convergents and approximation errors, ours are periods of closest
return, and the periods of the associated return errors. Let us unpack
this.

Given a Dynamical System $(X,T)$ with $(X,d)$ a metric space, recall
that a period of closest return (or quasiperiod) to $x_{0}$ is an
integer $n>0$ such that $d(T^{m}x_{0},x_{0})<d(T^{n}x_{0},x_{0})$
requires $m>n$. In the system $(\mathbb{T},R_{a_{0}^{\slash}})$
(being rotations of the circle through $a_{0}^{\slash}$ revolutions),
then the ``best possible'' result of classical Continued Fraction
theory are equivalent to the result that each $q_{n}$ is a quasiperiod.
The approximation error $q_{n}a_{0}^{\slash}-p_{n}=\frac{(-1)^{n}}{a_{n+1}^{\slash}q_{n}+q_{n-1}}$
is equivalent to saying that $q_{n}$ rotations through $a_{0}^{\slash}$
returns to a (directed) distance of $\frac{(-1)^{n}}{a_{n+1}^{\slash}q_{n}+q_{n-1}}$
revolutions from the starting point, ie it is the return error, which
has an \noun{error period} of $a_{n+1}^{\slash}q_{n}+q_{n-1}$. Since
this will be of primary interest we will give it its own notation:
motivated by the continuant relation for quasiperiods, namely $q_{n+1}=a_{n+1}q_{n}+q_{n-1}$,
we will denote the $n-$th error period $q_{n+1}^{\slash}\coloneqq a_{n+1}^{\slash}q_{n}+q_{n-1}$.\footnote{Historical note: the notation $q_{n+1}^{\slash}$ was in fact used
by Hardy \& Wright\cite{HardyWright1975BookNumTheory}, but purely
as a convenient local shorthand for the expression $a_{n+1}^{\slash}q_{n}+q_{n-1}$
in the proof of the inequality $\left|\alpha-\frac{p_{n}}{q_{n}}\right|<\frac{1}{q_{n}q_{n+1}}$.
It does not seem to have been named, or used in further theory development.
In fact Lang\cite{Lang1995introduction} does not use any shorthand
at all, always writing $a_{n+1}^{\slash}q_{n}+q_{n-1}$ in full. We
are not aware of distinct names being introduced for $(q_{n}),(q_{n}^{\slash})$
previously.}

In summary, the sequences $a_{n},a_{n}^{\slash}$ are first class
citizens of both classical theory and of the theory in this paper.
However classical theory then focuses on the convergents $\frac{p_{n}}{q_{n}}$
and their numerators/denominators, whereas our focus is on the periods
$q_{n},q_{n}^{\slash}$. For convenience we give a formal definition:
\begin{defn}
\label{def:Quasiperiods}Given a real number $a_{0}^{\slash}$ with
full quotient sequence $(a_{n}^{\slash})$ and partial quotient sequence
$(a_{n})$, we also define for $n\ge0$ the \noun{quasiperiod} sequence
of $a_{0}^{\slash}$ as $(q_{n}=a_{n}q_{n-1}+q_{n-1})$ and the \noun{error-period}
sequence of $a_{0}^{\slash}$ as $(q_{n}^{\slash}=a_{n}^{\slash}q_{n-1}+q_{n-1})$
where $q_{-1}=0,q_{-2}=1$. Although of less interest, we also define
$p_{n},p_{n}^{\slash}$ analogously. 
\end{defn}

Finally we remark that a rotation through $a_{0}^{\slash}\ge1$ is
equivalent to a rotation through $\alpha=\{a_{0}^{\slash}\}<1$. In
this paper we focus on the dynamics of irrational rotations, and so
we will adopt the convention that the symbol $\alpha$ is used to
represent a real number (usually irrational) with $0\le\alpha\coloneqq\{a_{0}^{\slash}\}<1$,
noting that $A(\alpha)=CF(0,a_{1},a_{2}\ldots)$, ie $a_{0}=0$ and
hence also $\alpha=a_{0}^{\slash}(\alpha)=\frac{1}{a_{1}^{\slash}}$.
It also gives $p_{0}=0,p_{1}=1$.

\begin{defn}[Diophantine type functions]
\label{def:Types}In estimating we find that it is useful to establish
bounds on the growth of the sequences $(q_{n}),(q_{n}^{\slash})$.
Since $q_{n}=a_{n}q_{n-1}+q_{n-2}$ we have for $n\ge1$ $a_{n}q_{n-1}\le q_{n}<(a_{n}+1)q_{n-1}$
so that $\prod_{r=1}^{n}a_{r}\le q_{n}<\prod_{r=1}^{n}(a_{r}+1)$
with equality only for $n=1$. Define $a_{n}^{\max}=\max_{r\le n}a_{r}$
then $q_{n}<(a_{n}^{\max}+1)^{n}$. Note $a_{n}^{\max}$ is an increasing
positive sequence. Define $a_{n}^{\max},a_{n}^{\slash\max}$ and $A_{n}=\max_{r\le n}\left(\frac{q_{r}}{q_{r-1}}\right)$
and $A_{n}^{\slash}=\max_{r\le n}\left(\frac{q_{r}^{\slash}}{q_{r-1}}\right)$.
Note $a_{n}^{\max}<a_{n}^{\slash\max}<a_{n}^{\max}+1$ and $A_{n}<A_{n}^{\slash}<A_{n}+1$
$a_{n}^{\max}<A_{n}<a_{n}^{\max}+1$ $a_{n}^{\slash\max}<A_{n}^{\slash}<a_{n}^{\slash\max}+1$.
All are $O(a_{n}^{\max})$ and $A_{n}^{\slash}<a_{n}^{\max}+2$.
\end{defn}

Estimates relating $n,q_{n}$
Since $q_{n}=a_{n}q_{n-1}+q_{n-2}$ we get $q_{n}=\left(a_{n}a_{n-1}+1\right)q_{n-2}+q_{n-3}\ge2q_{n-2}$
for $n\ge2$ (equality only for $n=2)$. Hence for $n$ even $q_{n}\ge2^{n/2}$
for $n\ge2$ and for $n$ odd $q_{n}\ge2^{(n-1)/2}q_{1}$ for $n\ge1$.
Hence $n\le\frac{2}{\log2}\log q_{n}$ and $n\le1+\frac{2}{\log2}(\log q_{n}-\log q_{1})$
for $n$ odd. 

Writing $\phi\coloneqq\frac{1}{2}\left(\sqrt{5}+1\right)$ for the
golden ration, we can get a slightly better estimate by observing
that $q_{n}$ increases with increasing $a_{n}$ and so the lowest
quasiperiods are those of the golden rotation, amely $\phi-1=\frac{1}{2}\left(\sqrt{5}-1\right)$
with $a_{i}=1$. Hence $q_{n}\ge F_{n+1}=\frac{1}{\sqrt{5}}\left(\phi^{n+1}-(-\phi)^{-(n+1)}\right)$
giving $n\le\frac{1}{\log\phi}\left(\log q_{n}+\log\sqrt{5}-\log\left(1-\phi^{-2(n+1)}\right)\right)-1$.
Since $\phi^{-2(n+1)}$ is a reducing function with $n$ we have for
$n\ge1$, $n\le\frac{1}{\log\phi}\left(\log q_{n}+\log\sqrt{5}-\log\left(1-\phi^{-4}\right)\right)-1$.
Now $1-\phi^{-4}=\phi^{-2}(\phi^{2}+\phi^{-2})=\phi^{-2}F_{2}\sqrt{5}$
and $F_{2}=1$ so that $\log\sqrt{5}-\log\left(1-\phi^{-4}\right)=2\log\phi$.
So we have $n\le\frac{\log q_{n}}{\log\phi}+1$ for $n\ge1$ with
equality only for $n=1$ and $q_{1}=1$ (and in fact the strict inequality
still holds for $n=0$).

Also $q_{n}=\prod_{r=1}^{n}\frac{q_{r}}{q_{r-1}}\le\prod_{r=1}^{n}A_{r}\le\left(A_{n}\right)^{n}$.
Hence 
\begin{equation}
\frac{\log q_{n}}{\log A_{n}}\le n\le\frac{\log q_{n}}{\log\phi}+1\label{eq:Estn}
\end{equation}
where the left inequality holds for $q_{n}>1$ (and hence also $A_{n}>1)$,
and the right inequality holds for $n\ge0$.\footnote{In this context it is tempting to index the $a_{r}$ from 0 so $\alpha=[0;a_{0},a_{1},\ldots]$
as this gives the nicely corresponding sequences $(a_{r})_{r=0}^{\infty},(b_{r})_{r=0}^{\infty},(q_{r})_{r=0}^{\infty}$
with $b_{r}\le a_{r}$. However this would be a significant departure
from a strong historic convention which works well in a wider context.
We will therefore adhere to the established convention so that $\alpha=[0;a_{1},a_{2},\ldots]$
and $q_{r+1}=a_{r+1}q_{r}+q_{r-1}$$a_{r},q_{r}$, as elsewhere in
the literature. } 

Continued Fractions Preliminary Results
We will need some additional simple results in following sections.
Given the maturity of CF theory, it is likely that these results are
not new, however we have not been able to find them (perhaps they
are hidden by notational differences), and so we will give a proof
here.
\begin{lem}
$p_{n}\equiv(-1)^{n+1}q_{n-1}^{-1}\bmod q_{n}$ where $r^{-1}$ signifies
the inverse of $r\bmod q$, ie $0\le r^{-1}<q$ with $rr^{-1}\equiv1\bmod q$.
For $\alpha=a_{0}^{\slash}<1$ this is equivalent to $p_{n}=q_{n-1}^{-1}$
($n$ odd) and $p_{n}=q_{n}-q_{n-1}^{-1}$ ($n$ even). 
\end{lem}

\begin{proof}
The first part follows directly from expressing (\ref{eq:pnqn}) $\bmod q_{n}$.
 For the second part, note for $\alpha<1$ we have $p_{n}<q_{n}$,
and the result follows.
\end{proof}
\begin{lem}
We have the equalities $q_{n}^{\slash}=\prod_{1}^{n}a_{r}^{\slash}$
and $\frac{p_{n}^{\slash}}{q_{n}^{\slash}}=a_{0}^{\slash}$
\end{lem}

\begin{proof}
Using the recurrence relations we have $q_{n+1}^{\slash}=a_{n+1}^{\slash}q_{n}+q_{n-1}=a_{n+1}^{\slash}\left(a_{n}q_{n-1}+q_{n-2}\right)+q_{n-1}$.
The anthyphairetic relation gives $a_{n+1}^{\slash}a_{n}+1=a_{n+1}^{\slash}a_{n}^{\slash}$
and hence $q_{n+1}^{\slash}=a_{n+1}^{\slash}\left(a_{n}^{\slash}q_{n-1}+q_{n-2}\right)=a_{n+1}^{\slash}q_{n}^{\slash}$.
This gives us $q_{n}^{\slash}=\left(\prod_{1}^{n}a_{r}^{\slash}\right)q_{0}^{\slash}$.
Now $q_{0}^{\slash}=a_{0}^{\slash}q_{-1}+q_{-2}=1$, and the first
result follows. A similar argument gives $p_{n}^{\slash}=\left(\prod_{1}^{n}a_{r}^{\slash}\right)p_{0}^{\slash}$
and $p_{0}^{\slash}=a_{0}^{\slash}p_{-1}+p_{-2}=a_{0}^{\slash}$,
and the second result follows. 
\end{proof}
\begin{defn}
We define the upper Gauss map $\Gamma:[0,\infty]\rightarrow[1,\infty]$
by $x\mapsto\frac{1}{x-\left\lfloor x\right\rfloor }$, noting this
defines the left shift on the sequence of full quotients, ie $\Gamma(a_{n}^{\slash})=a_{n+1}^{\slash}$.\footnote{It is also the conjugate of the lower Gauss map $\gamma:[0,1]\rightarrow[0,1]$
by $\Gamma=R\gamma R$ where $R$ is the involution $x\rightarrow\frac{1}{x}$
on $[0,\infty]$ } 
\end{defn}

Recall we define the dual rotation $\overline{\alpha}\coloneqq1-\alpha$.
For any sequence of observables $x_{r}\coloneqq x_{r}(\alpha)$ we
define the dual observable values $\overline{x}_{r}\coloneqq x_{r}(\overline{\alpha})$

$a_{0}^{\slash}=\alpha,a_{n}=\left\lfloor a_{n}^{\slash}\right\rfloor ,a_{n+1}^{\slash}=\frac{1}{a_{n}^{\slash}-a_{n}}$
so for $\alpha<\frac{1}{2}$ $a_{1}=\left\lfloor \frac{1}{\alpha}\right\rfloor \ge2,a_{0}^{\slash}=\alpha,a_{1}^{\slash}=\frac{1}{\alpha},\overline{a}_{0}^{\slash}=1-\alpha,\overline{a}_{1}^{\slash}=\frac{1}{1-\alpha}$
and hence$\overline{a}_{2}^{\slash}=\frac{1}{\overline{a}_{1}^{\slash}-\overline{a}_{1}}=\frac{1}{\frac{1}{1-\alpha}-1}=\frac{1-\alpha}{\alpha}=\frac{1}{\alpha}-1=a_{1}^{\slash}-1$.
Hence $\overline{a}_{2}=a_{1}-1$
\begin{lem}
For $0<\alpha<\frac{1}{2}$, the quotients $\overline{a}_{r}^{\slash},\overline{a}_{r}$
of $\overline{\alpha}$ satisfy the identity $\overline{x}_{r+1}=x_{r}$
for $r\ge2$, and further $\overline{q}_{r}^{\slash}$ satisfies it
for $r\ge1$, and $\overline{q}_{r}$ for $r\ge0$. In addition the
early sequence values are $\overline{a}_{0}^{\slash}=1-\alpha,\overline{a}_{1}^{\slash}=\frac{1}{1-\alpha},\overline{a}_{2}^{\slash}=\frac{\alpha}{1-\alpha}$,
$\overline{a}_{0}=0,\overline{a}_{1}=1,\overline{a}_{2}=a_{1}-1$,
$\overline{q}_{0}^{\slash}=1,\overline{q}_{1}^{\slash}=\frac{1}{1-\alpha}$
and $\overline{q}_{0}=1$. 
\end{lem}

\begin{proof}
Standard results from earlier in this section immediately give us
$\overline{a}_{0}^{\slash}=1-\alpha,\overline{q}_{0}=1,\overline{q}_{0}^{\slash}=1$,
and also since $1-\alpha<1$, $\overline{a}_{0}=0,\overline{q}_{1}^{\slash}=\overline{a}_{1}^{\slash}=\frac{1}{1-\alpha}$
(note that $\overline{q}_{1}^{\slash}\ne q_{0}^{\slash}$,$\overline{a}_{1}^{\slash}\ne a_{0}^{\slash}$). 

Now we use $0<\alpha<\frac{1}{2}$ so that $\frac{1}{2}<\overline{\alpha}<1$
and hence $\overline{a}_{1}^{\slash}=\frac{1}{1-\alpha}<2$, giving
$\overline{q}_{1}=\overline{a}_{1}=1=q_{0}$. We can now calculate
$\overline{a}_{2}^{\slash}=\frac{1}{\overline{a}_{1}^{\slash}-\overline{a}_{1}}=\frac{1}{\frac{1}{1-\alpha}-1}=\frac{1-\alpha}{\alpha}=\frac{1}{\alpha}-1=a_{1}^{\slash}-1$
and hence $\overline{a}_{2}=a_{1}-1$. These results then give $\overline{q}_{2}=(a_{1}-1).1+1=a_{1}=q_{1}$
and similarly $\overline{q}_{2}^{\slash}=(a_{1}^{\slash}-1).1+1=a_{1}^{\slash}=q_{1}^{\slash}$.
Finally $\overline{a}_{3}^{\slash}=\frac{1}{\overline{a}_{2}^{\slash}-\overline{a}_{2}}=\frac{1}{\left(a_{1}^{\slash}-1\right)-\left(a_{1}-1\right)}=a_{2}^{\slash}$
and hence also $\overline{a}_{3}=a_{2}$. This then gives $\overline{q}_{3}^{\slash}=\overline{a}_{3}^{\slash}\overline{q}_{2}^{\slash}=a_{2}^{\slash}q_{1}^{\slash}=q_{2}^{\slash}$,
and we have established the early sequence values.

Now for $n\ge2$, we can write $\overline{a}_{n+1}^{\slash}=\Gamma^{n-2}\overline{a}_{3}^{\slash}=\Gamma^{n-2}a_{2}^{\slash}=a_{n}^{\slash}$
and hence also $\overline{a}_{n+1}=a_{n}$. The recurrence relations
for $\overline{q}_{n+1}^{\slash},\overline{q}_{n+1}$ quickly give
the results for these sequences for $n\ge2$, and the early sequence
results extend them to the claimed starting values.
\end{proof}
We now develop some useful identities
\begin{lem}
\label{lem:identities}Identities:
\end{lem}

\begin{enumerate}
\item For $n\ge0$ and $a_{0}^{\slash}$ irrational we have $1\le a_{n+1}<a_{n+1}^{\slash}=a_{n+1}+\frac{1}{a_{n+2}^{\slash}}<a_{n+1}+1$ 
\item 
\item For $n>0$, $q_{n}<q_{n}^{\slash}<q_{n}+q_{n-1}<\min\left(\left(1+\frac{1}{a_{n}}\right)q_{n},q_{n+1}\right)\le2q_{n}$
Hence $\prod_{r=1}^{n}a_{r}<q_{n}<q_{n}^{\slash}=\prod_{r=1}^{n}a_{r}^{\slash}=\prod_{r=1}^{n}(a_{r}+\frac{1}{a_{r+1}^{\slash}})<\prod_{r=1}^{n}(a_{r}+1)$
\item $a_{n+1}^{\slash}(a_{n}^{\slash}-a_{n})=1$ (equivalent to $a_{n+1}^{\slash}a_{n}^{\slash}=a_{n+1}^{\slash}a_{n}+1$
hence $a_{n+1}^{\slash}a_{n}^{\slash}>2$ and $q_{n+1}^{\slash}>2q_{n-1}^{\slash}$)
and hence $q_{n}^{\slash}=q_{n}+\frac{q_{n-1}}{a_{n+1}^{\slash}}<q_{n}\left(1+\frac{1}{a_{n}a_{n+1}^{\slash}}\right)$.
This is $<\frac{3}{2}q_{n}$ unless $a_{n+1}=a_{n}=1$
\item $\frac{1}{q_{r+2}^{\slash}}+\frac{a_{r+1}}{q_{r+1}^{\slash}}=\frac{1}{q_{r}^{.\slash}}\le\frac{1}{q_{r}^{.}}$
(equality only for $r=0$)
\item $\frac{1}{q_{r}}-\frac{1}{q_{r}^{\slash}}=\frac{q_{r-1}}{q_{r}q_{r+1}^{\slash}}\le\min\{\frac{1}{q_{r+1}^{\slash}},\frac{1}{2q_{r}}\}\le\frac{1}{2q_{r}}$
with equality only possible for $r=1,q_{1}=1$
\end{enumerate}
\begin{proof}
$\frac{1}{q_{r+2}^{\slash}}+\frac{a_{r+1}}{q_{r+1}^{\slash}}=\frac{q_{r+1}^{\slash}+a_{r+1}q_{r+2}^{\slash}}{q_{r+2}^{\slash}q_{r+1}^{\slash}}=\frac{\left(1+a_{r+1}a_{r_{+2}}^{\slash}\right)q_{r+1}^{\slash}}{q_{r+2}^{\slash}q_{r+1}^{\slash}}=\frac{\left(a_{r+1}^{\slash}a_{r_{+2}}^{\slash}\right)}{q_{r+2}^{\slash}}=\frac{1}{q_{r}^{\slash}}$
and $q_{r}^{\slash}\ge q_{r}$ with equality only for $r=0$

$\frac{1}{q_{r}}-\frac{1}{q_{r}^{\slash}}=\frac{q_{r}^{\slash}-q_{r}}{q_{r}q_{r}^{\slash}}=\frac{\left(a_{r}^{\slash}-a_{r}\right)q_{r-1}}{q_{r}q_{r}^{\slash}}=\frac{\frac{1}{a_{_{r+1}}^{\slash}}q_{r-1}}{q_{r}q_{r}^{\slash}}=\frac{q_{r-1}}{q_{r}q_{r+1}^{\slash}}$.
The inequalities follow from $q_{r-1}\le q_{r}$ (equality only possible
for $r=1,\alpha>\frac{1}{2}$), and $q_{r+1}^{\slash}>q_{r+1}>2q_{r-1}$
\end{proof}
\label{subsec:Ostrowski-Representation}Ostrowski Representation
Ostrowski introduced this representation in his paper (\#\#paper)
studying the sum of remainders (the Birkhoff sum $\sum_{r=1}^{N}\left\{ r\alpha\right\} $).
The precise application of the tool in that paper does not seem to
have a natural generalisation, but the tool itself it has found many
other applications since, and it will be a primary tool in our study. 

We first recall the idea of a non-standard number system for representing
integers. In a standard number system we use a radix $q>1$ to represent
an integer $N$ by taking the sequence of weights $1=q^{0}<q^{1}<q^{2}...<q^{n}$
defined by $q^{n}\le N<q^{n+1}$ and finding integer coefficients
$b_{r}\ge0$ such that $N=\sum_{r=0}^{n}b_{r}q^{r}$. In general $N$
has many such representations, but if we impose the condition that
we maximise $b_{n}$, then $b_{n-1}$ and so on down to $b_{0}$,
then the $b_{r}$ are determined and we have a canonical representation.
In a non-standard number system we take an arbitrary unbounded sequence
of weights $1=q_{0}\le q_{1}\le q_{2}...$ and apply precisely the
same procedure to arrive at a canonical representation $N=\sum_{r=0}^{n}b_{r}q_{r}$.
(We can allow $q_{r}=q_{r+1}$ because maximising $b_{r+1}$ before
$b_{r}$ will result in $b_{r}=0$). 
\begin{defn}
For $N\in\mathbb{N}$ the Ostrowski representation $O_{\alpha}(N)$
is the canonical representation $N=\sum_{r=0}^{n}b_{r}q_{r}$ in the
non-standard number system defined by the quasiperiods $q_{r}$ of
$\alpha$.

\end{defn}

Preliminary results
In this section we will build a number of simple results from the
foregoing which will be useful in our main development.

We note the following characteristics of the Ostrowski representation:
\begin{lem}
\label{lem:ORresults}Let $\alpha$ have partial quotient sequence
$(a_{n})$ and quasiperiod sequence $(q_{n}),$and suppose $N$ has
Ostrowski Representation $\sum_{r=0}^{n}b_{r}q_{r}$, then:
\begin{enumerate}
\item $b_{r}\le a_{r+1}$
\item $b_{0}<a_{1}$
\item $b_{r}=a_{r+1}\,\Longrightarrow\,b_{r-1}=0$
\item 
\item $q_{1}=1\,\Longrightarrow b_{0}=0$ 
\end{enumerate}
\end{lem}

\begin{proof}
The crucial insight is that if $b_{r}q_{r}\ge q_{r+1}$ then $b_{r+1}$
is not the maximal coefficient of $q_{r+1}$. So in the canonical
representation we must have $b_{r}q_{r}<q_{r+1}$. Since $q_{r+1}=a_{r+1}q_{r}+q_{r-1}$
and $q_{r}>0$ for $r\ge0$, this gives $b_{r}<a_{r+1}+\frac{q_{r-1}}{q_{r}}$
for $r\ge0$. Recall $q_{-1}=0,q_{0}=1,q_{1}\ge1,q_{2}>1$ and so
for $r=0,$$b_{0}<a_{1}$ and for $r\ge1$, $b_{r}\le a_{r+1}$. Also
if $q_{1}=1$ then $a_{1}=1$ and then $b_{0}<a_{1}$ gives $b_{0}=0$.
Finally, refining the initial insight, we note that if $b_{r}=a_{r+1}$
then $b_{r}q_{r}+b_{r-1}q_{r-1}=q_{r+1}+(b_{r-1}-1)q_{r-1}$ which
again means that $b_{r+1}$ cannot be maximal if $b_{r-1}-1\ge0$. 
\end{proof}

\section{\label{sec:Dualities-on-Operators}Theory of Hom-set Magmas }

\subsection{Introduction}

Given the Birkhoff sum $S_{N}(\phi,x_{0})=\sum_{r=1}^{N}\phi(x_{0}+r\alpha)$,
we can rewrite this in curried form as $S_{N}^{x_{0}}\phi$, ie as
the action of an operator $S_{N}^{x_{0}}$ with signature $(\mathbb{TR})\mathbb{R}$
(see{*}{*}) mapping the function space of observables $\{\phi_{\mathbb{TR}}\}$
to $\mathbb{R}$. More specifically therefore, $S_{N}^{x_{0}}$ is
a \noun{functional}. In later sections we will decompose $S_{N}^{x_{0}}$
into a sum of many functionals of the same signature. This leads to
much repetitive work which can be saved by a structured understanding
of a number of related underlying dualities. 

In this section we will develop the algebraic theory which contains
these dualities, which seems of interest in itself. Amongst other
things, it delivers a set of 8 dual relations between functionals
of signature $(ST)T$ built from a Source structure $(S,\sigma_{S},R_{S})$
and a Target structure $(T,\tau_{T},R_{T})$ equipped with endomorphisms
$\sigma,\tau$ and endorelations $R_{S},R_{T}$. If any one of the
dual relations hold, we can immediately deduce the other 7. 

We will later (Section \ref{subsec:SourceTargetDualities}) apply
this theory to the environment of $(\mathbb{T},\overline{.},\le_{\mathbb{T}}),(\mathbb{R},-_{\mathbb{R}},\le_{\mathbb{R}})$.
Although this is an extremely simple environment in which to apply
the theory, we have still found this approach has saved duplicated
effort, and perhaps more importantly, confusion, by providing a rigorous
and systematic approach to manipulating these relations. 
The General Hom-set Magma 
For a brief summary of free magmas and the height function $h$,
and also of hom-sets, see the sections on generated structure and
on hom-sets in subsection \ref{par:AlgBasics}.

Let $X=\{X_{\alpha}\}$ be a set of sets indexed by another set $A$.
We will call each $X_{\alpha}$ a base set, or height 0 set, and set
$M_{0}(X)\coloneqq X$. Recall that we write $AB$ for the set of
functions from $A$ to $B$. For $n\ge1$ define $M_{n}(X)=\left\{ AB:A\in M_{n-1}(X)|B\in M_{n-1}(X)\right\} $
(ie at least one of $A,B$ is in $M_{n-1}(X)$). Each $m\in M_{n}(X)$
is a hom-set and we say that it has height $n$ (written $h(m)=n$)
and is of \noun{type} $(h(A),h(B))$. Note that $h(AB)=\max(h(A),h(B))+1$.
Finally we define $M(X)=\bigcup_{n=0}^{\infty}M_{n}(X)$ called the
Magma of hom-sets on $X$. In the sequel we will find it more convenient
to refer to the height of $AB$ as its \noun{Level}.

Note that if $A,B$ are sets then $AB$ is an element of $P^{2}\left(A\times B\right)$
where $P$ is the power set operator on sets, and hence together with
$M(X)$ itself satisfies the requirements of a set in $ZF$ so that
there are no foundational issues.

There is however another technical issue: we cannot easily guarantee
that the generating set $X$ is free. Hence $M(X)$ may not be a free
magma, and a given element $m$ may appear at different levels, with
different tree representations. We avoid this issue by regarding $M(X)$
as a family of sets indexed by the free magma on the symbols $'X_{\alpha}'$.
Then when we talk about $AB\in M(X)$ we must specify which representation
tree in the free magma we have in mind, and then the height, width
and type of $AB$ are well defined.

Further notes on notation and annotation
We will adopt the conventions of operator theory in the sequel, and
call functions of Level 2 and above \noun{operators.} Further, we
will call operators of type $(n,0)$ (which map functions to a base
set) \noun{functionals}. We will use greek letters for functions of
level 1 (which are necessarily of type $(0,0)$), eg $\phi_{X_{\alpha}X_{\beta}},$
and capital letters for functionals, eg $A_{(X_{\alpha}X_{\beta})X_{\beta}}$
(so that we can correctly write $\left(A_{(X_{\alpha}X_{\beta})X_{\beta}}\phi_{X_{\alpha}X_{\beta}}\right)_{X_{\beta}}$
). We also need to refer to relations and will use the symbol $'R'$
to indicate these. In particular $R_{XX}$ refers to an endorelation
on $X$.

For convenience we will also introduce a condensed \noun{endo annotation}
by writing the annotation $T^{\circ}$ for the annotation $TT$, and
$T^{\circ\circ}=(TT)^{\circ}=(TT)(TT)$ etc. This applies to both
endomorphisms and endorelations. For example the level 2 endomorphism
$\alpha_{(XY)(XY)}$ on $XY$ can be written $\alpha_{\left(XY\right)^{\circ}}$,
and the level 3 endomorphism $\sigma_{\left((ST)(ST)\right)\left((ST)(ST)\right)}$
on $(ST)^{\circ}$ can be written $\sigma_{(ST)^{\circ\circ}}$. Note
that the presence of $.\text{°}$ always indicates that the annotated
morphism is an endomorphism. Later when we hope the reader has become
comfortable with which symbols represent endomorphism or endorelations,
we will drop the $.^{\circ}$ annotation altogether, so that the (Level
2) annotated endomorphism $\sigma_{(ST)^{\circ}}$ will be annotated
even more succinctly as $\alpha_{ST}$. This of course has the risk
of confusing it with the (Level 1) function $\alpha_{ST}$. The quack
rule applies: if it looks like an endomorphism, and behaves like an
endomorphism, it probably is a higher level endomorphism rather than
a lower level function.

We will also adopt the convention of using annotation only when there
are changes in the referents, reading left to right. For example in
the formula $\sigma_{(ST)^{\circ}}\phi_{ST}=\phi\circ\sigma_{S^{\circ}}$,
the symbol $'\sigma'$ changes referents (from Level $2$ of type
$(1,1)$ to Level $1$ of type $(0,0)$) whereas the symbol $'\phi'$
has the same referent in both places.

Finally we use the symbol $'\circ'$ conventionally as the binary
operator which composes morphisms, so that the function composition
$(f_{TZ}\circ g_{ST})_{SZ}$ represents the function $s_{S}\mapsto f(g(s))$,
with the appropriate corresponding definition for relations.

Source-Target Magma 
In this section we will be working with Levels $0$ to $3$ of the
hom-set magma $M(X)$ on a pair of Level $0$ sets $X=\{S,T\}$. We
will be focusing on morphisms from $S$ to $T$, and call $S$ the
source, and $T$ the target structure. We will here point out the
hom-sets of particular interest.

Level 1 hom-sets contain functions of type $(0,0)$, ie between ordered
pairs of the Level 0 sets $S,T$. In particular the hom-sets $SS,TT$
(or $S^{\circ},T^{\circ}$) are the endomorphisms of $S$ and $T$
respectively, and $ST$ is the the hom-set of function from the source
to the target. The other Level $1$ hom-set is $TS$, but we shall
not focus on this here.

Level 2 hom-sets contain Level 2 operators (ie involving Level 1)
objects, and we shall focus on:
\begin{enumerate}
\item $(ST)(ST)$ or $(ST)^{\circ}$, the hom-set of type (1,1) endomorphisms
mapping functions in $ST$ to functions in $ST$ , 
\item $(ST)T$ the hom-set of type $(1,0)$ operators (functionals) which
map functions in $ST$ to ``values'', ie to elements of the base
set $T$.
\end{enumerate}
Level 3 hom-sets contain Level 3 operators (ie involving Level 2 objects),
and we shall focus on $\left((ST)T\right)\left((ST)T\right)$ or $\left((ST)T\right)^{\circ}$,
the hom-set of endomorphisms of type $(2,1)$ mapping functionals
to functionals in $(ST)T$. 

We will also develop the concept of a \noun{Pull Up. }As the name
suggests, Pull Ups map objects of one level to a higher level, unlike
functionals which map objects down to the base level. Technically
therefore pull ups are operators, but there is less value in surfacing
them explicitly, and we will simply explore their results. In particular
we will introduce pull ups of Level 1 objects to Levels 2 and 3, and
of Level 2 objects to Level 3. Objects may be functions or relations.
Pull Ups are specialisations of the well known pull back and push
forward constructions.

\subsection{\label{subsec:PullUpsMorphs}Pull Ups of EndoMorphisms}

Level 2 Pull Ups of Endomorphisms on $S,T$
\begin{defn}[Pull Up Endomorphisms]
Let $\sigma_{SS},\tau_{TT}$ be Level 1 endomorphisms of $S,T$
respectively. 

We now ``pull back'' $\sigma_{SS}$ to $ST$ to induce a level 2
endomorphism $\sigma_{(ST)(ST)}$ which is of type $(1,1)$ on $ST$.
Using the condensed endomorphism annotation, we have:
\begin{equation}
\sigma_{(ST)^{\circ}}\coloneqq\phi_{ST}\mapsto(\phi\circ\sigma_{S^{\circ}})_{ST}
\end{equation}
\end{defn}

We call $\sigma_{(ST)^{\circ}}$ the (source) \noun{pull up} of $\sigma_{S^{\circ}}$
from $S$ to $ST$ . Similarly we can ``push forward'' $\tau_{T^{\circ}}$
to induce the (target) pull up of $\tau_{T^{\circ}}$ from $T$ to
$ST$, which is again a level 2 endomorphism of type $(1,1)$:
\begin{equation}
\tau_{(ST)^{\circ}}\coloneqq\phi_{ST}\mapsto\left(\tau_{T^{\circ}}\circ\phi\right)_{ST}
\end{equation}
Note the reversal of the order of composition between the pull ups
of source and target endomorphisms. 

\begin{lem}
\label{lem:PullUpInvolution}The pull up of an involution is also
an involution. 
\end{lem}

\begin{proof}
Let $\sigma_{\left(ST\right)^{\circ}}$ be the source pull up of $\sigma_{S^{\circ}}$.
Then
\begin{equation}
\sigma_{(ST)^{\circ}}^{2}\phi=\sigma(\sigma\phi)=\sigma(\phi\circ\sigma_{S^{\circ}})=(\phi\circ\sigma_{S^{\circ}})\circ\sigma_{S^{\circ}}=\phi\circ\sigma_{S^{\circ}}^{2}
\end{equation}

so that if $\sigma_{S^{\circ}}$ is an involution (ie $\sigma^{2}=Id$)
then $\sigma_{(ST)^{\circ}}$ is also an involution. A similar argument
works for target pull ups. 
\end{proof}
\begin{lem}
\label{lem:PullUpCommute}Two source (or target) pull ups commute
if their underlying endomorphisms commute. Source and target pull
ups always commute.
\end{lem}

\begin{proof}
Let $\sigma_{S^{\circ}},\varsigma_{S^{\circ}}$ be two commuting endomorphisms.
Using the associativity of composition we get 

\begin{equation}
\left(\sigma_{(ST)^{\circ}}\circ\varsigma_{(ST)^{\circ}}\right)\phi_{ST}=\sigma(\varsigma(\phi))=(\phi\circ\varsigma_{S^{\circ}})\circ\sigma_{S^{\circ}}=\phi\circ(\varsigma\circ\sigma)=\phi\circ(\sigma\circ\varsigma)=(\varsigma_{(ST)^{\circ}}\circ\sigma_{(ST)^{\circ}})\phi\label{eq:pullpullCommute}
\end{equation}

A similar argument establishes the result for two target endomorphisms.
By another similar argument:

\begin{equation}
\left(\sigma_{(ST)^{\circ}}\circ\tau_{(ST)^{\circ}}\right)\phi_{ST}=\sigma(\tau(\phi))=(\tau_{TT}\circ\phi)\circ\sigma_{SS}=\tau\circ(\phi\circ\sigma)=\tau_{(ST)^{\circ}}(\sigma_{(ST)^{\circ}}(\phi))=(\tau\circ\sigma)\phi\label{eq:pullpushCommute}
\end{equation}
and so $\sigma_{(ST)^{\circ}},\tau_{(ST)^{\circ}}$ always commute. 
\end{proof}
\begin{rem}
If two involutions $\sigma,\tau$ commute then $\left(\sigma\tau\right)^{2}=\sigma^{2}\tau^{2}=Id$
so that $\sigma\tau$ is also an involution. 
\end{rem}

Level 3 Pull Ups of endomorphisms on $S,T$
We can now use the above results about $ST$ to deduce similar results
about $(ST)T$ by replacing $S$ with $ST$ and keeping $T$ unchanged.
Note that $(ST)T$ is a level 2 set of morphisms of type $(1,0)$
and so $\left((ST)T\right)^{\circ}$ is a level 3 set of type $(2,2)$
morphisms. 

Proceeding as in the previous subsection, the pullback of $\sigma_{\left(ST\right)^{\circ}}$
gives us the source pull up:
\begin{align}
\sigma_{\left((ST)T\right)^{\circ}} & \coloneqq A_{(ST)T}\mapsto A\circ\sigma_{\left(ST\right)^{\circ}}
\end{align}

and the pushforward of $\tau_{T^{\circ}}$ gives us the target pull
up:
\begin{equation}
\tau_{\left((ST)T\right)^{\circ}}\coloneqq A_{(ST)T}\mapsto\tau_{T^{\circ}}\circ A
\end{equation}

However this time we have an additional endomorphism on the source
$ST$, namely $\tau_{\left(ST\right)^{\circ}}$ which we can again
pull back. The natural labelling of this new pull up would be $\tau_{\left((ST)T\right)^{\circ}}$
but have just used that symbol and annotation for the pull up of $\tau_{T^{\circ}}$!
Instead, since the new pull up is a source pull up, we will label
it $\sigma_{\left((ST)T\right)^{\circ}}^{\slash}$ to get: 
\begin{align}
\sigma_{\left((ST)T\right)^{\circ}}^{\slash} & \coloneqq A_{(ST)T}\mapsto A\circ\tau_{\left(ST\right)^{\circ}}
\end{align}

By the second part of Lemma \ref{lem:PullUpCommute} we have immediately
that $\sigma_{\left((ST)T\right)^{\circ}},\tau_{\left((ST)T\right)^{\circ}}$
commute and that $\sigma_{\left((ST)T\right)^{\circ}}^{\slash},\tau_{\left((ST)T\right)^{\circ}}$
commute. The same part of the lemma also gives us that $\sigma_{\left(ST\right)^{\circ}},\tau_{\left(ST\right)^{\circ}}$
commute, and then we can apply the first part to find that $\sigma_{\left((ST)T\right)^{\circ}},\sigma_{\left((ST)T\right)^{\circ}}^{\slash}$
commute. 

It follows by repeated application of Lemma \ref{lem:PullUpInvolution}
that if $\sigma_{S^{\circ}}$ is an involution, so is $\sigma_{\left((ST)T\right)^{\circ}}$,
and if $\tau_{T^{\circ}}$ is an involution, so are $\tau_{\left((ST)T\right)^{\circ}}$
and $\sigma_{\left((ST)T\right)^{\circ}}^{\slash}$. If both $\sigma_{S^{\circ}},\tau_{T^{\circ}}$are
involutions, so are all the 8 compositions of the three Level 3 .
Hence we have shown:
\begin{thm}
\label{thm:InvolutionGroup}If $S,T$ are equipped with involutions
$\sigma_{S^{\circ}},\tau_{T^{\circ}}$ respectively, then the space
of functionals $(ST)T$ is equipped with a $\left(C_{2}\right)^{3}$
group of involutions, generated by the 3 involutions $\sigma,\sigma^{\slash},\tau$
of $\left((ST),T\right)^{\circ}$ which are themselves induced by
$\sigma_{S^{\circ}},\tau_{T^{\circ}}$. 
\end{thm}

\subsection{Pull Ups of EndoRelations}

The theory we have developed so far could be given purely within the
context of Category Theory (ie in terms of objects and morphisms without
defining what we mean by them). However at this stage our progress
is easiest if we introduce the standard set theoretic notion of Relations
as elements of sets. 
\begin{defn}
We define a relation $R_{ST}$ between $S,T$ as a subset of $S\times T$,
ie an element of the power set of the Cartesian product of $S,T$.
We write $s_{S}\,R_{ST}\,t_{T}$ to mean $(s,t)\in R_{ST}$ and we
also define the dual or opposite relation $R_{TS}^{Op}=\left\{ (t,s):(s,t)\in R_{ST}\right\} $.
If $S=T$, then $R$ and $R^{Op}$ are called endorelations on $S$.
\end{defn}

Note that $Op$ is an involution on relations, ie $(R^{Op})^{Op}=R$,
and that $sRt\Leftrightarrow tR^{Op}s$.

We now suppose that $S,T$ are equipped with endorelations $R_{S^{\circ}}\subseteq S\times S,R_{T^{\circ}}\subseteq T\times T$
in addition to the endomorphisms $\sigma_{S^{\circ}},\tau_{T^{\circ}}$.
As usual if we have $(s_{1},s_{2})\in R_{S^{\circ}}$ we will write
this as $s_{1}\,R_{S^{\circ}}\,s_{2}$, and we define the opposite
relation $R_{S^{\circ}}^{Op}$ by $s_{1}\,R_{S^{\circ}}\,s_{2}\Leftrightarrow s_{2}\,R_{S^{\circ}}^{Op}\,s_{1}$,
and then $R_{T^{\circ}}^{Op}$ similarly. 
\begin{defn}
\label{def:RelationHomomorphism}A relation homomorphism from $R_{S^{\circ}}$
to $R_{T^{\circ}}$ is a morphism $\phi_{ST}$ such that $\phi\times\phi$
is an element of $R_{S^{\circ}}R_{T^{\circ}}$, ie $s_{1}R_{S^{\circ}}s_{2}\Leftrightarrow(s_{1},s_{2})\in R_{S^{\circ}}\Rightarrow(\phi s_{1},\phi s_{2})\in R_{S^{\circ}}\Leftrightarrow\phi s_{1}R_{T^{\circ}}\phi s_{2}$.
A relation anti homomorphism from $R_{S^{\circ}}$ to $R_{T^{\circ}}$
is a relation homomorphism from $R_{S^{\circ}}$ to $R_{T^{\circ}}^{Op}$,
and has the opposite result, ie $s_{1}R_{S}s_{2}\Rightarrow\phi s_{1}R_{T}^{Op}\phi s_{2}\Leftrightarrow\phi s_{2}R_{T}\phi s_{1}$.
. If $R_{S^{\circ}},R_{T^{\circ}}$ are both functions we call $\phi_{ST}$
a function homomorphism. 
\end{defn}

\begin{rem}
\label{rem:Antimorphisms}Note that a relation anti-morphism of $R_{S^{\circ}}R_{T^{\circ}}$
transposes its images under $R_{T^{\circ}}$, so that a composition
of a relation morphism and anti-morphism is a relation anti-morphism,
whilst a composition of two relation morphisms or two relation anti-morphisms
is a relation morphism. Further if $\phi$ is an $R_{S^{\circ}}R_{T^{\circ}}$
morphism, then we have $s_{2}R_{S^{\circ}}^{Op}s_{1}\Leftrightarrow s_{1}R_{S^{\circ}}s_{2}\Rightarrow\phi s_{1}R_{T^{\circ}}\phi s_{2}\Leftrightarrow\phi s_{2}R_{T^{\circ}}^{Op}\phi s_{1}$
so that $\phi$ is also an $R_{S^{\circ}}R_{T^{\circ}}^{Op}$, $R_{S^{\circ}}^{Op}R_{T^{\circ}}$
anti morphism, and an $R_{S^{\circ}}^{Op}R_{T^{\circ}}^{Op}$ morphism.

Now suppose that $\sigma_{S^{\circ}}$ is an $R_{S^{\circ}}R_{S^{\circ}}$
anti-morphism, and that similarly $\tau_{T^{\circ}}$ is an $R_{T^{\circ}}R_{T^{\circ}}$
anti-morphism. Then if $\phi$ is also a relation anti-morphism of
$R_{S^{\circ}}R_{T^{\circ}}$, then $\sigma_{ST}(\phi)=\phi\circ\sigma_{S^{\circ}},\tau_{ST}(\phi)=\tau_{T}\circ\phi$
are both relation morphisms of $R_{S^{\circ}}R_{T^{\circ}}$ and $(\sigma\circ\tau)\circ\phi$
is a relation anti-morphism of $R_{S^{\circ}}R_{T^{\circ}}$. 
\end{rem}

---
\begin{rem}
\label{rem:FuncHomomorphism}When $R$ is a function we write $xRy$
as $y=R(x)$, so a function homomorphism has $s_{2}=R_{s}(s_{1})\Rightarrow\phi s_{2}=R_{T}(\phi s_{1})$
hence $\phi(R_{S}s_{1})=R_{T}(\phi s_{1})$ and so $\phi\circ R_{S}=R_{T}\circ\phi$.
\end{rem}

Level 2 Pull Ups of endorelations from $T$ 
We define in this section a pull up of an endorelation on $T$. This
pull up is \emph{not} the same operation as the pull up of a function:
there are many possible pull ups, and we will only define the operation
in one direction, as a pull back from the target. Nevertheless the
term pull up still seems appropriate.
\begin{defn}
For any subset $Z\subseteq S$, the pull up of the endorelation $R_{T}$
to be the endorelation $\left(R_{Z}\right)_{(ST)^{\circ}}$ on $ST$
defined by $\phi_{1}\,R_{Z}\,\phi_{2}\Leftrightarrow(\phi_{1}s\,R_{T^{\circ}}\,\phi_{2}s$
for all $s_{Z})$. If $Z=\emptyset$, $R_{Z}$ is the full relation. 
\end{defn}

\begin{rem}
\label{rem:ChangeAnnotation}Recall that that we may allow ourselves
at any point to change annotation at the risk of confusion. This seems
to be a point where improvement in legibility outweighs the risk. 

We will use $\sigma$ (without annotation) to signify $\sigma_{T^{\circ}}$,
$\sigma_{ST}$ to signify $\sigma_{(ST)^{\circ}}$, and analogously
for the symbol $\tau$. We will also call a morphism of $R_{Z}R_{Z}$
an endomorphism of $R_{Z}$. 
\end{rem}

\begin{rem}
\label{rem:inducedRelation}Note that if $Z^{\slash}\subseteq Z$
then $\phi_{1}\,R_{Z}\,\phi_{2}\Rightarrow\phi_{1}\,R_{Z^{\slash}}\,\phi_{2}$,
ie $R_{Z}\subseteq R_{Z^{\slash}}$.

Given $\phi_{1},\phi_{2}$ such that $\phi_{1}R_{Z}\phi_{2}$ we can
now deduce a number of dual results, ie given a result involving $R_{Z}$
we can deduce a corresponding result involving $R_{Z}^{Op}$. However
using we will find it possible, and in fact more convenient, to translate
these results back into the alternative form of results in terms of
$R_{Z}$. 
\end{rem}

\begin{prop}
\label{prop:TauST}If $\tau$ is an anti-endomorphism of $R_{T^{\circ}}$
then $\tau_{ST}$ is an anti-endomorphism of $R_{Z}$, ie $\phi_{1}\,R_{Z}\,\phi_{2}\Rightarrow\left(\tau_{ST}\phi_{2}\right)\,R_{Z}\,\left(\tau_{ST}\phi_{1}\right)$.
A similar result holds for endomorphisms.
\end{prop}

\begin{proof}
If $\tau$ is an anti morphism then for any $s_{S}$, $\phi_{1}s\,R_{T^{\circ}}\,\phi_{2}s\Rightarrow\tau(\phi_{2}s)\,R_{T^{\circ}}\,\tau(\phi_{1}s)\Leftrightarrow(\tau_{ST}\phi_{2})s\,R_{T^{\circ}}\,(\tau_{ST}\phi_{1})s$.
Since this holds trivially for all $s_{Z}$, the result for anti-morphisms
follows. The morphism case is similar.
\end{proof}
Recall that $\sigma Z\coloneqq\{\sigma z:z\in Z\}$. 
\begin{prop}
\label{prop:sigmaST}If $\sigma$ is an involution on $S$, then $\sigma_{ST}$
is an endomorphism of $R_{\sigma Z}$, ie $\phi_{1}R_{Z}\phi_{2}\Rightarrow\left(\sigma_{ST}\phi_{1}\right)\,R_{\sigma Z}\,\left(\sigma_{ST}\phi_{2}\right)$
\end{prop}

\begin{proof}
Since $\sigma$ is an involution, $\phi s=\phi(\sigma^{2}s)=(\phi\circ\sigma)(\sigma s)=(\sigma_{ST}\phi)(\sigma s)$,
so $\phi_{1}s\,R_{T^{\circ}}\,\phi_{2}s\Rightarrow(\sigma_{ST}\phi_{1})(\sigma s)\,R_{T^{\circ}}\,(\sigma_{ST}\phi_{2})(\sigma s)$
which establishes the result.
\end{proof}
\begin{prop}
\label{prop:SigmaTauST}If $\sigma$ is an involution on $S$, and
$\tau_{T}$ is an (anti) morphism on $R_{T}$, then $\sigma_{ST}\circ\tau_{ST}$
is an (anti) morphism of $R_{\sigma Z}$.
\end{prop}

\begin{proof}
This follows immediately by applying \ref{prop:TauST} followed by
\ref{prop:sigmaST}, eg in the anti morphism case, $\phi_{1}R_{Z}\phi_{2}\Rightarrow(\sigma_{ST}\circ\tau_{ST})\phi_{2}\,R_{\sigma Z}\,(\sigma_{ST}\circ\tau_{ST})\phi_{1}$
\end{proof}
Note that we established $\sigma_{ST}\circ\tau_{ST}=\tau_{ST}\circ\sigma_{ST}$
in (\ref{eq:pullpushCommute}), and so the order of application of
results above is immaterial.

Level 3 Pull Ups of endorelations from $T$ 
Recall (\ref{thm:InvolutionGroup}) that if $\sigma,\tau$ are involutions
on $S,T$ then $\sigma,\sigma^{\slash},\tau$ on $(ST)T$ are also
commuting involutions and generate an (Abelian) $(C_{2})^{3}$ group
of 8 involutions. We wish to investigate how these behave with respect
to induced relations on $(ST)T$. We first look at results we can
take from the previous section by replacing $S$ by $ST$, $Z\subseteq S$
by $\Psi\subseteq ST$, and $\phi_{ST}$ with $A_{(ST)T}$. Recall
that $\sigma_{\left((ST)T\right)^{\circ}}^{\slash}(A)=A\circ\tau_{\left(ST\right)^{\circ}}$.

Note we can apply proposition \ref{prop:TauST} to $\tau_{(ST)T}$,
and proposition \ref{prop:sigmaST} to $\sigma_{(ST)T},\sigma_{(ST)T}^{\slash}$.
\ref{prop:TauST} requires $\tau_{T}$ to be an anti-morphism on $R_{T}$,
and \ref{prop:sigmaST} requires $\sigma_{ST},\sigma_{ST}^{\slash}$
to be involutions on $ST$. Now $\sigma_{ST}$ is an involution if
$\sigma_{S}$ is, and $\sigma_{ST}^{\slash}$ is an involution if
$\tau_{T}$ is. If these hold then \ref{prop:SigmaTauST} will also
hold for the pairs $\sigma,\tau$ and $\sigma^{\slash},\tau$. Hence
\ref{prop:TauST}-\ref{prop:SigmaTauST} give dual results relating
to $Id,\sigma,\sigma^{\slash},\tau,\sigma\circ\tau,\sigma^{\slash}\circ\tau$
but we are missing results for $\sigma\circ\sigma^{\slash},\sigma\circ\sigma^{\slash}\circ\tau$
which we will now develop.
\begin{prop}
If $\sigma_{ST},\sigma_{ST}^{\slash}$ are involutions on $ST$, then
$\sigma_{(ST)T}\circ\sigma_{(ST)T}^{\slash}$ is itself an involution
and also an endomorphism of $R_{\sigma\sigma^{\slash}\Psi}$, ie $A_{1}R_{\Psi}A_{2}\Leftrightarrow(\sigma\sigma^{\slash})A_{1}\,R_{(\sigma\sigma^{\slash}\Psi)}\,(\sigma\sigma^{\slash})A_{2}$
\end{prop}

\begin{proof}
By Theorem \ref{thm:InvolutionGroup} $\sigma\circ\sigma^{\slash}$
is itself an involution on $ST$, and the result follows by applying
Proposition \ref{prop:sigmaST}
\end{proof}
\begin{prop}
If $\sigma_{ST},\sigma_{ST}^{\slash}$ are involutions on $ST$, and
$\tau_{T}$ is an (anti)morphism of $R_{T}$, then $\sigma_{(ST)T}\circ\sigma_{(ST)T}^{\slash}\circ\tau_{(ST)T}$
is an (anti)morphism of $R_{\sigma\sigma^{\slash}\Psi}$, ie $A_{1}R_{\Psi}A_{2}\Leftrightarrow(\sigma\sigma^{\slash}\tau)A_{2}\,R_{(\sigma\sigma^{\slash}\Psi)}\,(\sigma\sigma^{\slash}\tau)A_{1}$
\end{prop}

\begin{proof}
By the previous proposition $\sigma_{(ST)T}\circ\sigma_{(ST)T}^{\slash}$
is an involution, and the result follows by applying proposition \ref{prop:SigmaTauST} 
\end{proof}
We have now established the following:
\begin{thm}
\label{thm:Duals}If $\sigma_{S},\tau_{T}$ are involutions on $(S,\sigma_{S},R_{S}),(T,\tau_{T},R_{T})$
respectively, and $\tau_{T}$ is also an anti-morphism on $R_{T}$,
then we have 8 dual equivalences involving the induced endorelations
$R_{\Psi},R_{\sigma\Psi},R_{\tau\Psi},R_{\sigma\tau\Psi},$ on $(ST)T$
and the induced involutions $\sigma_{(ST)T},\sigma_{(ST)T}^{\slash},\tau_{(ST)T}$,
namely
\begin{align*}
A_{1}R_{\Psi}A_{2} & \Leftrightarrow\sigma A_{1}R_{\sigma\Psi}\sigma A_{2}\Leftrightarrow\sigma^{\slash}A_{1}R_{\tau\Psi}\sigma^{\slash}A_{2}\Leftrightarrow\sigma\sigma^{\slash}A_{1}R_{\sigma\tau\Psi}\sigma\sigma^{\slash}A_{2}
\end{align*}

\[
\Leftrightarrow\tau A_{1}R_{\Psi}^{Op}\tau A_{2}\Leftrightarrow\sigma\tau A_{1}R_{\sigma\Psi}^{Op}\sigma\tau A_{2}\Leftrightarrow\sigma^{\slash}\tau A_{1}R_{\tau\Psi}^{Op}\sigma^{\slash}\tau A_{2}\Leftrightarrow\sigma\sigma^{\slash}\tau A_{1}R_{\sigma\tau\Psi}^{Op}\sigma\sigma^{\slash}\tau A_{2}
\]

\end{thm}

\subsection{Special Cases and Application}

We now consider 2 special cases of Theorem \ref{thm:Duals}: the case
of self dual $\Psi$, and the case of function homomorphisms from
$\tau_{ST}$ to $\tau_{T}$ in $\Psi$.
When $\Psi$ is self-conjugate under the involution $\sigma\tau$
Given any $\Phi\subseteq ST$, let $\Psi(\Phi)=\Phi\bigcup\sigma\tau\Phi$.
Note then that $\sigma\tau\Psi=\Psi$ (ie $\Psi$ is self-conjugate
under the involution $\sigma\tau=\sigma\circ\tau$), which immediately
gives $\tau\Psi=\sigma\Psi$ since $\sigma,\tau$ are commuting involutions.
Hence both $\Psi(\Phi)$ and $\Psi(\sigma\Phi)$ are self-conjugate
under $\sigma\tau$ and we can interchange $R_{\Psi}$ with $R_{\sigma\tau\Psi}$,
and $R_{\sigma\Psi}$ with $R_{\tau\Psi}$ in each of the above dual
forms of Theorem \ref{thm:Duals}. 

\begin{align}
A_{1}R_{\Psi}A_{2} & \Leftrightarrow\sigma A_{1}R_{\sigma\Psi}\sigma A_{2}\Leftrightarrow\sigma^{\slash}A_{1}R_{\sigma\Psi}\sigma^{\slash}A_{2}\Leftrightarrow\sigma\sigma^{\slash}A_{1}R_{\Psi}\sigma\sigma^{\slash}A_{2}\label{eq:DualPsi}
\end{align}
\[
\Leftrightarrow\tau A_{1}R_{\Psi}^{Op}\tau A_{2}\Leftrightarrow\sigma\tau A_{1}R_{\sigma\Psi}^{Op}\sigma\tau A_{2}\Leftrightarrow\sigma^{\slash}\tau A_{1}R_{\sigma\Psi}^{Op}\sigma^{\slash}\tau A_{2}\Leftrightarrow\sigma\sigma^{\slash}\tau A_{1}R_{\Psi}^{Op}\sigma\sigma^{\slash}\tau A_{2}
\]

Now suppose $\sigma_{S}$ is a relation anti-homomorphism on $R_{S}$
in addition to being an involution. If $\Phi$ is a set of relation
(anti) homomorphisms then $\sigma\tau\Phi$ is also a set of relation
(anti) homomorphisms by Remark (\ref{rem:Antimorphisms}), and so
the full set $\Psi(\Phi)$ is also a set of relation (anti) homomorphisms.
Similarly $\sigma\Phi,\tau\Phi,\Psi(\sigma\Phi)$ are sets of relation
anti-homomorphisms (homomorphisms). Finally note that $\Phi,\sigma\tau\Phi\subseteq\Psi(\Phi)$,
and $\sigma\Phi,\tau\Phi\subseteq\sigma\Psi(\Phi)$$\sigma\tau$ and
by remark \ref{rem:inducedRelation}, if any of these relations hold,
we can then substitute any of the subsets for the full set, eg from
$A_{1}R_{\Psi}A_{2}$ we can deduce $\sigma\sigma^{\slash}A_{1}R_{\Phi}\sigma\sigma^{\slash}A_{2}$.

When $A_{1},A_{2}$ are function homomorphisms of $\tau_{ST}\tau_{T}$
There is in general no necessary relationship between $\sigma_{(ST)T}^{\slash}A_{(ST)T}=A\circ\tau_{ST}$
and $\tau_{(ST)T}A_{(ST)T}=\tau_{T}\circ A$. However this changes
when $A$ is a function homomorphism from $\tau_{ST}$ to $\tau_{T}$,
ie a homomorphism from the algebraic structure $(ST,\tau_{ST})$ to
$(T,\tau_{T})$, ie $A\left(\tau_{ST}\phi\right)=\tau_{T}\left(A\phi\right)$

\begin{lem}
$\sigma_{(ST)T}^{\slash}=\tau_{(ST)T}$ on function homomorphisms
from $\tau_{ST}$ to $\tau_{T}$
\end{lem}

\begin{proof}
By remark \ref{rem:FuncHomomorphism} $A_{(ST)T}$ is a function homomorphism
means $A\circ\tau_{ST}=\tau_{T}\circ A$. Then 
\[
\sigma_{(ST)T}^{\slash}(A)=A\circ\tau_{ST}=\tau_{T}\circ A=\tau_{(ST)T}(A)
\]
\end{proof}
This simplifies the dual forms of Theorem \ref{thm:Duals} when $A_{1},A_{2}$
are function homomorphisms to: 
\begin{align}
A_{1}R_{\Psi}A_{2} & \Leftrightarrow\sigma A_{1}R_{\sigma\Psi}\sigma A_{2}\Leftrightarrow\tau A_{1}R_{\tau\Psi}\tau A_{2}\Leftrightarrow\sigma\tau A_{1}R_{\sigma\tau\Psi}\sigma\tau A_{2}\label{eq:DualHom}
\end{align}
\[
\Leftrightarrow\tau A_{1}R_{\Psi}^{Op}\tau A_{2}\Leftrightarrow\sigma\tau A_{1}R_{\sigma\Psi}^{Op}\sigma\tau A_{2}\Leftrightarrow A_{1}R_{\tau\Psi}^{Op}A_{2}\Leftrightarrow\sigma A_{1}R_{\sigma\tau\Psi}^{Op}\sigma A_{2}
\]

The combined case
Finally let us consider both special cases. Then we have the simplified
equivalences of $\eqref{eq:DualHom}$ but with the addition of $\sigma\tau\Psi=\Psi$
and $\sigma\Psi=\tau\Psi$ gives: 
\begin{gather}
A_{1}R_{\Psi}A_{2}\Leftrightarrow\sigma A_{1}R_{\tau\Psi}\sigma A_{2}\Leftrightarrow\tau A_{1}R_{\tau\Psi}\tau A_{2}\Leftrightarrow\sigma\tau A_{1}R_{\Psi}\sigma\tau A_{2}\label{eq:DualFinal}
\end{gather}
\[
\Leftrightarrow\tau A_{1}R_{\Psi}^{Op}\tau A_{2}\Leftrightarrow\sigma\tau A_{1}R_{\tau\Psi}^{Op}\sigma\tau A_{2}\Leftrightarrow A_{1}R_{\tau\Psi}^{Op}A_{2}\Leftrightarrow\sigma A_{1}R_{\Psi}^{Op}\sigma A_{2}
\]

\label{subsec:Example}Example 
The theory we have developed is of quite wide application, but we
will give one simple example which will become useful to us in Section
\ref{sec:HomBirkhoffSums}. 

We let $S,T$ be 2 real intervals having relations $R_{S},R_{T}$
both as appropriate subsets of the $'\le_{\mathbb{R}}'$ order relation.
Let $\sigma_{S}$ be any involution on $S$ which is an anti-morphism
of $\le$, ie for $s_{1}\le s_{2}$ in $S$, $\sigma s_{2}\le\sigma s_{1}$,
and let $\tau_{T}$ be the same on $T$. Suppose $\Psi\subset ST$
is self-conjugate under $\sigma\tau$. 

Suppose further that $A_{1},A_{2}$ are functionals in $(ST)T$ and
that under the pull up of $R_{T}$ to $R_{\Psi}$ that we have $A_{1}R_{\Psi}A_{2}$,
ie for all $\phi_{\Psi}$ we have $A_{1}\phi\le_{T}A_{2}\phi$. It
makes sense then to write $R_{\Psi}=R_{\sigma\tau\Psi}=\le_{\Psi}$
and similarly $R_{\sigma\Psi}=R_{\tau\Psi}=\le_{\sigma\Psi}$. This
means we write $A_{1}\le_{\Psi}A_{2}$ to mean $A_{1}\le A_{2}$ on
$\Psi$, and similarly $A_{1}\le_{\sigma\Psi}A_{2}$ to mean $A_{1}\le A_{2}$
on  $\sigma\Psi$. Then by \eqref{eq:DualPsi} when $A_{1}\le_{\Psi}A_{2}$
we also have the dual results (recalling $A_{1}R^{Op}A_{2}\Leftrightarrow A_{2}RA_{1}$):
\begin{equation}
\sigma A_{1}\le_{\sigma\Psi}\sigma A_{2},\;\sigma^{\slash}A_{1}\le_{\sigma\Psi}\sigma^{\slash}A_{2},\;\sigma\sigma^{\slash}A_{1}\le_{\Psi}\sigma\sigma^{\slash}A_{2}\label{eq:DualPsiEg}
\end{equation}
\[
\tau A_{2}\le_{\Psi}\tau A_{1},\;\sigma\tau A_{2}\le_{\sigma\Psi}\sigma\tau A_{1},\;\sigma^{\slash}\tau\le_{\sigma\Psi}\sigma^{\slash}\tau A_{1},\;\sigma\sigma^{\slash}\tau A_{2}\le_{\Psi}\sigma\sigma^{\slash}\tau A_{1}
\]

\newpage{}

\section{\label{sec:Separation-of-Concerns}Separation of Concerns in Birkhoff
Sums}

\subsection{Introduction}

Recall that given a Dynamical System $(X,T_{XX})$ with a value space
$V_{Monoid}=(V,+_{V})$ we define the $N$th Birkhoff sum of an observable
$\phi_{XV}$ as $S_{N}(\phi,x)\coloneqq\sum_{r=1}^{N}\phi(T^{r}x)$
(where the summation takes place in $V$). As we have noted, previous
studies with the Birkhoff sum in this form have proceeded with properties
of $\phi$ being tightly bound with the entire reasoning process.
Our strategy is to separate the role of $\phi$ from the underlying
dynamics of the system $(X,T)$. The key result which enables us to
do this is: 
\begin{equation}
S_{N}(\phi,x)=(S_{N}\phi)x=\phi(S_{N}x)\label{eq:SN}
\end{equation}

The right side of this identity allows us to separate the study of
the composite Birkhoff sum $S_{N}(\phi,x)$ into two parts. First
we study the effects of the homonymic operator $S_{N}$ on $x$. Unlike
the original $S_{N}$ this newly constructed operator is independent
of $\phi$ and is purely dynamical. In Section \ref{sec:Distribution}
we use it to develop a theory of the sequential distribution of orbits
on the circle. This needs to done only once, and then the results
are available for use with any observable $\phi$. 

In section \ref{sec:Analysis} we will study the general classification
of unbounded observables, and develop some initial estimates for the
general Birkhoff sum. In section \ref{sec:Distribution} we will
develop the estimates for the case of irrational rotations of the
circle. Finally in section \ref{sec:Application} we apply the developed
theory to obtain more specialised estimates for the Birkhoff sum in
the case of some some important specific classes of observables.

The result \eqref{eq:SN} is probably best positioned as a result
in Abstract Algebra. In essence it captures a general insight into
a structural decomposition which, as with many such results, is ``obvious''
once seen, and can then be more quickly introduced by way of an ansatz.
However as it is fundamental to our approach we will provide two proofs
- a relatively formal derivation and a quick proof by ansatz. The
first proof is much longer but provides a deeper understanding of
the mechanisms at work. The ansatz simply delivers the result as quickly
as possible, and will suit those to whom the result is intuitive and
are anxious to proceed to the content of subsequent sections. However
the significant difference between the 2 approaches also helps to
show that, although the result may appear ``intuitive'' to some,
this intuition is built upon a fair amount of internalised mathematical
machinery. Also like much ``abstract nonsense'', the proof itself
is relatively simple - the challenge lies in developing the right
set of definitions and cocepts within which the result becomes natural.

\subsection{Formal development}

In the sequel we will fix $N\ge0$ and to simplify notation we will
write merely $\sum x_{r}$ to denote the formal sum $x_{1}+x_{2}\ldots+x_{N}$
(where $+$ is an associative binary operator, not necessarily commutative).

Recall three universal constructions

Given sets $X,Y,Z$ then $XY$ is the set of functions $\phi_{XY}\coloneqq x_{X}\mapsto(\phi x)_{Y}$
from $X$ to $Y$, and $(X\times Y)Z$ is the set of bifunctions (functions
of two arguments) $f\coloneqq(x,y)\mapsto\left(f(x,y)\right)_{Z}$.
\begin{description}
\item [{C1}] From a set $X$ we can construct the \noun{Kleene star} (or
free monoid) $X^{*}$ by $X^{*}=(X^{*},+_{X^{*}})$ where $X^{*}$
is the set of formal sums $x_{1}+x_{2}\ldots+x_{k}$ for $k\ge0$,
and $+$ is the concatenation of formal sums (with the empty sum as
unit). We regard $X$ as a subset of $X^{*}$ by regarding $x_{X}$
as a formal sum of length 1. Note that ``free'' here means free
of relations, ie $\sum_{r=1}^{m}x_{r}=\sum_{r=1}^{n}x_{r}^{\slash}$
means $m=n$ and $x_{r}=x_{r}^{\slash}$. We will write $\sum^{*}x_{r}$
to denote the formal sum $\sum_{r=1}^{N}x_{i}$ (recall $N$ is fixed).
\item [{C2}] From an element $x_{X},$we can construct the \noun{pull back
}function (or functional) $x_{(XY)Y}$ of $x_{X}$ by $x\phi\coloneqq\phi x$
for each $\phi_{XY}$. 
\item [{C3}] From a bifunction $f_{(X\times Y)Z}$ we can construct the
two \noun{curried functions }$f_{X(YZ)}^{c1}\coloneqq x\mapsto(y\mapsto f(x,y))$
so that $\left(fx\right)(y)=f(x,y)$, and $f_{Y(XZ)}^{c2}\coloneqq y\mapsto(x\mapsto f(x,y))$
so that $\left(fy\right)(x)=f(x,y)$.
\end{description}
Let us equip $Y$ with a binary operator $+_{Y}$, ie an element of
the set of bifunctions $(Y\times Y)Y$, so that $Y=(Y,+_{Y})$ is
a semi-group. Recall we can always add a unit if necessary to make
a semi-group into a monoid, so we assume $Y$ is a monoid. 
\begin{description}
\item [{C4}] From $+_{Y}$ we construct the pull back (binary) operator
$\left(+_{Y}\right)_{XY}$ on $XY$ by $(\phi_{1}+_{Y}\phi_{2})x\coloneqq\phi_{1}x+_{Y}\phi_{2}x$.
We call $(XY,+_{Y})$ the pull back monoid. We write $\sum^{Y}\phi_{r}$
to denote a sum in $XY$ using the pull back $+_{Y}$. 
\end{description}
Note that we now have two monoids constructed on $XY$, namely the
free monoid $\left((XY)^{*},+_{(XY)^{*}}\right)$ and the pull back
monoid $(XY,+_{XY})$. The general element of $(XY)^{*}$ has the
form $\left(\sum^{*}\phi_{r}\right)_{(XY)^{*}}$ and $\left(\sum^{Y}\phi_{r}\right)_{XY}$
respectively.
\begin{description}
\item [{C5}] Given the two monoid constructions $(XY)^{*}$ and $XY$ on
the underlying set $XY$, we construct the natural homomorphism $\theta_{(XY)^{*}(XY)}$
by $\theta\left(\sum^{*}\phi_{r}\right)\coloneqq\sum^{Y}\phi_{r}$
which is surjective but not generally injective. This also allows
us to define $\left(\sum^{*}\phi_{r}\right)x\coloneqq\left(\sum^{Y}\phi_{r}\right)x=\sum\left(\phi_{r}x\right)$.
Important case is when $Y=W^{*}$, ie $Y$ is already a free monoid.
Then the natural homomorphism is $\left(\sum^{*}\phi_{r}\right)_{(XY)^{*}}\mapsto\sum^{*}\phi_{r}$
which is now an isomorphism, ie $(XW^{*})^{*}\cong XW^{*}$
\item [{C6}] Finally, given a set function $\phi_{XY}$ we construct the
monoid homomorphism $\phi_{X^{*}Y}\coloneqq\left(\sum^{*}x_{i}\right)\mapsto\sum(\phi x_{i})$.
Note that this is well defined since $X^{*}$ is free. If this were
not the case $\phi_{X^{*}Y}$ may be multi-valued (ie no longer strictly
a function). 
\end{description}
\begin{thm}
Let $X,Y$ be sets and $Z=(Z,+_{Z})$ a monoid, then for any set of
bifunctions $f_{r}$ in $(X\times Y)Z$ we have $\left(\sum^{*}f_{r}\right)(x,y)=\left((\sum^{*}f_{r}^{c1})(x)\right)y=x\left((\sum^{*}f_{r}^{c2})(y)\right)$ 
\end{thm}

\begin{proof}
$\left(\sum^{*}f_{r}\right)(x,y)=_{C5}\left(\sum^{Z}f_{r}\right)(x,y)=_{C4}\sum f_{r}(x,y)=_{C3}\sum\left((f_{r}^{c2}y)(x)\right)=_{C2}\sum\left(x(f_{r}^{c2}y))\right)=_{C5}x\left(\sum^{*}(f_{r}^{c2}y)\right)=_{C5}x\left((\sum^{*}f_{r}^{c2})y\right)$
\end{proof}
\begin{cor}
Given a Dynamical System $(X,T_{XX})$ with a Value System $(V,+_{V})$,
we have $S_{N}(\phi,x)=(S_{N}\phi)x=\phi(S_{N}x)$
\end{cor}

\begin{proof}
From $T$ we construct $T_{(\Phi\times X)V}(\phi,x)\coloneqq\phi(Tx)$.
From the theorem we get $\left(\sum^{*}T_{r}\right)(\phi,x)=_{Thm}\phi\left((\sum^{*}T_{r}^{c2})(x)\right)=_{C6}\sum\phi\left(T_{r}^{c_{2}}x\right)$.
Now $\phi\left(T_{r}^{c_{2}}x\right)=_{C2}\left(T_{r}^{c_{2}}x\right)\phi=_{C3}T_{r}(\phi,x)=_{\textrm{def}}\phi(Tx)$
and so $\sum\phi\left(T_{r}^{c_{2}}x\right)=\sum\phi(T_{r}x)=_{C6}\phi\left(\sum^{*}(T_{r}x)\right)=_{C5}\phi\left((\sum^{*}T_{r})x)\right)$.

Now put $T_{r}=T^{r}$ and $S_{N}=\sum^{*}T_{XX}^{r}$, $S_{N}=\sum^{*}(T^{r})_{(\Phi\times X)V}.$ 

Alternate form gives $S_{N}(\phi,x)=(S_{N}\phi)x$ where $S_{N}$
is $\Phi\Phi^{*}$ but under $\sim$ this becomes $\Phi\Phi$ , ie
$S_{N}\phi$ can be regarded simply as another observable. 
\end{proof}

\subsection{Proof by Ansatz}

Define $S_{N}x\coloneqq\sum_{r=1}^{N}T^{r}x$ as a formal sum, and
$\phi(\sum_{r=1}^{N}x_{r})\coloneqq\sum_{r=1}^{N}\phi(x_{r})$ then
$\phi(S_{N}x)=\sum_{r=1}^{N}\phi(T^{r}x)=S_{N}(\phi,x)$

\section{\label{sec:Distribution}Distribution of orbit segments on the circle }

Given a Dynamical System $(X,T)$ we now develop the theory of orbit
segments $S_{N}x\coloneqq\sum_{r=1}^{N}T^{r}x$ which we introduced
in the previous section, but specifically for the case where $T$
is a rotation $R_{\alpha}$ of the circle through $\alpha_{\mathbb{R}}$
revolutions. In this case for a circle point $x_{\mathbb{T}}$ we
have $S_{N}x=\sum_{r=1}^{N}T^{r}x=\left(T^{r}x\right)_{r=1}^{N}=\left(x+_{\mathbb{T}}r\alpha\right)_{r=1}^{N}$
(where $+_{\mathbb{T}}$ is addition on the circle). Note that we
will find it useful to be able to switch between the formal sum notation
$\sum_{r=1}^{N}T^{r}x$ and the sequence notation $\left(T^{r}x\right)_{r=1}^{N}=\left(x+_{\mathbb{T}}r\alpha\right)_{r=1}^{N}$. 

The distribution of the orbit $(r\alpha)_{r=1}^{N}=(\alpha,2\alpha,3\alpha\ldots,N\alpha)$
is well understood classically when considered simply as a set, and
is a primary example of equidistribution. In this section we instead
study the distribution of the orbit as a sequence. This introduces
some additional and quite elegant structure. In later sections we
shall exploit this structure in deriving estimates for our anergodic
Birkhoff sums. However the results of this section seem of interest
in their own right.

\subsection{QuasiPeriod Decomposition}

\label{subsec:QPDecompIntro}Introduction/Motivation
Intuitively if $\alpha$ and $\beta$ lie ``close'' to each other,
then the orbits of a circle point $x$ under the rotations $R_{\alpha}$
and $R_{\beta}$ will continue to lie close to one another over suitably
short initial orbit segments. We make this notion precise. If $\frac{p_{r}}{q_{r}}$
is a convergent of $\alpha$, then we know (see \ref{eq:CFApproximation})
that $\alpha-\frac{p_{r}}{q_{r}}=\frac{(-1)^{r}}{q_{r}q_{r+1}^{\slash}}$,
and so the $t-$th points of the orbits of $x$ under $R_{\alpha},R_{\frac{p_{r}}{q_{r}}}$
are separated from each other by a distance of $\left\Vert t\alpha-t\frac{p_{r}}{q_{r}}\right\Vert =\left\Vert \frac{t}{q_{r}q_{r+1}^{\slash}}\right\Vert $.
Further for $1\le t\le q_{r}$ we have $0<(-1)^{r}(t\alpha-\frac{tp_{r}}{q_{r}})=\frac{t}{q_{r}q_{r+1}^{\slash}}\le\frac{1}{q_{r+1}^{\slash}}<\frac{1}{q_{r}}$.
But since the sequence $(tp_{r})_{t=1}^{q_{r}}$ is just a permutation
$\bmod q_{r}$ of $(u)_{u=1}^{q_{r}}$, this tells us that the orbital
points $\{t\alpha\}_{t=1}^{q_{r}}$ are 'pigeon-holed' by the circle
partition defined by the points $(\frac{u}{q_{r}})_{u=1}^{q_{r}}$:
in other words, for each $1\le u\le q_{r}$ the partition interval
$(\frac{u-1}{q_{r}},\frac{u}{q_{r}})$ contains precisely one orbital
point $t\alpha$.   Recall that each $q_{r}$ we call a quasiperiod
of $\alpha$. The preceding discussion now motivates decomposing the
initial orbit segment of length $N$ into sub-segments of quasiperiod
length. A sub-segment of length $q_{r}$ we will then relate (using
the observations above) to rational (periodic) orbits of $q_{r}$.
The Ostrowski representation (see \ref{subsec:Ostrowski-Representation})
will be used to give us precisely this decomposition. 

What we have said so far is a classical strategy developed by Koksma(\#\#paper)
and further developed by Herman(\#\#paper).  However we now introduce
two new elements:
\begin{enumerate}
\item Previously the approach was used directly in the decomposition of
the Birkhoff sum under study. However our strategy is to decouple
the Birkhoff sum itself from the underlying dynamics of the irrational
rotation. This is an important conceptual distinction which bears
fruit later. It comes at the expense of slightly more theory and notation,
in order to describe the decomposition of orbits rather than scalar
sums.
\item Previously the approach required the sum-function in the Birkhoff
sum to be of bounded variation. This is because the periodic orbits
could be allowed to float to begin at the start of a quasiperiod segment.
In order to study unbounded functions, we are forced instead to use
periodic orbits anchored to the same fixed initial point. This requires
more analysis, but also results in a deeper understanding of the distribution
of the points of an initial orbit segment.
\end{enumerate}

Ostrowski Decomposition of Orbit Segments
Given $\alpha$, the Ostrowski Representation (see \ref{subsec:Ostrowski-Representation})
$OR_{\alpha}(N)$ gives us a canonical way of representing an integer
$N\ge0$ as a sum of quasiperiods of $\alpha$, which we write $N=\sum_{r=0}^{n}b_{r}q_{r}$.
With a little care over notation we can use this to induce a corresponding
canonical decomposition of an orbit segment of length $N$ into segments
of quasiperiod length.

We first introduce a way of representing other integers $M$ in terms
of the Ostrowski representation of $N$:

\begin{defn}[Ostrowski Triples]
Let $\alpha,N$ be fixed, and $N$ has Ostrowski Representation $OR_{\alpha}(N)=\sum_{r=0}^{n}b_{r}q_{r}$
wrt $\alpha$. We say $(r,s,t)_{\mathbb{Z}^{3}}$ is an \noun{Ostrowski
triple} representing the integer $M_{\mathbb{Z}}$ (with respect to
$\alpha,N$) if $M=\sum_{u=r+1}^{n}b_{u}q_{u}+sq_{r}+t$. For convenience
we will use the homonymous formal notation $rst$ to represent both
the triple $(r,s,t)$ and the represented integer $M$ so that we
can write $M=rst=\sum_{u=r+1}^{n}b_{u}q_{u}+sq_{r}+t$. 
\end{defn}

Note from the definition that $rst=rs0+t=r00+sq_{r}+t=\sum_{u=r+1}^{n}b_{u}q_{u}+sq_{r}+t$.
In particular $n00=0$, $n01=1$, $(-1)00=N$, and $000=\sum_{u=1}^{n}b_{u}q_{u}=N-b_{0}q_{0}=N-b_{0}$.
For $0\le r\le n-1$, $r00=b_{n}q_{n}+b_{n-1}q_{n-1}\ldots+b_{r+1}q_{r+1}$
so that that $rst$ \emph{increases} as $r$ \emph{decreases} from
$n$ to $-1$.  Note also that $M$ may in general be represented
by many Ostrowski triples with respect to $N,\alpha$ but we can define
a distinguished representation as follows.
\begin{defn}

For $1\le M\le N$ we define \textbf{\emph{the}} Ostrowski triple
of $M$ (wrt $N,\alpha)$ to be the unique Ostrowski triple representing
$M$ defined by $r=\min\{k:k00<M\}$, $s=\left\lfloor \frac{M-1-r00}{q_{r}}\right\rfloor $,
$t=M-rs0$. For $M=1..N$ we call the Ostrowski triple of $M$ a canonical
triple of $(\alpha,N)$. 
\end{defn}

Note that if $rst$ is a canonical triple then from the definition
$sq_{r}\le M-1-r00$ giving $t\ge1$. In particular it is easy to
see then that for canonical triples we have $0\le r\le n,0\le s\le b_{r}-1<a_{r+1},1\le t\le q_{r}$,
and and $1\le rst\le N$. $N$ itself is represented by the canonical
triple $=0(b_{n}-1)q_{n}$ , and $rs0$ is \emph{never} a canonical
triple (since $t\ge1$).

We are now ready to decompose the orbit segment $S_{N}(x_{0})$ into
segments of quasiperiod length. If $M=rst$ we have the equivalent
notations $x_{M}=T^{M}x_{0}=T^{rst}x_{0}=x_{rst}$.
\begin{defn}[Ostrowski Decomposition]
Given $\alpha$ irrational, $N\ge0$ and a Dynamical system $(X,T)$,
we define the Ostrowski decomposition of the orbit segment $S_{N}(x_{0})=\sum_{u=1}^{N}T^{u}x_{0}=\sum_{u=1}^{N}x_{u}$
to be the reverse formal sum $S_{N}(x_{0})=\underleftarrow{\sum_{r=0}^{n}}S_{b_{r}q_{r}}(x_{r00})$.
We further decompose each $S_{b_{r}q_{r}}(y)$ into $b_{r}$ segments
of length $q_{r}$ to obtain the Extended Ostrowski Decomposition
$S_{N}(x_{0})=\underleftarrow{\sum_{r=0}^{n}}\left(\sum_{s=0}^{b_{r}-1}S_{q_{r}}(x_{rs0})\right)$.
\end{defn}

We can write this more explicitly (recalling $n00=0)$ as
\begin{align*}
S_{N}(x_{0}) & =S_{b_{n}q_{n}}(x_{0})+S_{b_{n-1}q_{n-1}}(x_{b_{n}q_{n}})+S_{b_{n-2}q_{n-2}}(x_{b_{n}q_{n}+b_{n-1}q_{n-1}})+\ldots S_{b_{0}q_{0}}(x_{N-b_{0}q_{0}})\\
 & =S_{q_{n}}(x_{n00})+S_{q_{n}}(x_{n10})+S_{q_{n}}(x_{n20})\ldots+S_{q_{n}}(x_{n(b_{n}-1)0})\\
 & +S_{q_{n-1}}(x_{(n-1)00})+S_{q_{n-1}}(x_{(n-1)10})+S_{q_{n-1}}(x_{(n-1)20})\ldots+S_{q_{n-1}}(x_{(n-1)(b_{n-1}-1)0})\\
 & \vdots\\
 & +S_{q_{0}}(x_{000})+S_{q_{0}}(x_{010})+S_{q_{0}}(x_{020})\ldots+S_{q_{0}}(x_{0(b_{n}-1)0})
\end{align*}

 Note that if $b_{r}=0$ the corresponding line is empty.

Also note that $0s0=(N-b_{0}q_{0})+s+0$ and $q_{0}=1$ so the last
line simplifies to $x_{001}+x_{011}+\ldots+x_{0(b_{n}-1)1}=x_{(N-b_{0})+1}+x_{(N-b_{0})+2}+\ldots+x_{N}$. 

\subsection{Application to irrational rotations}

When $T$ is an irrational rotation $R_{\alpha}$ of the circle we
have $T^{M}x_{0}=\left\{ x_{0}+M\alpha\right\} $ so that if $M=rst$
we have $x_{rst}=\left\{ x_{0}+(rst)\alpha\right\} $. In particular
we will write $\alpha_{rst}\coloneqq\left\{ (rst)\alpha\right\} $
and then the Ostrowski Decomposition is $S_{N}(x_{0})=\underleftarrow{\sum_{r=0}^{n}}\sum_{s=0}^{b_{r}-1}S_{q_{r}}(x_{0}+\alpha_{rs0})$
and $S_{q_{r}}(x_{0}+\alpha_{rs0})$ is the sequence of points $\left(\left\{ x_{0}+\alpha_{rst}\right\} \right)_{t=1}^{q_{r}}$.
We are now in a position to extend our discussion to each orbit segment
$S_{q_{r}}$ in the Ostrowski decomposition. In section \eqref{subsec:QPDecompIntro}
we introduced the notion of tracking quasiperiod segments of the orbit
under $R_{\alpha}$ by means of quasiperiod partitions. We introduce
a quantity which will help us measure the precision of this tracking
behaviour.
\begin{defn}[Tracking error]
We define the \noun{tracking error} $(\epsilon_{rst})_{\mathbb{R}}$of
$\alpha_{rst}$ to be the signed real number 
\[
\epsilon_{rst}=(-1)^{r}\alpha_{r00}+\left(\frac{s}{q_{r+1}^{\slash}}+\frac{t}{q_{r}q_{r+1}^{\slash}}\right)
\]
\end{defn}

We shall see later that the tracking error $\epsilon_{rst}$ is the
signed smallest distance of the circle point $\alpha_{rst}$ from
the circle point $t\text{\ensuremath{\frac{p_{r}}{q_{r}}}}$. Our
goal is to develop bounds to show that this error is always suitably
small, although at this point we cannot even rule out $\left|\epsilon_{rst}\right|>1$.
However we can observe immediately that  for $r=n$ we have $\alpha_{n0t}=t\alpha$
and so our earlier discussion (\ref{subsec:QPDecompIntro}) tells
us that for $t\in1..q_{n}$ we have $0<\epsilon_{n0t}=(-1)^{n}\left(t\alpha-t\frac{p_{n}}{q_{n}}\right)\le\frac{1}{q_{n+1}^{\slash}}<\frac{1}{a_{n+1}q_{n}}$.
.

The general case becomes more complex as we have to take into account
the effects of quasiperiod segments which do not start from the origin.
We first show that we can in fact recover $\alpha_{rst}$ from the
definition of $\epsilon_{rst}$, via the following simple result: 
\begin{prop}
\label{prop:BaseAlpha}$\alpha_{rst}=\left\{ t\frac{p_{r}}{q_{r}}+(-1)^{r}\epsilon_{rst}\right\} $
\end{prop}

\begin{proof}
By definition $\alpha_{rst}=\left\{ \alpha_{r00}+(sq_{r}+t)\alpha\right\} =\alpha_{r00}+(sq_{r}+t)\alpha+M$
for some integer $M$. But $\alpha=\frac{p_{r}}{q_{r}}+\frac{(-1)^{r}}{q_{r}q_{r+1}^{\slash}}$
so that $\alpha_{rst}=\alpha_{r00}+(-1)^{r}\left(\frac{s}{q_{r+1}^{\slash}}+\frac{t}{q_{r}q_{r+1}^{\slash}}\right)+t\frac{p_{r}}{q_{r}}+M=(-1)^{r}\epsilon_{rst}+t\frac{p_{r}}{q_{r}}+M$
and the result follows since $\alpha_{rst}=\left\{ \alpha_{rst}\right\} $. 
\end{proof}

We now investigate the two components of $\epsilon_{rst}$ separately,
namely $(-1)^{r}\alpha_{r00}$ and $\frac{s}{q_{r+1}^{\slash}}+\frac{t}{q_{r}q_{r+1}^{\slash}}$.
We start with the latter as it is the simpler.
\begin{lem}
For canonical $rst$ we have $0<\frac{s}{q_{r+1}^{\slash}}+\frac{t}{q_{r}q_{r+1}^{\slash}}\le\frac{s+1}{q_{r+1}^{\slash}}<\frac{1}{q_{r}^{\slash}}$
\end{lem}

\begin{proof}
For canonical $rst$ we have $0\le s\le b_{r}-1$and $1\le t\le q_{r}$
so that $0<\frac{s}{q_{r+1}^{\slash}}+\frac{t}{q_{r}q_{r+1}^{\slash}}\le\frac{s+1}{q_{r+1}^{\slash}}\le\frac{b_{r}}{q_{r+1}^{\slash}}\le\frac{a_{r+1}}{q_{r+1}^{\slash}}<\frac{a_{r+1}^{\slash}}{q_{r+1}^{\slash}}=\frac{1}{q_{r}^{\slash}}$ 
\end{proof}

We now turn to investigate the term $\alpha_{r00}$. Recall the definition
$\alpha_{r00}=\left\{ \sum_{u=r+1}^{n}b_{u}q_{u}\alpha\right\} $.
We might expect this sum to be capable of taking a good range of
values within $[0,1)$ but fortunately for our purposes it turns out
to be surprisingly constrained, and this rigidity is the foundational
result of this section.
\begin{lem}
If $b_{r}\ne0$ then $\left\Vert \alpha_{r00}\right\Vert <\frac{1}{q_{r+1}^{\slash}}$
and more precisely $\frac{-1}{q_{r+1}^{\slash}}+\frac{1}{q_{r+2}^{\slash}}<(-1)^{r}\left\{ \left\{ \alpha_{r00}\right\} \right\} <\frac{1}{q_{r+2}^{\slash}}$
\end{lem}

\begin{proof}
For $0\le r\le n$ we have $\alpha_{r00}=\sum_{u=r+1}^{n}b_{u}q_{u}\alpha=\sum_{u=r+1}^{n}b_{u}p_{u}+\sum_{u=r+1}^{n}\frac{\left(E_{u}-O_{u}\right)b_{u}}{q_{u+1}^{\slash}}$.
Hence 
\begin{equation}
\left\{ \left\{ \alpha_{r00}\right\} \right\} =\left\{ \left\{ \sum_{u=r+1}^{n}\frac{E_{u}b_{u}}{q_{u+1}^{\slash}}-\sum_{u=r+1}^{n}\frac{O_{u}b_{u}}{q_{u+1}^{\slash}}\right\} \right\} \label{eq:alphar}
\end{equation}
We now estimate the even and odd sums on the right side. We will primarily
use the fact that $b_{u}\le a_{u+1}$ but with an important refinement.
Recall (Lemma \ref{lem:ORresults}) that if $b_{r}\ne0$ then $b_{r+1}<a_{r+2}$
and so we have $0\le\sum_{u=r+1}^{n}\frac{E_{u}b_{u}}{q_{u+1}^{\slash}}\le\left(\sum_{u=r+1}^{n}\frac{E_{u}a_{u+1}}{q_{u+1}^{\slash}}\right)-\frac{E_{r+1}}{q_{r+2}^{\slash}}=\left(\sum_{u=r+1+E_{r}}^{n-O_{n}}\frac{E_{u}a_{u+1}}{q_{u+1}^{\slash}}\right)-\frac{O_{r}}{q_{r+2}^{\slash}}$,
and $0\le\sum_{u=r+1}^{n}\frac{O_{u}b_{u}}{q_{u+1}^{\slash}}\le\left(\sum_{u=r+1}^{n}\frac{O_{u}a_{u+1}}{q_{u+1}^{\slash}}\right)-\frac{O_{r+1}}{q_{r+2}^{\slash}}=\left(\sum_{u=r+1+O_{r}}^{n-E_{n}}\frac{O_{u}a_{u+1}}{q_{u+1}^{\slash}}\right)-\frac{E_{r}}{q_{r+2}^{\slash}}$.
We now estimate the sums over $a_{u}$. Since $\frac{a_{u+1}}{q_{u+1}^{\slash}}<\frac{1}{q_{u}^{\slash}}$,
and considering the parity of $r,n$ we can refine the sums to:

\begin{equation}
\sum_{u=r+1}^{n}\frac{E_{u}a_{u+1}}{q_{u+1}^{\slash}}<\sum_{u=r+1+E_{r}}^{n-O_{n}}\frac{E_{u}}{q_{u}^{\slash}}<\frac{2}{q_{r+1+E_{r}}^{\slash}}\label{eq:maxEven}
\end{equation}

\begin{equation}
\sum_{u=r+1}^{n}\frac{O_{u}a_{u+1}}{q_{u+1}^{\slash}}<\sum_{u=r+1+O_{r}}^{n-E_{n}}\frac{O_{u}}{q_{u}^{\slash}}<\frac{2}{q_{r+1+O_{r}}^{\slash}}\label{eq:maxOdd}
\end{equation}
Now $r+1+E_{r}\ge2$ and $q_{2}^{\slash}>2$ and so $\sum_{u=r+1}^{n}\frac{E_{u}a_{u+1}}{q_{u+1}^{\slash}}=\left\{ \sum_{u=r+1}^{n}\frac{E_{u}a_{u+1}}{q_{u+1}^{\slash}}\right\} $.
Similarly $r+1+O_{r}\ge2$ for $r\ge1$ and so $\sum_{u=r+1}^{n}\frac{O_{u}a_{u+1}}{q_{u+1}^{\slash}}=\left\{ \sum_{u=r+1}^{n}\frac{O_{u}a_{u+1}}{q_{u+1}^{\slash}}\right\} $
unless $r=0$ and $q_{1}^{\slash}<2$. But in the latter case $q_{1}=1$
and hence $\alpha>\frac{1}{2}$ which means $b_{0}=0$. Hence $\sum_{u=r+1}^{n}\frac{O_{u}a_{u+1}}{q_{u+1}^{\slash}}=\left\{ \sum_{u=r+1}^{n}\frac{O_{u}a_{u+1}}{q_{u+1}^{\slash}}\right\} $
whenever $b_{r}\ne0$. We now know both these sums are less than
$1$. 

Now recall $\alpha=\frac{p_{u}}{q_{u}}+\frac{(-1)^{u}}{q_{u}q_{u+1}^{\slash}}$
so that $\frac{1}{q_{u+1}^{\slash}}=(-1)^{u}\left(q_{u}\alpha-p_{u}\right)$.
Hence $\frac{a_{u+1}}{q_{u+1}^{\slash}}=(-1)^{u}a_{u+1}\left(q_{u}\alpha-p_{u}\right)=(-1)^{u}\left(q_{u+1}-q_{u-1}\right)\alpha-(-1)^{u}a_{u+1}p_{u}$.
Summing over $u$ even gives the telescoping sum:
\[
\sum_{u=r+1}^{n}\frac{E_{u}a_{u+1}}{q_{u+1}^{\slash}}=\left\{ \sum_{u=r+1}^{n}\frac{E_{u}a_{u+1}}{q_{u+1}^{\slash}}\right\} =\left\{ \sum_{u=r+1+O_{r+1}}^{n-O_{n}}E_{u}\left(q_{u+1}-q_{u-1}\right)\alpha\right\} =\left\{ \left(q_{n+1-O_{n}}-q_{r+O_{r+1}}\right)\alpha\right\} =\left\{ -\frac{1}{q_{n+2-O_{n}}^{\slash}}+\frac{1}{q_{r+1+O_{r+1}}^{\slash}}\right\} <\frac{1}{q_{r+1+E_{r}}^{\slash}}
\]

Again using $r+1+E_{r}\ge2$ gives $\frac{1}{q_{r+1+E_{r}}^{\slash}}<\frac{1}{2}$and
so $0\le\sum_{u=r+1}^{n}\frac{E_{u}b_{u}}{q_{u+1}^{\slash}}<\frac{1}{q_{r+1+E_{r}}^{\slash}}-\frac{E_{r+1}}{q_{r+2}^{\slash}}<\frac{1}{2}$.
Following the same argument for $u$ odd gives $0\le\sum_{u=r+1}^{n}\frac{O_{u}b_{u}}{q_{u+1}^{\slash}}<\frac{1}{q_{r+1+O_{r}}^{\slash}}-\frac{O_{r+1+O_{r}}}{q_{r+2}^{\slash}}<\frac{1}{2}$
but again with the proviso that $b_{r}\ne0$

We can now use these results in \eqref{eq:alphar} to obtain (for
$b_{r}\ne0$):

\[
-\frac{1}{2}<\left(-\frac{1}{q_{r+1+O_{r}}^{\slash}}+\frac{E_{r}}{q_{r+2}^{\slash}}\right)<\left(-\sum_{u=r+1}^{n}\frac{O_{u}b_{u}}{q_{u+1}^{\slash}}\right)<\,\,\left\{ \left\{ \alpha_{r00}\right\} \right\} \,\,<\left(\sum_{u=r+1}^{n}\frac{E_{u}b_{u}}{q_{u+1}^{\slash}}\right)<\left(\frac{1}{q_{r+1+E_{r}}^{\slash}}-\frac{O_{r}}{q_{r+2}^{\slash}}\right)<\frac{1}{2}
\]

It is easily checked that this is a restatement of the theorem result.
\end{proof}
We now combine the previous lemmas of this section to obtain:

\begin{lem}
\label{lem:Epsilon}For $rst$ canonical, $\epsilon_{rst}$ has lower
and upper bounds $\epsilon_{rst}^{L}$ and $\epsilon_{rst}^{U}=\epsilon_{rst}^{L}+\frac{1}{q_{r+1}^{\slash}}$
respectively, satisfying: 
\[
-\frac{1}{q_{r+1}^{\slash}}<\left(\frac{s-1}{q_{r+1}^{\slash}}+\frac{1}{q_{r+2}^{\slash}}\right)<\left(\frac{(s-1)q_{r}+t}{q_{r}q_{r+1}^{\slash}}+\frac{1}{q_{r+2}^{\slash}}\right)=\epsilon_{rst}^{L}<\epsilon_{rst}<\epsilon_{rst}^{U}=\left(\frac{sq_{r}+t}{q_{r}q_{r+1}^{\slash}}+\frac{1}{q_{r+2}^{\slash}}\right)\le\left(\frac{s+1}{q_{r+1}^{\slash}}+\frac{1}{q_{r+2}^{\slash}}\right)\le\frac{1}{q_{r}^{\slash}}
\]

. 

In particular $\left|\epsilon_{rst}\right|<\frac{1}{q_{r}^{\slash}}$
and if $s\ne0$ $\epsilon_{rst}>0$. Further for $r=n$ $\epsilon_{rst}=\frac{sq_{r}+t}{q_{r}q_{r+1}^{\slash}}>0$,
and for $r<n,s=0,t=q_{r}$ $\epsilon_{rst}>\frac{1}{q_{r+2}^{\slash}}>0$
\end{lem}

\begin{proof}
Using $0\le s\le b_{r}-1,1\le t\le q_{r}$ for $rst$ canonical gives
$0<\frac{sq_{r}+t}{q_{r}q_{r+1}^{\slash}}\le\frac{s+1}{q_{r+1}^{\slash}}$.
Further $\frac{1}{q_{r+2}^{\slash}}+\frac{s+1}{q_{r+1}^{\slash}}\le\frac{1}{q_{r+2}^{\slash}}+\frac{a_{r+1}}{q_{r+1}^{\slash}}=\frac{1}{q_{r}^{\slash}}$.
Hence $\frac{-1}{q_{r+1}^{\slash}}<\epsilon<\frac{1}{q_{r}^{\slash}}$
and for $q_{r}\ge2$ this means $\frac{-1}{2}<\epsilon<\frac{1}{2}$
so that $\epsilon_{rst}=\{\{\epsilon_{rst}\}\}$. The various inequalities
of the first result then follow and we turn to the remaining results.
For $r=n$ we have $\alpha_{n00}=0$ and so $\epsilon_{rst}=\frac{sq_{r}+t}{q_{r}q_{r+1}^{\slash}}>0$.
For $s\ne0$, $\epsilon_{rst}>\frac{s-1}{q_{r+1}^{\slash}}+\frac{1}{q_{r+2}^{\slash}}>0$.
If $s=0$ but $t=q_{r}$, $\epsilon_{rst}>\frac{(s-1)q_{r}+t}{q_{r}q_{r+1}^{\slash}}+\frac{1}{q_{r+2}^{\slash}}=\frac{1}{q_{r+2}^{\slash}}>0$.

\end{proof}
The results can be packaged into one simple and quite elegant result
which includes all of the cases above:

\begin{cor}[Parity Duality]
\label{cor:ParityDuality}\emph{Writing $l_{rst}=\left\{ (-1)^{r}\frac{tp_{r}}{q_{r}}\right\} +\epsilon_{rst}^{L}$
and $u_{rst}=\left\{ (-1)^{r}\frac{tp_{r}}{q_{r}}\right\} +\epsilon_{rst}^{U}$,
then $0<\frac{1}{q_{r+2}^{\slash}}<l_{rst}=u_{rst}-\frac{1}{q_{r+1}^{\slash}}$
 and $u_{rst}\le1-\frac{1}{q_{r}}+\frac{1}{q_{r}^{\slash}}<1$ ($r>0)$
and $u_{rst}\le1-\frac{1}{q_{r+1}^{\slash}}<1$ ($r=0$). In addition:
 }

\emph{}

\begin{align*}
l_{rst} & <\text{\ensuremath{\alpha_{rst}}\ensuremath{<}\ensuremath{u_{rst}}}\qquad r\,\mathrm{even}\\
l_{rst} & <\text{\ensuremath{\overline{\alpha_{rst}}}\ensuremath{<}\ensuremath{u_{rst}} }\qquad r\,\mathrm{odd}
\end{align*}
\end{cor}

\begin{proof}
Note that $\left\{ (-1)^{r}\frac{tp_{r}}{q_{r}}\right\} \le1-\frac{1}{q_{r}}$
 and by (\ref{lem:Epsilon}) $\epsilon_{rst}^{U}\le\frac{1}{q_{r}^{\slash}}$
giving $u_{rst}\le1-\frac{1}{q_{r}}+\frac{1}{q_{r}^{\slash}}$. But
$\frac{1}{q_{r}^{\slash}}<\frac{1}{q_{r}}$ for $r>0$ and then $u_{rst}<1$. 

For $r=0$ we have $\left\{ (-1)^{r}\frac{tp_{r}}{q_{r}}\right\} =0$
but we now have $q_{0}^{\slash}=q_{0}=1$ which does not achieve the
result. However also from (\ref{lem:Epsilon}) if $r=0$ $\epsilon_{0st}^{U}\le\frac{1}{q_{0}^{\slash}}-\frac{1}{q_{1}^{\slash}}=1-\frac{1}{q_{1}^{\slash}}$,
and $q_{1}^{\slash}>1$ so that $u_{0st}<1$ and so $u_{rst}<1$ for
any $r\ge0$.  We now consider $l_{rst}$. 

If $\left\{ \frac{tp_{r}}{q_{r}}\right\} \ne0$ then $q_{r}>1$ and
we have $\left\{ (-1)^{r}\frac{tp_{r}}{q_{r}}\right\} \ge\frac{1}{q_{r}}$.
From (\ref{lem:Epsilon}) $\epsilon_{rst}^{U}>\frac{t}{q_{r}q_{r+1}^{\slash}}+\frac{1}{q_{r+2}^{\slash}}$
and so $l_{rst}>\frac{1}{q_{r}}+\frac{1}{q_{r+2}^{\slash}}-\frac{1}{q_{r+1}^{\slash}}>\frac{1}{q_{r+2}^{\slash}}>0$. 

If $\left\{ \frac{tp_{r}}{q_{r}}\right\} =0$ then $l_{rst}=\epsilon_{rst}^{U}-\frac{1}{q_{r+1}^{\slash}}$
and either $q_{r}=1$ (and then $t=1=q_{r}$) or $t=q_{r}$ - in either
case $t=q_{r}$. Then $\epsilon_{rsq_{r}}^{U}>\frac{1}{q_{r+1}^{\slash}}+\frac{1}{q_{r+2}^{\slash}}$
and hence $l_{rst}>\frac{1}{q_{r+2}^{\slash}}>0$. 

Finally $\epsilon_{rst}^{L}=\epsilon_{rst}^{U}-\frac{1}{q_{r+1}^{\slash}}$
in (\ref{lem:Epsilon}), which gives us $l_{rst}=u_{rst}-\frac{1}{q_{r+1}^{\slash}}$. 

The first result follows, and we turn to consideration of the results
involving $\alpha_{rst}$.

Recall from (\ref{prop:BaseAlpha}) $\alpha_{rst}=\left\{ \frac{tp_{r}}{q_{r}}+(-1)^{r}\epsilon_{rst}\right\} $
and $\left|\epsilon_{rst}\right|<\frac{1}{q_{r}^{\slash}}$ and $\epsilon_{rst}^{L}<\epsilon_{rst}<\epsilon_{rst}^{U}$

For $\left\{ \frac{tp_{r}}{q_{r}}\right\} \ne0$ (so that $q_{r}>1$)
$\left\{ \frac{tp_{r}}{q_{r}}+(-1)^{r}\epsilon_{rst}\right\} =\left\{ \frac{tp_{r}}{q_{r}}\right\} +(-1)^{r}\epsilon_{rst}$.

For $r$ even this gives $\alpha_{rst}=\left\{ \frac{tp_{r}}{q_{r}}\right\} +\epsilon_{rst}$. 

If $r$ is odd then $\ensuremath{\overline{\alpha_{rst}}}=1-\alpha_{rst}=1-\left(\left\{ \frac{tp_{r}}{q_{r}}\right\} -\epsilon_{rst}\right)=\left\{ -\frac{tp_{r}}{q_{r}}\right\} +\epsilon_{rst}$. 

The results follow in both cases from $\epsilon_{rst}^{L}<\epsilon_{rst}<\epsilon_{rst}^{U}$
and the definitions of $u_{rst},l_{rst}$.

If $\left\{ \frac{tp_{r}}{q_{r}}\right\} =0$ we have $\left\{ \frac{tp_{r}}{q_{r}}+(-1)^{r}\epsilon_{rst}\right\} =\left\{ (-1)^{r}\epsilon_{rst}\right\} $.
We also have $u_{rst}=\epsilon_{rst}^{U},l_{rst}=\epsilon_{rst}^{L}$.
Hence $0<l_{rst}=\epsilon_{rst}^{L}<\epsilon_{rst}<\epsilon_{rst}^{U}=u_{rst}<1$
which also gives us $\left\{ \epsilon_{rst}\right\} =\epsilon_{rst}$.

For $r$ even this gives $\alpha_{rst}=\left\{ \epsilon_{rst}\right\} =\epsilon_{rst}$. 

If $r$ is odd then $\ensuremath{\overline{\alpha_{rst}}}=1-\left\{ -\epsilon_{rst}\right\} =\left\{ \epsilon_{rst}\right\} =\epsilon_{rst}$.

The results follow directly.
\end{proof}
Note that we can rewrite the second inequality using $\overline{x}=1-x$
to obtain the dual result $\overline{l_{rst}}>\alpha_{rst}>\overline{u_{rst}}$
for $r$ odd.

We have now developed a lot of understanding of where the point $\alpha_{rst}$
lies in the partition defined by the points $t\frac{p_{r}}{q_{r}}$,
but there are also some important special cases where we can go further
(and which are also important when we come to consider the application
to anergodic Birkhoff sums).

Corollary \ref{cor:ParityDuality} tells us that $\alpha_{rst}$ always
lies in one of the two intervals of length $1/q_{r}$either side of
$tp_{r}/q_{r}$, and in the majority of cases ($s\ensuremath{\ne0}$or
$r=n$ or $t=q_{r}$) it lies in the primary interval. With a little
extra work we can identify the interval more precisely. 
The partition interval of $\alpha_{rst}$

\begin{defn}
Fixing $r$, the circle points $\frac{k}{q_{r}}$ ($k=1..q_{r}$)
form a partition defining $q_{r}$ distinct open intervals of length
$\frac{1}{q_{r}}$. For $q_{r}=1$ there is only such one point and
interval, but for $q_{r}>1$ each point is surrounded by two distinct
intervals, one on each side, which we will label $I_{k}$ and $-I_{k}$
(the positive and negative intervals of the point $\frac{k}{q_{r}}$),
which we will write generically as $(-1)^{u}I_{k}$ for some $u_{\mathbb{Z}}$.
We define $(-1)^{u}I_{k}$ specifically as the set of points $\left\{ \frac{k\,+\,(-1)^{u}\nu}{q_{r}}\right\} _{Set}$
for $\nu\in(0,1)$. 
\end{defn}

Note that in fact the latter definition also holds for $q_{r}=1$,
but in this case the two intervals coincide ($I_{k}=-I_{k}$). 

Recall that circle points $\frac{k}{q_{r}},\frac{k+mq_{r}}{q_{r}}$
coincide for any integer $m_{\mathbb{Z}}$, so that $(-1)^{u}I_{k}=(-1)^{u}I_{k+mq_{r}}$. 

The set of positive intervals and the set of negative intervals represent
two different ways of labelling the same underlying set of intervals
which make up the partition. There is therefore a natural bijection
between the two sets of labels. We can capture this explicitly as
follows:
\begin{lem}
$(-1)^{u}I_{k}=(-1)^{u+1}I_{k+(-1)^{u}}$
\end{lem}

\begin{proof}
$(-1)^{u}I_{k}=\left\{ \frac{k\,+\,(-1)^{u}\nu}{q_{r}}\right\} =\left\{ \frac{k+(-1)^{u}\,-\,(-1)^{u}(\nu-1)}{q_{r}}\right\} =(-1)^{u+1}I_{k+(-1)^{u}}$.
\end{proof}

\begin{lem}
,  $\alpha_{rst}\in(-1)^{r}\sgn(\epsilon_{rst})I_{k}$ where $k=tp_{r}\bmod q_{r}$.
\end{lem}

\begin{proof}
This follows immediately from $\alpha_{rst}=\left\{ t\frac{p_{r}}{q_{r}}+(-1)^{r}\epsilon_{rst}\right\} $
and $\left|\epsilon_{rst}\right|<\frac{1}{q_{r}}$. 
\end{proof}

\begin{prop}
For $r\ge0$, $p_{r}=(-1)^{r+1}q_{r-1}^{-1}\bmod q_{r}$
\end{prop}

\begin{proof}
This follows immediately from the identity$q_{r}p_{r-1}-p_{r}q_{r-1}=(-1)^{r}$for
$r\ge0$ 
\end{proof}
\begin{lem}
For canonical $rst$, and any integer $u$, the points $\alpha_{rst}$
which lie in $(-1)^{u}I_{k}$ are precisely those for which either
(a) $t=(-1)^{r+1}kq_{r-1}\bmod q_{r}$ and $\sgn(\epsilon_{rst})=(-1)^{u+r}$,
or (b) $t=(-1)^{r+1}\left(k+(-1)^{u}\right)q_{r-1}\bmod q_{r}$ and
$\sgn(\epsilon_{rst})=(-1)^{u+r+1}$ 
\end{lem}

\begin{proof}
From previous 3 lemmas
\end{proof}
Special case - the intervals around the origin

\begin{cor}
Given $r$, the canonical triples for which $\alpha_{rst}$ lies in
an interval around the origin (ie $\pm I_{0}$) are given by:
\end{cor}

\begin{enumerate}
\item For $(-1)^{r}I_{0}$: $t=q_{r}$ (for any $s$) or $r<n,s=0,t=q_{r}-q_{r-1}$
and $\sgn(\epsilon_{r0(q_{r}-q_{r-1})})=-1$). Then we get the lower
bound $l_{r0(q_{r}-q_{r-1})}>\frac{1}{q_{r}}+\frac{-1}{q_{r+1}^{\slash}}+\frac{1}{q_{r+2}^{\slash}}>\frac{1}{2q_{r}}$.
\item For $(-1)^{r+1}I_{0}$: $t=q_{r-1}\bmod q_{r}$ and $\sgn(\epsilon_{rst})=+1$.
Then we get the upper bound $u_{rsq_{r-1}}<1-\frac{a_{r+1}^{\slash}-s}{q_{r+1}^{\slash}}$.
\end{enumerate}
\begin{proof}
For canonical $rst$, the points $\alpha_{rst}$ which lie in $(-1)^{u}I_{0}$
are precisely those for which either (a) $t=0\bmod q_{r}$ and $\sgn(\epsilon_{rst})=(-1)^{u+r}$,
or (b) $t=(-1)^{r+1}\left((-1)^{u}\right)q_{r-1}\bmod q_{r}$ and
$\sgn(\epsilon_{rst})=(-1)^{u+r+1}$ 

(a) For $t=q_{r}$ we have from (\ref{lem:Epsilon}) that always $\sgn(\epsilon_{rst})=+1$
and so $u+r=0\bmod2$ and $\alpha_{rsq_{r}}$ always lies in $(-1)^{r}I_{0}$. 

(b) If $u=r\bmod2$, the possibilities are $\sgn(\epsilon_{rst})=-1$
and $t=q_{r}-q_{r-1}$ giving $\alpha_{rst}$ in $(-1)^{r}I_{0}$.
But in this case we can only have $\sgn(\epsilon_{rst})=-1$ if $r<n,s=0$.
Then $l_{r0(q_{r}-q_{r-1})}>\frac{1}{q_{r}}+\left(\frac{-1}{q_{r+1}^{\slash}}+\frac{q_{r}-q_{r-1}}{q_{r}q_{r+1}^{\slash}}+\frac{1}{q_{r+2}^{\slash}}\right)>\frac{1}{2q_{r}}$.

If $u=r+1\bmod2$, the possibilities are $\sgn(\epsilon_{rst})=+1$
and $t=q_{r-1}$ giving $\alpha_{rst}$ in $(-1)^{r+1}I_{0}$. But
in this case we can only have $\sgn(\epsilon_{rst})\ne+1$ if $s=0$.
And $u_{rsq_{r-1}}=\left(1-\frac{1}{q_{r}}\right)+\left(\frac{sq_{r}+q_{r-1}}{q_{r}q_{r+1}^{\slash}}+\frac{1}{q_{r+2}^{\slash}}\right)<1-\frac{a_{r+1}^{\slash}-s}{q_{r+1}^{\slash}}$
by (\ref{cor:ParityDuality}). 
\end{proof}

\emph{}

\section{\label{sec:Analysis}Analysis of spaces of unbounded observables}

The analysis of spaces of unbounded (and non integrable) functions
is not well developed, and so in this section we lay some foundations
suitable for our purposes. The usual approaches to functional analysis
via the theory of Banach spaces and Schauder bases do not seem immediately
helpful here. Instead, noting that fundamentally we are interested
in summing functions over large orbits, we look for help from the
elementary theory of integration. In particular we find that extending
the ideas of Bounded Variation and Jordan Decomposition leads to a
natural set of function spaces suitable for our use. 

For precise definitions of circle terminology in this section, see
\ref{subsec:Circle}.

\subsection{Definitions and Basic Topology }

Recall that a real observable on the circle is a (total) function
$\phi_{\mathbb{T}\mathbb{R}}$. Our first inclination may be to define
an unbounded observable as a function which can take values on the
extended real line $\mathbb{R}\bigcup\{\pm\infty\}$. However this
is not necessary, and instead it proves technically simpler to restrict
the codomain of unbounded observables to $\mathbb{R}$ for the following
reasons:
\begin{enumerate}
\item The sum $\sum_{i=1}^{n}\phi(x_{i})$ is always well defined (we avoid
the difficulty of defining $\infty-\infty$)
\item The function $\sum_{i=1}^{n}\phi_{i}$ is always well defined (again
we avoid the difficulty of defining $\infty-\infty$)
\item We can derive observables (total functions) $\phi_{\mathbb{T}\mathbb{R}}$
from partial functions by using the simple convention that the resulting
observable vanishes outside the domain of the partial function. We
will write $\phi_{X\mathbb{R}}(x)=\psi_{Y\mathbb{R}}(x)$ for $Y\subseteq X$
to mean $\phi$ coincides with $\psi$ on $Y$ and vanishes on $X$
outside $Y$.
\end{enumerate}
\begin{example}
The function $\phi_{\mathbb{I}\mathbb{R}}(x)=\frac{1}{x}$ on $\mathbb{I}=(0,1)$
is a partial function on $\mathbb{T}=[0,1)$ whose domain is $(0,1).$
It extends to the observable $\phi_{\mathbb{T}\mathbb{R}}$, a total
function on $[0,1)$ by $\phi_{\mathbb{T}\mathbb{R}}(0)\coloneqq0$.
The observable $\phi$ is unbounded but defined on the whole of $[0,1)$,
and has codomain $\mathbb{\mathbb{R}}$ (there is no need to extend
$\mathbb{\mathbb{R}}$ with points at infinity). We allow ourselves
to write $\phi_{\mathbb{T}\mathbb{R}}(x)=\frac{1}{x}$ with the convention
that $\phi(x)=0$ where $\frac{1}{x}$ is undefined.
\end{example}

The following treatment of locally bounded variation is slightly non-standard
but again allows us to avoid some technical difficulties. We first
extend the definition of bounded variation to non-closed intervals:

\begin{defn}[Bounded Variation on an Interval]
Let $J$ be any interval (not necessarily closed) of the real line
or the circle, and let $Part(J)$ the set of partitions of $J$ (ie
ordered sequences of the form $(x_{i})_{i=1}^{n}$ for $n\ge1$).
Given a function $\phi_{J\mathbb{\mathbb{R}}}$, if $J$ is not the
full circle, we define the \noun{variation} of $\phi$ over $J$ as
$\var_{J}\phi=\sup_{Part(J)}\sum_{r=1}^{n-1}\left|\phi(x_{r+1})-\phi(x_{r})\right|$.
If $\var_{J}\phi$ exists we say $\phi$ is of \noun{Bounded Variation}
$(\phi$ is BV) on $J$, otherwise $\phi$ is of \noun{Unbounded Variation}
($\phi$ is UBV) on $J$. If $J$ is the full circle, we define $\var_{\mathbb{T}}\phi=\sup_{Part(\mathbb{T})}\sum_{r=1}^{n}\left|\phi(x_{r+1})-\phi(x_{r})\right|$
where $x_{n+1}\coloneqq x_{1}$. 
\end{defn}

Note that this definition coincides with the classical definitions
of Bounded Variation on closed intervals and on the full circle. 
\begin{example}
Given the observable $\phi x=\frac{1}{x}$ on the circle (using our
convention that $\phi(0)=0$) then $\phi$ is BV on the directed intervals
$[\frac{1}{2},1)$,$[\frac{1}{2},1]$ but UBV on $[0,\frac{1}{2})$
and $(0,\frac{1}{2})$.
\end{example}

We now give a simple but powerful principle which is not specifically
related to unboundedness, but includes it.
\begin{defn}[Localisation]
Given a topological space $(X,\mathscr{F})$, and a proposition $P$
defined on the open sets  of $X$. We derive the proposition \noun{locally}
$P_{X}$ (written $LP_{X}$) which is defined at points of $X$ as
follows: $x$ is locally $P$ (ie $LP_{X}(x)$) iff $x$ has an open
neighbourhood $N_{x}\in\mathscr{F}$ satisfying $P_{X}(N_{x})$. 
In particular if the proposition is defined in terms of a function
$\phi_{XY}$, ie $P\coloneqq P(\phi_{XY})$, then when $x$ is locally
$P$ (ie $LP(x)$) we will also say that $\phi$ is locally $P$ at
$x$. We will denote the set of points at which $x$ is locally $P$
as $LP(\phi)$.
\end{defn}

\begin{example}
\label{exa:LBV}Let $(J,\mathscr{F})$ be an interval of the real
line or the circle equipped with the ambient subspace topology, and
$\phi_{J\mathbb{R}}$ a fixed function. Let $P=BV$, the proposition
defined on each $K\in\mathscr{F}$ that $\phi$ is of Bounded Variation
on $K$. Then $\phi_{J\mathbb{R}}$ is of \noun{Locally Bounded Variation
}(ie $LP=LBV$) at $x_{J}$ iff $x$ has a neighbourhood $K\in\mathscr{F}$
on which $\phi$ is of Bounded Variation. We then write $LBV_{J}(x)$.
Note that if $\phi$ is not $LBV$ at $x$, then there is no neighbourhood
of $J$ containing $x$ on which $\phi$ is $LBV$, which means in
particular that $\phi$ is $UBV$ on \emph{every} open interval containing
$x$ and we will write $UBV(x)$. Finally $LBV_{J}(\phi)$ is the
set of points in $J$ at which $\phi$ is of locally bounded variation,
and $UBV_{J}(\phi)$ is the set of points at which $\phi$ is of Unbounded
Variation. 
\end{example}

\begin{prop}
Given a topological space $(X,\mathscr{F})$, an observable $\phi_{X\mathbb{R}}$,
and a proposition $P_{X}$ defined on open sets of $X$, then the
set of points $LP_{X}(\phi)$ (the points at which $\phi$ is locally
$P$) is open, and its boundary belongs to the complement of $LP(\phi)$,
the closed set of points at which $\phi$ is not locally $P$
\end{prop}

\begin{proof}
Let $x\in LP(\phi)$ so there is a neighbourhood $N_{x}$ with $x\in N_{x}$
and $P(N_{x})$. By definition there is an open set $O_{x}\in\mathscr{F}$
with $x\in O_{x}\subseteq N_{x}$. Now for any $y\in O_{x}$, $N_{x}$
is also a neighbourhood of $y$, and since $P(N_{x})$ we have $y\in LP(\phi)$.
Hence $O_{x}\subseteq LP(\phi)$ and so $x$ is an interior point
of $LP(\phi)$ and $LP(\phi)$ is open. 

Now $LP(x)$ is defined for every $x_{X}$ and so either $LP(x)$
or $!LP(x)$, and so the points at which $!LP(x)$ form the complement
of $LP(\phi)$. Since $LP(\phi)$ is open, its complement is closed
and the boundary of $LP(\phi)$ lies completely in that complement. 
\end{proof}
\begin{cor}[LBV topology]
\label{cor:LBV Topology}In an interval $J$ with the subspace topology,
the points $LBV_{J}(\phi)$ at which a function $\phi_{J\mathbb{\mathbb{R}}}$
is of locally bounded variation, form an at most countable set of
disjoint open intervals in $J$, and the boundaries of these intervals
in $J$ are of unbounded variation.  
\end{cor}

\begin{proof}
Applying the Proposition to example \ref{exa:LBV} gives us that $LBV_{J}(\phi)$
(the set of points of $J$ at which $\phi$ is of locally bounded
variation) is an open set in $J$, and hence (by Proposition \ref{prop: Countable intervals})
is an at most countable union of disjoint open intervals. The boundary
of the set $LBV_{J}(\phi)$ in $J$ is therefore the set of boundaries
of these intervals and the proposition tells us these lie in $!LBV(\phi)$,
ie they are in $UBV(\phi)$. 

\end{proof}
Note that the endpoints of an interval $J$ itself are \emph{not}
boundary points in the subspace topology, so that the endpoints of
$J$ may still be $LBV_{J}(\phi).$ For example let $\phi x=1/x$
with $J=[\frac{3}{4},1],$then $\phi_{J\mathbb{\mathbb{R}}}$ is $LBV$
on the whole of $J$ including both endpoints, and so $LBV(\phi)$
is the single interval $J$ which is open without boundary in the
subspace topology of $J$. Contrast this with the case of $J=[0,\frac{1}{4}]$,
when $LBV(\phi)$ is now $(0,\frac{1}{4}]$ and this interval has
the boundary point $0$ in $J$ and is $UBV$, but $\frac{1}{4}$
is $LBV$ and is not a boundary point of $LBV(\phi)$ in $J$. 

Also from Corollary \ref{cor:LBV Topology} we have $!LBV(\phi)=UBV(\phi)$,
ie $\phi$ is of Unbounded Variation at all points (including endpoints)
outside the open intervals of $LBV(\phi)$. Note that $UBV(\phi)$
itself may have positive measure. For example the Dirichlet observable
(the indicator function on the rationals in $[0,1)$) is $UBV$ everywhere
and $LBV(\phi)$ is empty. 

We summarise our findings in the following Theorem: 
\begin{thm}
For any interval $J$ of the circle or real line (with the subspace
topology), and observable $\phi_{J\mathbb{\mathbb{R}}}$, there is
a canonical decomposition of $J$ into the disjoint union of the closed
set $UBV_{J}(\phi)$ with an at most countable set of open subintervals
$\{I_{k}\}$ on which $\phi$ is $LBV$. Further, writing $\phi_{U}=\left\llbracket x\in UBV(\phi)\right\rrbracket \phi,\phi_{k}=\left\llbracket x\in I_{k}\right\rrbracket \phi$
we have 
\begin{equation}
\phi=\phi_{U}+\sum_{k}\phi_{k}\label{eq:phi}
\end{equation}

where each $\phi_{k}$ is $LBV$ with disjoint open support $I_{k}$,
and the boundary (endpoints) of each $I_{k}$ are in $UBV(\phi)$ 
\end{thm}

\begin{defn}[Primitive functions]
Given intervals $I\subseteq J$, a function $\phi_{J\mathbb{R}}$
is a \noun{primitive} on $I$ if it vanishes\footnote{this means $I$ contains the \emph{set theoretic} support of $\phi$
in $J$, but not necessarily the \emph{topological} support as this
is the closure of the set theoretic support)} outside $I$, is monotone on $I$, and has a bound on $I$ (ie $\inf_{I}\phi$
or $\sup_{I}\phi$ exists, or both). If it has both bounds it is a
bounded primitive, otherwise it is an upper or lower bounded primitive.
\end{defn}

Notes:
\begin{enumerate}
\item A monotone function on an interval is $LBV$ on that interval, so
that any primitive on $I$ is $LBV$ on $I$.
\item The monotone constraint rules out the existence of primitives which
are bounded functions of unbounded variation (eg $\sin1/x$ on $(0,1)$).
\item We need to take the same care over primitives and interval endpoints
as we do for any $LBV$functions. So a primitive may be unbounded
on non-included endpoints, eg $\phi_{\mathbb{TR}}x=1/x$ is a lower
bounded primitive on $I=(0,1)$ but unbounded at $0$ (where its value
is $0$). If $I$ includes an endpoint $a$, then $\phi$ must be
bounded on $I$ at $a$ though it may be unbounded outside $I$ at
$a$, eg $\phi_{\mathbb{TR}}x=1/x$ is also a bounded primitive on
$I=[\frac{1}{2},1]$, where it is bounded at $1$ in $I$, but unbounded
at $1$ in $\mathbb{T}$. 
\end{enumerate}

\subsection{Some initial function classes}

Normalised functions
\begin{defn}
We say an observable $\phi$ is \noun{normalised} if $\phi(x)=0$
on $UBV(\phi)$. (Note this requires $UBV(\phi)$ to have empty interior).
The \noun{normalisation} of $\phi$ is the normalised function $\phi^{Norm}(x)=\left\llbracket x\not\in UBV(\phi)\right\rrbracket \phi(x)$.
 We say $\phi$ is equivalent to $\psi$ under normalisation if the
equivalence relation $\phi^{Norm}=\psi^{Norm}$ holds. 
\end{defn}

\begin{example}
The observable $\phi x=1/x$ is normalised (since $UBV(\phi)=\{0\}$
and $\phi0=0$). However the Dirichlet function is not normalised,
since it is UBV everywhere but not $0$ everywhere. Its normalisation
is the $0$ function.
\end{example}

When considering Birkhoff sums $S_{N}\phi=\sum_{r=1}^{N}\phi(x_{0}+r\alpha)$
we are primarily interested only in orbits which avoid the unbounded
points of $\phi$. In these cases the value of $\phi$ at each unbounded
point has no effect on the Birkhoff sum, ie $S_{N}\phi=S_{N}\phi^{Norm}$.
We will henceforth regard observables as normalised unless otherwise
indicated. 

Note that if $\phi$ is normalised then $\phi_{U}=0$ in \eqref{eq:phi},
and so this equation simplifies to $\phi=\sum_{k}\phi_{k}$. In the
next section we will investigate the decomposition of the functions
$\phi_{k}$.

Banach congruence classes
Given an additive group $G$ of real valued functions, let $G_{B}=\{\phi:\sup|\phi|<\infty\}$
be the Banach space (and subgroup) of bounded functions. Then each
congruence class of $G/G_{B}$ consists of unbounded functions whose
differences lie in $G_{B}$. 
\begin{defn}
We say that two observables $\phi,\psi$ are of \noun{finite difference}
if $\phi-\psi$ is bounded, and we write ($\phi\stackrel{B}{\sim}\psi$)
noting that this is an equivalence relation. 
\end{defn}

\begin{example}
On $[0,1)$ it is easily shown that $\pi\cot\pi x\stackrel{B}{\sim}\frac{1-2x}{x(1-x)}\stackrel{B}{\sim}\frac{1}{\{\{x\}\}}$
but $\csc\pi x\not\sim\cot\pi x$ as their difference is unbounded
as $x\rightarrow1$.
\end{example}

If $\phi,\psi$ are of bounded difference, this means $S_{N}\phi-S_{N}\psi=O(N)$
and further if $\oint(\phi-\psi)=0$ then $S_{N}\phi-S_{N}\psi=O(\log N)^{2+\epsilon}$
({*}{*})  for ae $\alpha$. In particular this means if $S_{N}\phi$
has a super-linear growth rate (ie greater than $O(N)$), then $S_{N}\psi$
has the same growth rate and the same set of unbounded points, and
that this super-linear growth rate is determined purely by the behaviour
of $\phi,\psi$ at unbounded points. 

We now refine some existing understanding and terminology related
to boundedness to enable us to study unboundedness.
Monotone functions
 We will use the word ``increasing'' always to mean monotone increasing
unless we specify otherwise. The word ``decreasing'' is dual to
increasing, and can be substituted in the following paragraph.

\subsection{Classification of primitive functions on an interval }

\begin{figure}
\begin{centering}
\includegraphics{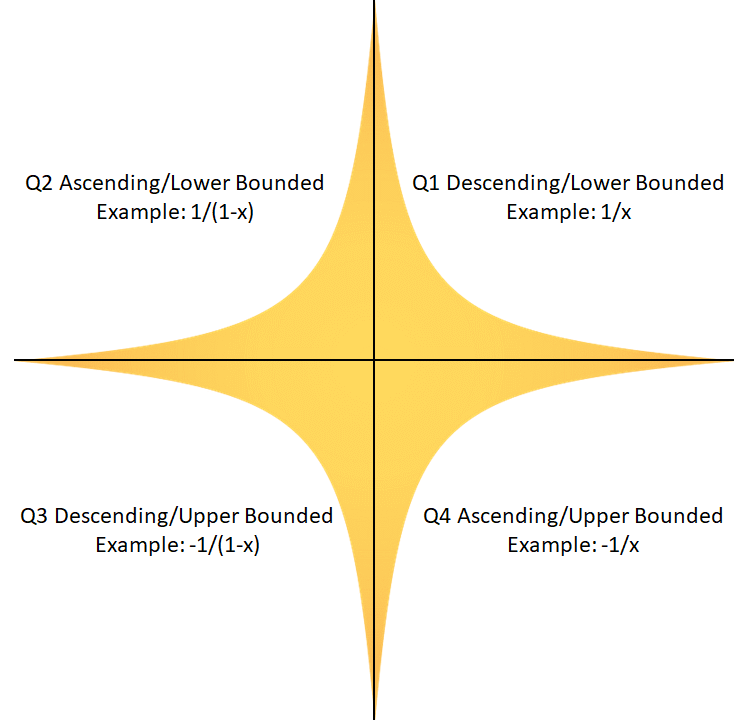}
\par\end{centering}
\caption{Quadrants of Primitives}

\end{figure}

Let $I\subseteq J$ be an interval with endpoint coordinates $a,b$,
such that $J$ contains the endpoints $a,b$. (If $J$ is the circle
we allow $a,b$ to represent the same circle point, but require $b=a+1$).
Recall $\phi_{J\mathbb{\mathbb{R}}}$ is a primitive on $I$ if $\phi$
vanishes outside $I$, is monotone on $I$, and has a bound (ie at
least one of $\inf\phi$ or $\sup\phi$ exists). Note that a primitive
function which is constant is both ascending and descending on $I$,
but any other primitive is either ascending or descending.

A primitive is called upper bounded if $\sup\phi$ exists, and lower
bounded if $\inf\phi$ exists. It is called a bounded primitive if
both bounds exist, and an unbounded primitive otherwise. We can therefore
classify an unbounded primitive by whether it is ascending or descending
(A or D) and whether it is upper or lower bounded (U or L). 

Let us write the set of \textbf{D}escending, \textbf{L}ower bounded
primitives on $I$ as DL, we similarly define the sets AL, AU, and
DU. An unbounded primitive belongs to precisely one of these sets,
but a bounded primitive is both upper and lower bounded and so will
belong to both AU and AL, or to both DU and DL, except for the case
of a constant function which belongs to all 4 sets. Note that DL is
closed under addition and contains the constant primitive 0. It therefore
forms a commutative monoid. Further it forms a module over the semi-ring
$\mathbb{R}^{+}=\{x\in\mathbb{R}:x\ge0\}$. The same applies to each
of the other 3 sets. 

Note that if $\phi$ is a primitive on $I$, then so is $-\phi$.
Further, if $I$ is either open or closed (ie not semi-open) then
$\overline{\phi}(x)\coloneqq\phi(a+b-x)$ is also a primitive on $I$,
and hence so is $-\overline{\phi}$. Finally $\overline{-\phi}(x)=(-\phi)(a+b-x)=-\overline{\phi}(x)$.

Let $\mathcal{F}$ be a family of primitives of $I$. Then $-\mathcal{F}\coloneqq\{\phi:-\phi\in\mathcal{F}\}$
is clearly also a family of primitives of $I$. If $I$ is either
open or closed we can also define another family of primitives $\overline{\mathcal{F}}\coloneqq\{\phi:\overline{\phi}(a+b-x)\in\mathcal{F}\}$,
and then $\overline{-\mathcal{F}}=-\overline{\mathcal{F}}=\{\phi:-\overline{\phi}\in\mathcal{F}\}$
is a 4th family. We call the set $\{\mathcal{F},-\mathcal{F}.\overline{\mathcal{F}},-\overline{\mathcal{F}}\}$
the dual quadrants of $\mathcal{F}$ (under the dualities $-,\bar{.}$).
Note that the set $\{DL,AU,DU,AL\}$ comprises the dual quadrants
of each of its 4 members when $I$ is open or closed (but not semi-open),
and fixing $DL$ this gives $AU=-DL,AL=\overline{DL},DU=-\overline{DL}$.

If $\phi$ is a primitive of $I$ which has a lower bound of 0 (ie
$\phi\ge0$ on $I$) then we say $\phi$ is a positive primitive on
$I$, and if $\phi\le0$ we say $\phi$ is negative. We denote the
set of descending positive primitives on $I$ as $\Phi(I)$, and note
this is a sub-monoid of $DL$ ($\Phi(I)\subset DL(I)$). Then the
dual quadrant $-\Phi(I)\subset AU(I)$ is the monoid of ascending
negative primitives. If $I$ is open or closed, we can also identify
the other dual quadrants as the monoids $\overline{\Phi}(I)\subset AL(I),-\overline{\Phi}(I)\subset DU(I)$.

Positive and negative primitives have the important properties that
they can be extended and translated in $J$. 

If $\phi\in\Phi(I)$ and $K$ is a right extension of $I$ (ie $x\in K\backslash I\,\Rightarrow\,x\ge b$
where $b$ is the right endpoint of $I$), then $\phi\in\Phi(K)$
so that $\Phi(I)\subset\Phi(K)$. Similarly $-\Phi(I)\subset-\Phi(K)$,
and if $K^{\prime}$ is a left extension of $I$ (and open or closed
with $I$) then also $\overline{\Phi}(I)\subset\overline{\Phi}(K^{\prime}),-\overline{\Phi}(I)\subset-\overline{\Phi}(K^{\prime})$.

We now assume $I=(a,b)$ is open. Let $K=|c,d|$ (where we set $c=0,d=1$
when $K$ is the circle $[0,1)$). If $\phi\in\Phi(I)$ we can define
the translation $T\phi(x)=\phi(x+(a-c))$ and then $T\phi\in\Phi(c,c+(b-a))$.
By the remarks on extension above we can write $T\phi\in\Phi(c,d)$
and $-T\phi\in-\Phi(c,d)$. In particular when $K$ is the circle,
$T\phi\in\Phi(0,1),-T\phi\in-\Phi(0,1)$. Similarly if $\phi\in\overline{\Phi}(I)$
we define the translation $\overline{T}\phi(x)=\phi((d-b)+x))$ and
then $\overline{T}\phi\in\overline{\Phi}(d-(b-a),d)$ which we can
extend to $\overline{T}\phi\in\overline{\Phi}(c,d)$ and $-\overline{T}\phi\in-\overline{\Phi}(c,d)$.
In particular when $J$ is the circle, $\overline{T}\phi\in\overline{\Phi}(0,1),-\overline{T}\phi\in-\overline{\Phi}(0,1)$.
In summary, positive and negative primitives on an open interval $I$
can be translated to primitives on the interior of $K$.

Note that a monotone function on a closed interval is BV. Since a
primitive $\phi$ on $I$ vanishes outside $I$, if $I$ is closed
$\phi$ is BV (and hence LBV) on the whole of $J$. If $I$ is open
then $\phi$ is LBV on $I$ (by \ref{cor:LBV Topology}) and $LBV$
outside the closure of $I$, and also $LBV$ at the bounded endpoint
of $I$.

\subsection{Representation of functions by primitives}

Recall we denote the set of points of locally bounded variation $LBV(\phi)$,
and its complement $UBV(\phi)$. We have shown $LBV(\phi)$ is a union
of disjoint open intervals, although it is important to note that
this includes the possibilities $LBV(\phi)=\mathbb{T}$ (when $\phi$
is BV on the whole circle) and $LBV=\emptyset$ (for example $\phi(x)=\left\llbracket x=\frac{p}{q},(p,q)=1\right\rrbracket q$
is unbounded everywhere and so is not locally bounded anywhere). Otherwise
there is an at most countable set of disjoint open intervals $I_{k}=(x_{k}^{L},x_{k}^{R})$
with $LBV(\phi)=\bigcup_{k}I_{k}$ and $x_{k}^{L},x_{k}^{R}\in UBV(\phi)$.
We will assume $LBV(\phi),UBV(\phi)\ne\emptyset$ going forward. We
can now define the observables $\phi_{U}x=\left\llbracket x\in UBV(\phi)\right\rrbracket \phi x$
and $\phi_{i}=\left\llbracket x\in(x_{i}^{L},x_{i}^{R})\right\rrbracket \phi$,
which enables us to write \label{eq:phiDecomp1}$\phi=\phi_{U}+\sum_{i}\phi_{i}$
with $UBV(\phi_{i})=\{x_{i}^{L},x_{i}^{R}\}$. In the special case
of $\phi$ having a single UBV point $x$, we get $x_{1}^{L}=x_{1}^{R}=x$,
otherwise the endpoints are all distinct.

Note that for a normalised function, by definition $\phi_{U}=0$.

We can now extend the Jordan Decomposition theorem from BV functions
to LBV functions:
\begin{lem}[Jordan Decomposition of LBV]
\label{lem:Decomp}If $\phi$ is LBV on an (non-empty) open interval
$I=(a,b)$then for each point $m\in(a,b)$ there is a canonical decomposition
into primitives on $I$, namely $\phi=\phi(m)+\sum_{i=1}^{4}\phi_{i}$
where $\phi_{1}\in\Phi(a,m),\,\phi_{2}\in-\Phi(a,m),\,\phi_{3}\in\overline{\Phi}(m,b),\,\phi_{4}\in-\overline{\Phi}(m,b)$. 

\end{lem}

\begin{proof}
We partition $\phi$ into left and right functions $\phi^{L}=\left\llbracket x\in(a,m)\right\rrbracket (\phi(x)-\phi(m)),\,\phi^{R}=\left\llbracket x\in(m,b)\right\rrbracket (\phi(x)-\phi(m))$.We
now define the left and right variations of $\phi$ for $a<x<b$ by
$V^{L}(x)=\left\llbracket x\in(a,m)\right\rrbracket \var_{[x,m]}\phi,\,V^{R}(x)=\left\llbracket x\in(m,b)\right\rrbracket \var_{[m,x]}\phi$
(the total variation of $\phi$ over $(x,m),(m,x)$ respectively).
These functions are well defined since $\phi$ is BV on each interval
$(x,m)$ or $(m,x)$. We now define the difference functions $D^{L}(x)=V^{L}(x)-\phi^{L}(x),\,D^{R}(x)=V^{R}(x)-\phi^{R}(x)$.
Note that these 4 defined functions all vanish outside $(a,m)$ or
$(m,b)$ and are monotone and positive and hence positive primitives
on their domains. $V^{L},D^{L}$ are both decreasing and so $V^{L},D^{L}\in\Phi(a,m)$,
whereas $V^{R},D^{R}$ are increasing and so $V^{R},D^{R}\in\overline{\Phi}(m,b)$.
But now $\left\llbracket x\in(a,b)\right\rrbracket \phi(x)=\phi(m)+V^{L}(x)-D^{L}(x)+V^{R}(x)-D^{R}(x)$
and the result follows.
\end{proof}
We can combine \eqref{eq:phi} and Lemma \eqref{lem:Decomp} to obtain:
\begin{thm}[Decomposition Theorem]
Given an  observable $\phi_{J\mathbb{R}}$ then $LBV(\phi)$ is
an at most countable union of disjoint open intervals $I_{k}=(a_{k},b_{k})$,
and $\phi=\phi_{U}+\sum_{k}\phi_{k}$ where $\phi_{U}=\left\llbracket x\in UBV(\phi)\right\rrbracket \phi$
and $\phi_{k}$ is an LBV function on $I_{k}$. Further, given $m_{k}\in I_{k}$
we have the representation of $\phi_{k}$ in terms of primitives given
by $\phi_{k}=\phi(m_{k})+\sum_{j=1}^{4}\phi_{kj}$ where $\phi_{k1}\in\Phi(a_{k},m_{k}),\,\phi_{k2}\in-\Phi(a_{k},m_{k}),\,\phi_{k3}\in\overline{\Phi}(a_{k},m_{k}),\,\phi_{k4}\in-\overline{\Phi}(a_{k},m_{k})$.
\end{thm}

\section{\label{sec:HomBirkhoffSums}Homogeneous Birkhoff Sums of unbounded
observables}

\subsection{Introduction}

In the last section we established that if $\phi_{\mathbb{TR}}$ is
an observable on the circle we have the decomposition $\phi=\phi_{U}+\sum_{k\in K}\phi_{k}$
into (at most countable) functions of disjoint support, where each
$\phi_{k}$ is LBV with its support in an open interval $I_{k}$.
We also showed we can decompose each $\phi_{k}$ into a sum of 4 primitive
functions $\phi_{ki}$ from each quadrant of $\Phi(I)$. Since $S_{N}$
is linear this means we have reduced the study of $S_{N}(\phi,x)$
to the study of $S_{N}(\phi_{U},x)$ and the study of $S_{N}(\phi_{ki},x)$
with $\phi_{ki}$ primitive. We will generally be concerned with sums
for which $x+r\alpha$ never coincides with an unbounded point  of
$\phi$ so that $S_{N}(\phi_{U},x)=0$.  This means we can replace
$\phi$ by its normalised form $\phi^{Norm}$ (which vanishes on the
set of unbounded points of $\phi$) without changing the overall sum.
In this section we concentrate our attention on the study of $S_{N}(\phi,x)$
for primitive $\phi$. 

Now suppose $\phi$ is a primitive on the directed interval $(a,b)$.
If $\phi$ is unbounded at $a$  then $\phi_{a}\coloneqq x_{\mathbb{T}}\rightarrow\phi\{x+a\}$
is a primitive on $(0,1)$, unbounded at $0^{+}$. Similarly if $\phi$
is unbounded at $b$ then $\phi_{b}\coloneqq x_{\mathbb{T}}\rightarrow\phi\{x-b\}$
is a primitive on $(0,1)$, unbounded at $1^{-}$. Then $S_{N}(\phi,x_{0})=S_{N}(\phi_{a},x_{0}-a)$
and $S_{N}(\phi,x_{0})=S_{N}(\phi_{b},x_{0}+b)$ which means we can
reduce the study of Birkhoff sums of general primitives to the study
of Birkhoff sums of primitives on $(0,1)$. We will now assume $\phi$
is a primitive on $(0,1)$.

Recall that we call $S_{N}(\phi,x)$ a homogeneous sum when $x=0$
and inhomogeneous otherwise.

The general case of an inhomogeneous sum $S_{N}(\phi,x)$ is challenging.
This may seem surprising because we can easily convert an inhomogeneous
sum to a homogeneous sum by defining $\phi_{c}(x)=\phi(x+c)$, and
then $S_{N}(\phi,x_{0})=S_{N}(\phi_{x_{0}},0)$. The problem is that
in general $\phi_{x_{0}}$ will not then be a primitive on $(0,1)$,
and the underlying problem has not been simplified. 

Fortunately homogeneous sums are sufficient to address a high proportion
of problems of interest (including the cases described in the abstract
of this paper).. Given that this paper is already somewhat lengthy,
we therefore content ourselves with postponing consideration of inhomogeneous
sums to another paper, and restrict this paper to the homogeneous
case.  

We will now use the results on the sequential distribution of $(r\alpha)$
from section \ref{sec:Distribution} to develop estimates of the homogeneous
sum $S_{N}\phi\coloneqq S_{N}(\phi,0)=\sum_{r=1}^{N}\phi(r\alpha)$
when $\phi$ is a monotonic function on $(0,1)$. We will use another
variant of notation $S_{N}(x)\coloneqq S_{N}(\phi,x)=\sum_{r=1}^{N}\phi(x+r\alpha)$,
noting carefully that this is semantically distinct from the homonymous
notation $S_{N}(x)\coloneqq\sum_{r=1}^{N}x_{r}$ of section \ref{sec:Separation-of-Concerns}.

The restriction to monotone functions might at first seem restrictive.
Classical studies on sums of bounded functions tend to focus on functions
of bounded variation as the primary objects of interest. However since
any function of bounded variation can be represented as the sum of
two monotonic functions (one increasing, one decreasing), a focus
on monotone functions is not a limitation, but rather a recognition
of the fact that monotonic functions in fact lie slightly deeper. 

 We will first establish some duality results  which reduce the
number of cases we will need to consider. We will use the Source-Target
Magma theory developed in Section \ref{sec:Dualities-on-Operators},
where we put the source $S=(0,1)$ with involution $\sigma_{S}\coloneqq x\mapsto(1-x)$
and target $T=\mathbb{R}$ with involution $\tau_{S}\coloneqq x\mapsto-x$.

\subsection{\label{subsec:Use-of-Duality-1}Dual Relations for Birkhoff Functionals}

In this section we will explore the use of dualities to reduce the
number of cases we will need to consider when we look at bounds for
Birkhoff sums. We will exploit both the quite general Source and Target
dualities investigated in Section \ref{sec:Dualities-on-Operators},
and also a more specialised quasiperiodic duality which arises in
the context of Continued Fraction theory.

These dualities arise under the operation of various involutions,
and we will call fixed points of an involution \noun{self-conjugate}
under that involution.

We will also define a duality which combines the effects of the Source
and Quasiperiodic involutions. As this will be an important duality
for us, we will call it simply the \noun{Double Duality}, and its
fixed points are the \noun{double self-conjugate} points.

\label{subsec:SourceTargetDualities}Source-Target Dualities in Birkhoff
Sums on the circle
Recall that a Birkhoff sum $S_{N}\phi$ can be regarded as a functional
operator $\left(S_{N}\right)_{(\mathbb{TR})\mathbb{R}}$ acting on
the function $\phi_{\mathbb{TR}}$. Note that since the orbit $(r\alpha)$
avoids the origin, the sum $S_{N}\phi_{\mathbb{TR}}=\sum_{1}^{N}\phi_{\mathbb{TR}}(r\alpha)$
is identical in value with the sum $S_{N}\phi_{S\mathbb{R}}\left\{ r\alpha\right\} $
where $\phi_{\mathbb{TR}}$ has been replaced by its restriction to
the interval $S=(0,1)$, and the Birkhoff Sum functional $\left(S_{N}\right)_{(\mathbb{TR})\mathbb{R}}$
becomes $\left(S_{N}\right)_{(S\mathbb{R})\mathbb{R}}$. 

Using this observation, we can now apply the theory developed in Section
\ref{sec:Dualities-on-Operators}. In particular we will build on
the example application given in Subsection \ref{subsec:Example}.
We take as the real intervals $S\coloneqq(0,1)$ and $T\coloneqq\mathbb{R}$
(again equipped with the endorelations $R_{S}=R_{T}=\le_{\mathbb{R}}$).
We now set the source involution $\sigma_{S}$ to be the circle involution
$\overline{.}:x\mapsto\{1-x\}$ but restricted to $(0,1)$ (where
we note it is still an involution and can be written $x\mapsto1-x$),
and we set the target involution to be $\tau_{T}\coloneqq x_{\mathbb{R}}\mapsto-x$. 

Note that now the source pull up of $\sigma_{S}$ is $\sigma_{ST}=(\overline{.})_{ST}\coloneqq\phi_{ST}\mapsto\phi\circ\text{\ensuremath{\sigma_{S}}}$,
ie $\overline{\phi}\alpha\coloneqq\phi\overline{\alpha}$. Similarly
the target pull up of $\tau_{T}$ is $\tau_{ST}=-_{ST}\coloneqq\phi\mapsto-_{T}\circ\phi$,
ie $\left(-_{ST}\phi\right)\alpha\coloneqq-_{\mathbb{R}}\left(\phi\alpha\right)$.
Again we can pull up to Level 2, to get $\sigma_{(ST)T}\coloneqq A_{(ST)T}\mapsto A\circ\ensuremath{\sigma_{ST}}$,
ie $\overline{A}\phi\coloneqq A\overline{\phi}$ and similarly $\left(-_{(ST)T}A\right)\phi\coloneqq-_{\mathbb{R}}\left(A\phi\right)$. 

If we further assume that $A$ is a $\tau-$morphism this means $A(-_{ST}\phi)=-_{\mathbb{R}}\left(A\phi\right)$. 

In this source-target environment we also have further structure than
in $\ref{subsec:Example}$ since $T=\left(\mathbb{R},+_{\mathbb{R}},\times_{\mathbb{R}}\right)$
is a field. The pull up the binary operations of $T$ to $ST$ induces
an algebra of functions over $T$ (where $\lambda_{T}\phi_{ST}\,+_{ST}\,\mu_{T}\psi_{ST}$
and $\phi\,\times_{ST}\,\psi$ have their natural meanings). It is
easily verified that the involutions $\sigma_{ST},\tau_{ST}$ are
are then linear in both $+_{ST},\times_{ST}$. The same argument applies
analogously to the algebra $(ST)T$ of functionals over $T$. 

If $\Psi\subset ST$ is self-conjugate under $\sigma\tau$, the Example
application of Subsection \ref{subsec:Example} now applies, giving
us:

\begin{align}
A_{1}\le_{\Psi}A_{2} & \Leftrightarrow\sigma A_{1}\le_{\overline{\Psi}}\sigma A_{2}\Leftrightarrow\sigma^{\slash}A_{1}\le_{\overline{\Psi}}\sigma^{\slash}A_{2}\Leftrightarrow\sigma\sigma^{\slash}A_{1}\le_{\Psi}\sigma\sigma^{\slash}A_{2}\label{eq:DualPsi-1}
\end{align}
\[
\Leftrightarrow\tau A_{1}\le_{\Psi}^{Op}\tau A_{2}\Leftrightarrow\sigma\tau A_{1}\le_{\overline{\Psi}}^{Op}\sigma\tau A_{2}\Leftrightarrow\sigma^{\slash}\tau A_{1}\le_{\overline{\Psi}}^{Op}\sigma^{\slash}\tau A_{2}\Leftrightarrow\sigma\sigma^{\slash}\tau A_{1}\le_{\Psi}^{Op}\sigma\sigma^{\slash}\tau A_{2}
\]

We can go further in the case that $A_{1},A_{2}$ are $\tau-$morphisms.
We use the results of \eqref{eq:DualHom} to obtain from $A_{1}\le_{\Psi}A_{2}$
: 

\begin{gather}
A_{1}\le_{\Psi}A_{2}\Leftrightarrow\sigma A_{1}\le_{\overline{\Psi}}\sigma A_{2}\Leftrightarrow\tau A_{1}\le_{\overline{\Psi}}\tau A_{2}\Leftrightarrow\sigma\tau A_{1}\le_{\Psi}\sigma\tau A_{2}\label{eq:DualFinal-1}
\end{gather}
\[
\Leftrightarrow\tau A_{1}\le_{\Psi}^{Op}\tau A_{2}\Leftrightarrow\sigma\tau A_{1}\le_{\overline{\Psi}}^{Op}\sigma\tau A_{2}\Leftrightarrow A_{1}\le_{\overline{\Psi}}^{Op}A_{2}\Leftrightarrow\sigma A_{1}\le_{\Psi}^{Op}\sigma A_{2}
\]

We pick out the first and last dual relations as being of particular
use to us later in this Section: 
\begin{equation}
A_{1}\le_{\Psi}A_{2}\Leftrightarrow\sigma A_{2}\le_{\Psi}\sigma A_{1}\label{eq:SigmaDual}
\end{equation}

\begin{rem}
\label{rem:STDualsMonotonic}Let $\Theta\subset ST$ be a set of functions
which are descending (ascending) on $S$. Now $\sigma_{ST}\tau_{ST}\Theta=-\overline{\Theta}$
and this is also a set of descending (ascending) functions. Recall
 $\Psi(\Theta)\coloneqq\Theta\bigcup\sigma\tau\Theta=\Theta\bigcup-\overline{\Theta}$
is self-conjugate under $\sigma\tau$, as is also $\sigma\Psi=\overline{\Psi}=\overline{\Theta}\bigcup-\Theta$.
So in these cases of $\Phi$, $\Psi$ is a self-conjugate set of descending
(ascending) functions, and $\overline{\Psi}$ is a self-conjugate
set of ascending (descending) functions for which the dual relations
in \eqref{eq:DualPsi-1} apply, and further \eqref{eq:DualFinal-1}
apply when $A_{1},A_{2}$ are $\tau-$morphisms. There are important
two special cases: 
\end{rem}

\begin{enumerate}
\item When $\Theta=\Psi$ is the set of all monotonic descending (or monotonic
ascending, or simply all monotonic) functions 
\item When $\Psi=\Phi\bigcup\sigma\tau\Phi$ where $\Phi$ is the monoid
of positive descending primitives from ({*}{*}).
\end{enumerate}

\label{subsec:QPDuality}Quasi-Period Duality 

In this section we explore a specialised duality which arises in the
context of continued fraction developments.   We will use some of
the Source/Target Duality theory from the previous section, but now
in our specific environment of Birkhoff sums over rotations. 
Indexed Families
Recall that it is sometimes convenient to call a function $\iota_{IX}:I\rightarrow X$
an \noun{index function}. We then call $I$ the \noun{index set},
and $x_{i}\coloneqq\iota(i)$ an \noun{indexed element}. We call the
collection $(x_{i})_{i\in I}$ an \noun{indexed family} (indexed by
$I$). If $I=\{i\}_{i=1}^{n}$ we also write the indexed family as
$(x_{i})_{i=1}^{n}$.
\begin{rem}
\label{rem:IndexedFamilyProblem}It is particularly important to note
that $(x_{i})$ is not a set if $\iota$ is not injective (it is a
collection containing identical elements). This means we must be particularly
careful if we wish to define functions on the set $X=\{x_{i}\}$ making
use of the index $i$. For example, given an endomorphism $f_{II}$
of $I$ we can define the pull back function $g_{IX}=\iota\circ f$
which is a well-defined index function from $I$ to $X$. It is then
tempting to think that we can then also define a derived endomorphism
$h_{XX}:x_{i}\mapsto x_{f(i)}$ but this is not the case in general:
if $h$ exists and $\iota(i)=\iota(j)$ then $x_{f(i)}=h(x_{i})=h(x_{j})=x_{f(j)}$
and this is not the case for general $f$, even if $f$ is a bijection.
In general then, $h_{XX}$ is a relation (or multi-valued function).
\end{rem}

Quasi-Period Indexed Families
Recall \eqref{def:Quasiperiods} that we denote the sequence of quasiperiods
of $\alpha_{\mathbb{R}\backslash\mathbb{Q}}$ as $(q_{r}^{\alpha}){}_{r=0}^{\infty}$,
and that for $\alpha\in(0,\frac{1}{2})\backslash\mathbb{Q}$ we have
$q_{r}^{\alpha}=q_{\overline{r}}^{\overline{\alpha}}$ where $\overline{\alpha}=\{1-\alpha\}$
and $\overline{r}=r+1$. This result reveals a non-trivial duality
which we will develop in this section. 

Note first that we have here two maps with the homonymous notation
$'\overline{.}'$ . We can extend both to involutions as follows.
The first is defined on $\mathbb{I}=[0,1)$ by $\left(\overline{.}\right)_{\mathbb{I}}\coloneqq\alpha\mapsto\{1-\alpha\}$
(note that $\overline{0}=0$ is fixed point). The second needs slightly
more care in definition: recalling $\mathbb{N}=\mathbb{N}^{+}\bigcup\{0\}$
we take the disjoint union $\mathbb{N}_{2}=\mathbb{N}\bigsqcup\mathbb{N}^{+}$
and define the involution $\left(\overline{.}\right)_{\mathbb{N}_{2}}$
by: $\overline{r_{\mathbb{N}}}\coloneqq(r+1)_{\mathbb{N}^{+}}$and
$\overline{r_{\mathbb{N}^{+}}}\coloneqq(r-1)_{\mathbb{N}}$. We can
now use these two involutions to define a product involution $\left(\overline{.}\right)_{\mathbb{I}}\times\left(\overline{.}\right)_{\mathbb{N}_{2}}$
on a suitable set as follows:
\begin{defn}
The \noun{QuasiPeriod index set} $QPI=\left(QPI_{Set},\overline{.}\right)$
 has as its underlying set the strict subset $QPI_{Set}\subset\mathbb{I}\times\mathbb{N}$
defined by 
\[
QPI\coloneqq\left\{ (\alpha,r):\alpha\in\mathbb{I}\backslash\mathbb{Q}\:\&\:r\ge\left\llbracket \alpha\ge\frac{1}{2}\right\rrbracket \right\} 
\]
It is equipped with the Quasiperiod involution $Q_{QPI}$ defined
by $Q(\alpha,r)\coloneqq\overline{(\alpha,r)}=(\overline{\alpha},\overline{r})$
where $\overline{\alpha}\coloneqq\{1-\alpha\}$ and $\overline{r}\coloneqq r+(-1)^{\left\llbracket \alpha>\frac{1}{2}\right\rrbracket }$. 

A \noun{QP indexed family} (QPI family) $X=\left\{ (x_{r}^{\alpha}),\iota_{(QPI)X}\right\} $
is a family $X=(x_{r}^{\alpha})$ indexed by the QP index set $QPI$
using the index map $\iota_{(QPI)X}:(\alpha,r)_{QPI}\mapsto x_{r}^{\alpha}$).
There is usually no need to specify the index map, in which case we
simply write $X=(x_{r}^{\alpha})$. Points satisfying $x_{\overline{r}}^{\overline{\alpha}}=x_{r}^{\alpha}$
we call QP self-conjugate. . If the index map is independent of $\alpha$,
ie for each $r$, $\iota(\alpha,r)=x_{r}$ for every $\alpha$, we
will also allow ourselves to write the family as $(x_{r}^{\alpha})=(x_{r})$.
Similarly if $\iota$ is independent of $r$, we will write the family
as $(x^{\alpha})$, and as $(x)$ if $\iota$ is a constant map. 
\end{defn}

\begin{rem}
Note that $(\alpha,r)\in QPI$ means $r\ge0$ for $\alpha<\frac{1}{2}$
and $r\ge1$ for $\alpha>\frac{1}{2}$. In particular this means that
a QPI family $(x_{r}^{\alpha})$ contains no element $x_{0}^{\alpha}$
for $\alpha>\frac{1}{2}$.

Note that a QP indexed family may be further indexed by other index
sets, eg the family $(x_{rz}^{\alpha})$ is indexed by $QPI\times Z$,
but the QP involution only affects the QP index within the family,
ie $x_{rz}^{\alpha}=\iota\left((\alpha,r),z\right)$ and $\iota\left((\overline{\alpha},\overline{r}),z\right)=x_{\overline{r}z}^{\overline{\alpha}}$. 

In line with Remark \ref{rem:IndexedFamilyProblem} we cannot assume
that there is an induced qp involution $Q_{X}\coloneqq x_{r}^{\alpha}\mapsto x_{\overline{r}}^{\overline{\alpha}}$
on a QPI family $X=(x_{r}^{\alpha})$ if the index map is not injective.
However in many (but not all ) important cases, the function $x_{r}^{\alpha}\mapsto x_{\overline{r}}^{\overline{\alpha}}$
is well-defined (and then the induced function is also an involution
since $Q_{QPI}$ is itself an involution). When this is the case some
individual proofs below could be simplified. However our overall presentation
is simplified by having a single set of proofs which work whether
an induced involution is defined or not. Hence we will NOT generally
assume $Q\coloneqq x_{r}^{\alpha}\mapsto x_{\overline{r}}^{\overline{\alpha}}$
is well-defined in the sequel.
\end{rem}

We now define a duality which will prove important in later sections.
This is a hybrid using both the QP involution and the source involution
$\left(\overline{.}\right)_{\mathbb{II}}\coloneqq\alpha\mapsto{1-\alpha}$.
\begin{defn}[Double Dual]
\label{def:DoubleDual}Given a QPI family of functions $(\phi_{r}^{\alpha})$,
we call the derived family  $\left(\overline{\phi_{\overline{r}}^{\overline{\alpha}}}\right)$
the family of the \noun{double dual} functions of $(\phi_{r}^{\alpha})$.
 For convenience we will use the simplified notation $\overline{\phi}_{r}^{\alpha}\coloneqq\overline{\phi_{\overline{r}}^{\overline{\alpha}}}$.
When $\overline{\phi}_{r}^{\alpha}=\phi_{r}^{\alpha}$ we say $\phi_{r}^{\alpha}$
is double self-conjugate.
\end{defn}

\begin{rem}
It is important note that following Remark ({*}{*}) we cannot regard
the Double Dual as a function map because $\overline{\phi}_{r}^{\alpha}\mapsto\phi_{r}^{\alpha}$
may be multi-valued.
\end{rem}

\begin{example}
\label{exa:QPIFamilyExamples}Important examples of QPI families
\end{example}

\begin{enumerate}
\item The quasiperiod index set $QPI$ is itself a QP indexed family using
the index map $Id_{QPI}$. Since $Id$ is injective, in this case
the involution $Q:(\alpha,r)\mapsto\overline{(\alpha,r)}$ exists
trivially.
\item A scalar qp family $\left(\lambda_{r}^{\alpha}\right)$ with each
$\lambda_{r}^{\alpha}\in\mathbb{R}$. Important examples are $q_{r}^{\alpha},a_{r+1}^{\alpha},b_{r}^{\alpha}$
(where these symbols have their standard meanings). Note that these
particular examples are also qp self-conjugate (eg $q_{r}^{\alpha}=q_{\overline{r}}^{\overline{\alpha}}$)
with the single exception of $a_{r+1}^{\alpha}$ for $\alpha<\frac{1}{2},r=0$
when $a_{2}^{\overline{\alpha}}=a_{1}^{\alpha}-1$. Scalar families
will not generally have injective index maps ($Q_{_{\left(\lambda_{r}^{\alpha}\right)}}$is
not well defined), although the identity function serves as a partial
qp involution when restricted to the set of self-conjugate elements.
\item The particular scalar qp family $\left(E_{r}^{\alpha}\right)$ where
$E_{r}^{\alpha}\coloneqq E_{r}=\left\llbracket r\,\mathrm{even}\right\rrbracket $
(so the family $\left(E_{r}^{\alpha}\right)$ is independent of $\alpha$).
Note $E_{\overline{r}}^{\overline{\alpha}}=O_{r}^{\alpha}$ and $O_{\overline{r}}^{\overline{\alpha}}=E_{r}^{\alpha}$
and so $\left(E_{r}^{\alpha}\right)=\left(O_{r}^{\alpha}\right)$.The
index map $\iota:(\alpha,r)\mapsto E_{r}^{\alpha}=E_{r}$ is clearly
not injective, but in fact the induced involution $Q_{\left(E_{r}^{\alpha}\right)}$
does exist.
\item Given a scalar qp family $\{\lambda_{r}^{\alpha}\}$ we can form new
scalar families such as $\left\{ f_{\mathbb{\mathbb{R}R}}(\lambda_{r}^{\alpha})\right\} $
or $\left\{ \phi_{\mathbb{IR}}(\lambda_{r}^{\alpha})\right\} $ (when
$\lambda_{r}^{\alpha}\in\mathbb{I}$), and $\left\llbracket P(\lambda_{r}^{\alpha})\right\rrbracket $
(where $P$ is a proposition). If $\lambda_{r}^{\alpha}$ is qp self
conjugate, so are $f(\lambda_{r}^{\alpha})$, $\phi(\lambda_{r}^{\alpha})$
and $\left\llbracket P(\lambda_{r}^{\alpha})\right\rrbracket $. In
particular $\phi\left(\frac{1}{q_{r}^{\alpha}}\right)$ is qp self
conjugate for $r>0$. 
\item The qp family of qp index functions $n^{\alpha}:\mathbb{N}^{+}\rightarrow\mathbb{N}$
defined by $n^{\alpha}(N)\coloneqq\max\{t:q_{t}^{\alpha}\le N\}$.
Note we have dropped $r$ from the notation to indicate that these
functions are independent of $r$, ie for any $r$, $n_{r}^{\alpha}=n^{\alpha}$.
Note also that we can now write the Ostrowski representation of $N$
as $N=\sum_{r=\left\llbracket \alpha>\frac{1}{2}\right\rrbracket }^{n^{\alpha}(N)}b_{r}q_{r}$,
and also that we can now write the relationship between the indexes
$n^{\alpha},n^{\overline{\alpha}}$ as $n^{\overline{\alpha}}(N)=n^{\alpha}(N)+(-1)^{\left\llbracket \alpha>\frac{1}{2}\right\rrbracket }$.
In particular this gives $q_{n^{\alpha}}^{\alpha}=q_{n^{\overline{\alpha}}}^{\overline{\alpha}}.$
\item The qp scalar family $c_{rst}^{\alpha}$ for some symbol $c$: 
\begin{enumerate}
\item When $c$ is the null symbol we will simply write $\alpha_{rst}$.
The QP conjugate is $\overline{\alpha}_{\overline{r}st}$, and since
$\overline{\alpha}_{\overline{r}st}=\{1-\alpha_{rst}\}=\overline{\alpha_{rst}}$
qp involution and source involution coincide on $\mathbb{I}\backslash\mathbb{Q}$
\item When $c$ is an observable symbol, eg $\phi$, we write $\phi_{rst}^{\alpha}\coloneqq\phi(\alpha_{rst})$
with QP conjugate $\phi_{\overline{r}st}^{\overline{\alpha}}=\phi(\overline{\alpha}_{\overline{r}st})=\phi(\overline{\alpha_{rst}})=\overline{\phi}_{rst}^{\alpha}$,
and again qp involution and source involution coincide
\end{enumerate}
\end{enumerate}
 . 
QPI Families of Operators

Let $\mathscr{X}$ be the set of QPI families of functionals (families
whose elements lie in $(ST)T$), and let $\left(X_{r}^{\alpha}\right),\left(Y_{r}^{\alpha}\right)$
2 members of $\mathscr{X}$. Define $Z_{r}^{\alpha}\coloneqq\lambda X_{r}^{\alpha}+\mu Y_{r}^{\alpha}$.
By this definition $Z_{\overline{r}}^{\overline{\alpha}}\coloneqq\lambda X_{\overline{r}}^{\overline{\alpha}}+\mu Y_{\overline{r}}^{\overline{\alpha}}$
and hence $\left(Z_{r}^{\alpha}\right)$ is also a QPI family, and
we can write $\left(Z_{r}^{\alpha}\right)=\lambda_{T}\left(X_{r}^{\alpha}\right)\,+_{\mathscr{X}}\,\mu_{T}\left(Y_{r}^{\alpha}\right)$
to mean for each $(\alpha,r)$ that $Z_{r}^{\alpha}=\lambda X_{r}^{\alpha}+\mu Y_{r}^{\alpha}$.
Note then $\mathscr{X}$ is a vector space over $T$ , and $+_{\mathscr{X}}$
is the pull up of $+_{T}$ to $\mathscr{X}$. Similarly we can also
pull up $\times_{T}$ to $\mathscr{X}$ so that $\left(X_{r}^{\alpha}\right)\times\left(Y_{r}^{\alpha}\right)$
is again a QPI family, and so $\mathscr{X}$ is also an algebra over
$T$. 

Recall also that that when $ST$ is equipped with an involution $\upsilon{}_{ST}$
then this induces a pullup involution $\upsilon_{(ST)T}$ defined
by $\upsilon_{(ST)T}(X)\coloneqq X\circ\upsilon_{ST}$ which is also
linear in $+_{ST},\times_{ST}$. Again given a QPI family $\left(X_{r}^{\alpha}\right)$
we can define the pullup involution $\upsilon_{\mathscr{X}}$ by $\upsilon_{\mathscr{X}}\left(\,\left(X_{r}^{\alpha}\right)\,\right)=\left(\upsilon_{(ST)T}X_{r}^{\alpha}\right)$
and again this is linear in $+_{\mathscr{X}},\times_{\mathscr{X}}$,
so that $\upsilon_{\mathscr{X}}$ is an algebra automorphism. 

In particular we have proved:
\begin{prop}[Linearity of Double Dual]
\label{prop:DDLinearity}If $Z_{r}^{\alpha}=\sum_{i=1}^{k}X_{ir}^{\alpha}\times Y_{ir}^{\alpha}$
then $\left(Z_{r}^{\alpha}\right)$ is a QPI family and  $\overline{Z_{\overline{r}}^{\overline{\alpha}}}=\sum_{i=1}^{k}\overline{X_{\overline{r}}^{\overline{\alpha}}}\times\overline{Y_{\overline{r}}^{\overline{\alpha}}}$,
ie the double dual of an algebraic combination of $QPI$ families
is the same algebraic combination of the double duals
\end{prop}

\begin{proof}
We showed above that $\left(Z_{r}^{\alpha}\right)$ is a QPI family.
The double dual of $Z_{r}^{\alpha}$ is $\sigma_{(ST)T}Z_{\overline{r}}^{\overline{\alpha}}$
which by definition of $Z_{r}^{\alpha}$ is $\sigma_{(ST)T}\left(\sum_{i=1}^{k}X_{\overline{r}}^{\overline{\alpha}}\times Y_{\overline{r}}^{\overline{\alpha}}\right)$
and the result follows since $\sigma$ is an algebra homomorphism.
\end{proof}

\begin{example}
\label{exa:OperatorQPDualities}Important examples of qp operator
families
\end{example}

\begin{enumerate}
\item Given an operator $X$ we can define the constant qp family $\mathscr{X}_{X}$
in which $X_{r}^{\alpha}\phi=X\phi$ for all $(\alpha,r)\in QPI$.
This family is automatically qp self conjugate ($X_{\overline{r}}^{\overline{\alpha}}\phi=X\phi=X_{r}^{\alpha}\phi$).
\item The partition sum operator family $P_{q_{r}^{\alpha}}$ (see Subsection
\ref{subsec:Partition-Sums} for details) is a linear operator family
which is Source self conjugate ($P_{q_{r}^{\alpha}}=\overline{P_{q_{r}^{\alpha}}}$),
QP self conjugate $(P_{q_{r}^{\alpha}}=P_{q_{\overline{r}}^{\overline{\alpha}}})$,
and hence also double self conjugate ($P_{q_{r}^{\alpha}}=\overline{P}_{q_{r}^{\alpha}}\coloneqq\overline{P_{q_{\overline{r}}^{\overline{\alpha}}}}$)
\item $\phi\mapsto\phi_{rst}^{\alpha}\coloneqq\phi\left(\alpha_{rst}\right)$
is an (anonymous) linear operator family which is neither source
nor qp self conjugate, but \emph{is} double self conjugate $($$\phi_{rst}^{\alpha}=\overline{\phi}_{rst}^{\alpha}\coloneqq\overline{\phi_{\overline{r}st}^{\overline{\alpha}}}$).
\item The family $S_{rs}^{\alpha}\phi=S_{q_{r}}^{\alpha}(\phi,\alpha_{rs})\coloneqq\sum_{t=1}^{q_{r}}\phi_{rst}^{\alpha}$
is double self conjugate since $\phi\mapsto\phi_{rst}^{\alpha}$ is
double self conjugate, and hence also the sum $\sum_{s=0}^{b_{r}-1}S_{rs}^{\alpha}$
is double self-conjugate.
\item $\phi\mapsto\phi(f(q_{r}^{\alpha}))$ is an (anonymous) linear operator
family which is not source self conjugate, but which is qp self conjugate
since the qp family $\left(q_{r}^{\alpha}\right)$ is qp self conjugate
($q_{r}^{\alpha}=q_{\overline{r}}^{\overline{\alpha}}$). In particular
$\phi\mapsto\phi(\frac{1}{q_{r}^{\alpha}})$ is qp self conjugate. 
\end{enumerate}

Recall that in this section we are considering the Source-Target Magma
where $S=(0,1),T=\mathbb{R}$. 
\begin{defn}
Given two QPI operator families $\left(X_{r}^{\alpha}\right),\left(Y_{r}^{\alpha}\right)$
we will say $\left(X_{r}^{\alpha}\right)$ is \noun{dominated} by
$\left(Y_{r}^{\alpha}\right)$ on $\Theta\subseteq ST$ if for every
$\text{(\ensuremath{\alpha,r)\in QPI}}$ we have $X_{r}^{\alpha}\le_{\Theta}Y_{r}^{\alpha}$.
We write this as $\left(X_{r}^{\alpha}\right)\le_{\Theta}\left(Y_{r}^{\alpha}\right)$

\end{defn}

\begin{lem}[Domination Duals]
\label{lem:DominationDuals}Let $\Psi\subset ST$ be self-conjugate
under the involution $\sigma_{ST}\tau_{ST}$. Given two \emph{linear}
QPI families $\left(X_{r}^{\alpha}\right),\left(Y_{r}^{\alpha}\right)$
of functionals from $(ST)T$, then we have the following dual results:
\[
\left(X_{r}^{\alpha}\right)\le_{\Psi}\left(Y_{r}^{\alpha}\right)\Leftrightarrow\left(X_{\overline{r}}^{\overline{\alpha}}\right)\le_{\Psi}\left(Y_{\overline{r}}^{\overline{\alpha}}\right)\Leftrightarrow\left(\overline{Y_{r}^{\alpha}}\right)\le_{\Psi}\left(\overline{X_{r}^{\alpha}}\right)\Leftrightarrow\left(\overline{Y_{\overline{r}}^{\overline{\alpha}}}\right)\le_{\Psi}\left(\overline{X_{\overline{r}}^{\overline{\alpha}}}\right)
\]
\end{lem}

\begin{proof}
($1\Leftrightarrow2$): By definition $\left(X_{r}^{\alpha}\right)=\left(X_{\overline{r}}^{\overline{\alpha}}\right)$,
and so trivially $\left(X_{r}^{\alpha}\right)\le_{\Psi}\left(Y_{r}^{\alpha}\right)\Leftrightarrow\left(X_{\overline{r}}^{\overline{\alpha}}\right)\le_{\Theta}\left(Y_{\overline{r}}^{\overline{\alpha}}\right)$.

($1\Leftrightarrow3$): Since $\Psi$ is self-conjugate, and $X_{r}^{\alpha},Y_{r}^{\alpha}$
are linear (and hence $\tau-$morphisms), then by $\ref{eq:SigmaDual}$
we have $X_{r}^{\alpha}\le_{\Psi}Y_{r}^{\alpha}\Leftrightarrow\overline{Y_{r}^{\alpha}}\le_{\Psi}\overline{X_{r}^{\alpha}}$,
and the result follows.

($2\Leftrightarrow4$): This follows by taking the duality ($1\Leftrightarrow3$)
and applying the duality ($1\Leftrightarrow2$).
\end{proof}
Analogous results follow for $\overline{\Psi}$.

Recall that $S_{rs}^{\alpha}\phi\coloneqq\sum_{t=1}^{q_{r}}\phi_{rst}^{\alpha}$

\begin{cor}
\label{cor:Sr dual inequality} If $\Psi$ is self-conjugate and $\left(X_{r}^{\alpha}\right)$
is linear, $\left(\sum_{s=0}^{b_{r}^{\alpha}-1}S_{rs}^{\alpha}\right)\le_{\Psi}\left(X_{r}^{\alpha}\right)$
on $\Psi$ if and only if $\left(\overline{X_{\overline{r}}^{\overline{\alpha}}}\right)\le_{\Psi}\left(\sum_{s=0}^{b_{r}^{\alpha}-1}S_{rs}^{\alpha}\right)$ 
\end{cor}

\begin{proof}
This follows directly from the lemma putting $Y_{r}^{\alpha}=\sum_{s=0}^{b_{r}-1}S_{rs}^{\alpha}$
and using the fact that $b_{r}^{\alpha}$and $S_{rs}^{\alpha}$ are
double self conjugate, and hence by linearity of the double dual (Proposition
\ref{prop:DDLinearity}) so is the sum $Y_{r}^{\alpha}$ .
\end{proof}

\subsection{\label{subsec:Partition-Sums}Partition Sums}

We now introduce a set of linear functionals on observables which
will form an important part of our Birkhoff Sum estimates.

\begin{defn}
\textbf{Partition Sums}. Consider the partition of $[0,1)$ into
$k\ge1$  equal sub-intervals starting from $0$. We call the $k-1$
points with coordinates $\left\{ \frac{t}{k}\right\} _{t=1}^{k-1}$
(ie not including $0$) the interior points, and define the \noun{partition
sum} of $\phi$ on the partition to be the sum of values of $\phi$
over the interior points, namely $P_{k}(\phi)=\sum_{t=1}^{k-1}\phi(\frac{t}{k})$.
\end{defn}

\begin{rem}
Note that $P_{k}$ is a linear functional of $\phi$, and that $P_{1}(\phi)$
is an empty sum taking the value $0$, and that $P_{2}(\phi)=\phi(\frac{1}{2})$.
\end{rem}

Recall from the previous section that we are using the notation $\left(\overline{.}\right)_{\mathbb{T}^{\circ}}$
for the Source involution of the circle $x\mapsto\overline{x}=\{1-x\}$.
This has the pull ups $\left(\overline{.}\right)_{\left(\mathbb{TR}\right)^{\circ}}:\phi\mapsto\overline{\phi}$,
$\left(\overline{.}\right)_{\left((\mathbb{TR})T\right)^{\circ}}:A\mapsto\overline{A}$
defined by $\overline{\phi}x=\phi\overline{x},\overline{A}\phi=A\overline{\phi}$
\begin{prop}
$P_{k}$ is self-conjugate under the Source involution $\left(\overline{\text{.}}\right)_{\mathbb{T}}$,
and if $\phi_{\mathbb{TR}}$ is anti-symmetric then $P_{k}(\phi)=0$
\end{prop}

\begin{proof}
By definition, putting $u=k-t$:
\[
\overline{P_{k}}(\phi)=P_{k}(\overline{\phi})=\sum_{t=1}^{k-1}\phi(1-\frac{t}{k})=\sum_{u=1}^{k-1}\phi(\frac{u}{k})=P_{k}(\phi)
\]
Now if $\phi$ is anti-symmetric on the circle we have $\phi=-\overline{\phi}$.
Since $P_{k}$ is linear $P_{k}\left(\phi\right)=P_{k}\left(-\overline{\phi}\right)=-P_{k}\left(\overline{\phi}\right)=-P_{k}\left(\phi\right)$,
hence $P_{k}(\phi)=0$.
\end{proof}
We will usually be concerned with partition sums $P_{q_{r}^{\alpha}}$
where $q_{r}^{\alpha}$ is a quasiperiod of $\alpha_{\mathbb{T}}$. 
\begin{prop}
$P_{q_{r}^{\alpha}}$ is self-conjugate under the quasiperiod involution
$(\alpha,r)\mapsto(\overline{\alpha},\overline{r})$
\end{prop}

\begin{proof}
Recall that $q_{r}^{\alpha}$ is itself self-conjugate, ie $q_{r}^{\alpha}=q_{\overline{r}}^{\overline{\alpha}}$,
and the result follows trivially. 
\end{proof}

\subsection{Birkhoff sums of monotonic functions on $(0,1)$}

Recall $S_{N}^{\alpha}\phi=\sum_{r=0}^{n^{\alpha}}\sum_{s=0}^{b_{r}-1}S_{rs}^{\alpha}\phi$
where $S_{rs}^{\alpha}\phi=S_{q_{r}^{\alpha}}(\phi,\alpha_{rs})=\sum_{t=1}^{q_{r}^{\alpha}}\phi(\alpha_{rs}+t\alpha)=\sum_{t=1}^{q_{r}^{\alpha}}\phi_{rst}^{\alpha}$.
We will fix $\alpha$ for the moment and drop it from the annotation. 

When $q_{r}=1$ we may have $q_{1}=q_{0}=1$. However in this case
we also have $\alpha>\frac{1}{2}$. Also from (\ref{lem:ORresults})$b_{0}=0$
so that the $S_{q_{0}}$ term is not of interest and we can write
$S_{N}^{\alpha}\phi=\sum_{r=\left\llbracket \alpha>\frac{1}{2}\right\rrbracket }^{n}\sum_{s=0}^{b_{r}-1}S_{rs}^{\alpha}\phi$

Now for $q_{r}=1$, $t\in(1..q_{r})$ gives $t\in(1)$ and so $S_{q_{r}}(\alpha_{rs})=\sum_{t=1}^{q_{r}}\phi_{rst}=\phi_{rsq_{r}}$
and similarly for $q_{r}=2$ we have $S_{q_{r}}(\alpha_{rs})=\sum_{t=1}^{q_{r}}\phi_{rst}=\phi_{rsq_{r}}+\phi_{rsq_{r-1}}$
(since $q_{r-1}=1$). Combining these results we have for $q_{r}\le2$
(and where $b_{r}=0$ gives the empty sum $0$):
\begin{equation}
\sum_{s=0}^{b_{r}-1}S_{rs}^{\alpha}\phi=\sum_{s=0}^{b_{r}-1}\left(\phi_{rsq_{r}^{\alpha}}^{\alpha}+\left\llbracket q_{r}^{\alpha}>1\right\rrbracket \phi_{rsq_{r-1}^{\alpha}}^{\alpha}\right)\label{eq:qr<3}
\end{equation}

\begin{defn}[Bounds Operators]
Given $\alpha,N$, the Bounds Functionals  $\left(B_{r}^{\alpha}\right)$
form a QP family of operators in $\left(\mathbb{TR}\right)\mathbb{R}$
defined on observables $\phi$ by: 
\begin{equation}
B_{r}^{\alpha}\phi=\left\llbracket b_{r}^{\alpha}>0\right\rrbracket \left(\left\llbracket q_{r}^{\alpha}>1\right\rrbracket b_{r}^{\alpha}\left(P_{q_{r}^{\alpha}}(\phi)-\overline{\phi}(\frac{1}{q_{r}^{\alpha}})\right)+\left\llbracket 2<q_{r}^{\alpha}<q_{n^{\alpha}}^{\alpha}\right\rrbracket E_{r}^{\alpha}\left(\phi_{r0,q_{r}-q_{r-1}}^{\alpha}-\overline{\phi}\left(\frac{2}{q_{r}^{\alpha}}\right)\right)+\sum_{s=0}^{b_{r}^{\alpha}-1}\left(\phi_{rsq_{r}}^{\alpha}+\left\llbracket q_{r}^{\alpha}>1\right\rrbracket \phi_{rsq_{r-1}}^{\alpha}\right)\right)\label{eq:BoundsDef}
\end{equation}
\end{defn}

\begin{rem}
Note this reduces to \eqref{eq:qr<3} for $q_{r}^{\alpha}\le2$.

Further note that each Bounds Functional is an algebraic combinations
of other linear QPI functionals and hence linear . 

\end{rem}

\begin{lem}
The Bounds Functionals $\left(B_{r}^{\alpha}\right)$ have QP duals
$\left(B_{\overline{r}}^{\overline{\alpha}}\right)$ and Double Duals
$\left(\overline{B_{\overline{r}}^{\overline{\alpha}}}\right)$ which
can be written (when $\alpha$ is fixed): 
\begin{equation}
B_{\overline{r}}^{\overline{\alpha}}\phi=\left\llbracket b_{r}>0\right\rrbracket \left(\left\llbracket q_{r}>1\right\rrbracket b_{r}\left(P_{q_{r}}(\phi)-\overline{\phi}(\frac{1}{q_{r}})\right)+\left\llbracket 2<q_{r}<q_{n}\right\rrbracket O_{r}\left(\overline{\phi}_{r0,q_{r}-q_{r-1}}-\overline{\phi}\left(\frac{2}{q_{r}}\right)\right)+\sum_{s=0}^{b_{r}-1}\left(\overline{\phi}_{rsq_{r}}+\left\llbracket q_{r}>1\right\rrbracket \overline{\phi}_{rsq_{r-1}}\right)\right)\label{eq:BoundsDefQP}
\end{equation}
\begin{equation}
\overline{B_{\overline{r}}^{\overline{\alpha}}}\phi=\left\llbracket b_{r}>0\right\rrbracket \left(\left\llbracket q_{r}>1\right\rrbracket b_{r}\left(P_{q_{r}}(\phi)-\phi(\frac{1}{q_{r}})\right)+\left\llbracket 2<q_{r}<q_{n}\right\rrbracket O_{r}\left(\phi_{r0,q_{r}-q_{r-1}}-\phi\left(\frac{2}{q_{r}}\right)\right)+\sum_{s=0}^{b_{r}-1}\left(\phi_{rsq_{r}}+\left\llbracket q_{r}>1\right\rrbracket \phi_{rsq_{r-1}}\right)\right)\label{eq:BoundsDefDual}
\end{equation}
\end{lem}

\begin{proof}
Applying the QP involution to the definition \eqref{eq:BoundsDef}
gives us the QP dual:

\[
B_{\overline{r}}^{\overline{\alpha}}\phi=\left\llbracket b_{\overline{r}}^{\overline{\alpha}}>0\right\rrbracket \left(\left\llbracket q_{\overline{r}}^{\overline{\alpha}}>1\right\rrbracket b_{\overline{r}}^{\overline{\alpha}}\left(P_{q_{\overline{r}}^{\overline{\alpha}}}(\phi)-\overline{\phi}(\frac{1}{q_{\overline{r}}^{\overline{\alpha}}})\right)+\left\llbracket 2<q_{\overline{r}}^{\overline{\alpha}}<q_{n^{\overline{\alpha}}}^{\overline{\alpha}}\right\rrbracket E_{\overline{r}}^{\overline{\alpha}}\left(\phi_{\overline{r}0,q_{\overline{r}}^{\overline{\alpha}}-q_{\overline{r}-1}^{\overline{\alpha}}}^{\overline{\alpha}}-\overline{\phi}\left(\frac{2}{q_{\overline{r}}^{\overline{\alpha}}}\right)\right)+\sum_{s=0}^{b_{\overline{r}}^{\overline{\alpha}}-1}\left(\phi_{\overline{r}sq_{\overline{r}}^{\overline{\alpha}}}^{\overline{\alpha}}+\left\llbracket q_{\overline{r}}^{\overline{\alpha}}>1\right\rrbracket \phi_{\overline{r}sq_{\overline{r}-1}^{\overline{\alpha}}}^{\overline{\alpha}}\right)\right)
\]

But now we have by ({*}{*}) that $b_{r}^{\alpha},q_{r}^{\alpha},P_{q_{r}^{\alpha}}$
and $q_{n^{\alpha}}^{\alpha}$  are all self-conjugate under the
QP involution, so that for example $b_{r}^{\alpha}=b_{\overline{r}}^{\overline{\alpha}}$
and we can and will denote either of them by $b_{r}$ when $\alpha$
is fixed. Further $E_{\overline{r}}^{\overline{\alpha}}=O_{r}^{\alpha}=O_{r}$
and $\phi_{\overline{r}st}^{\overline{\alpha}}=\overline{\phi}_{rst}^{\alpha}$
and the QP dual result \eqref{eq:BoundsDefQP} follows. For the double
dual we have $\overline{B_{\overline{r}}^{\overline{\alpha}}}\phi\coloneqq B_{\overline{r}}^{\overline{\alpha}}\overline{\phi}$.
From \eqref{eq:BoundsDefQP} we then obtain: 

\[
B_{\overline{r}}^{\overline{\alpha}}\overline{\phi}=\left\llbracket b_{r}>0\right\rrbracket \left(\left\llbracket q_{r}>1\right\rrbracket b_{r}\left(P_{q_{r}}(\overline{\phi})-\phi(\frac{1}{q_{r}})\right)+\left\llbracket 2<q_{r}<q_{n}\right\rrbracket O_{r}\left(\phi_{r0,q_{r}-q_{r-1}}-\phi\left(\frac{2}{q_{r}}\right)\right)+\sum_{s=0}^{b_{r}-1}\left(\phi_{rsq_{r}}+\left\llbracket q_{r}>1\right\rrbracket \phi_{rsq_{r-1}}\right)\right)
\]

Now note that $P_{q_{r}}$ is also self-conjugate under $\phi\mapsto\overline{\phi}$
and \eqref{eq:BoundsDefDual} follows.
\end{proof}

\begin{thm}[Bounds for monotonic functions]
\label{lem:BaseInequalities}Given $\alpha,N,r$ and the associated
Bounds Functionals $B_{r}^{\alpha},\overline{B_{\overline{r}}^{\overline{\alpha}}}$,
let $\phi$ be a monotonic decreasing observable on $(0,1)$. Then
we have: 
\begin{align}
\overline{B_{\overline{r}}^{\overline{\alpha}}}\phi & \le\sum_{s=0}^{b_{r}-1}S_{rs}^{\alpha}\phi\le B_{r}^{\alpha}\phi\label{eq:BaseInequalities}
\end{align}

Further, these are equalities for constant $\phi$ (and also if $b_{r}=0$
or $q_{r}\le2$).
\end{thm}

\begin{proof}
Given the right hand inequality $\sum_{s=0}^{b_{r}-1}S_{rs}^{\alpha}\phi\le B_{r}^{\alpha}\phi$
we can deduce the left hand inequality via the corollary to the domination
lemma, Corollary \ref{cor:Sr dual inequality}. It remains to prove
the right hand. We first deal with the equality claims and $q_{r}\le2$.

If $b_{r}=0$ then we have the equality $0=\overline{B_{\overline{r}}^{\overline{\alpha}}}\phi=B_{r}^{\alpha}\phi=\sum_{s=0}^{b_{r}-1}S_{q_{r}}(\alpha_{rs})$.
For $q_{r}\le2$, we have $P_{2}=\phi(\frac{1}{2})=\overline{\phi}(\frac{1}{2})$
and $P_{1}=0$, and so for $q_{r}\le2$ both the first and mid terms
of $\overline{B_{\overline{r}}^{\overline{\alpha}}},B_{r}^{\alpha}$
vanish and we are left with (\ref{eq:qr<3}) which is an equality.
If $q_{r}>2$ and $\phi(x)=c$ then the mid terms vanish, the first
term is $b_{r}\left((q_{r}-1)c-c\right)$ and the final term becomes
$b_{r}(c+c)$ so that $A_{r}\phi=B_{r}\phi=b_{r}q_{r}c$ which is
also $\sum_{s=0}^{b_{r}-1}S_{q_{r}}(\alpha_{rs})$. We now deal with
the general case for $q_{r}>2$.

We now use the results of Section \ref{sec:Distribution} to understand
where $\alpha_{rst}$ lies in the regular $q_{r}$ partition of the
circle.

Case: $s>0$ or $s=0\&r=n$. 

In this case for each $t$, the co-siting condition holds, ie $\alpha_{rst}$
lies in $I(t\alpha)$ so that $\phi_{rst}$ lies between the values
of $\phi$ at the interval endpoints. Summing over increasing values
of $\alpha_{rst}$ (ie in the direction $0$ to $1$) gives us:

For $r$ even we have $S_{q_{r}}(\alpha_{rs})=\sum_{t=1}^{q_{r}}\phi(\alpha_{rst})\ge\phi_{rsq_{r}}+\left(P_{q_{r}}-\phi(\frac{1}{q_{r}})\right)+\phi_{rsq_{r-1}}$
and $S_{q_{r}}(\alpha_{rs})\le\phi_{rsq_{r}}+\left(P_{q_{r}}-\phi(1-\frac{1}{q_{r}})\right)+\phi_{rsq_{r-1}}$.

Similarly for $r$ odd we have $\phi_{rsq_{r-1}}+\left(P_{q_{r}}-\phi(1-\frac{1}{q_{r}})\right)+\phi_{rsq_{r}}\ge S_{q_{r}}(\alpha_{rs})\ge\phi_{rsq_{r-1}}+\left(P_{q_{r}}-\phi(\frac{1}{q_{r}})\right)+\phi_{rsq_{r}}$
which is structurally identical to the $r$ even case, even though
$\alpha_{rsq_{r}},\alpha_{rsq_{r-1}}$ are switched in their positions
around $0$.

Case: $s=0\,\&\,r<n,q_{r}>1$. 

In this case for each $t\ne q_{r}$, $\alpha_{rst}$ may be shifted
by one partition interval from $I(t\alpha)$. In particular, $\alpha_{rs(q_{r}-q_{r-1})}$
may be shifted into the same partition interval as $\alpha_{rsq_{r}}$.

For $r$ even, the possible shift is backward (towards 0), so that
the lower bound is not affected, but the upper bound increases to
$S_{q_{r}}(\alpha_{rs})\le\left(P_{q_{r}}-\phi(1-\frac{1}{q_{r}})\right)+\phi_{r0q_{r}}+\phi_{r0,q_{r}-q_{r-1}}+\left(\phi_{rsq_{r-1}}-\phi(1-\frac{2}{q_{r}})\right)=\left(P_{q_{r}}+\phi_{rsq_{r}}+\phi_{rsq_{r-1}}-\phi(1-\frac{1}{q_{r}})\right)+\left(\phi_{r0,q_{r}-q_{r-1}}-\phi(1-\frac{2}{q_{r}})\right)$. 

For $r$ odd, the possible shift is forward (towards 1), so that the
upper bound is not affected, but the lower bound decreases to $S_{q_{r}}(\alpha_{rs})\ge\left(P_{q_{r}}-\phi\left(\frac{1}{q_{r}}\right)\right)+\left(\phi_{rsq_{r-1}}-\phi(\frac{2}{q_{r}})\right)+\phi_{rsq_{r}}+\phi_{r0,q_{r}-q_{r-1}}=\left(\phi_{rsq_{r}}+P_{q_{r}}-\phi(\frac{1}{q_{r}})+\phi_{rsq_{r-1}}\right)+\left(\phi_{r0,q_{r}-q_{r-1}}-\phi(\frac{2}{q_{r}})\right)$. 

The results follow on reorganising the terms.

\end{proof}

\begin{cor}
$\sum_{r=0}^{n}\overline{B_{\overline{r}}^{\overline{\alpha}}}\phi\le S_{N}\phi\le\sum_{r=0}^{n}B_{r}^{\alpha}\phi$
\end{cor}

\begin{proof}
This follows from $S_{N}\phi=\sum_{r=0}^{n}\sum_{s=0}^{b_{r}-1}S_{rs}^{\alpha}\phi$
\end{proof}

\begin{cor}[Duality results]
\label{cor:BSDualities}Let $\Psi\subset ST$ be a $\sigma\tau$
self-conjugate set of decreasing functions on $(0,1)$. Then we have:

\begin{alignat*}{3}
\overline{B_{\overline{r}}^{\overline{\alpha}}} & \:\le_{\Psi} & \sum_{s=0}^{b_{r}-1}S_{rs}^{\alpha} & \:\le_{\Psi}\; & B_{r}^{\alpha}\\
\sigma\overline{B_{\overline{r}}^{\overline{\alpha}}} & \:\le_{\overline{\Psi}} & \sigma\sum_{s=0}^{b_{r}-1}S_{rs}^{\alpha} & \:\le_{\overline{\Psi}}\; & \sigma B_{r}^{\alpha}\\
\tau\overline{B_{\overline{r}}^{\overline{\alpha}}} & \:\le_{\overline{\Psi}} & \tau\sum_{s=0}^{b_{r}-1}S_{rs}^{\alpha} & \:\le_{\overline{\Psi}}\; & \tau B_{r}^{\alpha}\\
\sigma\tau\overline{B_{\overline{r}}^{\overline{\alpha}}} & \:\le_{\Psi} & \;\sigma\tau\sum_{s=0}^{b_{r}-1}S_{rs}^{\alpha} & \:\le_{\Psi}\; & \sigma\tau B_{r}^{\alpha}
\end{alignat*}
\end{cor}

\begin{alignat*}{3}
\tau\overline{B_{\overline{r}}^{\overline{\alpha}}} & \:\ge_{\Psi} & \tau\sum_{s=0}^{b_{r}-1}S_{rs}^{\alpha} & \:\ge_{\Psi}\; & \tau B_{r}^{\alpha}\\
\tau\sigma\overline{B_{\overline{r}}^{\overline{\alpha}}} & \:\ge_{\overline{\Psi}} & \;\tau\sigma\sum_{s=0}^{b_{r}-1}S_{rs}^{\alpha} & \:\ge_{\overline{\Psi}}\; & \tau\sigma B_{r}^{\alpha}\\
\overline{B_{\overline{r}}^{\overline{\alpha}}} & \:\ge_{\overline{\Psi}} & \sum_{s=0}^{b_{r}-1}S_{rs}^{\alpha} & \:\ge_{\overline{\Psi}}\; & B_{r}^{\alpha}\\
\sigma\overline{B_{\overline{r}}^{\overline{\alpha}}} & \:\ge_{\Psi} & \sigma\sum_{s=0}^{b_{r}-1}S_{rs}^{\alpha} & \:\ge_{\Psi}\; & \sigma B_{r}^{\alpha}
\end{alignat*}

\begin{proof}
This is an application of the dualities in Subsection \ref{subsec:SourceTargetDualities}.
\end{proof}

\begin{cor}
\label{cor:asymmetric-1}If $\phi$ is a decreasing anti-symmetric
function then 

\begin{flalign}
B_{r}^{\alpha}\phi & =\left\llbracket b_{r}>0\right\rrbracket \left(\sum_{s=0}^{b_{r}-1}\left(\phi_{rsq_{r}}+\left\llbracket q_{r}>1\right\rrbracket \left(\phi_{rsq_{r-1}}+\phi(\frac{1}{q_{r}})\right)\right)+\left\llbracket 2<q_{r}<q_{n}\right\rrbracket E_{r}\left(\phi_{r0,q_{r}-q_{r-1}}+\phi\left(\frac{2}{q_{r}}\right)\right)\right)\label{eq:BaseInequalityAsymm}\\
\overline{B_{\overline{r}}^{\overline{\alpha}}}\phi & =\left\llbracket b_{r}>0\right\rrbracket \left(\sum_{s=0}^{b_{r}-1}\left(\phi_{rsq_{r}}+\left\llbracket q_{r}>1\right\rrbracket \left(\phi_{rsq_{r-1}}-\phi(\frac{1}{q_{r}})\right)\right)+\left\llbracket 2<q_{r}<q_{n}\right\rrbracket O_{r}\left(\phi_{r0,q_{r}-q_{r-1}}-\phi\left(\frac{2}{q_{r}}\right)\right)\right)\label{eq:BaseInequalityAsymm2}
\end{flalign}

Further for $q_{r}>1$ each sum $\sum_{s=0}^{b_{r}-1}\left(\phi_{rsq_{r}}+\phi_{rsq_{r-1}}\right)$
is positive  for $r$ even, and negative for $r$ odd. 

 .
\end{cor}

\begin{proof}
Note if $\phi$ is anti-symmetric then $\phi=-\overline{\phi}$ and
$P_{q_{r}}(\phi)=0$. The expressions (\ref{eq:BaseInequalityAsymm},\ref{eq:BaseInequalityAsymm2})
then follow easily from the theorem. We now examine $X_{r}=\sum_{s=0}^{b_{r}-1}\left(\phi_{rsq_{r}}+\phi_{rsq_{r-1}}\right)$.
 For $r$ even $\phi_{rsq_{r}}+\phi_{rsq_{r-1}}\le\phi\left(\frac{s+1/a_{r+2}^{\slash}}{q_{r+1}^{\slash}}\right)-\phi\left(\frac{a_{r+1}-s}{q_{r+1}^{\slash}}\right)$.

Reversing the order of second terms gives us $X_{r}=\sum_{s=0}^{b_{r}-1}\left(\phi\left(\frac{s+1/a_{r+2}^{\slash}}{q_{r+1}^{\slash}}\right)-\phi\left(\frac{a_{r+1}-(b_{r}-1)+s}{q_{r+1}^{\slash}}\right)\right)$.
Since $b_{r}\le a_{r+1},$ each term in this sum is positive, and
hence $X_{r}\ge0$ for $r$ even. Similarly $X_{r}\le0$ for $r$
odd. 
\end{proof}

\subsection{Further bounds for primitive functions }

The results of the previous section hold for any monotonic decreasing
$\phi$. In this section we will develop the upper bound results further
for those functions which are lower bounded, and in particular for
the positive primitives in $\Phi$.. Recall from Subsection \ref{eq:phiDecomp1}
we can reduce any unbounded observable to a sum of primitives. 

The key new ingredient arises from the fact that if $\phi{}_{\Phi}$
is lower bounded at say $1^{-}$ with bound $c$, then for irrational
$0<\alpha<\frac{1}{2}$ the sequence of extremal points $\alpha_{rsq_{r}}$
tend to $1^{-}$ for $r$ odd and hence $\phi_{rsq_{r}}\rightarrow c$.

Recall our notation for limits when they exist, for example $\phi(1^{-})=\lim_{x\uparrow1}\phi(x)$.
Note the limit exists whenever $\phi$ is LBV left of $1$, and in
particular for $\phi\in\Phi\bigcup-\Phi$. Also note that $\phi\mapsto\phi(1^{-})$
is then a linear functional, and in particular $(-\phi)(1^{-})=-\phi(1^{-})$.
Analogous remarks apply to $\phi(0^{+})=\lim_{x\downarrow0}\phi(x)$.
Note that if $\phi(1^{-})$ exists, so does $\overline{\phi}(0^{+})$
and the two limits are equal. 

Finally for $\phi$ monotonic descending with $\phi(1^{-})=C<0$ then
$\phi-C\in\Phi$ and $S_{N}\phi=S_{N}(\phi-C)+CN$, which means we
can restrict our attention to $\phi\in\Phi$. Similar remarks apply
to all quadrants. 

\begin{lem}[Upper bounds for primitive functions]
Let $\phi\in\Phi$. Then for $q_{r}>2$ we have: 

\begin{align}
B_{r}^{\alpha}(\phi) & \le\left\llbracket b_{r}>0\right\rrbracket \left(b_{r}P_{q_{r}}+\left\llbracket 2<q_{r}<q_{n}\right\rrbracket E_{r}\left(\phi_{r0,q_{r}-q_{r-1}}-\overline{\phi}\left(\frac{1}{q_{r}}\right)\right)+\sum_{s=0}^{b_{r}-1}\left(E_{r}\phi_{rsq_{r}}+O_{r}\phi_{rsq_{r-1}}\right)\right)\label{eq:NewB-2}\\
\overline{B_{\overline{r}}^{\overline{\alpha}}}\overline{\phi} & \le\left\llbracket b_{r}>0\right\rrbracket \left(b_{r}P_{q_{r}}+\left\llbracket 2<q_{r}<q_{n}\right\rrbracket O_{r}\left(\overline{\phi}_{r0,q_{r}-q_{r-1}}-\overline{\phi}\left(\frac{1}{q_{r}}\right)\right)+\sum_{s=0}^{b_{r}-1}\left(O_{r}\overline{\phi}_{rsq_{r}}+E_{r}\overline{\phi}_{rsq_{r-1}}\right)\right)\label{eq:NewALower-1}
\end{align}

Further, these inequalities are reversed under the duality $\phi\mapsto-\phi$,
and are equalities for constant $\phi$ (and trivially also for $b_{r}=0$).
\end{lem}

\begin{proof}
All terms are linear functionals, so that results for $-\phi$ follow
by duality. The equalities follow the same argument as Lemma \ref{lem:BaseInequalities}.

Since $\phi\in\Phi$, $\phi$ is descending and we have from Lemma
\ref{lem:BaseInequalities} for $q_{r}>2$

\begin{flalign}
B_{r}^{\alpha}(\phi) & =\left\llbracket b_{r}>0\right\rrbracket \left(b_{r}\left(P_{q_{r}}-\overline{\phi}(\frac{1}{q_{r}})\right)+\left\llbracket 2<q_{r}<q_{n}\right\rrbracket E_{r}\left(\phi_{r0,q_{r}-q_{r-1}}-\overline{\phi}\left(\frac{2}{q_{r}}\right)\right)+\sum_{s=0}^{b_{r}-1}\left(\phi_{rsq_{r}}+\phi_{rsq_{r-1}}\right)\right)\label{eq:BaseInequalityBCopy-1}\\
\overline{B_{\overline{r}}^{\overline{\alpha}}}(\phi) & =\left\llbracket b_{r}>0\right\rrbracket \left(b_{r}\left(P_{q_{r}}-\phi(\frac{1}{q_{r}})\right)+\left\llbracket 2<q_{r}<q_{n}\right\rrbracket O_{r}\left(\phi_{r0,q_{r}-q_{r-1}}-\phi\left(\frac{2}{q_{r}}\right)\right)+\sum_{s=0}^{b_{r}-1}\left(\phi_{rsq_{r}}+\phi_{rsq_{r-1}}\right)\right)\label{eq:BaseInequalityACopy-1}
\end{flalign}

Also since $\phi$ is descending, for $q_{r}>1$ we have $O_{r}\phi_{rsq_{r}}\le O_{r}\overline{\phi}(\frac{1}{q_{r}})$,
and for $s>0$ or $r=n$ we also have $E_{r}\phi_{rsq_{r-1}}\le E_{r}\overline{\phi}(\frac{1}{q_{r}})$.
However for $s=0\&r<n$ we may have $\alpha_{r0q_{r-1}}$ pulled back
to its alternate interval. 

However in this case, since $q_{r}>2$ we have $E_{r}\phi_{r0q_{r-1}}\le E_{r}\overline{\phi}(\frac{2}{q_{r}})$
which then gives us $\phi_{rsq_{r}}+\phi_{rsq_{r-1}}-\overline{\phi}(\frac{1}{q_{r}})\le E_{r}\phi_{rsq_{r}}+O_{r}\phi_{rsq_{r-1}}+\left\llbracket s=0\&r<n\right\rrbracket E_{r}(\overline{\phi}(\frac{2}{q_{r}})-\overline{\phi}(\frac{1}{q_{r}}))$.
For $b_{r}>0$ we can now substitute this result into \eqref{eq:BaseInequalityBCopy-1}
to obtain \eqref{eq:NewALower-1}. The result for $B_{r}(\overline{\phi})$
in \eqref{eq:NewALower-1}) follows by analogous argument using $\overline{\phi}$
ascending. 

\end{proof}

\begin{lem}
Let $\phi\in\Phi$. Then for $q_{r}=2$ we have: 

\begin{align}
B_{r}^{\alpha}\phi & \le\left\llbracket b_{r}>0\right\rrbracket \left(b_{r}P_{q_{r}}+\left\llbracket 2=q_{r}<q_{n}\right\rrbracket E_{r}(\phi_{r0q_{r-1}}^{\alpha}-\overline{\phi}(\frac{1}{q_{r}}))+\sum_{s=0}^{b_{r}-1}\left(E_{r}\phi_{rsq_{r}}^{\alpha}+O_{r}\phi_{rsq_{r-1}}^{\alpha}\right)\right)\label{eq:NewB}\\
\overline{B_{\overline{r}}^{\overline{\alpha}}}\overline{\phi} & \le\left\llbracket b_{r}>0\right\rrbracket \left(b_{r}P_{q_{r}}+\left\llbracket 2=q_{r}<q_{n}\right\rrbracket O_{r}(\overline{\phi}_{r0q_{r-1}}^{\alpha}-\overline{\phi}(\frac{1}{q_{r}}))+\sum_{s=0}^{b_{r}-1}\left(O_{r}\overline{\phi}_{rsq_{r}}^{\alpha}+E_{r}\overline{\phi}_{rsq_{r-1}}^{\alpha}\right)\right)\label{eq:NewALower}
\end{align}

Writing $X_{r}^{\alpha}(\phi)=\left\llbracket 2=q_{r}<q_{n}\right\rrbracket E_{r}\left(\phi_{r0q_{r-1}}^{\alpha}-\overline{\phi}(\frac{1}{q_{r}})\right)$
for $q_{r}=2$ we have $\left|X_{r}(\phi)\right|<\phi(\frac{1}{3})-\phi(\frac{1}{2})$.
If $X_{r}\phi\ne0$ then $-\sgn X_{r}\overline{\phi}=\left\llbracket \left\{ \alpha\right\} <\frac{1}{2}\right\rrbracket =\sgn X_{r}\phi$,
so that the $X_{r}$ term can be ignored in \ref{eq:NewALower} for
$\left\llbracket \left\{ \alpha\right\} <\frac{1}{2}\right\rrbracket $,
and in \ref{eq:NewALower} for $\left\llbracket \left\{ \alpha\right\} >\frac{1}{2}\right\rrbracket $.

Further, these inequalities are reversed under the duality $\phi\mapsto-\phi$,
and are equalities for constant $\phi$ (and also trivially for $b_{r}=0$).
\end{lem}

\begin{proof}
All terms are linear functionals, so that results for $-\phi$ follow
by duality. The equalities again follow the same argument as Lemma
\ref{lem:BaseInequalities} above.

Since $\phi\in\Phi$, $\phi$ is descending and we have from Lemma
\ref{lem:BaseInequalities} for $q_{r}=2$

\begin{flalign}
B_{r}(\phi) & =\left\llbracket b_{r}>0\right\rrbracket \left(b_{r}\left(P_{q_{r}}-\overline{\phi}(\frac{1}{q_{r}})\right)+\sum_{s=0}^{b_{r}-1}\left(\phi_{rsq_{r}}+\phi_{rsq_{r-1}}\right)\right)\label{eq:BaseInequalityBCopy}\\
\overline{B}_{r}(\phi) & =\left\llbracket b_{r}>0\right\rrbracket \left(b_{r}\left(P_{q_{r}}-\phi(\frac{1}{q_{r}})\right)+\sum_{s=0}^{b_{r}-1}\left(\phi_{rsq_{r}}+\phi_{rsq_{r-1}}\right)\right)\label{eq:BaseInequalityACopy}
\end{flalign}

Also since $\phi$ is descending, for $q_{r}>1$ we have $O_{r}\phi_{rsq_{r}}\le O_{r}\overline{\phi}(\frac{1}{q_{r}})$,
and for $s>0$ or $r=n$ we also have $E_{r}\phi_{rsq_{r-1}}\le E_{r}\overline{\phi}(\frac{1}{q_{r}})$.
However for $s=0\&r<n$ we may have $\alpha_{r0q_{r-1}}$ pulled back
to its alternate interval which gives us $\phi_{rsq_{r}}+\phi_{rsq_{r-1}}-\overline{\phi}(\frac{1}{q_{r}})\le E_{r}\phi_{rsq_{r}}+O_{r}\phi_{rsq_{r-1}}+\left\llbracket s=0\&r<n\right\rrbracket E_{r}(\phi_{r0q_{r-1}}-\overline{\phi}(\frac{1}{q_{r}}))$.
This establishes \ref{eq:NewB} and then \ref{eq:NewALower} follows
as its qp-dual. 

If $\phi$ is a constant function, the mid-term is zero. We now investigate
requirements for the mid-terms to be either zero or positive (non-negative).
Note that if $q_{r}=2$ we have $q_{r-1}=1$ and $r\in\{1,2\}$.

For $r=1$ we have $E_{r}=0$ and so the mid-term vanishes in $B_{1}\phi$.
However $O_{r}=1$ and so the mid-term in $B_{1}^{\alpha}\overline{\phi}$
is $\overline{\phi}_{101}^{\alpha}-\overline{\phi}(\frac{1}{2})=\overline{X}_{1}^{\alpha}\overline{\phi}$
which is non-negative for $\alpha_{101}\ge\frac{1}{2}$. But $\alpha_{101}<1-\left(\frac{1}{2}+\frac{1}{q_{r+2}^{\slash}}-\frac{1}{2q_{r+1}^{\slash}}\right)<\frac{1}{2}+\frac{1}{2q_{2}^{\slash}}$and
$q_{2}^{\slash}>3$ so $\alpha_{101}<\frac{1}{2}+\frac{1}{2.3}=\frac{2}{3}$.
Also for $\alpha_{101}\ge\frac{1}{2}$ we must have $\frac{1}{q_{r+2}^{\slash}}\le\frac{1}{2q_{r+1}^{\slash}}$
giving $a_{r+2}^{\slash}\ge2$, requiring $a_{3}\ge2$. Finally $q_{1}=q_{r}=2$
which means $a_{1}=2$ and hence $\frac{1}{3}<\alpha<\frac{1}{2}$.

For $r=2$ we have $O_{r}=0$ and so the mid-term vanishes in $B_{2}\overline{\phi}$.
However $E_{r}=1$ and so the mid-term in $B_{2}\phi$ is $\phi_{201}-\overline{\phi}(\frac{1}{2})$.
We now use the identity $\phi_{rst}^{\alpha}=\overline{\phi}_{\overline{r}st}^{\overline{\alpha}}$
to get $\phi_{201}^{\alpha}-\overline{\phi}(\frac{1}{2})=\overline{\phi}_{101}^{\overline{\alpha}}-\overline{\phi}(\frac{1}{2})$.
But this is the mid term of $B_{1}^{\overline{\alpha}}\overline{\phi}$
and so by the previous result, non-negative requires $\frac{1}{3}<\overline{\alpha}<\frac{1}{2},\frac{1}{2}<\overline{\alpha}_{101}<\frac{2}{3}$
and $a_{1}^{\alpha}=1,a_{2}^{\alpha}=a_{1}^{\overline{\alpha}}-1=1,a_{4}^{\alpha}=a_{3}^{\overline{\alpha}f}\ge2$
which establishes the result.
\end{proof}
Note that the constant $X_{r}\phi$ in (\ref{eq:NewALower}) is generally
a small constant which will be insignificant for most purposes, and
is often negative in which case it can be ignored entirely in these
inequalities. The conditions for it to be positive are necessary but
not sufficient. In the following example $X_{2}\phi>0$ for strictly
descending $\phi$, showing that this term cannot be eliminated in
the inequality.
\begin{example}
$\alpha=[2,1,3,100,...]\approx0.36,N=8=2*3+1*2+0*1,(101)=7,\alpha_{101}=\{7\alpha\}\approx0.54>\frac{1}{2}$.
\end{example}

Analysis of Upper Bound Components 

\begin{align}
B_{r}^{\alpha}\phi & \le\left\llbracket b_{r}>0\right\rrbracket \left(b_{r}P_{q_{r}}+\left\llbracket 2<q_{r}<q_{n}\right\rrbracket E_{r}\left(\phi_{r0,q_{r}-q_{r-1}}^{\alpha}-\overline{\phi}\left(\frac{1}{q_{r}}\right)\right)+\sum_{s=0}^{b_{r}-1}\left(E_{r}\phi_{rsq_{r}}^{\alpha}+\left\llbracket q_{r}>1\right\rrbracket O_{r}\phi_{rsq_{r-1}}^{\alpha}\right)\right)\label{eq:NewB-2-1}\\
\overline{B_{\overline{r}}^{\overline{\alpha}}}\overline{\phi} & \le\left\llbracket b_{r}>0\right\rrbracket \left(b_{r}P_{q_{r}}+\left\llbracket 2<q_{r}<q_{n}\right\rrbracket O_{r}\left(\overline{\phi}_{r0,q_{r}-q_{r-1}}^{\alpha}-\overline{\phi}\left(\frac{1}{q_{r}}\right)\right)+\sum_{s=0}^{b_{r}-1}\left(O_{r}\overline{\phi}_{rsq_{r}}^{\alpha}+\left\llbracket q_{r}>1\right\rrbracket E_{r}\overline{\phi}_{rsq_{r-1}}^{\alpha}\right)\right)\label{eq:NewALower-1-1}
\end{align}

We now analyse further the third terms above. 
Using values of $\alpha_{rst}$ 
We now exploit bounds for $\alpha_{rst}$ developed in Section \ref{sec:Distribution}. 

First we need to treat the case $q_{r}=1$ separately.
\begin{lem}
For $q_{r}=1$ we have $P_{q_{r}}=0$ and then $B_{r}^{\alpha}\phi=\sum_{s=0}^{b_{r}-1}E_{r}\phi_{rsq_{r}}^{\alpha}$

\begin{equation}
\sum_{s=0}^{b_{r}-1}\phi\left(\frac{s+1}{q_{r+1}^{\slash}}+\frac{1}{q_{r+2}^{\slash}}\right)\le\overline{B}_{r}(\phi)=\sum_{s=0}^{b_{r}-1}\phi_{rsq_{r}}=B_{r}(\phi)\le\sum_{s=0}^{b_{r}-1}\phi\left(\frac{s}{q_{r+1}^{\slash}}+\frac{1}{q_{r+2}^{\slash}}\right)\label{eq:A0}
\end{equation}

where the inequalities are reversed under the dualities $\phi\mapsto-\phi,\phi\mapsto\overline{\phi}$,
and become equalities for constant $\phi$. The upper bounds also
give $\sum_{s=0}^{b_{r}-1}\phi_{rsq_{r}}=B_{r}\phi\le\phi\left(\frac{1}{q_{r+2}^{\slash}}\right)+\sum_{s=1}^{b_{r}-1}\phi\left(\frac{s}{q_{r+1}^{\slash}}\right)$
and $\sum_{s=0}^{b_{r}-1}\overline{\phi}_{rsq_{r}}=B_{r}\overline{\phi}\le\sum_{s=1}^{b_{r}}\overline{\phi}\left(\frac{a_{r+1}^{\slash}-s}{q_{r+1}^{\slash}}\right)$
\end{lem}

\begin{proof}
Since $\phi\in\Phi$, $\phi$ is monotonic descending and we have
by definition (see \ref{eq:BoundsDef}) for any monotonic descending
$\phi$ and $q_{r}=1$: 

\begin{flalign}
B_{r}(\phi) & =\overline{B}_{r}(\phi)=\left\llbracket b_{r}>0\right\rrbracket \sum_{s=0}^{b_{r}-1}\phi_{rsq_{r}}\label{eq:BaseInequalityBCopy-2-1}
\end{flalign}

If $b_{r}=0$ the sum is empty so the condition $\left\llbracket b_{r}>0\right\rrbracket $
becomes redundant which establishes the central equalities in (\ref{eq:A0})
The rest of (\ref{eq:A0}) then follows from (\ref{cor:ParityDuality}).
The duality result follows from (\ref{DSDuality}).

The second upper bound for $\sum_{s=0}^{b_{r}-1}\phi_{rsq_{r}}$ is
just a rewrite of (\ref{eq:A0}). Using duality we have $\sum_{s=0}^{b_{r}-1}\overline{\phi}_{rsq_{r}}\le\sum_{s=0}^{b_{r}-1}\overline{\phi}\left(\frac{s+1}{q_{r+1}^{\slash}}+\frac{1}{q_{r+2}^{\slash}}\right)$

Now using $\frac{1}{q_{r+2}^{\slash}}+\frac{a_{r+1}-s}{q_{r+1}^{\slash}}=\frac{a_{r+1}^{\slash}-s}{q_{r+1}^{\slash}}$
from (\ref{lem:identities}) we get $\frac{s+1}{q_{r+1}^{\slash}}+\frac{1}{q_{r+2}^{\slash}}=\frac{a_{r+1}^{\slash}-(a_{r+1}-s-1)}{q_{r+1}^{\slash}}=\frac{a_{r+1}^{\slash}-u}{q_{r+1}^{\slash}}$
where $u=a_{r+1}-s-1$ so $\sum_{s=0}^{b_{r}-1}\overline{\phi}\left(\frac{s+1}{q_{r+1}^{\slash}}+\frac{1}{q_{r+2}^{\slash}}\right)=\sum_{u=a_{r+1}-b_{r}}^{a_{r+1}-1}\overline{\phi}\left(\frac{a_{r+1}^{\slash}-u}{q_{r+1}^{\slash}}\right)\le\sum_{u=a_{r+1}-b_{r}-t}^{a_{r+1}-1-t}\overline{\phi}\left(\frac{a_{r+1}^{\slash}-u}{q_{r+1}^{\slash}}\right)$
for $t\ge0$.  We put $t=a_{r+1}-b_{r}-1$ to get $\sum_{u=1}^{b_{r}}\overline{\phi}\left(\frac{a_{r+1}^{\slash}-u}{q_{r+1}^{\slash}}\right)$.

\end{proof}
We now look at the general case where $q_{r}>1$. For convenience
we define some further shorthand notation:
\begin{defn}
\label{def:Crs}Define $C_{rs}^{\alpha}\phi\coloneqq E_{r}\phi_{rsq_{r}}^{\alpha}+\left\llbracket q_{r}>1\right\rrbracket O_{r}\phi_{rsq_{r-1}}^{\alpha}$
(from \eqref{eq:NewALower-1-1}and its double -dual $\overline{C}_{rs}^{\alpha}\overline{\phi}\coloneqq O_{r}\overline{\phi}_{rsq_{r}}^{\alpha}+\left\llbracket q_{r}>1\right\rrbracket E_{r}\overline{\phi}_{rsq_{r-1}}^{\alpha}$
\end{defn}

For $\phi\in\Phi$ (ie monotone decreasing, bounded below), we now
investigate upper bounds for the $q_{r}>1$ terms in $C_{rs}\phi$. 

 Using the bounds on $\alpha_{rst}$ from \ref{prop:BaseAlpha} ({*}{*}Need
a clear summary of $\alpha_{rst}$ results for here!!) gives: 

\begin{center}
\begin{flalign}
C_{rs}\phi=\left(\phi_{rsq_{r}}^{E}+\phi_{rsq_{r-1}}^{O}\right) & =\phi^{E}\left(\frac{s+1}{q_{r+1}^{\slash}}\right)+\phi^{O}\left(\frac{1}{q_{r+2}^{\slash}}+\frac{a_{r+1}-s}{q_{r+1}^{\slash}}\right) & r=n\label{eq:SB1}\\
 & \ge\phi^{E}\left(\frac{1}{q_{r+2}^{\slash}}+\frac{s+1}{q_{r+1}^{\slash}}\right)+\phi^{O}\left(\frac{a_{r+1}-\left(s-1\right)}{q_{r+1}^{\slash}}\right) & r<n\nonumber \\
 & \le\phi^{E}\left(\frac{1}{q_{r+2}^{\slash}}+\frac{s}{q_{r+1}^{\slash}}\right)+\phi^{O}\left(\frac{a_{r+1}-s}{q_{r+1}^{\slash}}\right) & r<n\nonumber 
\end{flalign}
\par\end{center}

and 
\begin{flalign}
\overline{C}_{rs}\overline{\phi}=\left(\overline{\phi}_{rsq_{r}}^{O}+\overline{\phi}_{rsq_{r-1}}^{E}\right) & =\phi^{O}\left(\frac{s+1}{q_{r+1}^{\slash}}\right)+\phi^{E}\left(\frac{1}{q_{r+2}^{\slash}}+\frac{a_{r+1}-s}{q_{r+1}^{\slash}}\right) & r=n\label{eq:SB2}\\
 & \ge\phi^{O}\left(\frac{1}{q_{r+2}^{\slash}}+\frac{s+1}{q_{r+1}^{\slash}}\right)+\phi^{E}\left(\frac{a_{r+1}-\left(s-1\right)}{q_{r+1}^{\slash}}\right) & r<n\nonumber \\
 & \le\phi^{O}\left(\frac{1}{q_{r+2}^{\slash}}+\frac{s}{q_{r+1}^{\slash}}\right)+\phi^{E}\left(\frac{a_{r+1}-s}{q_{r+1}^{\slash}}\right) & r<n\nonumber 
\end{flalign}

Grouping terms by quasiperiod

Note in the sums (\ref{eq:SB1}) that the denominator of each fraction
is of order $q_{r+1}^{\slash}$ with the exceptions of the upper
bounds for $\phi_{rsq_{r}}^{E},\overline{\phi}_{rsq_{r}}^{O}$ for
$s=0,r<n$ ({*}{*}check r=0???) which are of order $q_{r+2}^{\slash}$. 

But the size of $\phi\left(\frac{1}{q_{r}^{\slash}}\right)$ is sensitively
dependent on the size of $q_{r}^{\slash}$ when $\phi$ is unbounded
at $0^{+}$, and so we will now group terms by denominator, as follows.
\begin{defn}
Let $C_{r}$ be the sum of the terms in the double sum $DS\phi=\sum_{r=0}^{n}\sum_{s=1}^{b_{r}-1}C_{rs}\phi$
with denominator $q_{r+1}^{\slash}$, so that $DS(\phi)=\sum_{r=0}^{n}C_{r}(\phi)$.

For $q_{r}=1$, we have from (\ref{eq:A0}) $C_{r}(\phi)=\sum_{s=1}^{b_{r}-1}\phi\left(\frac{s}{q_{r+1}^{\slash}}\right)$
and $\overline{C}_{r}(\phi)=\sum_{s=1}^{b_{r}}\overline{\phi}\left(\frac{a_{r+1}-s}{q_{r+1}^{\slash}}\right)$.
\end{defn}

Then by inspection of \eqref{eq:SB1}, promoting the terms with denominators
$q_{r+2}^{\slash}$ we have for $r\ge1$:

\begin{flalign}
C_{n}(\phi) & =\sum_{s=0}^{b_{n}-1}\phi^{E}\left(\frac{s+1}{q_{n+1}^{\slash}}\right)+\sum_{s=0}^{b_{n}-1}\phi^{O}\left(\frac{a_{n+1}^{\slash}-s}{q_{n+1}^{\slash}}\right)\,\,+\left\llbracket b_{n-1}>0\right\rrbracket O_{r}\phi\left(\frac{1}{q_{n+1}^{\slash}}\right) & r=n\label{eq:SB3}\\
C_{r}(\phi) & \le\sum_{s=1}^{b_{r}-1}\phi^{E}\left(\frac{1}{q_{r+2}^{\slash}}+\frac{s}{q_{r+1}^{\slash}}\right)+\sum_{s=0}^{b_{r}-1}\phi^{O}\left(\frac{a_{r+1}-s}{q_{r+1}^{\slash}}\right)\,\,+\left\llbracket b_{r-1}>0\right\rrbracket O_{r}\phi\left(\frac{1}{q_{r+1}^{\slash}}\right) & r<n\nonumber 
\end{flalign}

Note that at this point we have simply rearranged terms using the
originating inequalities (\ref{eq:SB1}). However it is difficult
to make use of all the information in (\ref{eq:SB3}) except in special
circumstances such as periodic partial quotients. Our focus in this
paper is on the general case, and so we now develop some relaxations
of (\ref{eq:SB3}) which are less precise but more tractable for our
purposes. Since $\phi$ is decreasing we have $\phi^{E}\left(\frac{1}{q_{r+2}^{\slash}}+\frac{s}{q_{r+1}^{\slash}}\right)\le\phi^{E}\left(\frac{s}{q_{r+1}^{\slash}}\right)$
and $\phi^{O}\left(\frac{a_{r+1}^{\slash}-s}{q_{r+1}^{\slash}}\right)\le\phi^{O}\left(\frac{a_{r+1}-s}{q_{r+1}^{\slash}}\right)$,
which with some rewriting results in the slightly less precise but
more uniform inequalities:

\begin{align}
C_{n}(\phi) & \le\sum_{s=1}^{b_{n}}\phi^{E}\left(\frac{s}{q_{n+1}^{\slash}}\right)+\left\llbracket b_{n-1}>0\right\rrbracket \phi^{O}\left(\frac{1}{q_{n+1}^{\slash}}\right)+\sum_{s=0}^{b_{n}-1}\phi^{O}\left(\frac{a_{n+1}-s}{q_{n+1}^{\slash}}\right)\label{eq:Sr}\\
C_{r}(\phi) & \le\sum_{s=1}^{b_{r}-1}\phi^{E}\left(\frac{s}{q_{r+1}^{\slash}}\right)+\left\llbracket b_{r-1}>0\right\rrbracket \phi^{O}\left(\frac{1}{q_{r+1}^{\slash}}\right)+\sum_{s=0}^{b_{r}-1}\phi^{O}\left(\frac{a_{r+1}-s}{q_{r+1}^{\slash}}\right)\nonumber 
\end{align}

Note that for $r$ even $C_{r}$ has $b_{r}-1$ or $b_{r}$ (for $r=n$
only) non-zero terms, whereas for $r$ odd $C_{r}$ has $b_{r}$ or
$b_{r}+1$ (for $b_{r-1}>0$) non-zero terms.  We will use this form
for one particular application in the next section, but for many purposes
we can afford to relax the inequalities still further to the point
where we can eliminate the asymmetry. 

We first use the fact that $\phi$ is descending so that for $b_{r}<a_{r+1}$
we have $\sum_{s=0}^{b_{r}-1}\phi^{O}\left(\frac{a_{r+1}-s}{q_{r+1}^{\slash}}\right)=\sum_{s=a_{r+1}-(b_{r}-1)}^{a_{r+1}}\phi^{O}\left(\frac{s}{q_{r+1}^{\slash}}\right)\le\sum_{s=2}^{b_{r}+1}\phi^{O}\left(\frac{s}{q_{r+1}^{\slash}}\right)$.
Using $\left\llbracket b_{r-1}>0\right\rrbracket =1-\left\llbracket b_{r-1}=0\right\rrbracket $
and $\phi$ descending gives 
\[
\left\llbracket b_{r-1}>0\right\rrbracket \phi^{O}\left(\frac{1}{q_{r+1}^{\slash}}\right)+\sum_{s=0}^{b_{r}-1}\phi^{O}\left(\frac{a_{r+1}-s}{q_{r+1}^{\slash}}\right)\le\sum_{s=1}^{b_{r}+1}\phi^{O}\left(\frac{s}{q_{r+1}^{\slash}}\right)-\left\llbracket b_{r-1}=0\right\rrbracket \phi^{O}\left(\frac{1}{q_{r+1}^{\slash}}\right)
\]
If $b_{r}=a_{r+1}$ we have $\sum_{s=0}^{b_{r}-1}\phi^{O}\left(\frac{a_{r+1}-s}{q_{r+1}^{\slash}}\right)=\sum_{s=1}^{a_{r+1}}\phi^{O}\left(\frac{s}{q_{r+1}^{\slash}}\right)$,
and we also have from(\ref{lem:ORresults}) that $b_{r-1}=0$ so that
$\sum_{s=0}^{b_{r}-1}\phi^{O}\left(\frac{a_{r+1}-s}{q_{r+1}^{\slash}}\right)=\sum_{s=1}^{a_{r+1}}\phi^{O}\left(\frac{s}{q_{r+1}^{\slash}}\right)$. 

In addition we have $\phi^{E}\left(\frac{1}{q_{r+2}^{\slash}}+\frac{s}{q_{r+1}^{\slash}}\right)\le\phi^{E}\left(\frac{s}{q_{r+1}^{\slash}}\right)$
and $\phi^{O}\left(\frac{a_{r+1}^{\slash}-s}{q_{r+1}^{\slash}}\right)\le\phi^{O}\left(\frac{a_{r+1}-s}{q_{r+1}^{\slash}}\right)$.
Hence, writing $c_{r}^{\alpha}=c(N,\alpha,r)=E_{r}\left(b_{r}-1+\left\llbracket r=n\right\rrbracket \right)+O_{r}\min(b_{r}+1,a_{r+1})$,
we can combine all these results for $2\le r\le n$ into: 

\begin{align}
C_{r}(\phi) & \le\sum_{s=1}^{c_{r}^{\alpha}}\phi\left(\frac{s}{q_{r+1}^{\slash}}\right)\,\,-\left\llbracket b_{r-1}=0\right\rrbracket \phi^{O}\left(\frac{1}{q_{r+1}^{\slash}}\right)\label{eq:Sr-1}
\end{align}

Finally, if $\phi\left(\frac{1}{q_{r+1}^{\slash}}\right)\ge0$ we
can reduce this to $C_{r}\le\sum_{s=1}^{c_{r}}\phi\left(\frac{s}{q_{r+1}^{\slash}}\right)$.
Note that for $r$ odd, this adds significant imprecision if $b_{r-1}=0$,
but then we cannot use this information if we do not know $b_{r-1}=0$,
and in general we will not.

The argument for $\overline{\phi}$ is precisely analogous giving
for $\overline{c}_{r}^{\alpha}=c_{\overline{r}}^{\overline{\alpha}}=O_{r}\left(b_{r}-1+\left\llbracket r=n\right\rrbracket \right)+E_{r}\min(b_{r}+1,a_{r+1})$,
and $1\le r\le n$:
\begin{align}
\overline{C}_{r}^{\alpha}(\overline{\phi}) & \le\sum_{s=1}^{\overline{c}_{r}^{\alpha}}\phi\left(\frac{s}{q_{r+1}^{\slash}}\right)\,\,-\left\llbracket b_{r-1}=0\right\rrbracket \phi^{E}\left(\frac{1}{q_{r+1}^{\slash}}\right)\label{eq:Sr-1-1}
\end{align}

\subsection{\label{subsec:Summary-of-Results}Summary of Results}

For convenience we summarise here the results of this Section. 

Recall $S_{N}\phi=\sum_{r=0}^{n}\sum_{s=0}^{b_{r}-1}S_{rs}^{\alpha}\phi$
where $S_{rs}^{\alpha}\phi=S_{q_{r}}^{\alpha}(\phi,\alpha_{rs})=\sum_{t=1}^{q_{r}}\phi(\alpha_{rs}+t\alpha)=\sum_{t=1}^{q_{r}}\phi_{rst}^{\alpha}$. 
\begin{lem}[Bounds for monotonic functions]
\label{lem:BaseInequalities-1-1}Let $\phi$ be a monotonic decreasing
observable. Then for $r\ge0$ we have $\overline{B_{\overline{r}}^{\overline{\alpha}}}\phi\le\sum_{s=0}^{b_{r}-1}S_{rs}^{\alpha}\phi\le B_{r}^{\alpha}\phi$
 where $B_{r}^{\alpha},\overline{B_{\overline{r}}^{\overline{\alpha}}}$
are the Bound Functionals: 

\begin{flalign}
B_{r}^{\alpha}\phi & =\left\llbracket b_{r}>0\right\rrbracket \left(\left\llbracket q_{r}>1\right\rrbracket b_{r}\left(P_{q_{r}}-\overline{\phi}(\frac{1}{q_{r}})\right)+\left\llbracket 2<q_{r}<q_{n}\right\rrbracket E_{r}\left(\phi_{r0,q_{r}-q_{r-1}}-\overline{\phi}\left(\frac{2}{q_{r}}\right)\right)+\sum_{s=0}^{b_{r}-1}\left(\phi_{rsq_{r}}+\left\llbracket q_{r}>1\right\rrbracket \phi_{rsq_{r-1}}\right)\right)\label{eq:BaseInequalityB-1-2}\\
\overline{B_{\overline{r}}^{\overline{\alpha}}}\phi & =B_{\overline{r}}^{\overline{\alpha}}\overline{\phi}=\left\llbracket b_{r}>0\right\rrbracket \left(\left\llbracket q_{r}>1\right\rrbracket b_{r}\left(P_{q_{r}}-\phi(\frac{1}{q_{r}})\right)+\left\llbracket 2<q_{r}<q_{n}\right\rrbracket O_{r}\left(\phi_{r0,q_{r}-q_{r-1}}-\phi\left(\frac{2}{q_{r}}\right)\right)+\sum_{s=0}^{b_{r}-1}\left(\phi_{rsq_{r}}+\left\llbracket q_{r}>1\right\rrbracket \phi_{rsq_{r-1}}\right)\right)\label{eq:BaseInequalityA-1-2}
\end{flalign}
\end{lem}

\begin{lem}[Upper bounds for primitive functions]
Let $\phi\in\Phi$, the set of monotone descending positive functions
on $(0,1)$.. Then for $q_{r}>2$ we have: 
\end{lem}

\begin{align}
B_{r}^{\alpha}\phi & \le\left\llbracket b_{r}>0\right\rrbracket \left(b_{r}P_{q_{r}}+\left\llbracket 2<q_{r}<q_{n}\right\rrbracket E_{r}\left(\phi_{r0,q_{r}-q_{r-1}}^{\alpha}-\overline{\phi}\left(\frac{1}{q_{r}}\right)\right)+\sum_{s=0}^{b_{r}-1}\left(E_{r}\phi_{rsq_{r}}^{\alpha}+O_{r}\phi_{rsq_{r-1}}^{\alpha}\right)\right)\label{eq:NewB-2-1-1}\\
\overline{B_{\overline{r}}^{\overline{\alpha}}}\overline{\phi} & \le\left\llbracket b_{r}>0\right\rrbracket \left(b_{r}P_{q_{r}}+\left\llbracket 2<q_{r}<q_{n}\right\rrbracket O_{r}\left(\overline{\phi}_{r0,q_{r}-q_{r-1}}^{\alpha}-\overline{\phi}\left(\frac{1}{q_{r}}\right)\right)+\sum_{s=0}^{b_{r}-1}\left(O_{r}\overline{\phi}_{rsq_{r}}^{\alpha}+E_{r}\overline{\phi}_{rsq_{r-1}}^{\alpha}\right)\right)\label{eq:NewALower-1-1-1}
\end{align}

\begin{defn}
We define $C_{rs}^{\alpha}\phi\coloneqq E_{r}\phi_{rsq_{r}}^{\alpha}+O_{r}\phi_{rsq_{r-1}}^{\alpha}$
and its quasiperiod dual $\overline{C}_{rs}^{\alpha}\overline{\phi}\coloneqq O_{r}\overline{\phi}_{rsq_{r}}^{\alpha}+E_{r}\overline{\phi}_{rsq_{r-1}}^{\alpha}$.
We let $C_{r}$ be the sum of the terms in the double sum $DS\phi=\sum_{r=0}^{n}\sum_{s=1}^{b_{r}-1}C_{rs}\phi$
with denominator $q_{r+1}^{\slash}$, so that $DS(\phi)=\sum_{r=0}^{n}C_{r}(\phi)$.
And $S_{N}^{\alpha}\phi\le\sum_{r}B_{r}^{\alpha}\phi\le\sum_{r}SS_{r}^{\alpha}\phi+C_{r}^{\alpha}\phi$
\end{defn}

\begin{flalign}
C_{n}(\phi) & =\sum_{s=1}^{b_{n}}\phi^{E}\left(\frac{s}{q_{n+1}^{\slash}}\right)+\sum_{s=0}^{b_{n}-1}\phi^{O}\left(\frac{a_{n+1}^{\slash}-s}{q_{n+1}^{\slash}}\right)\,\,+\left\llbracket b_{n-1}>0\right\rrbracket O_{r}\phi\left(\frac{1}{q_{n+1}^{\slash}}\right) & r=n\label{eq:SB3-2}\\
C_{r}(\phi) & \le\sum_{s=1}^{b_{r}-1}\phi^{E}\left(\frac{1}{q_{r+2}^{\slash}}+\frac{s}{q_{r+1}^{\slash}}\right)+\sum_{s=0}^{b_{r}-1}\phi^{O}\left(\frac{a_{r+1}-s}{q_{r+1}^{\slash}}\right)\,\,+\left\llbracket b_{r-1}>0\right\rrbracket O_{r}\phi\left(\frac{1}{q_{r+1}^{\slash}}\right) & r<n\nonumber 
\end{flalign}

Writing $c_{r}^{\alpha}=c(N,\alpha,r)=E_{r}\left(b_{r}-1+\left\llbracket r=n\right\rrbracket \right)+O_{r}\min(b_{r}+1,a_{r+1})$
this gives

\begin{align}
C_{r}(\phi) & \le\sum_{s=1}^{c_{r}^{\alpha}}\phi\left(\frac{s}{q_{r+1}^{\slash}}\right)\,\,-\left\llbracket b_{r-1}=0\right\rrbracket \phi^{O}\left(\frac{1}{q_{r+1}^{\slash}}\right)\label{eq:Sr-1-2}
\end{align}

\section{\label{sec:Application}Application to the function family $\theta^{\beta}(x)=x^{-\beta},\beta\ge1$}

\begin{lyxgreyedout}
Could also cover $0\le\beta<1$ but it's tedious - is it worth it??? %
\end{lyxgreyedout}

In the previous section we established inequalities applicable to
any suitable decreasing observable $\phi$. In this section we will
further develop the inequalities for the special family of functions
$\theta^{\beta}(x)=x^{-\beta}$ for $\text{\ensuremath{\beta\ge}1}.$
(For $\beta\le1$ $S_{N}\theta^{\beta}$ can be estimated using the
Denjoy-Koksma({*}{*}) result, combining it with the technique of truncated
forms of the observables for $0<\beta<1$).  

We recall from the previous section that the lower bound for $S_{N}\phi$
can be studied in a straightforward way, but the upper bound is more
complex and we will it in parts. In particular $S_{N}$ splits naturally
into a sum of 3 parts: $S_{N}\phi\le\left(S_{N}^{1}+S_{N}^{2}+S_{N}^{3}\right)\phi$
where

\begin{align}
S_{N}^{1}\phi & =\sum_{r=0}^{n}\left\llbracket b_{r}>0\right\rrbracket b_{r}P_{q_{r}}(\phi)\label{eq:S1}\\
S_{N}^{2}\phi & =\sum_{r=0}^{n}\left\llbracket b_{r}>0\right\rrbracket \left\llbracket 2<q_{r}<q_{n}\right\rrbracket E_{r}\left(\phi_{r0,q_{r}-q_{r-1}}-\overline{\phi}\left(\frac{1}{q_{r}}\right)\right)\label{eq:S2}\\
S_{N}^{3}\phi & =\sum_{r=0}^{n}\left\llbracket b_{r}>0\right\rrbracket \sum_{s=0}^{b_{r}-1}\left(E_{r}\phi_{rsq_{r}}+O_{r}\phi_{rsq_{r-1}}\right)=\sum_{r=0}^{n}C_{r}\phi\label{eq:S3}
\end{align}

Note that the first two sums are single sums whilst the third is a
double sum. Also $S_{N}^{1}\phi$ vanishes for $\phi$ anti-symmetric
({*}{*}). 

From this point we assume $\alpha$ is fixed and can be dropped from
the notation, and that the canonical Ostrowski representation of $N$
is $N=\sum_{r=0}^{n}b_{r}q_{r}$ . 

\subsection{\label{subsec:ancillary}Some ancillary lemmas }

The following function will play a large role in this section:
\begin{defn}[Generalised Harmonic Function]
\label{def:HarmonicFun} For $y>0$ we define $H_{k}^{\beta}(y)\coloneqq\sum_{s=0}^{k-1}(y+s)^{-\beta}$
(with $H_{k}^{\beta}(y)\coloneqq0$ for $k\le0$). We will also write
$H_{k}^{\beta}\coloneqq H_{k}^{\beta}(1)$.
\end{defn}

Note that this is a positive function decreasing in $y$ and increasing
in $k$. It follows that for $y\ge1$ that $H_{k}^{\beta}(y)\le H_{k}^{\beta}<H_{\infty}^{\beta}=\zeta(\beta)$.
So for $\beta>1$ $H_{k}^{\beta}(y)$ is $O(1)$, but the constant
becomes arbitrarily large as $\beta\downarrow1$. 

\begin{prop}
\label{prop:bnqn>N} $N\ge b_{n}q_{n}>\frac{b_{n}}{b_{n}+1}N\ge\frac{1}{2}N$
\end{prop}

\begin{proof}
By definition $N<\left(b_{n}+1\right)q_{n}$ and the left hand inequality
follows immediately. The right hand follows since $b_{n}\ge1$.
\end{proof}
This gives yet less sharp, but again more digestible alternatives.
\begin{lem}
Suppose $\beta\ge1$ then $\sum_{r=0}^{n-1}\left\llbracket b_{r}>0\right\rrbracket q_{r}^{\beta}\le\sum_{r=0}^{n-1}b_{r}q_{r}^{\beta}<q_{n}^{\beta}$
\end{lem}

\begin{proof}
The first inequality is trivial. For $\beta\ge1$, $\sum_{r=0}^{n-1}b_{r}q_{r}^{\beta}\le\left(\sum_{r=0}^{n-1}b_{r}q_{r}\right)^{\beta}$
and $\sum_{r=0}^{n-1}b_{r}q_{r}=N-b_{n}q_{n}<q_{n}$.
\end{proof}
\begin{cor}
\label{cor:sumqr}$\sum_{r=0}^{n}\left\llbracket b_{r}>0\right\rrbracket q_{r}^{\beta}\le\min\left(2q_{n}^{\beta},N^{\beta}\right)$
and $\sum_{r=0}^{n-1}\left\llbracket b_{r}>0\right\rrbracket q_{r}^{\beta}<\min\left(2q_{n-1}^{\beta},q_{n}^{\beta}\right)$
\end{cor}

\begin{cor}
For $\phi x=(1-x)^{-\beta},\beta\ge1$ From (\ref{cor:ParityDuality}),
for $r>0$ we have both $\phi_{r0,q_{r}-q_{r-1}}^{E}\le\overline{\phi}(\frac{1}{2q_{r}})=(2q_{r})^{\beta}$
and also $\phi_{r0,q_{r}-q_{r-1}}^{O}\le\overline{\phi}(\frac{1}{2q_{r}})$
and hence also $\sum_{r=0}^{n-1}\left\llbracket b_{r}>0\right\rrbracket \left(\phi_{r0,q_{r}-q_{r-1}}^{E}+\phi_{r0,q_{r}-q_{r-1}}^{O}\right)\le\sum_{r=0}^{n-1}\left\llbracket b_{r}>0\right\rrbracket \overline{\phi}(\frac{1}{2q_{r}})\le2^{\beta}q_{n}^{\beta}.$
By the lemma $\sum_{r=1}^{n-1}\phi_{r0,q_{r}-q_{r-1}}^{E}+\phi_{r0,q_{r}-q_{r-1}}^{O}\le2^{\beta}q_{n}^{\beta}$.
\end{cor}

Similarly $\sum_{r=1}^{n-1}\overline{\phi}^{E}\left(\frac{2}{q_{r}}\right)+\overline{\phi}^{O}\left(\frac{2}{q_{r}}\right)\le\left(\frac{1}{2}\right)^{\beta}q_{n}^{\beta}$
\begin{lem}[Generalised Harmonic Function]
\label{lem:H Estimates} The following results hold:
\end{lem}

$P_{q_{r}}\theta^{\beta}\le\sum_{t=1}^{q_{r}-1}\frac{q_{r}^{\beta}}{t^{\beta}}=q_{r}^{\beta}H_{q_{r}-1}^{\beta}$.

For $\beta=1,k\ge1$ we have $\log\left(1+\frac{k}{y}\right)<H_{k}(y)\le\frac{1}{y}+\log\left(1+\frac{k-1}{y}\right)$
 . For $y=1$ this gives $\log\left(1+k\right)<H_{k}\le1+\log k$

For $\beta>1,k\ge1$ we have $\frac{1}{\beta-1}\left(\frac{1}{y^{\beta-1}}-\frac{1}{(y+k)^{\beta-1}}\right)<H_{k}^{\beta}(y)\le\frac{1}{y^{\beta}}+\frac{1}{\beta-1}\left(\frac{1}{y^{\beta-1}}-\frac{1}{(y+k-1)^{\beta-1}}\right)$.

Finally a useful trivial estimate is for $n\ge0$ $H_{n}^{\beta}\le\left\llbracket n>0\right\rrbracket \left(1+\left(\frac{1}{2^{\beta}}\right)(n-1)\right)$
(equality for $n\in\{1,2\}$) which also gives $H_{n+1}\le1+\frac{1}{2}n$
and $H_{n-1}\le\frac{1}{2}n$

By definition $H_{n}^{\beta}=0$ for $n\le0$. 

For $i\ge j\ge0$, 
\begin{equation}
H_{i-j}^{\beta}(j+1)=H_{i}^{\beta}-H_{j}^{\beta}=\sum_{k=j+1}^{i}\frac{1}{t^{\beta}}\le\frac{1}{\left(j+1\right)^{\beta}}(i-j)\label{eq:Hi-j}
\end{equation}
 and for $i\ge1$, $H_{i}^{\beta}\le1+\sum_{t=2}^{i}\frac{1}{t^{\beta}}\le1+\frac{1}{2^{\beta}}(i-1)$.
In ptic for $\beta=1$ we get for $n\ge0$ 
\begin{equation}
H_{n}\le\left\llbracket n>0\right\rrbracket \frac{1}{2}\left(n+1)\right)\label{eq:HnHalf}
\end{equation}
 (equality for $n\le2$) which also gives $H_{n+1}\le1+\frac{1}{2}n$
and $H_{n-1}\le\frac{1}{2}n$

$Q_{r}\le E_{r}\left(\frac{1}{2}b_{r}+1\right)+O_{r}\left(\frac{1}{2}b_{r}\right)=E_{r}+\frac{1}{2}b_{r}$

\subsection{Lower bound for $S_{N}\theta^{\beta}$ using the generalised harmonic
function $H^{\beta}$}

Most interest in the mathematical community has been on upper bounds
for these sums, and we will also focus primarily on these upper bounds.
However first we take a little time here to derive some simple results
for a lower bound also. 

Since $\theta^{\beta}>0$ we have the simple and crude lower bound
$S_{N}\theta^{\beta}=\sum_{r=0}^{n}S_{b_{r}q_{r}}\theta^{\beta}\ge S_{b_{n}q_{n}}\theta^{\beta}$
(where $N=\sum_{r=0}^{n}b_{r}q_{r}$). Recall we are regarding $\alpha$
as fixed and drop it from the notation, so that from \eqref{eq:BaseInequalities},\eqref{eq:BoundsDefDual}
we get: 

\begin{multline*}
\sum_{s=0}^{b_{r}-1}S_{rs}\phi\ge\overline{B}_{r}\phi=\\
\left\llbracket b_{r}>0\right\rrbracket \left(\left\llbracket q_{r}>1\right\rrbracket b_{r}\left(P_{q_{r}}-\phi(\frac{1}{q_{r}})\right)+\left\llbracket 2<q_{r}<q_{n}\right\rrbracket O_{r}\left(\phi_{r0,q_{r}-q_{r-1}}-\phi\left(\frac{2}{q_{r}}\right)\right)+\sum_{s=0}^{b_{r}-1}\left(\phi_{rsq_{r}}+\left\llbracket q_{r}>1\right\rrbracket \phi_{rsq_{r-1}}\right)\right)
\end{multline*}

\begin{multline*}
\sum_{s=0}^{b_{r}-1}S_{rs}\overline{\phi}\ge B_{r}^{\alpha}\overline{\phi}=\\
\left\llbracket b_{r}>0\right\rrbracket \left(\left\llbracket q_{r}>1\right\rrbracket b_{r}\left(P_{q_{r}}-\phi(\frac{1}{q_{r}})\right)+\left\llbracket 2<q_{r}<q_{n}\right\rrbracket E_{r}\left(\overline{\phi}_{r0,q_{r}-q_{r-1}}-\phi\left(\frac{2}{q_{r}}\right)\right)+\sum_{s=0}^{b_{r}-1}\left(\overline{\phi}_{rsq_{r}}+\left\llbracket q_{r}>1\right\rrbracket \overline{\phi}_{rsq_{r-1}}\right)\right)
\end{multline*}

Hence $S_{N}\theta^{\beta}\ge S_{b_{n}q_{n}}\theta^{\beta}=\sum_{s=0}^{b_{n}-1}S_{ns}\theta^{\beta}\ge\overline{B}_{n}\theta^{\beta}$. 

Note for $q_{n}>1$, $\overline{B}_{n}\theta^{\beta}=b_{n}\left(P_{q_{n}}-\theta^{\beta}(\frac{1}{q_{n}})\right)+\sum_{s=0}^{b_{n}-1}\left(\theta_{nsq_{n}}^{\beta}+\theta_{nsq_{n-1}}^{\beta}\right)$.
And for any $1<q_{r}\le q_{n}$, $P_{q_{r}}-\theta^{\beta}(\frac{1}{q_{r}})=q_{r}^{\beta}H_{q_{r}-1}^{\beta}-q_{r}^{\beta}=q_{r}^{\beta}\left(H_{q_{r}-1}^{\beta}-1\right)$. 

For $n$ even $\sum_{s=0}^{b_{n}-1}\theta_{nsq_{n}}^{\beta}=q_{n+1}^{\slash\beta}H_{b_{n}}^{\beta}$
and $\sum_{s=0}^{b_{n}-1}\theta_{nsq_{n-1}}^{\beta}>\sum_{s=0}^{b_{n}-1}1=b_{n}$.
For $n$ odd $\sum_{s=0}^{b_{n}-1}\theta_{nsq_{n}}^{\beta}\ge b_{n}$
and $\theta_{nsq_{n-1}}^{\beta}>\left(\frac{q_{n+1}^{\slash}}{a_{n+1}^{\slash}-s}\right)^{\beta}>q_{n}^{\slash\beta}$
giving $\sum_{s=0}^{b_{n}-1}\theta_{nsq_{n-1}}^{\beta}>b_{n}q_{n}^{\slash\beta}$.
Hence for $q_{n}>1$

\begin{equation}
S_{b_{n}q_{n}}\theta^{\beta}>b_{n}q_{n}^{\beta}\left(H_{q_{n}-1}^{\beta}-1\right)+E_{n}q_{n+1}^{\slash\beta}H_{b_{n}}^{\beta}+O_{n}b_{n}q_{n}^{\slash\beta}+b_{n}\label{eq:EstLBSingle}
\end{equation}

Using QP duality this also gives us for $q_{n}>1$
\begin{equation}
S_{b_{n}q_{n}}\overline{\theta}^{\beta}>b_{n}q_{n}^{\beta}\left(H_{q_{n}-1}^{\beta}-1\right)+O_{n}q_{n+1}^{\slash\beta}H_{b_{n}}^{\beta}+E_{n}b_{n}q_{n}^{\slash\beta}+b_{n}
\end{equation}

\begin{equation}
S_{b_{n}q_{n}}\left(\theta^{\beta}+\overline{\theta}^{\beta}\right)>2b_{n}q_{n}^{\beta}\left(H_{q_{n}-1}^{\beta}-1\right)+q_{n+1}^{\slash\beta}H_{b_{n}}^{\beta}+b_{n}q_{n}^{\slash\beta}+2b_{n}\label{eq:EstLBSymm}
\end{equation}

We derived estimates for the $H$ terms in (\ref{lem:H Estimates}).

\subsection{Upper bound for $S_{N}\theta^{\beta}$ using the generalised harmonic
function $H^{\beta}$}

Recall we have decomposed the Birkhoff sum $S_{N}\phi$ into three
component sums $S_{N}^{1}\phi+S_{N}^{2}\phi+S_{N}^{3}\phi$,. We now
study each of these component sums for the function $\theta^{\beta}$. 

The Single Sum Components $S_{N}^{1}\theta^{\beta}$ and $S_{N}^{2}\theta^{\beta}$ 
From \eqref{eq:S1} we have $S_{N}^{1}(\theta^{\beta})=\sum_{r=0}^{n}b_{r}P_{q_{r}}(\theta^{\beta})\le\sum_{r=0}^{n}b_{r}q_{r}^{\beta}H_{q_{r}-1}^{\beta}$.
Since $P_{q_{r}}$ is symmetric, it follows that $S_{N}^{1}$ is also
symmetric, ie $S_{N}^{1}(\overline{\theta}^{\beta})=S_{N}^{1}(\theta^{\beta})$.
Also from \eqref{eq:S2} we have $S_{N}^{2}(\theta^{\beta})=\sum_{r=0}^{n}\left\llbracket b_{r}>0\right\rrbracket \left\llbracket 2<q_{r}<q_{n}\right\rrbracket E_{r}\left(\theta_{r0,q_{r}-q_{r-1}}^{\beta}-\overline{\theta}^{\beta}\left(\frac{1}{q_{r}}\right)\right)\le\sum_{r=0}^{n}\left\llbracket b_{r}>0\right\rrbracket \left\llbracket 2<q_{r}<q_{n}\right\rrbracket E_{r}\left(\theta^{\beta}\left(\frac{1}{2q_{r}}\right)-1\right)$

Noting $E_{r}\left(\theta^{\beta}\left(\frac{1}{2q_{r}}\right)-1\right)=E_{r}\left((2q_{r})^{\beta}-1\right)$,
and the invoking $QP$ duality we get:

\begin{align*}
\left(S_{N}^{1}+S_{N}^{2}\right)\theta^{\beta} & \le\sum_{r=0}^{n}\left\llbracket b_{r}>0\right\rrbracket \left(b_{r}q_{r}^{\beta}H_{q_{r}-1}^{\beta}+\left\llbracket 2<q_{r}<q_{n}\right\rrbracket E_{r}\left((2q_{r})^{\beta}-1\right)\right)\\
\left(S_{N}^{1}+S_{N}^{2}\right)\overline{\theta}^{\beta} & \le\sum_{r=0}^{n}\left\llbracket b_{r}>0\right\rrbracket \left(b_{r}q_{r}^{\beta}H_{q_{r}-1}^{\beta}+\left\llbracket 2<q_{r}<q_{n}\right\rrbracket O_{r}\left((2q_{r})^{\beta}-1\right)\right)
\end{align*}

The Double Sum Component $S_{N}^{3}\theta^{\beta}$
Initial estimates using the harmonic function 
From ({*}{*}) we have $S_{N}^{3}\theta^{\beta}=\sum_{r=0}^{n}C_{r}\theta^{\beta}$
and from ({*}{*}) we have 

\begin{flalign}
C_{n}(\phi) & =E_{n}\sum_{s=1}^{b_{n}}\phi\left(\frac{s}{q_{n+1}^{\slash}}\right)+O_{n}\left(\sum_{s=0}^{b_{n}-1}\phi\left(\frac{a_{n+1}^{\slash}-s}{q_{n+1}^{\slash}}\right)\,\,+\left\llbracket b_{n-1}>0\right\rrbracket \phi\left(\frac{1}{q_{n+1}^{\slash}}\right)\right) & r=n\label{eq:SB3-1}\\
C_{r}(\phi) & \le E_{r}\sum_{s=1}^{b_{r}-1}\phi\left(\frac{s+(a_{r+1}^{\slash}-a_{r+1})}{q_{r+1}^{\slash}}\right)+O_{r}\left(\sum_{s=0}^{b_{r}-1}\phi\left(\frac{a_{r+1}-s}{q_{r+1}^{\slash}}\right)\,\,+\left\llbracket b_{r-1}>0\right\rrbracket \phi\left(\frac{1}{q_{r+1}^{\slash}}\right)\right) & r<n\nonumber 
\end{flalign}

Since $\theta^{\beta}\left(\frac{s}{q_{r+1}^{\slash}}\right)=\frac{1}{s^{\beta}}q_{r+1}^{\slash\beta}$,
the substitution $\phi=\theta^{\beta}$ in the equations above gives
us $C_{r}\left(\theta^{\beta}\right)=Q_{r}q_{r+1}^{\slash\beta}$
for some coefficient $Q_{r}$. The sums on the right hand side above
can now be written in terms of the generalised harmonic function as: 

\begin{flalign}
C_{n}(\phi)=Q_{n}q_{n+1}^{\slash\beta}\le & \left(E_{n}H_{b_{n}}^{\beta}+O_{n}\left(H_{b_{n}}^{\beta}(a_{n+1}^{\slash}-b_{n}+1)+\left\llbracket b_{n-1}>0\right\rrbracket \right)\right)q_{n+1}^{\slash\beta}\label{eq:SB3-1-1}\\
\le & \left(E_{n}H_{b_{n}}^{\beta}+O_{n}\left(H_{b_{n}}^{\beta}(a_{n+1}-b_{n}+1)+\left\llbracket b_{n-1}>0\right\rrbracket \right)\right)q_{n+1}^{\slash\beta} & r=n\\
C_{r}(\phi)=Q_{r}q_{r+1}^{\slash\beta}\le & \left(E_{r}H_{b_{r}-1}^{\beta}(1+(a_{r+1}^{\slash}-a_{r+1}))+O_{r}\left(H_{b_{r}}^{\beta}(a_{r+1}-b_{r}+1)+\left\llbracket b_{r-1}>0\right\rrbracket \right)\right)q_{r+1}^{\slash\beta}\nonumber \\
\le & \left(E_{r}H_{b_{r}-1}^{\beta}+O_{r}\left(H_{b_{r}}^{\beta}(a_{r+1}-b_{r}+1)+\left\llbracket b_{r-1}>0\right\rrbracket \right)\right)q_{r+1}^{\slash\beta} & r<n
\end{flalign}

Some Refined Estimates\label{par:EstQr}

For $r$ odd we have $Q_{r}=H_{b_{r}}^{\beta}(a_{r+1}-b_{r}+1)+\left\llbracket b_{r-1}>0\right\rrbracket $
which we will now examine more closely. Recall if $b_{r}=a_{r+1}$
then $b_{r-1}=0$ and so then $Q_{r}=H_{a_{r+1}}^{\beta}$. If $b_{r}<a_{r+1}$
then $H_{b_{r}}^{\beta}(a_{r+1}-b_{r}+1)\le H_{b_{r}}^{\beta}(2)=H_{b_{r}+1}^{\beta}-1$
and so (for $r$ odd) $Q_{r}\le H_{b_{r}+1}^{\beta}-\left\llbracket b_{r-1}=0\right\rrbracket $.
Hence we can set 
\begin{equation}
c_{r}^{\alpha}=c(N,\alpha,r)=E_{r}\left(b_{r}-\left\llbracket b_{r}>0\right\rrbracket +\left\llbracket r=n\right\rrbracket \right)+O_{r}\min(b_{r}+\left\llbracket b_{r-1}>0\right\rrbracket ,a_{r+1})\le b_{r}+1\label{eq:estcr}
\end{equation}
 to get 
\begin{equation}
Q_{r}\le H_{c_{r}}^{\beta}-O_{r}\left\llbracket b_{r-1}=0\right\rrbracket \label{eq:EstQr-1}
\end{equation}

and we have $H_{c_{r}}^{\beta}\le\min\left\{ \zeta(\beta),\left\llbracket c_{r}>0\right\rrbracket +\log^{+}c_{r}\right\} \le\min\left\{ \zeta(\beta),1+\log a_{r+1}\right\} $
and also $H_{c_{r}}\le\left\llbracket c_{r}>0\right\rrbracket \frac{1}{2}(1+c_{r})$
- see lemma \ref{lem:H Estimates} and hence for $\beta=1$ and $r<n$
we have 
\begin{equation}
Q_{r}\le E_{r}\frac{1}{2}\left(1+(b_{r}-1)\right)+O_{r}\left(\frac{1}{2}b_{r}+\left\llbracket b_{r-1}>0\right\rrbracket \right)\le O_{r}+\frac{1}{2}b_{r}\label{eq:EstQr-2}
\end{equation}
.

Note that in both even and odd cases, the estimate for $Q_{r}$ increase
with $b_{r}$. However the even estimates begin at $H_{1}=1$ and
the increments decrease in size, whereas odd estimates begin lower
at $H_{1}^{\beta}(a_{r+1}^{\slash})=\frac{1}{a_{r+1}^{\slash\beta}}$
but the increments increase in size. The estimates converge as $b_{r}$
increases. Each $Q_{r}$ estimate is therefore maximal for $b_{r}=a_{r+1}$.
However the sum $\sum Q_{r}q_{r+1}^{\slash\beta}$ may not be maximal
with $b_{r}=a_{r+1}$ for each $r$ because each $b_{r}=a_{r+1}$
forces $b_{r-1}=0$ which reduces the size of $Q_{r-1}$. To maximise
the double sum component, we therefore also need to investigate the
case of $b_{r}=a_{r+1}-1$: \\
For $a_{r+1}>1$ we then have for $r$ odd $Q_{r}<H_{b_{r}}^{\beta}(2)+\left\llbracket b_{r-1}>0\right\rrbracket $
and if $b_{r-1}>0$ this gives $Q_{r}\le H_{a_{r+1}-1}^{\beta}$ (with
equality only for $a_{r+1}=1,b_{r-1}>0$). \\
In the case of $r$ even we get $Q_{r}\le H_{a_{r+1}-2}^{\beta}$.
\\
In all cases: 
\begin{equation}
Q_{r}\le H_{c_{r}}^{\beta}\le\min\{\zeta(\beta),1+\log^{+}c_{r}\}\le\min\{\zeta(\beta),1+\log a_{r+1}\}\label{eq:EstQrMain}
\end{equation}
Now $\zeta(\beta)<2$ for $\beta>\beta_{0}\approx1.73$ and $\log a_{r+1}>1$
for $a_{r+1}\ge3$ so for larger values of $\beta$, $Q_{r}<\zeta(\beta)$
will be the better estimate, but for values of $\beta$ close to $1$,
$Q_{r}<1+\log c_{r}$ will be the better estimate. Of course for $\beta=1$
$Q_{r}\le1+\log^{+}c_{r}$ (equality only for $c_{r}\le1$) will always
be the better estimate. The case of $a_{r+1}=1$ needs further consideration.

Hence the estimate is maximal for $N=\sum b_{r}q_{r}$ with $b_{r}=a_{r+1}-1$
(except $a_{r+1}=1$)

\subsection{Using Estimates for the harmonic function }

Here we use the ancillary estimates established in subsection \ref{subsec:ancillary}.
The Single Sum Components $S_{N}^{1}\theta^{\beta},S_{N}^{2}\theta^{\beta}$
We have $SS_{r}(\theta^{\beta})=S_{N}^{1}\theta^{\beta}+S_{N}^{2}\theta^{\beta}\le\left\llbracket b_{r}>0\right\rrbracket \left(b_{r}q_{r}^{\beta}H_{q_{r}-1}^{\beta}+\left\llbracket 2<q_{r}<q_{n}\right\rrbracket E_{r}\left((2q_{r})^{\beta}-1\right)\right)$
for $q_{r}\ge1$. 

Now for $\beta\ge1$, $H_{q_{r}-1}^{\beta}<\zeta(\beta)$ and so $\sum_{r=0}^{n}b_{r}q_{r}^{\beta}H_{q_{r}-1}^{\beta}<\sum_{r=0}^{n}b_{r}^{\beta}q_{r}^{\beta}\zeta(\beta)\le N^{\beta}\zeta(\beta)$
and also $\sum_{r=0}^{n}b_{r}q_{r}^{\beta}H_{q_{r}-1}^{\beta}\le\sum_{r=0}^{n}b_{r}q_{r}H_{q_{r}-1}\le N(1+\log q_{n})$,
giving us:

\begin{equation}
S_{N}^{1}\theta^{\beta}=S_{N}^{1}\overline{\theta}^{\beta}<N^{\beta}\min\left\{ \zeta(\beta),1+\log q_{n}\right\} \label{eq:EstS1}
\end{equation}

By \eqref{cor:sumqr} we have $\sum_{r=0}^{n}\left\llbracket b_{r}>0\right\rrbracket \left\llbracket 2<q_{r}<q_{n}\right\rrbracket E_{r}q_{r}^{\beta}<\sum_{r=0}^{n-1}\left\llbracket b_{r}>0\right\rrbracket E_{r}q_{r}^{\beta}\le O_{n}\min\left(\frac{1}{1-2^{-\beta}}q_{n-1}^{\beta},q_{n}^{\beta}\right)+E_{n}\min\left(\frac{1}{1-2^{-\beta}}q_{n-2}^{\beta},q_{n-1}^{\beta}\right)$

From subsection \ref{subsec:ancillary} $\sum_{r=0}^{k}\frac{1}{2^{r\beta}}<\frac{1}{1-\frac{1}{2^{\beta}}}=\frac{2^{\beta}}{2^{\beta}-1}=1+\frac{1}{2^{\beta}-1}$
and so 
\begin{equation}
S_{N}^{2}\theta^{\beta}\le\sum_{r=0}^{n}\left\llbracket b_{r}>0\right\rrbracket \left\llbracket 2<q_{r}<q_{n}\right\rrbracket E_{r}\left((2q_{r})^{\beta}-1\right)<2^{\beta}\left(O_{n}\min\left\{ \frac{q_{n-1}^{\beta}}{1-2^{-\beta}},q_{n}^{\beta}\right\} +E_{n}\min\left\{ \frac{q_{n-2}^{\beta}}{1-2^{-\beta}},q_{n-1}^{\beta}\right\} \right)\le2^{\beta}q_{n}^{\beta}\label{eq:EstS2}
\end{equation}

Note that since $q_{r}=a_{r}q_{r-1}+q_{r-2}$, if $a_{r}>1$ then
$q_{r}^{\beta}>(2q_{r-1})^{\beta}\ge\frac{q_{r-1}^{\beta}}{1-2^{-\beta}}$
for $\beta\ge1$. If $a_{r}=1$ the situation is more complex. However
in the case $\beta=1$, we do have $q_{r}<2q_{r-1}$ for $a_{r}=1$.
By duality, the same result holds for $\overline{\theta}^{\beta}$
but with $E_{n},O_{n}$ interchanged. This also gives us

\begin{equation}
S_{N}^{2}\left(\theta^{\beta}+\overline{\theta}^{\beta}\right)<2^{\beta}\left(\min\left\{ \frac{q_{n-1}^{\beta}}{1-2^{-\beta}},q_{n}^{\beta}\right\} +\min\left\{ \frac{q_{n-2}^{\beta}}{1-2^{-\beta}},q_{n-1}^{\beta}\right\} \right)\le2^{\beta}(q_{n}^{\beta}+q_{n-1}^{\beta})\label{eq:EstS2Symm}
\end{equation}

Double Sum Component $S_{N}^{3}\theta^{\beta}$ - initial Landau estimates
In this section we will introduce the 3 main approaches to estimating
the double sum, and use them to develop quick estimates using Landau
(big O) notation. We will use the approaches more carefully in the
following section to estimate explicit constants.

We now consider the double sum component of $S_{N}\theta^{\beta}$,
namely $S_{N}^{3}\theta^{\beta}=\sum_{r=0}^{n}C_{r}\left(\theta^{\beta}\right)=\sum_{r=0}^{n}Q_{r}q_{r+1}^{\slash\beta}$.
This is by far the most complex of the terms to estimate partly because
it is irregular and partly because the estimates are aesthetically
rather unsatisfying. Those who require beauty in their estimates should
look elsewhere! 

Another problem is that whereas $q_{r+1}^{\slash}=O(N)$ for $r<n$,
for $r=n$ the term $q_{n+1}^{\slash}$ is arbitrarily larger than
$N$. This means we cannot expect to develop a general estimate purely
in terms of $N$ unless we impose restrictions on the type of $\alpha$
(eg $\alpha$ of constant type gives $q_{n+1}^{\slash\beta}=O(N)$),
and this has often been an approach used. Otherwise the estimate must
be dependent on both $N$ and $\alpha$. We will work first with general
$\alpha$, but then also look at the effects of restricting $\alpha$. 

A number of approaches have been developed in the literature in studying
both the function $\theta$ and a number of related functions. These
approaches are presented very differently in different contexts, and
can be difficult to compare in terms of concepts, terminology and
notation. However the theoretical structure we have developed in this
paper suggests we can usefully group them into one or other of 2 broad
approaches. We will also consider a 3rd approach in this paper which
tends to give slightly better results. However each of the approaches
has strengths and weaknesses, and in fact for each approach we can
find combinations of $\alpha,N$ for which that approach produces
the best estimate. 

We start by noting that we can write $S_{N}^{3}\theta^{\beta}=\sum_{r=0}^{n}C_{r}\theta^{\beta}=\sum_{r=0}^{n}Q_{r}q_{r+1}^{\slash\beta}$where
$Q_{r}$ is a positive real number. This means we are seeking an estimate
of the sum of products $\sum_{r=0}^{n}Q_{r}q_{r+1}^{\slash\beta}$
with all terms positive. The most obvious technique to use is that
of partial summation, but unfortunately it does not seem to help us
greatly here. We must content ourselves therefore with less sophisticated
approaches which we will call estimates by extraction:
\begin{defn}
Given sequences of positive real numbers $a_{i},b_{i},c_{i}$ (with
$c_{i}\ne0$), and the sum of products $\sum_{i=0}^{n}a_{i}b_{i}$
then we define ``the estimate by extraction of $c_{i}$'' as $\left(\max_{i\le n}(\frac{b_{i}}{c_{i}})\right)\sum_{i=0}^{n}a_{i}c_{i}$.
In the special case $c_{i}=1$ the definition becomes $\left(\max_{i\le n}b_{i}\right)\sum_{i=0}^{n}a_{i}$
which we will call instead ``the estimate by extraction of $b_{i}$''.
\end{defn}

Note that $\sum_{i=0}^{n}a_{i}b_{i}=\sum_{i=0}^{n}a_{i}\frac{b_{i}}{c_{i}}c_{i}\le\left(\max_{i\le n}(\frac{b_{i}}{c_{i}})\right)\sum_{i=0}^{n}a_{i}c_{i}$,
so that the estimate is an upper bound for $\sum_{i=0}^{n}a_{i}b_{i}$. 

If we now return to our sum $\sum_{r=0}^{n}Q_{r}q_{r+1}^{\slash\beta}$
we can now see immediately that there are 3 natural ways of estimating
this sum by extraction, namely setting $c_{i}=1$ and extracting either
$Q_{r}$ or $q_{r+1}^{\slash\beta}$, or extracting another term $c_{i}\ne1$.
Most approaches in the literature extract the quasiperiod term $q_{r+1}^{\slash\beta}$
and the remaining approaches extract a ``type'' of $\alpha$ (recall
that a type gives a bound on the growth of the quasiperiods $q_{r}$).
From this analysis there is an obvious 3rd candidate, namely the extraction
of $Q_{r}$, the coefficient of $q_{r+1}^{\slash\beta}$. We shall
see that the latter approach tends to give best results, but we have
not found it in the literature. This seems surprising, and so it may
be that we have just missed it. 

We will label these approaches, namely of QP extraction, type extraction,
and coefficient extraction as $A,B,C$ respectively. We shall see
below that this also places them in a general sense in order of improving
accuracy. For convenience, given a sequence $(x_{i})$ of real numbers,
we also introduce the notation $M_{n}x_{i}\coloneqq\max\{x_{i}:i\le n\}$. 

 We now make some high level remarks about each of the 3 approaches
$A,B,C$. 

Approach A (Quasiperiod Extraction): $S_{N}^{3}\theta^{\beta}=\sum_{r=0}^{n}Q_{r}q_{r+1}^{\slash\beta}\le q_{n+1}^{\slash\beta}\sum_{r=0}^{n}Q_{r}$
 This approach extracts the largest term, namely the quasiperiod
term ($q_{n+1}^{\slash\beta}$), and will not usually give the best
results. However it has been the most common approach used historically,
and for $\beta=1$ it allows for the use of the elegant inequality
$\sum_{r=0}^{n}\log a_{r+1}=\log\prod_{r=0}^{n}a_{r+1}\le\log q_{n+1}$
(equality only for $n=0$). This inequality helps aesthetically, but
unfortunately does not improve accuracy.
Order of Magnitude of $S_{N}^{3A}\theta^{\beta}\coloneqq q_{n+1}^{\slash\beta}\sum_{r=0}^{n}Q_{r}$
We have $\sum_{r=0}^{n}Q_{r}\le\min\left\{ (n+1)\zeta(\beta),(n+1)+\sum_{r=0}^{n}\log c_{r}\right\} $.
Now for $q_{n}>1$, $n=O(\log q_{n})$ and $\sum_{r=0}^{n}\log c_{r}\le\sum_{r=0}^{n}\log a_{r+1}\le\log q_{n+1}$.
Hence, using $q_{n+1}^{\slash}=O(q_{n+1})$ we obtain for $\beta>1$,
$S_{N}^{3a}\theta^{\beta}=O(q_{n+1}^{\slash\beta}\log q_{n})$ whereas
for $\beta=1$ $S_{N}^{3A}\theta^{\beta}=O(q_{n+1}^{\slash}\log q_{n+1})$

Approach B (Type Extraction): $\sum_{r=0}^{n}Q_{r}q_{r+1}^{\slash\beta}\le\left(\max_{r\le n}\frac{q_{r+1}^{\slash\beta}}{q_{r}^{\beta}}\right)\sum_{r=0}^{n}Q_{r}q_{r}^{\beta}$
. This approach extracts the upper type function $A_{n+1}^{\slash\beta}=\max_{r\le n}q_{r+1}^{\slash\beta}/q_{r}^{\beta}$,
a term which is much smaller term than $q_{n+1}^{\slash\beta}$ in
approach A. It is not as small as the term extracted in approach C,
but it has the benefit of naturally producing results involving $N$
rather than $q_{n}$. This approach was introduced by Lang, but we
will apply it in a different way from Lang to produce an improved
result (see \#\#paper for comparison).
Order of Magnitude of $S_{N}^{3B}\theta^{\beta}\coloneqq A_{n+1}^{\slash\beta}\sum_{r=0}^{n}Q_{r}q_{r}^{\beta}$
Note $A_{n+1}^{\slash}=O(A_{n+1})$ and for $\beta\ge1$ $\sum_{r=0}^{n}Q_{r}q_{r}^{\beta}=O\left(\sum_{r=0}^{n}(1+b_{r})q_{r}^{\beta}\right)=O(N^{\beta})=O(q_{n+1}^{\beta})$.
Hence $S_{N}^{3b}\theta^{\beta}=O(A_{n+1}^{\beta}q_{n+1}^{\beta})$
($\beta\ge1$).
Approach C (Coefficient Extraction): $\sum_{r=0}^{n}Q_{r}q_{r+1}^{\slash\beta}\le\left(\max_{r\le n}Q_{r}\right)\sum_{r=0}^{n}q_{r+1}^{\slash\beta}$
 This approach extracts the smallest term, namely the maximum coefficient
($\max_{r\le n}Q_{r}$) and will generally produce the most accurate
results. Alas it is also produces the least elegant results.
Order of Magnitude of $S_{N}^{3C}\theta^{\beta}\coloneqq\left(\max_{r\le n}Q_{r}\right)\sum_{r=0}^{n}q_{r+1}^{\slash\beta}$
Note $\sum_{r=0}^{n}q_{r+1}^{\slash\beta}=O(q_{n+1}^{\beta})$ and
for $\beta>1$ $\max_{r\le n}Q_{r}=O(1)$, for $\beta=1$ $Q_{r}\le1+\log^{+}c_{r}\le1+\log A_{n+1}\}$.
Hence $S_{N}^{3C}\theta^{\beta}=O(q_{n+1}^{\beta})$ ($\beta>1$)
and $S_{N}^{3C}\theta=O(q_{n+1}\log A_{n+1})$ ($\beta=1).$
Summary
We summarise the order of magnitude estimates for each approach in
the following table:
\begin{center}
Order Estimates for $S_{N}^{3}\theta^{\beta}$\\
\begin{tabular}{|c|c|c|}
\hline 
 & $\beta=1$ & $\beta>1$\tabularnewline
\hline 
\hline 
Method A & $q_{n+1}\log q_{n+1}$ & $q_{n+1}^{\beta}\log q_{n}$\tabularnewline
\hline 
Method B & $q_{n+1}A_{n+1}$ & $q_{n+1}^{\beta}A_{n+1}^{\beta}$\tabularnewline
\hline 
Method C & $q_{n+1}\log A_{n+1}$ & $q_{n+1}^{\beta}$\tabularnewline
\hline 
\end{tabular}
\par\end{center}

Refined Double Sum Estimates 
We now introduce optimisations of each approach in the previous section,
which improve on the results of the basic approach. We first introduce
two types of optimisation which are applicable to more than one approach.

Head and Tail Optimisation: For each $r<n$ we have $q_{r+1}^{\slash\beta}=O(N^{\beta})$,
but for $r=n$ the term $q_{n+1}^{\slash\beta}$ may be arbitrarily
large in comparison with $N$. We will therefore usually split the
sum $\sum_{r=0}^{n}Q_{r}q_{r+1}^{\slash\beta}$ into the sum of its
head term $H=Q_{n}q_{n+1}^{\slash\beta}$ and the tail sum $T=\sum_{r=0}^{n-1}Q_{r}q_{r+1}^{\slash\beta}$. 

Extraction of constants: If $a_{r}=c+d_{r}$ where $c$ is a constant
and $d_{r}$ increasing, then $\sum_{n}a_{r}b_{r}\le c\sum b_{r}+\left(\max b_{r}\right)\sum d_{r}$
is a better estimate than the simple extraction $\sum_{n}a_{r}b_{r}\le\max b_{r}\sum a_{r}$.

We are now ready to examine each approach in detail.

The estimates in \eqref{eq:SB3-1-1} contain enough information to
allow us to exploit specific patterns in the distribution of the coefficients
$b_{r}$ (eg if we are given $b_{r}=0$ for $r$ even we can exploit
this). Here we shall only be concerned with generic estimates, and
we will discard some of the information to simplify our results. 

We will use the generic estimate $Q_{r}\le\min\{\zeta(\beta),1+\log^{+}c_{r}\}$
for simpler results.

Also for $r<n$ we have $Q_{r}\le E_{r}\frac{1}{2}\left(1+(b_{r}-1)\right)+O_{r}\left(\frac{1}{2}(b_{r})+1\right)=O_{r}+\frac{1}{2}b_{r}$

Approach A: Quasiperiod Extraction
We use the approach to estimate the tail sum $T=\sum_{r=0}^{n-1}Q_{r}q_{r+1}^{\slash\beta}.$
Note that $\max_{r\le n-1}\left(q_{r+1}^{\slash\beta}\right)=q_{n}^{\slash\beta}$. 

The estimate $Q_{r}\le\zeta(\beta)$ gives us $T\le n\zeta(\beta)q_{n}^{\slash\beta}<\zeta(\beta)q_{n}^{\slash\beta}\left(\frac{\log q_{n}}{\log\phi}+1\right)$. 

The estimate $Q_{r}\le1+\log^{+}c_{r}$ gives, by separation of constants,
$T\le\sum_{r=0}^{n-1}q_{r+1}^{\slash\beta}+q_{n}^{\slash\beta}\sum_{r=0}^{n-1}\log^{+}c_{r}$.
Now $c_{r}\le a_{r+1}$ and $q_{n}\ge\Pi_{r=1}^{n}a_{r}$ (equality
only for $r\le1$) and hence $\sum_{r=0}^{n-1}\log^{+}c_{r}\le\log q_{n}$. 

Hence $T\le\min\left\{ n\zeta(\beta)q_{n}^{\slash\beta},\,\sum_{r=0}^{n-1}q_{r+1}^{\slash\beta}+q_{n}^{\slash\beta}\log q_{n}\right\} $.
And  $\sum_{r=0}^{n-1}q_{r+1}^{\slash\beta}\le\frac{1}{1-2^{-\beta}}\left(q_{n}^{\slash\beta}+q_{n-1}^{\slash\beta}\right)$
which is $O(q_{n}^{\slash\beta})$.   

It follows that the two estimates are asymptotically $n\zeta(\beta)q_{n}^{\slash\beta}$
and $q_{n}^{\slash\beta}\log q_{n}$, and the question of which is
the lower estimate depends upon the relative values of $\zeta(\beta)$
and $\log q_{n}^{1/n}$.

Approach B: Type extraction
Recall that the upper type function is defined as $A_{n+1}^{\slash}=\max_{r\le n}\frac{q_{r+1}^{\slash}}{q_{r}}=\max_{r\le n}\left(a_{r+1}^{\slash}+\frac{q_{r-1}}{q_{r}}\right)$
and that $a_{r+1}^{\slash}+\frac{q_{r-1}}{q_{r}}<a_{r+1}^{\slash}+\frac{1}{a_{r}}<a_{r+1}+2$
giving $\sum_{r=0}^{n}Q_{r}q_{r+1}^{\slash\beta}\le A_{n+1}^{\slash\beta}\sum_{r=0}^{n}Q_{r}q_{r}^{\beta}$. 

In this approach results naturally involve $N$ and it becomes counter-productive
to split the head and tail. We will use the alternative estimates
for $Q_{r}$from ({*}{*})

Estimate:

ODD: For $r\le n$ $O_{r}Q_{r}\le O_{r}\left(\left\llbracket b_{r-1}>0\right\rrbracket +\frac{1}{2}b_{r}+\frac{1}{2}\left\llbracket b_{r}=a_{r+1}\right\rrbracket \right)\le O_{r}\left(1+\frac{1}{2}b_{r})\right)$

EVEN: For $r=n$ (in which case $b_{r}\ge1$) $Q_{r}^{E}\le E_{r}\left(1+\frac{1}{2}(b_{r}-1)\right)=E_{r}\frac{1}{2}\left(1+b_{r})\right)$ 

and for $r<n$ $Q_{r}^{E}\le E_{r}\left\llbracket b_{r}>1\right\rrbracket \left(1+\frac{1}{2}(b_{r}-2)\right)=E_{r}\left\llbracket b_{r}>1\right\rrbracket \frac{1}{2}b_{r}$.
Combine to $E_{r}Q_{r}\le E_{r}\left\llbracket b_{r}>1\right\rrbracket \frac{1}{2}b_{r}+\left\llbracket r=n\right\rrbracket \frac{1}{2}\left(1+\left\llbracket b_{r}=1\right\rrbracket \right)$.
Or for $r<n$ $Q_{r}\le O_{r}+\frac{1}{2}b_{r}$

So for $n$ even $\sum_{r=0}^{n}Q_{r}q_{r}^{\beta}\le\frac{1}{2}\left(1+b_{r})\right)q_{n}^{\beta}+\sum_{r=0}^{n-1}\left(O_{r}+\frac{1}{2}b_{r}\right)q_{r}^{\beta}=\frac{1}{2}\sum_{r=0}^{n}b_{r}q_{r}^{\beta}+\sum_{r=0}^{n-1}O_{r}q_{r}^{\beta}\le\frac{1}{2}N^{\beta}+\frac{1}{2}q_{n}^{\beta}+\frac{1}{1-2^{-\beta}}q_{n-1}^{\beta}$
and for $n$ odd $\frac{1}{2}N^{\beta}+\frac{1}{1-2^{-\beta}}q_{n}^{\beta}$

Also $Q_{r}^{E}\le E_{r}\left(b_{r}-\left\llbracket r<n,b_{r}>0\right\rrbracket \right)$
and $Q_{r}^{O}\le O_{r}\left(b_{r}+\left\llbracket b_{r-1}>0\right\rrbracket \right)$
hence $Q_{r}\le b_{r}+O_{r}\left\llbracket b_{r-1}>0\right\rrbracket -E_{r}\left\llbracket r<n,b_{r}>0\right\rrbracket $
which gives us for $r$ odd $Q_{r}q_{r}^{\beta}+Q_{r-1}q_{r-1}^{\beta}<b_{r}q_{r}^{\beta}+\left\llbracket b_{r-1}>0\right\rrbracket \left(q_{r}^{\beta}-q_{r-1}^{\beta}\right)$.

Hence $\sum_{t=0}^{2k+1}Q_{t}q_{t}^{\beta}=\sum_{r=0}^{k}\left(Q_{2r+1}q_{2r+1}^{\beta}+Q_{2r}q_{2r}^{\beta}\right)\le\sum_{r=0}^{2k+1}b_{r}q_{r}^{\beta}+\sum_{r=0}^{k}\left\llbracket b_{2r}>0\right\rrbracket \left(q_{2r+1}^{\beta}-q_{2r}^{\beta}\right)$.
Let $q_{L}$ be the largest quasiperiod in the second sum (or $0$
if the sum is empty), so that we can telescope the second sum to obtain
$\sum_{r=0}^{k}\left\llbracket b_{2r}>0\right\rrbracket \left(q_{2r+1}^{\beta}-q_{2r}^{\beta}\right)\le q_{L}^{\beta}$,
and so $\sum_{t=0}^{2k+1}Q_{t}q_{t}^{\beta}\le\sum_{r=0}^{2k+1}b_{r}q_{r}^{\beta}+q_{L}^{\beta}$.
Formally 
\begin{equation}
L=\max_{-1\le r\le n}\left\{ r\textrm{ odd}:(r=-1)|(b_{r-1}>0)\right\} \label{eq:defL}
\end{equation}
 (recalling $q_{-1}=0$). If $n$ is odd take $k=(n-1)/2$ and if
$n$ is even take $k=(n/2)-1$ and note $Q_{n}\le b_{n}$ giving in
both cases $\sum_{r=0}^{n}Q_{r}q_{r}^{\beta}\le\sum_{r=0}^{n}b_{r}q_{r}^{\beta}+q_{L}^{\beta}\le N^{\beta}+q_{L}^{\beta}$
from ({*}{*}). Note $q_{L}\le q_{n}$ for $n$ odd and $q_{L}\le q_{n-1}$
for $n$ even.

Approach C: Coefficient Extraction
In this case $T<\left(\max_{r\le n-1}Q_{r}\right)\sum_{r=0}^{n-1}q_{r+1}^{\slash\beta}$.
Now $\max_{r\le n-1}Q_{r}\le\max_{r\le n-1}H_{c_{r}}^{\beta}\le\min\left(\zeta(\beta),1+\log\left(c_{n-1}^{max}\right)\right)$
and $c_{n-1}^{max}=\max_{r\le n-1}c_{r}\le\max_{r\le n-1}a_{r+1}=a_{n}^{\max}$.

Hence $T\le\min\left(\zeta(\beta),1+\log c_{n}^{max}\right)\,\sum_{r=0}^{n-1}q_{r+1}^{\slash\beta}$.
Unless $a_{r}=1$ for $r\le n$, we have $A_{n}\ge2$ which gives
$1+\log A_{n}\ge1.69$ whilst $\zeta(2)=\pi^{2}/6<1.65$ so that the
$\zeta(\beta)$ form is better for $\beta\ge2$ once $a_{n}\ge2$.

Compared with Approach B where we extract $A_{n}^{\slash\beta}$,
we are here extracting $1+\log c_{n-1}^{max}<1+\log A_{n}$ . The
difference is small for small $A_{n}$ but can be significant with
well approximable $\alpha$. On the other hand we are left with the
sum $\sum_{r=0}^{n-1}q_{r+1}^{\slash\beta}$ which proves not to lend
itself well to aesthetically pleasing estimates.

\subsection{\label{subsec:SummaryThetaBeta}Summary of upper bound results for
$S_{N}\theta^{\beta}$}

We have $S_{N}\theta^{\beta}\le\left(S_{N}^{1}+S_{N}^{2}+S_{N}^{3}\right)\theta^{\beta}$
where:

\[
S_{N}^{1}\theta^{\beta}<N^{\beta}\min\left\{ \zeta(\beta),1+\log q_{n}\right\} 
\]

\[
S_{N}^{2}\theta^{\beta}<2^{\beta}\left(O_{n}\min\left\{ \frac{2^{\beta}}{2^{\beta}-1}q_{n-1}^{\beta},q_{n}^{\beta}\right\} +E_{n}\min\left\{ \frac{2^{\beta}}{2^{\beta}-1}q_{n-2}^{\beta},q_{n-1}^{\beta}\right\} \right)
\]

Using 3 methods, we have derived several estimates for $S_{N}^{3}\theta^{\beta}$
each of which may be best in particular circumstance. However Method
C will tend to give the best results.

$S_{N}^{3}\theta^{\beta}=\sum_{r=0}^{n}Q_{r}q_{r+1}^{\slash\beta}=Q_{n}q_{n+1}^{\slash\beta}+\sum_{r=0}^{n-1}Q_{r}q_{r+1}^{\slash\beta}$.
The high order term $Q_{n}q_{n+1}^{\slash\beta}$ is split out in
Methods A and C and we have 
\begin{multline*}
Q_{n}=\left(\,E_{n}H_{b_{n}}^{\beta}+O_{n}\left(\left\llbracket b_{n-1}>0\right\rrbracket +H_{a_{n+1}^{\slash}}^{\beta}-H_{a_{n+1}^{\slash}-b_{n}}^{\beta}\right)\,\right)\\
\le\min\left\{ \zeta(\beta),\,E_{n}\left(1+\log b_{n}\right)+O_{n}\left(\left\llbracket b_{n-1}>0\right\rrbracket +\left\llbracket b_{n}<a_{n+1}\right\rrbracket \log\frac{a_{n+1}}{a_{n+1}-b_{n}}+\left\llbracket b_{n}=a_{n+1}\right\rrbracket \left(1+\log a_{n+1}\right)\right)\right\} 
\end{multline*}

\begin{align}
 & \mathrm{Method\,A} & S_{N}^{3A}\theta^{\beta} & \,\,\le Q_{n}q_{n+1}^{\slash\beta}+\min\begin{cases}
n\zeta(\beta)q_{n}^{\slash\beta}\\
\left(\sum_{r=0}^{n-1}q_{r+1}^{\slash\beta}\right)+q_{n}^{\slash\beta}\log q_{n}\\
q_{n}^{\slash\beta}\left(1+\left(1+\frac{1}{\log\phi}\right)\log q_{n}\right)
\end{cases}\label{eq:MethodA}\\
 & \mathrm{Method\,B} & S_{N}^{3B}\theta^{\beta} & \,\,\le A_{n+1}^{\slash\beta}\min\begin{cases}
\frac{1}{2}N^{\beta}+O_{n}\left(\frac{1}{1-2^{-\beta}}q_{n}^{\beta}\right)+E_{n}\left(\frac{1}{2}q_{n}^{\beta}+\frac{1}{1-2^{-\beta}}q_{n-1}^{\beta}\right)\\
N^{\beta}+q_{L}^{\beta}
\end{cases}\label{eq:MethodB}\\
 & \mathrm{Method\,C} & S_{N}^{3C}\theta^{\beta} & \,\,\le Q_{n}q_{n+1}^{\slash\beta}+\,\left(\sum_{r=0}^{n-1}q_{r+1}^{\slash\beta}\right)\min\begin{cases}
\zeta(\beta)\\
1+\log^{+}c_{n-1}^{max}
\end{cases}\label{eq:MethodC}
\end{align}

where:

and from \eqref{eq:estcr} we have $c_{n-1}^{\max}\le\max_{r\le n-1}\left(E_{r}\left(b_{r}-\left\llbracket b_{r}>0\right\rrbracket \right)+O_{r}\min(b_{r}+1,a_{r+1})\right)\le a_{n}^{\max}$.
Further note $a_{n}^{\max}<A_{n}=\max_{r\le n}q_{r}/q_{r-1}$ so $\log^{+}c_{n-1}^{max}\le\max_{r\le n}\left(\log q_{r}-\log q_{r-1}\right)$

and from \eqref{eq:defL} $L=\max_{-1\le r\le n}\left\{ r\textrm{ odd}:(r=-1)|(b_{r-1}>0)\right\} $
giving $q_{L}\le O_{n}q_{n}+E_{n}q_{n-1}\le q_{n}\le N$.

and from \eqref{eq:Estn} $n\le\frac{\log q_{n}}{\log\phi}+1$

and $\sum_{r=0}^{n}q_{r}^{\slash\beta}<\sum_{r=0}^{n}E_{r}q_{r}^{\slash\beta}+\sum_{r=0}^{n}O_{r}q_{r}^{\slash\beta}<\frac{1}{1-2^{-\beta}}\left(q_{n}^{\slash\beta}+q_{n-1}^{\slash\beta}\right)<\frac{2^{\beta}}{1-2^{-\beta}}\left(q_{n}^{\beta}+q_{n-1}^{\beta}\right)$.
For $\beta=1$ we can improve this slightly by using $q_{r}^{\slash}<q_{r}+q_{r-1}$
to get $\sum_{r=0}^{n}q_{r}^{\slash}<\sum_{r=0}^{n}E_{r}(q_{r}+q_{r-1})+\sum_{r=0}^{n}O_{r}(q_{r}+q_{r-1})<2(q_{n}+q_{n-1})+2(q_{n-1}+q_{n-2})<3q_{n}+4q_{n-1}$.
It should be noted that these are very much worst case estimates.

\subsection{\label{subsec:Duality-results}Duality results}

Finally the duality results (see \ref{cor:BSDualities}) gives us
that $S_{N}\overline{\theta}^{\beta}$ (for $n>2$ for $\alpha>\frac{1}{2}$
to ensure $a_{n+1}$ is self-conjugate) has the same upper bounds
except that instance of $E_{n},O_{n}$ are interchanged. 

We can of course combine appropriate estimates to obtain estimates
for $S_{N}(\theta^{\beta}+\overline{\theta}^{\beta})$. 

Finally $\theta^{\beta}\ge1$ on $(0,1)$ so that $S_{N}(\theta^{\beta}-\overline{\theta}^{\beta})\le S_{N}\theta^{\beta}-N$.
However a better result can be obtained by taking ({*}{*})

\newpage{}

\section{Comparisons with Prior Art }

\subsection{Introduction}

As we discussed in the introduction to this paper, there have been
many studies of particular anergodic Birkhoff sums over the circle,
using techniques specific to the observable sum being studied. Our
aim in this section is to show that the generic technique developed
in this paper can now be applied relatively easily and quickly to
generate equivalent or better results for each of these particular
sums. The main challenge lies not in generating results, but in comparing
them with prior art because of the many different forms and notations
which have been used to date.

Typically these estimates are not aesthetically pleasing! Some of
the gruesome detail is removed by working in Landau notation (ie without
providing details of the constants involved). In some cases there
may be no easy way to estimate the constants in the first place. Whether
for these reasons of others, many studies have restricted themselves
to Landau type results. In most cases our methods lend themselves
readily to providing details of the constants, and where constants
have been calculated previously in the literature, we show that we
can provide in most cases improved constants. However where only Landau
results exist in the literature, we have contented ourselves with
also providing such results. The proofs and results are then less
clouded with distracting detail, but the proofs are often easily extended
to provide constants if so desired.

 We will typically find that we will need to sacrifice precision
somewhere in order to derive commensurable results. When this is necessary
we will carry out this sacrifice on our own results rather than those
in the literature. This means that our results will typically be somewhat
better in their original form rather than the form they are forced
into for the purpose of comparison.

We will continue the convention of the previous section in setting
$\theta x=\frac{1}{x}$, $\overline{\theta}x=\frac{1}{\left(1-x\right)}$
on $(0,1)$ and using $\alpha$ for an irrational rotation number.
Recall that we say $\phi,\psi$ are equivalent if $\phi-\psi$ is
bounded on $(0,1)$, and hence $S_{N}\phi\,-\,S_{N}\psi=O(N)$ (ie
their Birkhoff sums differ by at most a fixed multiple of $N$). 

In this section we will now review our estimates and prior art for
the following Birkhoff sums: 
\begin{enumerate}
\item The sum of reciprocals of fractional parts $\sum_{r=1}^{N}\frac{1}{\left\{ r\alpha\right\} }$
\item The sum of reciprocals of distances to the nearest integer $\sum_{r=1}^{N}\frac{1}{\left\Vert r\alpha\right\Vert }$
\item The sum of signed remainders $\sum_{r=1}^{N}\frac{1}{\left\{ \left\{ r\alpha\right\} \right\} }$ 
\item The sum of cotangents $\sum_{r=1}^{N}\cot\pi r\alpha$ 
\item The double exponential sum $\sum_{u=0}^{N-1}\sum_{v=0}^{N-1}e^{2\pi iuv\alpha}$ 
\end{enumerate}
In addition we will also cover a number of closely related sums which
seem of interest but for which we are not aware of previous work.

\subsection{Summary of our results for $\theta x=\frac{1}{x}$ }

We shall see that all the functions in the list above are equivalent
to a form of $S_{N}\theta$ (ie $S_{N}\theta^{\beta}$ with $\beta=1$)
in either its native $(\theta)$, symmetric $(\theta+\overline{\theta})$
or anti-symmetric $(\theta-\overline{\theta})$ form. We therefore
give here the summary table from Subsection \ref{subsec:SummaryThetaBeta}
of estimates for $\theta^{\beta}$, simplified to show only the case
$\beta=1$.

We have: 
\begin{equation}
S_{N}\theta\le\left(S_{N}^{1}+S_{N}^{2}+S_{N}^{3}\right)\theta\label{eq:SNSplitin3}
\end{equation}
\begin{equation}
S_{N}^{1}\theta<N\left(1+\log q_{n}\right)\label{eq:SN1}
\end{equation}
\begin{equation}
S_{N}^{2}\theta<2\left(O_{n}\min\left\{ 2q_{n-1},q_{n}\right\} +E_{n}\min\left\{ 2q_{n-2},q_{n-1}\right\} \right)\label{eq:SN2}
\end{equation}

We have derived several estimates for $S_{N}^{3}\theta$ using 3 methods
labelled $A,B,C$. Each of these may give the best results in particular
circumstances, although Method C will tend to give the best results
in on average.

We start from the identity $S_{N}^{3}\theta=\sum_{r=0}^{n}Q_{r}q_{r+1}^{\slash}=Q_{n}q_{n+1}^{\slash}+\sum_{r=0}^{n-1}Q_{r}q_{r+1}^{\slash}$
(where the right hand expression simply splits out the high order
term) where:
\begin{multline*}
Q_{n}=\left(\,E_{n}H_{b_{n}}+O_{n}\left(\left\llbracket b_{n-1}>0\right\rrbracket +H_{a_{n+1}^{\slash}}-H_{a_{n+1}^{\slash}-b_{n}}\right)\,\right)\le\\
E_{n}\left(1+\log b_{n}\right)+O_{n}\left(\left\llbracket b_{n-1}>0\right\rrbracket +\left\llbracket b_{n}<a_{n+1}\right\rrbracket \log\frac{a_{n+1}}{a_{n+1}-b_{n}}+\left\llbracket b_{n}=a_{n+1}\right\rrbracket \left(1+\log a_{n+1}\right)\right)
\end{multline*}

The high order term $Q_{n}q_{n+1}^{\slash}$ is split out in Methods
A and C, and we have the following upper bounds: 

\begin{align}
 & \mathrm{Method\,A} & S_{N}^{3A}\theta & \,\,\le Q_{n}q_{n+1}^{\slash}+\min\begin{cases}
\left(\sum_{r=0}^{n-1}q_{r+1}^{\slash}\right)+q_{n}^{\slash}\log q_{n}\\
q_{n}^{\slash}\left(1+\left(1+\frac{1}{\log\phi}\right)\log q_{n}\right)
\end{cases}\label{eq:MethodA-2}\\
 & \mathrm{Method\,B} & S_{N}^{3B}\theta & \,\,\le A_{n+1}^{\slash}\min\begin{cases}
\frac{1}{2}N+O_{n}\left(2q_{n}\right)+E_{n}\left(\frac{1}{2}q_{n}+2q_{n-1}\right)\\
N+q_{L}
\end{cases}\label{eq:MethodB-2}\\
 & \mathrm{Method\,C} & S_{N}^{3C}\theta & \,\,\le Q_{n}q_{n+1}^{\slash}+\,\left(\sum_{r=0}^{n-1}q_{r+1}^{\slash}\right)\left(1+\log^{+}c_{n-1}^{max}\right)\label{eq:MethodC-2}
\end{align}

where:

and from \eqref{eq:estcr} we have $c_{n-1}^{\max}\le\max_{r\le n-1}\left(E_{r}\left(b_{r}-\left\llbracket b_{r}>0\right\rrbracket \right)+O_{r}\min(b_{r}+1,a_{r+1})\right)\le a_{n}^{\max}$.
Further note $a_{n}^{\max}\le\max_{r\le n}q_{r}/q_{r-1}\ge$ so also
$\log^{+}c_{n-1}^{max}\le\max_{r\le n}\left(\log q_{r}-\log q_{r-1}\right)$

and from \eqref{eq:defL} $L=\max_{-1\le r\le n}\left\{ r\textrm{ odd}:(r=-1)|(b_{r-1}>0)\right\} $
giving $q_{L}\le O_{n}q_{n}+E_{n}q_{n-1}\le q_{n}\le N$.

and from \eqref{eq:Estn} $n\le\frac{\log q_{n}}{\log\phi}+1$, and
$\phi$ is the golden ratio.

and using $q_{r}^{\slash}<q_{r}+q_{r-1}$ we get $\sum_{r=0}^{n}q_{r}^{\slash}<\sum_{r=0}^{n}E_{r}(q_{r}+q_{r-1})+\sum_{r=0}^{n}O_{r}(q_{r}+q_{r-1})<2(q_{n}+q_{n-1})+2(q_{n-1}+q_{n-2})<3q_{n}+4q_{n-1}$.
It should be noted that these are very much worst case estimates.

\subsection{The sum of reciprocals of fractional parts $\sum_{r=1}^{N}\frac{1}{\{r\alpha\}}$
}

Analysis
This is the simplest case for comparison as we have the direct equality
$\sum_{r=1}^{N}\frac{1}{\{r\alpha\}}=S_{N}\theta$.
Prior art

Although it dates back to 1966, Lang's estimate of this sum \cite{Lang1995introduction}
still stands: 

\begin{equation}
\sum_{r=1}^{N}\frac{1}{\{r\alpha\}}<2N\log N+20Ng(N)+K_{0}\label{eq:Lang}
\end{equation}
Here $k_{0}\ge0$ is an arbitrary integer, $K_{0}=\sum_{r=1}^{k_{0}}\frac{1}{\{r\alpha\}}$
and $g(N)$ is a ``co-type'' of $\alpha$ for $N>k_{0}$, defined
as follows:
\begin{defn}
(Lang) The function $g(N)$ is a co-type of $\alpha$ for $N>k_{0}$
if $g$ is a monotonic increasing function, and for any $N>k_{0}$
there is always quasiperiod $q_{n}$ of $\alpha$ such that $N<q_{n}\le Ng(N)$

Remarks:

The minimal co-type for $k_{0}=0$ is in fact our type function $A(N)$
from \ref{def:Types}. To see this, let $g$ be a co-type for $N>0$.
If we choose $N=q_{n}$then we have $N<q_{n+1}=\frac{q_{n+1}}{q_{n}}q_{n}\le Ng(N)$
so $g(q_{n})\ge\frac{q_{n+1}}{q_{n}}$, Since $g$ is increasing so
for any $q_{n}\le N<q_{n+1},r\le n$ we have $g(N)\ge g(q_{n})\ge g(q_{r})\ge\frac{q_{r+1}}{q_{r}}$,
so that our type sequence $A_{n+1}=\max_{r\le n}\frac{q_{r+1}}{q_{r}}\le g(N)$.
\end{defn}

The constant terms $k_{0},K_{0}$ simply allow us to disregard large
initial outliers in the initial set $\{a_{r}\}_{r\le k_{0}}$(consider
for example the effects of setting $a_{1}=10^{100},a_{r}=1$ for $r>1$).
It is straightforward to introduce this refinement into our results
also, by defining and using the refined type function $A_{n+1}^{k_{0}}=\max_{n\ge r>k_{0}}\{q_{r+1}/q_{r}\}$).
However we will leave this as an exercise for the reader to keep the
core comparison clearer.  Effectively this means we will take $k_{0}=0$
so that $K_{0}=0$, and \eqref{eq:Lang} becomes for comparison purposes:

\[
\sum_{r=1}^{N}\frac{1}{\{r\alpha\}}<2N\log N+20NA_{n+1}
\]

Comparison
Lang's estimate is simple and attractive. It concentrates the effects
of the choice of $\alpha$ into the type function $g$. We will compare
it with our estimate via Method\textbf{ }B. Although we know Method
B is not generally as sharp as that of Method C, it provides here
for an easier comparison.

By the remarks above, we will focus on comparing the results of our
method with the estimate $2N\log N+20NA_{n+1}$.

Our results for comparison are (using Method B):
\begin{lem}
$\sum_{r=1}^{N}\frac{1}{\{r\alpha\}}=S_{N}\theta<N\log q_{n}+\left(N+q_{n}\right)A_{n+1}+(2N+3q_{n})$
\end{lem}

\begin{proof}
We have from \eqref{eq:SNSplitin3} $S_{N}\theta\le\left(S_{N}^{1}+S_{N}^{2}+S_{N}^{3B}\right)\theta=N\left(1+\log q_{n}\right)+2\left(E_{n}q_{n}+O_{n}q_{n-1}\right)+A_{n+1}^{\slash}\left(N+q_{L}\right)\le N\log q_{n}+N+2q_{n}+A_{n+1}^{\slash}\left(N+q_{n}\right)$.
Now from \ref{def:Types} $A_{n+1}^{\slash}<A_{n+1}+1$ which gives
the result.
\end{proof}
Note that we have taken worst assumptions to obtain this result, namely
$n$ even, $N\le2q_{n}$, $A_{n+1}^{\slash}=A_{n+1}+1$, but the result
gives better constants than $2N\log N+20NA_{n+1}$ for little effort.

\subsection{The sum of reciprocals of the distance to nearest integer function
$\sum_{r=1}^{N}\frac{1}{\left\Vert r\alpha\right\Vert }$ }

The Birkhoff sum $\sum_{r=1}^{N}\frac{1}{\left\Vert r\alpha\right\Vert }$
has been studied by many authors. The most recent major study is by
Beresnevich et al\cite{VelaniRecipFractional2020} who included explicit
(with constants) upper and lower bounds. 

In fact this is the only sum (of which we are aware) for which lower
bounds have been investigated, and so Beresnevich et al's results
are the only test case for our method on lower bound. 

We will treat upper and lower bounds separately.
Analysis
First we cast the problem appropriately as a suitable anergodic sum.
\begin{prop}
\label{prop:SumNearest}$\sum_{r=1}^{N}\frac{1}{\left\Vert r\alpha\right\Vert }=S_{N}\left(\theta+\overline{\theta}-\psi\right)$
where $\theta x=\frac{1}{x}$ and $S_{N}\psi\in2N\log2\pm2d_{\alpha}(N)$
\end{prop}

\begin{proof}
Recall that we define the nearest integer function on real numbers
by $\left\Vert x\right\Vert =\min_{n\in\mathbb{Z}}\left|x-n\right|=\left\llbracket 0\le\{x\}<\frac{1}{2}\right\rrbracket {x}+\left\llbracket \frac{1}{2}\le\{x\}<1\right\rrbracket {1-x}$.
Since $\left\Vert x\right\Vert =\left\Vert \{x\}\right\Vert $ we
may assume $0<x<1$so that $\frac{1}{\left\Vert x\right\Vert }=\frac{\left\llbracket 0\le x<\frac{1}{2}\right\rrbracket }{x}+\frac{\left\llbracket \frac{1}{2}\le x<1\right\rrbracket }{1-x}=\left(\frac{1}{x}+\frac{1}{1-x}\right)-\left(\frac{\left\llbracket 0\le x<\frac{1}{2}\right\rrbracket }{1-x}+\frac{\left\llbracket \frac{1}{2}\le x<1\right\rrbracket }{x}\right)$.
We put $\psi x=\frac{\left\llbracket 0\le x<\frac{1}{2}\right\rrbracket }{1-x}+\frac{\left\llbracket \frac{1}{2}\le x<1\right\rrbracket }{x}$
so that $\sum_{r=1}^{N}\frac{1}{\left\Vert r\alpha\right\Vert }=S_{N}\left(\theta+\overline{\theta}-\psi\right)$.
Now $\psi$ has the integral $2\log2$ and variation $2$ so can immediately
be estimated by Denjoy-Koksma({*}{*}) as $S_{N}\psi\in2N\log2\pm2d_{\alpha}(N)$,
where $d_{\alpha}(N)=\sum_{0}^{n}b_{r}$, the sum of digits in the
Ostrowski representation of $N$ with $1\le d_{\alpha}(N)<q_{n}^{1/n}\left(\frac{\log q_{n}}{\log\phi}+1\right)$(Equality
when $N=q_{n}$).
\end{proof}
Prior Art - Upper Bound
In this section a factor of 2 appears in all terms of the sum - for
convenience therefore we will estimate the half sum $\frac{1}{2}\sum_{r=1}^{N}\frac{1}{\left\Vert r\alpha\right\Vert }$. 

\begin{equation}
\frac{1}{2}\sum_{r=1}^{N}\frac{1}{\left\Vert r\alpha\right\Vert }\le2q_{n+1}\left(1+\log\left(1+\frac{N}{q_{n}}\right)\right)+32N\log q_{n}+q_{3}N\label{eq:VelaniUpper}
\end{equation}

We note that this is $O\left(q_{n+1}\log\left(1+\frac{N}{q_{n}}\right)+N\log q_{n}\right)$
where either term can dominate since $\frac{q_{n+1}}{N}$ can be arbitrarily
large or $O(q_{n})$ depending upon $\alpha$. 
Results from our theory
We first derive an estimate using our own methods.
\begin{lem}
$\frac{1}{2}\sum_{r=1}^{N}\frac{1}{\left\Vert r\alpha\right\Vert }<(1+\log c_{n})q_{n+1}^{\slash}+N\log q_{n}+(1+\log c_{n-1}^{max})\sum_{r=1}^{n}q_{n}^{\slash}+\left(\left(1-\log2\right)N+q_{n}+q_{n-1}\right)+d_{\alpha}(N)$
\end{lem}

\begin{proof}
From proposition \ref{prop:SumNearest} we have $\sum_{r=1}^{N}\frac{1}{\left\Vert r\alpha\right\Vert }=S_{N}\left(\theta+\overline{\theta}-\psi\right)\le S_{N}\left(\theta+\overline{\theta}\right)-2N\log2+2d_{\alpha}(N)$.
Now \eqref{eq:MethodB-2} gives us an upper bound for $S_{N}\theta,$and
by duality (Subsection \ref{subsec:Duality-results}) $S_{N}\overline{\theta}$
has the same bound with $E_{n}/O_{n}$ interchanged. The full result
follows from combining these partial results
\end{proof}

Comparison
To compare our result with \eqref{eq:VelaniUpper} we start by noting
$c_{n}\le b_{n}+1=1+\left\lfloor \frac{N}{q_{n}}\right\rfloor $,
and $q_{n+1}^{\slash}<q_{n+1}+q_{n}$ so that $(1+\log c_{n})q_{n+1}^{\slash}<\left(q_{n+1}+q_{n}\right)\left(1+\log\left(1+\frac{N}{q_{n}}\right)\right)$.
Next by ({*}{*}) $\sum_{r=1}^{n}q_{n}^{\slash}<3q_{n}+4q_{n-1}$ and
$d_{\alpha}(N)<${*}{*}?. Putting these results together gives
\[
\frac{1}{2}\sum_{r=1}^{N}\frac{1}{\left\Vert r\alpha\right\Vert }<\left(q_{n+1}+q_{n}\right)\left(1+\log\left(1+\frac{N}{q_{n}}\right)\right)+N\log q_{n}+\left(3q_{n}+4q_{n-1}\right)\log c_{n-1}^{max}+\left(N\left(1-\log2\right)+4q_{n}+5q_{n-1}\right)+d_{\alpha}(N)
\]

The main obstacle to comparison with Beresnevich now lies with the
third term, and in particular with the component $c_{n-1}^{\max}$.

Note for ae $\alpha$ we have $c_{n}=O(\log q_{n})^{1+\epsilon}$
so that for ae $\alpha$ $\log c_{n-1}^{\max}=O(L^{2}q_{n-1})$ and
so this term is asymptotically insignificant compared with the first
two terms and we get 
\[
\frac{1}{2}\sum_{r=1}^{N}\frac{1}{\left\Vert r\alpha\right\Vert }\le(q_{n+1}+q_{n})\left(1+\log\left(1+\frac{N}{q_{n}}\right)\right)+N\log q_{n}+O(q_{n}L^{2}q_{n-1}+N)
\]
For all $\alpha$ we can make the (very) coarse estimate $c_{n-1}^{\max}\le a_{n}^{\max}<\max_{r\le n}\frac{q_{r}}{q_{r-1}}<q_{n}$
to give

\[
\frac{1}{2}\sum_{r=1}^{N}\frac{1}{\left\Vert r\alpha\right\Vert }<\left(q_{n+1}+q_{n}\right)\left(1+\log\left(1+\frac{N}{q_{n}}\right)\right)+\left(N+3q_{n}+4q_{n-1}\right)\log q_{n}+O(N)
\]

which shows the same asymptotic order as Beresnevich but with improved
bounds. Beresnevich point out that they have not sought best possible
results, but then neither have we - these results are simplifications
read off our more general results. 

Prior Art - Lower Bound
Beresnevich et al show (also using elementary methods) that $\sum_{r=1}^{N}\frac{1}{\left\Vert r\alpha\right\Vert }\ge LB2$
where

\begin{equation}
LB2=q_{n+1}\log\left(1+\left\lfloor \frac{N}{q_{n}}\right\rfloor \right)+\frac{1}{24}N\log q_{n}-(\frac{1}{3}\log q_{2}+\frac{1}{2})N\label{eq:VelaniSumNearest}
\end{equation}
\footnote{Strictly, the paper has $\frac{N}{q_{n}}$ rather than $\left\lfloor \frac{N}{q_{n}}\right\rfloor $
but this makes negligible difference and appears to be a typographical
error {*}{*} Check this wording with Sebastian}

They also give an alternative result derived using Minkowski's Convex
Body Theorem showing that for $N\ge2$ we have $\sum_{r=1}^{N}\frac{1}{\left\Vert r\alpha\right\Vert }\ge LB3$
where:
\begin{equation}
LB3=N\log N+N(1-\log2)+2\label{eq:Minkowski}
\end{equation}
The latter is an interesting result because it is derived in such
a different manner - it exploits the convexity of the particular function
$\frac{1}{\left\Vert x\right\Vert }$ but does not use any properties
of the rotation number $\alpha$. We therefore expect it to give less
good results when specific properties of the rotation number become
important, and indeed this is what we find. 
Results from our theory
Given the Ostrowski representation $N=\sum_{r=0}^{n}b_{r}q_{r}$
we will use the rather crude estimate from \eqref{eq:EstLBSymm} where
we exploit the fact that $\phi\ge0$ and so $S_{N}\phi\ge S_{b_{n}q_{n}}\phi$.
Clearly this lower bound of $S_{b_{n}q_{n}}\phi$ is at its best when
$N=b_{n}q_{n}$ and becomes less good as $N$ is increased towards
its maximum of $(b_{n}+1)q_{n}-1$ (or in the case of $b_{n}=a_{n+1}$
a maximum of $q_{n+1}-1$). The proportional error will be greatest
when $b_{n}=1$ and $N=2q_{n}-1$. More refined estimates can be made
by considering additional terms from the Ostrowski representation.
However as we shall see, this initial estimate already compares well
with other published results.
\begin{lem}
For $q_{n}>1$ we have the lower bound $\sum_{r=1}^{b_{n}q_{n}}\frac{1}{\left\Vert r\alpha\right\Vert }>LB1$
where 
\begin{equation}
LB1=q_{n+1}^{\slash}\log^{*}(1+b_{n})+b_{n}q_{n}\left(2\log^{*}q_{n}-(1+2\log2)\right)\label{eq:OurSumNearest-1}
\end{equation}

where $\log^{*}x=\max\{1,\log x\}$. (Note that $2\log^{*}q_{n}-(1+2\log2)>0$
for $q_{n}^{2}>4e$, and hence for $q_{n}>3$).

\end{lem}

\begin{proof}
From proposition \ref{prop:SumNearest} $\sum_{r=1}^{N}\frac{1}{\left\Vert r\alpha\right\Vert }=S_{N}\left(\theta+\overline{\theta}-\psi\right)\ge S_{b_{n}q_{n}}\left(\theta+\overline{\theta}-\psi\right)$
where $\psi(x)=\left\llbracket x<\frac{1}{2}\right\rrbracket \frac{1}{1-x}+\left\llbracket x\ge\frac{1}{2}\right\rrbracket \frac{1}{x}$.
Note that $\var\psi=2$ and so by Denjoy-Koksma({*}{*}) $S_{N}\psi<2N\log2+2d_{\alpha}(N)$.
Now $d_{\alpha}(b_{n}q_{n})=b_{n}$ so this gives $S_{b_{n}q_{n}}\psi<2b_{n}q_{n}\log2+2b_{n}$. 

From \eqref{eq:EstLBSymm} $S_{b_{n}q_{n}}\left(\theta+\overline{\theta}\right)>2b_{n}q_{n}\left(H_{q_{n}-1}-1\right)+q_{n+1}^{\slash}H_{b_{n}}+b_{n}q_{n}^{\slash}+2b_{n}>2b_{n}q_{n}\left(\log^{*}q_{n}-1\right)+q_{n+1}^{\slash}\log^{*}(1+b_{n})+b_{n}q_{n}+2b_{n}$
(since $q_{n}^{\slash}>q_{n}$). The result follows.
\end{proof}
Comparison of Lower Bound results
\textbf{$LB1$vs $LB2$:}

Let us compare the three terms of our result $LB1$ and Beresnevich's
result $LB2$. To compare first terms we note that $\left\lfloor \frac{N}{q_{n}}\right\rfloor =b_{n}$
and that $q_{n+1}^{\slash}-q_{n+1}=(a_{n+1}^{\slash}-a_{n+1})q_{n}=\frac{1}{a_{n+2}^{\slash}}q_{n}$.
To compare the second and third terms we use $N<(b_{n}+1)q_{n}$. 

This gives us: $LB1-LB2>\frac{1}{a_{n+2}^{\slash}}q_{n}\log^{*}(1+b_{n})+\left(2-\frac{1+\frac{1}{b_{n}}}{24}\right)b_{n}q_{n}\log^{*}q_{n}+b_{n}q_{n}\left((\frac{1}{3}\log q_{2}+\frac{1}{2})(1+\frac{1}{b_{n}})-1-2\log2\right)$.
This is positive for $q_{n}\ge q_{2}$ and dominated by the central
term which has a minimum of $\frac{23}{12}b_{n}q_{n}\log q_{n}$.

\textbf{$LB1$vs $LB3$:}

To compare $LB1$ and $LB3$ we define $b=\frac{N}{q_{n}}$ so that
$b_{n}\le b<b_{n}+1$ (also $b<a_{n+1}+\frac{q_{n-1}}{q_{n}}$). This
gives $LB3=bq_{n}(\log q_{n}+\log b)+bq_{n}(1-\log2)+2$ and hence
$LB1-LB3=\left(q_{n+1}^{\slash}\log^{*}(1+b_{n})-bq_{n}\log b\right)+q_{n}\left((2b_{n}\log^{*}q_{n}-b\log q_{n})-b_{n}(1+2\log2)-b(1-\log2)\right)-2$.

Now $bq_{n}=N<q_{n+1}^{\slash}$ and $b<1+b_{n}$ so that the first
term is always positive. We now investigate the coefficient of the
second ($q_{n}$) term which we designate $C=(2b_{n}\log^{*}q_{n}-b\log q_{n})-b_{n}(1+2\log2)-b(1-\log2)$.
Note that $Cq_{n}>2$ gives $LB1>LB3$.  
\begin{lem}
If $b_{n}>1$ then $LB1>LB3$ for $q_{n}\ge27$ 
\end{lem}

\begin{proof}
Using $b<b_{n}+1$ gives $C\ge(b_{n}-1)\log q_{n}-b_{n}(2+\log2)-(1-\log2)$.
Hence for $b_{n}>1$ $C$ is dominated by the term $(b_{n}-1)q_{n}\log q_{n}$.
Further $C\ge0$ for $\log q_{n}\ge\frac{1}{b_{n}-1}\left(b_{n}(2+\log2)-(1-\log2)\right)+\frac{2}{q_{n}}=2+\log2+\frac{1}{b_{n}-1}+\frac{2}{q_{n}}$
or $q_{n}\ge2e^{2.5+2/q_{n}}$ which holds for integers $q_{n}\ge27$.
\end{proof}

As we noted above, the proportional error as a result of using $S_{b_{n}q_{n}}$
as a lower bound is greatest when $b_{n}=1$ and $N$ is close to
its maximal value of $2q_{n}-1$. Indeed our next lemma guarantees
$LB1>LB3$ when $b_{n}=1$ for most values of $N$, but not for the
largest values. 
\begin{lem}
If $b_{n}=1$ then $LB1>LB3$ if $a_{n+1}^{\slash}\ge(3+2\log2)-\frac{q_{n-1}-2}{q_{n}}$
or if $N/q_{n}\le2-\frac{3}{\log q_{n}+1-\log2}$. If neither condition
holds, we may still have $LB1>LB3$ but it is no longer guaranteed. 
\end{lem}

\begin{proof}
For $b_{n}=1$ we have $C\ge-3$ 

$C=(2-b)\log q_{n}-b(1-\log2)-(1+2\log2)$ and $b<2$. Put $b=2-\lambda$
then $C=\lambda\log q_{n}+\lambda(1-\log2)-3$ so $C\ge0$ for $\lambda\ge\frac{3}{\log q_{n}+1-\log2}$.
Hence given any fixed $\lambda$ we will always have $C>0$ for large
enough $q_{n}$, but it is also true that $C<0$ if $\lambda$ is
small enough. Put $\lambda=2-\frac{N}{q_{n}}=\frac{k}{q_{n}}$ for
some $1\le k\le q_{n}$, so for large $q_{n}$ there will always be
values of $N\le2q_{n}-1$ which give $C<0$, namely $N=2q_{n}-k$
where $k<\frac{3q_{n}}{\log q_{n}+1-\log2}$ and $k\le q_{n}$. For
these, taking $b<2$ gives $LB1-LB3>\left(q_{n+1}^{\slash}-2q_{n}\log2\right)-3q_{n}-2$
which is positive if $q_{n+1}^{\slash}\ge(3+2\log2)q_{n}+2$. This
gives $a_{n+1}^{\slash}q_{n}+q_{n-1}\ge(3+2\log2)q_{n}+2$ or $a_{n+1}^{\slash}\ge(3+2\log2)-\frac{q_{n-1}-2}{q_{n}}$,
which is guaranteed true for $a_{n+1}>4$ and often true for $a_{n+1}=4$
(specifically when $\frac{1}{a_{n+2}^{\slash}}>2\log2-1-\frac{q_{n-1}-2}{q_{n}})$.
When it is false, we may still have $LB1>LB3$, but it is no longer
guaranteed.
\end{proof}
In summary $LB1>LB3$ if $b_{n}>1$ or if $b_{n}=1$ and $a_{n+1}^{\slash}\ge(3+2\log2)-\frac{q_{n-1}-2}{q_{n}}$
or $N/q_{n}\le2-\frac{3}{\log q_{n}+1-\log2}$. It is possible that
given $\alpha$ and taking into account more terms of the Ostrowski
expansion we could obtain give a refined lower bound always guaranteed
to exceed $LB3$ for large enough $N$. On the other hand $LB3$ exploits
the convexity of $\frac{1}{\left\Vert x\right\Vert }$ which our method
does not. Possibly this guarantees $LB3>LB1$ in certain circumstances.
We leave this as a conjecture.
\begin{conjecture}
The estimate for $S_{N}\frac{1}{\left\Vert x\right\Vert }$ in \eqref{eq:EstLBSymm}
can be improved beyond the current estimate for $S_{b_{n}q_{n}}\frac{1}{\left\Vert x\right\Vert }$
to the point where it always exceeds $LB3$ for large enough $N$.
\end{conjecture}

\subsection{The sum of the anti-symmetrisation of reciprocals $S_{N}(\theta-\overline{\theta})=\sum_{r=1}^{N}\frac{1}{\{r\alpha\}}-\frac{1}{1-\{r\alpha\}}$ }

Analysis
Anti-symmetric functions have been less thoroughly investigated than
the lower bounded functions studied above. As far as we are aware,
published results to date have not included estimates of the asymptotic
constants. We will also restrict ourselves to Landau notation, and
focus instead on showing how we can obtain such results very quickly.
Note that for anti-symmetric functions, we can expect lower bounds
to be similar to the negative of the upper bounds, so we will consider
only the upper bound here.

Recall that if $g$ is a positive function, $f=O(g)$ means $\left|f\right|<Cg$
for some constant $C$, and in particular $f$ may take negative values. 

The following result is trivial but useful:
\begin{rem}
\label{rem:RationalDiff}Recall $S_{N}\phi\coloneqq\sum_{1}^{N}\phi(r\alpha)$
and so depends only on the values of $\phi$ along the orbit $(r\alpha)$.
Hence if $\phi,\psi$ agree at the (countable) set of orbital points
we have $S_{N}\phi=S_{N}\psi$. In particular if $\alpha$ is irrational,
and $\phi,\psi$ disagree only at some set of rational points, then
$S_{N}\phi=S_{N}\psi$. 
\end{rem}

\begin{lem}
\label{lem:BrIsO(N)}For $\alpha$ of constant type, let $\phi$ be
a monotonic decreasing function satisfying $\phi(x)=O\left(\frac{1}{\left\Vert x\right\Vert }\right)$
on $(0,1)$, then $B_{r}\phi=b_{r}P_{q_{r}}+O(q_{r})$ and the same
result holds for $\overline{B}_{r}\phi$
\end{lem}

\begin{proof}
Recall the definition of $B_{r}\phi$: 
\begin{equation}
B_{r}\phi=\left\llbracket b_{r}>0\right\rrbracket \left(\left\llbracket q_{r}>1\right\rrbracket b_{r}\left(P_{q_{r}}-\overline{\phi}(\frac{1}{q_{r}})\right)+\left\llbracket 2<q_{r}<q_{n}\right\rrbracket E_{r}\left(\phi_{r0,q_{r}-q_{r-1}}-\overline{\phi}\left(\frac{2}{q_{r}}\right)\right)+\sum_{s=0}^{b_{r}-1}\left(\phi_{rsq_{r}}+\left\llbracket q_{r}>1\right\rrbracket \phi_{rsq_{r-1}}\right)\right)\label{eq:Br}
\end{equation}
We will need the following results:

Since $\text{\ensuremath{\phi(x)=O(\frac{1}{\left\Vert x\right\Vert })}}$
we have for $q_{r}>1$, $\overline{\phi}(\frac{1}{q_{r}})=\phi(1-\frac{1}{q_{r}})=O(q_{r})$,
and similarly for $q_{r}>2$, $\overline{\phi}\left(\frac{2}{q_{r}}\right)=O(q_{r})$.

Since $\alpha$ is of constant type, $q_{r+2}^{\slash}=O(q_{r})$.
But $\left\Vert \alpha_{rst}\right\Vert >\frac{1}{q_{r+2}^{\slash}}$
and so $\phi_{rst}=\phi(\alpha_{rst})=O(q_{r})$.

Also since $\alpha$ is of constant type $b_{r}\le a_{r+1}=O(1)$
and so $\sum_{s=0}^{b_{r}-1}\phi_{rst}=O(q_{r})$

Combining these results in \eqref{eq:Br} gives $B_{r}\phi=b_{r}P_{q_{r}}+O(q_{r})$.
A similar argument gives also $\overline{B}_{r}\phi=b_{r}P_{q_{r}}+O(q_{r})$
\end{proof}
Note that a monotonic decreasing anti-symmetric real function on $(0,1)$
must satisfy $\phi(\frac{1}{2})=0$ and hence $\phi(x)\ge0$ on $(0,\frac{1}{2}]$.
Also since in this case $\phi(x)=-\phi(1-x)$, the constraint $\phi(x)=O\left(\frac{1}{\left\Vert x\right\Vert }\right)$
on $(0,1)$ is equivalent to $\phi(x)=O\left(\frac{1}{x}\right)$on
$(0,\frac{1}{2})$.
\begin{thm}
\label{thm:Antisymm}For $\alpha$ of constant type, let $\phi$ be
a monotonic decreasing anti-symmetric function satisfying $\phi(x)=O\left(\frac{1}{x}\right)$on
$(0,\frac{1}{2})$, then $S_{N}\phi=O(N)$ 
\end{thm}

\begin{proof}
From Lemma \ref{lem:BaseInequalities} for any decreasing $\phi$
we have $S_{N}\phi\le\sum_{r=0}^{n}B_{r}(\phi)$.

Since $\phi$ is anti-symmetric we have $P_{q_{r}}=\sum_{t=1}^{q_{r}-1}\phi(\frac{t}{q_{r}})=0$.
Hence in this case using Lemma \ref{lem:BrIsO(N)} gives us $B_{r}\phi=O(q_{r})$. 

Now $\sum_{r=0}^{n}q_{r}<2(q_{n}+q_{n-1})<4q_{n}$ and hence $S_{N}\phi\le\sum_{r=0}^{n}B_{r}\phi=O(\sum_{r=0}^{n}q_{r})=O(q_{n})=O(N)$. 

A similar argument gives $S_{N}\phi\ge\sum_{r=0}^{n}\overline{B}_{r}\phi=O(N)$
and so $S_{N}\phi=O(N)$.
\end{proof}
 
\begin{cor}
\label{cor:asymm}For $\alpha$ of constant type, then $S_{N}\phi=O(N)$
for each of $\phi_{1}(x)=\frac{1}{x}-\frac{1}{1-x},\phi_{2}(x)=\cot\pi x,\phi_{3}(x)=\frac{1}{\{\{x\}\}}$
\end{cor}

\begin{proof}
Note that each of $\phi_{1},\phi_{2}$ is natively anti-symmetric,
and so is the derived function $\phi_{3}^{\slash}(x)=\left\llbracket x\ne\frac{1}{2}\right\rrbracket \phi_{3}(x)$.
All 3 are monotonic decreasing. We also have:
\begin{enumerate}
\item $\phi_{1}(x)=\frac{1}{x}-\frac{1}{1-x}<\frac{1}{x}$ on $(0,\frac{1}{2})$ 
\item $\pi\cot\pi x<\frac{1}{x}$ on $(0,\frac{1}{2})$ 
\item $\frac{\left\llbracket x\ne\frac{1}{2}\right\rrbracket }{\{\{x\}\}}=\frac{1}{x}$
on $(0,\frac{1}{2})$ 
\end{enumerate}

Hence Theorem \ref{thm:Antisymm} applies, and the result is proven
for $\phi_{1},\phi_{2},\phi_{3}^{\slash}$. Now $\phi_{3}^{\slash},\phi_{3}$
differ only at the rational point $x=\frac{1}{2}$ , so by Remark
\ref{rem:RationalDiff} $S_{N}\phi_{3}=S_{N}\phi_{3}^{\slash}=O(N)$.
\end{proof}

\subsection{The series $\sum_{r=1}^{N}\frac{1}{r^{\gamma}}\phi(r\alpha)$}

Analysis
Note that this series not itself a Birkhoff sum (since $r^{\gamma}$
is not a function of the value of $r\alpha$), but we can estimate
it using Birkhoff sums.
\begin{thm}
\label{thm:partial}Let $\phi$ be any function with $S_{N}\phi=O\left(N^{\beta}\right)$
for some $\beta\ge1$, then for any $\gamma_{\mathbb{R}}$the series
$S=\sum_{r=1}^{N}\frac{1}{r^{\gamma}}\phi(r\alpha)$ is $O(1)$ for
$\gamma>\beta$, is $O(\log N$) for $\gamma=\beta$, and is $O(N^{\beta-\gamma})$
for $\gamma<\beta$. 
\end{thm}

\begin{proof}
By partial summation $\sum_{r=1}^{N}\frac{1}{r^{\gamma}}\phi(r\alpha)=\sum_{r=1}^{N-1}\left(S_{r}\phi\right)\left(\frac{1}{r^{\gamma}}-\frac{1}{(1+r)^{\gamma}}\right)+\left(S_{N}\phi\right)\left(\frac{1}{N^{\gamma}}\right)$.
Note the final term is $O(N^{\beta-\gamma})$. Now $\left(\frac{1}{r^{\gamma}}-\frac{1}{(1+r)^{\gamma}}\right)=r^{-\gamma}\left(1-\left(1+r\right)^{-\gamma}\right)=\gamma r^{-(\gamma+1)}+O(r^{-(\gamma+2)})$.
Hence $\sum_{r=1}^{N-1}\left(S_{r}\phi\right)\left(\frac{1}{r^{\gamma}}-\frac{1}{(1+r)^{\gamma}}\right)=O\left(\sum_{r=1}^{N-1}r^{\beta}r^{-(\gamma+1)}\right)$.
Now $\sum_{r=1}^{N-1}r^{-(\gamma+1-\beta)}$converges for $\gamma+1-\beta>1$
and is $O(\log N)$ for $\gamma+1-\beta=1$ and $O(N^{\beta-\gamma})$
for $\gamma+1-\beta<1$ as required.
\end{proof}
\begin{cor}
For each of the functions $\phi(x)=\frac{1}{x}-\frac{1}{1-x},\phi(x)=\cot\pi x,\phi(x)=\frac{1}{\{\{x\}\}}$
we have $\sum_{r=1}^{N}\frac{1}{r}\phi(r\alpha)=O(\log N)$
\end{cor}

\begin{proof}
This is a direct consequence of Corollary \ref{cor:asymm} and Theorem
\ref{thm:partial} with $\gamma=\beta=1$
\end{proof}

\subsection{\label{subsec:The-double-exponential}The double exponential sum
$Exp2(N)=\sum_{u=0}^{N-1}\sum_{v=0}^{N-1}e^{2\pi iuv\alpha}$}

For convenience we will write $e(x)\coloneqq e^{2\pi ix}$, noting
that $\left|1-e(x)\right|\le2$.

Analysis
This sum was studied by Sinai \& Ulcigrai \cite{Sinai2009} (in connection
with a problem in quantum computing) who showed that for $\alpha$
of bounded type and $N=q_{n},$the sum is $O(q_{n})$. Although the
double sum is not itself a Birkhoff sum, several Birkhoff sums are
involved in estimating it. 

The inner single sum over $v$ \emph{is} a Birkhoff sum, although
this fact is unimportant as we can simply sum it as a geometric progression
giving $\sum_{v=0}^{N-1}e(uv\alpha)=\left\llbracket u=0\right\rrbracket N+\left\llbracket u\ne0\right\rrbracket \frac{1-e(uN\alpha)}{1-e((u\alpha)}$. 

Hence  $Exp2(N)=N+\sum_{u=1}^{N-1}\frac{1-e(uN\alpha)}{1-e((u\alpha)}$
and the second term we can now write as another Birkhoff sum $S_{N-1}f$
where $f(x)=\frac{1-e(Nx)}{1-e((x)}$. We are left with showing $S_{N-1}f=O(N)$
when $N=q_{n}$. 

Now $\frac{1}{1-e(x)}=\frac{1-e(-x)}{2-\left(e((x)+e(-x)\right)}=\frac{(1-\cos2\pi x)+i\sin2\pi x}{2(1-\cos2\pi x)}=\frac{1}{2}\left(1+i\cot\pi x\right)$
and hence $f(x)=\frac{1}{2}\left(1-e(Nx)\right)+\frac{i}{2}\cot\pi x-\frac{i}{2}\cot\pi x\,e(Nx)$.
Denoting the 3 terms of $f$ as $T_{1},T_{2},T_{3}$ we have $\left|T_{1}(x)\right|=\frac{1}{2}\left|1-e(Nx)\right|\le1$
so that $S_{N-1}T_{1}=O(N)$, and also by Corollary \ref{cor:asymm}
$S_{N-1}\cot\pi x=O(N)$. Hence $Exp2(N)=O(N)+S_{N-1}T_{3}$. We are
left with showing $S_{N-1}T_{3}=O(N)$ for $N=q_{n}$. In fact this
is the central result proved by Sinai \& Ulcigrai, and it uses substantial
machinery. More precisely they reduce $S_{N-1}T_{3}$ to a sum involving
$\{\{x\}\}$ instead of $\cot\pi x$ and then prove the result $S_{q_{n}-1}\left(\frac{e(q_{n}x)}{\{\{x\}\}}\right)=O(q_{n})$

. We will prove a more general result which provides both $Exp2(N)=O(N)$
and $S_{N-1}T_{3}=O(N)$ for $N=q_{n}$ as particular cases.

Recall that $\theta:X\rightarrow Y$ is a Lipschitz continuous map
between two metric spaces  if there is a $C\ge0$ with $d_{Y}(y_{1},y_{2})\le Cd_{X}(x_{1},x_{2})$
for any $x_{1},x_{2}\in X$. In particular when $X=\mathbb{T}$ we
will take the metric to be $\left\Vert .\right\Vert $, and note that
since $\mathbb{T}$ is compact, $\theta$ is bounded. 

\begin{lem}[Partial Birkhoff Summation]
Let $\psi,\phi$ be two complex valued observables on the circle
$\mathbb{T}$ such that $\psi$ is Lipschitz continuous with constant
$C$, and $\left|S_{N}\phi\right|\le BN$ for some constant $B$ (ie
$S_{N}\phi$ is $O(N)$). Given an irrational rotation number $\alpha$
we then have $\left|S_{N}(\psi_{M}\phi)\right|<BN\left(\frac{1}{2}C\left\Vert M\alpha\right\Vert (N-1)+\left|\psi\right|\right)$
where $\psi_{M}(x)\coloneqq\psi(\{Mx\})$

\end{lem}

\begin{proof}
Using partial summation we get $S_{N}\left(\psi_{M}\phi\right)=\sum_{r=1}^{N}\psi(rM\alpha)\phi(r\alpha)=\sum_{r=1}^{N-1}\left(S_{r}\phi\right)\left(\psi(Mr\alpha)-\psi(r+1)M\alpha\right)+\left(S_{N}\phi\right)\left(\psi(NM\alpha)\right)$.
Now the Lipschitz condition ensures that $\left|\psi\right|=\sup_{x\in\mathbb{T}}\left|\psi(x)\right|<\infty$
and that $\left|\psi(Mr\alpha)-\psi(r+1)M\alpha\right|<C\left\Vert rM\alpha-(r+1)M\alpha\right\Vert =C\left\Vert M\alpha\right\Vert $.
Also $\left|S_{r}\phi\right|\le Br$ so that $\left|\sum_{r=1}^{N-1}\left(S_{r}\phi\right)\left(\psi(Mr\alpha)-\psi(r+1)M\alpha\right)\right|\le\sum_{r=1}^{N-1}\left(Br\right)\left(C\left\Vert M\alpha\right\Vert \right)=\frac{1}{2}N(N-1)B\left(C\left\Vert M\alpha\right\Vert \right)$.
Finally $\left|\left(S_{N}\phi\right)\left(\psi(NM\alpha)\right)\right|\le BN\left|\psi\right|$
and the result follows.

\end{proof}
\begin{cor}
We have $\left|S_{N}(\psi_{M}\phi)\right|=O(N)$ whenever $\left\Vert M\alpha\right\Vert =O(\frac{1}{N})$
holds. Given $k_{0},c_{0}$, then when $N<q_{n+1}$ this holds if
$\left\Vert M\alpha\right\Vert <\frac{k_{0}}{q_{n+1}}$, and in particular
if $M=kq_{n}$ for $k\le k_{0}$.  
\end{cor}

\begin{proof}
This is a simple consequence of the lemma and $q_{n-c}>A^{-c}q_{n}$

We can now deduce the major result of Sinai \& Ulcigrai:
\end{proof}
\begin{cor}[Sinai \& Ulcigrai]
$S_{q_{n}-1}\left(\cot\pi x\,e(q_{n}x)\right)=O(q_{n})$ and $S_{q_{n}-1}\left(\frac{e(q_{n}x)}{\{\{x\}\}}\right)=O(q_{n})$
and $Exp2(q_{n})=\sum_{u=0}^{q_{n}-1}\sum_{v=0}^{q_{n}-1}e^{2\pi iuv\alpha}=O(q_{n})$
\end{cor}

\begin{proof}
$e(x)$ is Lipschitz continuous on the circle with constant $C=2$.
For $\alpha$ of constant type we also have from Corollary \ref{cor:asymm}
$S_{N}(\cot\pi x)=O(N)$, and $S_{N}\frac{1}{\{\{x\}\}}=O(N)$. The
results follow immediately. 
\end{proof}

Note that the key ingredient in the result above is that $M\alpha$
is sufficiently small, and this restricts the values of $M$. However
this does not necessarily mean that the growth rate of $S_{N}(\psi_{M}\phi)$
will be higher for other values of $M$. Indeed Sinai \& Ulcigrai
wonder if $Exp2(N)=O(N)$ for all $N$ (which is equivalent to $S_{N}(\psi_{N}\phi)=O(N)$)
, and more generally we wonder if 
\begin{conjecture}
Under the conditions of lemma 84, $S_{N}(\psi_{M}\phi)=O(N)$ for
any fixed $M$, and the constant is independent of $M$ 
\end{conjecture}

It seems unlikely that our methods will answer this question. It
seems dependent upon showing that the sign of $\psi_{M}\phi$ is well
distributed so as not to destroy the cancellative effect of a sum
across an anti-symmetric function.

\section{Conclusion}

We have developed a general method for estimating the series $\sum_{r=1}^{N}\phi(r\alpha)$
in anergodic cases where $\phi$ has an unbounded singularity at $0$.
Although many such series have been studied case by case in the literature,
the current paper introduces a single unified approach, which in many
cases also leads to improved results.

We outline a number of obvious directions for further study: 
\begin{enumerate}
\item Extension of the theory to more general circle morphisms which are
``sufficiently'' close to rotations to benefit from the theory developed
in the current paper
\item Extension of the theory to the inhomogenous case $\sum_{r=1}^{N}\phi(x_{0}+r\alpha)$
with $x_{0}\ne0$. It may be simpler to start with specific values
eg $x_{0}=\frac{1}{2}$, $x_{0}$ rational.
\item Extension to higher dimensions. Here the general strategy of separating
concerns remains valid, but there is an obstacle to be overcome: the
estimation of sums analogous to $S_{q_{r}}\phi$ is problematic in
the absence of an easy extension of Continued Fraction theory to higher
dimensions.
\end{enumerate}
\bibliographystyle{plain}
\bibliography{Anergodic}

\appendix

\section{\label{sec:Context-Annotation}Context Annotation}

This annotation has been developed whilst working on this paper to
be a more efficient way of indicating context whilst working in related
spaces with related morphisms. The annotation seems widely useful,
particularly in higher order categories (ie in situations where we
wish to discuss induced functions, or functions of functions). The
main benefit is that it provides the flexibility to introduce formal
precision when it is needed but without binding ourselves to heavy
notational machinery when it is not needed. Of course this places
more responsibility on the author to make good stylistic decisions!

Mathematical notation is extremely abbreviated - mathematicians tend
to use very short signifiers (usually individual symbols) to signify
mathematical objects. Since the pool of signifiers is severely limited
(``not enough letters'' syndrome), we end up ``abusing notation'',
ie we will use the same signifier to signify different mathematical
objects. The term ``abuse'' taken too literally is unfair - ``reuse
of notation'' is probably a better term: we are simply reusing scarce
notational resources: done well this is a good thing, but done badly
it can introduce confusion or even incomprehensibility. Taken not
too literally, the term ``abuse'' is useful as a warning that there
is a trade-off to be managed: the economy of notation comes at the
expense of ambiguity. Done well, the ambiguity is easily resolved
by the reader from the context, but sometimes it is useful or even
necessary to make the context explicit. 

There are two classical approaches to disambiguation which we compare
briefly. Given two groups $(X,+),(Y,+)$, how do we make the meaning
of $x+y$ clear? The more common approach is to write ``For $x,y\in X$,
$x+y$''. A less common approach is to write ``$x\,+_{X}\,y$ to
mean that $+$ is taken from $(X,+)$ and so the reader can deduce
that $x,y$ are required to lie in $X$. The latter approach is less
familiar but it is arguably more economical to write and read. It
is normally restricted for use with binary operations, but we will
here extend and formalise it for all mathematical objects.
\begin{defn}
If we wish to record the fact that the mathematical object signified
by $'x'$ is ``contextualised by'' another object signified by $'X'$
we may do so by annotating the first signifier with the second, for
example by writing $'x_{X}'$. We call $X$ a \noun{context} of $x$.
\end{defn}

We will deal with the semantics of ``contextualised by'' in a moment,
but first we need to do some unpacking of the definition.

It is crucial to note that this is an \emph{annotation}, not part
of the \emph{notation}: it is an optional extra which can be used
by the author stylistically, and does not need to be rigorously applied. 

For example we might write the anonymous function $x_{\mathbb{R}}\mapsto x^{2}$
rather than $x_{\mathbb{R}}\mapsto x_{\mathbb{\mathbb{\mathbb{\mathbb{R}}}}}^{2}$.
We may also be flexible in positioning the annotation: given an object
$x_{i}$ we could write it for example as $x_{i}^{X}$ or $(x_{i})_{X}$. 

In addition, whereas the \emph{notations} $'x','x_{i}'$ generally
signify different objects, $'x'$ and the annotation $'x_{X}'$ generally
signify the \emph{same} object. A key exception is when there is
abuse of notation, for example using $'x'$ to signify elements of
two disjoint sets $X_{1},X_{2}$. Now, writing $x_{X_{1}},x_{X_{2}}$
successfully disambiguates the situation, but if we choose $x=x_{X_{1}}$,
then we have chosen $x,x_{X_{2}}$ to be different objects.

The semantics of ``contextualised by $X$'' needs be defined for
any given context $X$, and $x$ generally will have many contexts
which could be used. However a sensible context is often obvious,
as is the semantics.

If $X$ is a set, the obvious default semantics is that $'x_{X}'$
means $x\in X$, but more specialised signifiers have a different
obvious default. For example, in the case of a classical function
symbol such as $'f'$, the meaning of $f_{X}$ defaults to mean $f\in X^{X}$,
ie $f$ is a function on $X$. Similarly $+_{X}$ will mean $+_{X}\in X^{X^{2}}$,
ie $+_{X}$ is a binary operation on $X$: and generally if $\circ$
is an $n_{\mathbb{N}}$-ary operation on $X$, ie $\circ\in X^{X^{n}}$,
we will still write $\circ_{X}$.

If $X$ is not a set, we are free to assign $x_{X}$ our meaning of
choice, but often it will be obvious. For example if $'Set'$, $'Grp'$
are our signifiers for the categories of sets, groups respectively,
then $G_{Grp}$ means informally that $G$ is a group, or more formally
it means $G$ satisfies the predicates of the proper class $ob(Grp)$
(and of course the class $ob(Grp)$ is itself an object in the the
category $Grp$). We may formally define $G_{Grp}\coloneqq(X_{Set},+_{X})$,
but by abuse of notation we might also write $G\coloneqq(G,+)$ when
we feel the meaning is clear. We can also define $x\in_{G}\,G_{Grp}$
to mean $x\in G_{Set}$ and then by abuse of notation write simply
$x\in G$. This gives us the freedom to introduce formal precision
when it is needed, without binding ourselves to heavy notation when
it is not needed. At the same time it places more responsibility on
the author to make good choices. 

Finally we introduce a useful abbreviation. Given objects $X$ and
$Y$ of a category, the morphisms from $X$ to $Y$ are normally designated
$Hom(X,Y)$, so a morphism $\phi:X\rightarrow Y$ could be annotated
$\phi_{Hom(X,Y)}$. We will abbreviate $Hom(X,Y)$ to $XY$ and write
$\phi_{XY}$ to mean $\phi:X\rightarrow Y$. Note that writing $\phi_{XY}x$
now tells us that we have $x_{X}$ and $(\phi_{XY}x)_{Y}$. We can
also express composition of morphisms by writing $\phi_{YZ}\phi_{XY}=\phi_{XZ}$
(which is of course even more expressive in reverse Polish notation,
giving us $\phi_{XY}\phi_{YZ}=\phi_{XZ}$).
\end{document}